\documentclass[smallextended,numbook,runningheads]{svjour3}     
\smartqed  

\usepackage{graphicx}
\usepackage{amsmath}
\usepackage{epstopdf} 
\usepackage{mathptmx}      
\usepackage[toc,page,title,titletoc,header]{appendix}

\usepackage{graphicx}
\usepackage{subfigure}
\usepackage[misc]{ifsym}
\usepackage{amssymb,amsfonts,mathrsfs,amsmath,multirow,mathtools,color}
\usepackage[hidelinks=false,linkcolor=blue,citecolor=blue,urlcolor=blue,colorlinks=true]{hyperref}
\usepackage{hypernat}
\usepackage{cite,hyperref,caption}
\usepackage{algorithm}
\usepackage{algorithmic}
\usepackage{float}
\allowdisplaybreaks[4]

\newcommand{\be}{\begin{equation}}
\newcommand{\ee}{\end{equation}}
\newcommand\bea{\begin{eqnarray}}
\newcommand\eea{\end{eqnarray}}
\newcommand{\bean}{\begin{eqnarray*}}
\newcommand{\eean}{\end{eqnarray*}}

\usepackage{cases}
\newcommand\bcase{\begin{numcases}{}}
\newcommand\ecase{\end{numcases}}

\hypersetup{
colorlinks = true,
linkcolor = black
}

\journalname{}
\begin{document}

\title{Error estimates of invariant-preserving difference schemes for the rotation-two-component Camassa--Holm system with small energy
\thanks{Qifeng Zhang was supported by
Zhejiang Provincial Natural Science Foundation of China (Grant No. LZ23A010007).
Zhimin Zhang was supported by the National Natural Science Foundation of China (Grant Nos. 12131005 and 11871092).
}}


\author{Qifeng Zhang \and Jiyuan Zhang \and Zhimin Zhang
}


\institute{Qifeng Zhang \at Department of Mathematics, Zhejiang Sci-Tech University, Hangzhou, 310018, China \\ \email{zhangqifeng0504@gmail.com}
\and Jiyuan Zhang \at Department of Mathematics, Zhejiang Sci-Tech University, Hangzhou, 310018, China\\
               \email{z018283@126.com}
\and Zhimin Zhang \at Department of Mathematics, Wayne State University, Detroit, Michigan 48202, USA\\
               \email{ag7761@wayne.edu}
}

\date{Received: date / Accepted: date}

\maketitle

\begin{abstract}
A rotation-two-component Camassa-Holm (R2CH) system was proposed recently to describe the motion of shallow water waves under the influence of gravity. This is a highly nonlinear and strongly coupled system of partial differential equations. A crucial issue in designing numerical schemes is to preserve invariants as many as possible at the discrete level. In this paper, we present a provable implicit nonlinear difference scheme which preserves at least three discrete conservation invariants: energy, mass, and momentum, and prove the existence of the difference solution via the Browder theorem. The error analysis is based on novel and refined estimates of the bilinear operator in the difference scheme. By skillfully using the energy method, we prove that the difference scheme not only converges unconditionally when the rotational parameter diminishes, but also converges without any step-ratio restriction for the small energy case when the rotational parameter is nonzero. The convergence orders in both settings (zero or nonzero rotation parameter) are $O(\tau^2 + h^2)$ for the velocity in the $L^\infty$-norm and the surface elevation in the $L^2$-norm, where $\tau$ denotes the temporal stepsize and $h$ the spatial stepsize, respectively. The theoretical predictions are confirmed by a properly designed two-level iteration scheme. Comparing with existing numerical methods in the literature, the proposed method demonstrates its effectiveness for long-time simulation over larger domains and superior resolution for both smooth and non-smooth initial values.

\end{abstract}

\section{Introduction}\label{Sec:1}
In this paper we propose, analyze and test a two-level invariant-preserving nonlinear difference
scheme for solving a rotation-two-component Camassa-Holm (R2CH) system of the form
\begin{subequations}
\label{equa1.1}
\begin{numcases}{}
u_t-u_{xxt}-\kappa u_x+3uu_x=\sigma(2u_x u_{xx}+u u_{xxx})-\mu u_{xxx}-(1-2\Omega \kappa)\rho\rho_x+2\Omega \rho (\rho u)_x, \nonumber \\
 \qquad\qquad\qquad\qquad\qquad\qquad\qquad  x\in \mathbb{R},\;  t\in[0, T], \label{eq1.1}\\
\rho_t + (\rho u)_x =0, \quad x\in \mathbb{R},\; t\in[0, T], \label{eq1.2}
\end{numcases}
\end{subequations}
subject to the initial value conditions
\begin{align}
u(x, 0) = u^{0}(x),\quad \rho(x, 0) = \rho^{0}(x), \quad x\in \mathbb{R},\label{eq1.3}
\end{align}
and periodic boundary value conditions
\begin{align}
u(x,t)=u(x+L,t), \quad \rho(x,t)=\rho(x+L,t),\quad x\in \mathbb{R},\; t\in[0, T].   \label{eq1.4}
\end{align}
The system \eqref{equa1.1}--\eqref{eq1.4} was introduced by Fan, Gao, and Liu \cite{FGL2016} to depict the motion of shallow water waves at a free surface involving the Coriolis force caused by the Earth's rotation in the equatorial ocean regions. The variable $u(x,t)$ in \eqref{equa1.1} denotes the horizontal fluid velocity along with the $x$-direction, and the variable $\rho(x,t)$ is the altitude from the free surface elevation to equilibrium. The coefficient $\kappa$ denotes a underlying linear shear flow, and the parameter $\sigma>0$ provides the competition/balance index. $\mu$ is a real dimensionless constant, and $\Omega$ defines the average rotational angular velocity of the Earth. Throughout the paper, we always assume $\Omega\in [0,1/4)$, and $1-2\Omega \kappa>0$, see e.g., \cite{CFGL2017}.

It can be easily shown that at least three invariants for the R2CH system \eqref{equa1.1} can be expressed: (see e.g., \cite{FGL2016})
\begin{itemize}
  \item Energy:
  \begin{align}
    E(u,\rho) = \frac{1}{2} \int_{\mathbb{R}}\big[u^2+u_x^2+(1-2\Omega \kappa)(\rho-1)^2\big]{\rm d}x. \label{eq1.2a}
  \end{align}
  \item Momentum:
  \begin{align}
    H(u,\rho) = \int_{\mathbb{R}} \big[u+\Omega(\rho-1)^2\big]{\rm d}x.\label{eq1.2b}
  \end{align}
  \item Mass:
  \begin{align}
    I(u,\rho) = \int_{\mathbb{R}} (\rho-1){\rm d}x. \label{eq1.2c}
  \end{align}
\end{itemize}

As of now, there have been intense theoretical studies on the R2CH system \eqref{equa1.1} such as the existence of global solutions \cite{Zhang2017,Moon2017}, the wave-breaking of solitary waves \cite{CFGL2017,LPZ2019}, the local well-posedness \cite{ZL2018}, the peakon-delta weak solutions \cite{FY2019}, the persistence properties \cite{YLQ2020}, the traveling waves \cite{ZW2021}, and so forth.
Once the Earth's rotation parameter equals zero, the system \eqref{equa1.1} degenerates to a two-component Camassa--Holm (2CH) system, which was introduced in \cite{OR1996} for the first time, and was later re-derived in \cite{CI2008}. The extensive analyses with respect to the theoretical solutions are referred to \cite{CL2011,GL2011,GY2011,Henry2009}.
Furthermore, when the variable $\rho(x,t)$ disappears, \eqref{equa1.1} reduces to the classical CH equation, which was first proposed as a bi-Hamiltonian system by Fuchssteiner et al. in \cite{FF1981}. Subsequently, it
was reformulated as a physical model to simulate the unidirectional propagation
of shallow water waves by Camassa and Holm in \cite{CH1993}. Its theoretical studies have been very extensive, see e.g., \cite{BC2007,Dan2003,HX2008}. In addition, some researchers considered more general three-component, four-component, and multi-component CH systems, see e.g., \cite{GX2011,KLQ2021,LLP2014,HI2010}. Several significant research timelines are summarized in Figure \ref{CH_Evolution}.
\begin{figure}[htbp]
  \centering
  \includegraphics[width=9cm]{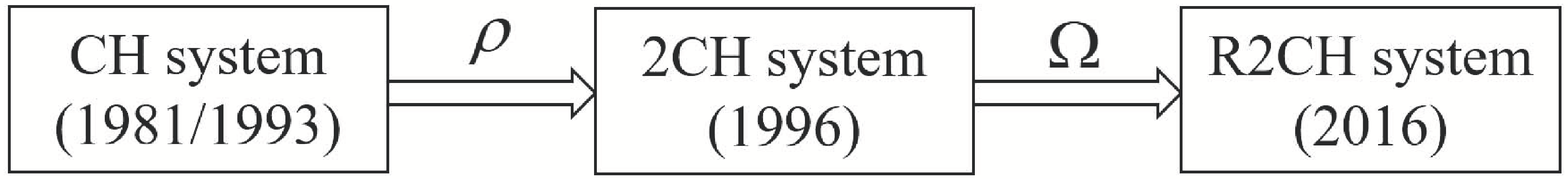}
  \caption{Milestones of the research timeline}\label{CH_Evolution}
   \vspace{-3mm}
\end{figure}

In contrast with extensive and in-depth theoretical studies, most of the present numerical studies are limited to the special cases of the R2CH system. For example, with regard to the classical CH equation, there are plentiful papers devoted
to investigating the numerical behavior of weak solutions and peakons, see e.g., \cite{HR2006a,HR2006b,CKR2008,COR2008,CR2012,XS2008,LX2016}.
Other important advances include the particle method \cite{CL2008}, multi-symplectic/multi-symplectic wavelet collocation schemes \cite{DBX2008,ZST2011}, the spectral/spectral element/spectral projection methods \cite{WX2015,YYW2021,KR2006}, the Galerkin finite element method \cite{ADM2019}, and the multiple scalar auxiliary variables/the invariant energy quadratization methods \cite{JGCW2020,JWG2021}. In the meantime, for the 2CH system, Cohen et al. investigated a multi-symplectic structure-preserving scheme \cite{CMR2014}, Liu et al. developed an invariant-preserving difference scheme \cite{LP2016}, Li et al. proposed a multi-symplectic compact method \cite{LQZS2017}, Yu et al. presented a three-step iterative algorithm \cite{YFS2018}, Chertock et al. applied a finite volume particle method \cite{CKL2020}, and Galtung et al. studied a variational discretization in Lagrangian variables \cite{GG2021}. Almost for sure, there is in no way an exhaustive list, but these results indicate the trend and efforts in the development of invariant-preserving numerical methods for the CH equation and 2CH system.

So far as is known to us, the numerical methods for the R2CH system are still lacking. The first numerical attempt for the R2CH system is attributed to Zhang et al. \cite{ZLZ2022}, who designed a conditionally convergent, and energy/mass-preserving finite difference scheme. However, there are several drawbacks for the scheme including: (i) It is difficult to determine whether the difference scheme preserves the momentum; (ii) When $\Omega=0$, a convergent condition $\tau\lesssim h$ in the error estimate is indispensable;
(iii) When $\Omega\neq0$, the error estimate is missing;
(iv) Numerical simulation will distort over a long-time simulation on large domain unless utilizing a more refined grid.

The original motivation in the present paper is to construct numerical schemes, which are able to fully preserve invariants \eqref{eq1.2a}--\eqref{eq1.2c} in the R2CH system \eqref{equa1.1}. Other motivations include overcoming the difficulties listed above and remedying deficiencies during the numerical analysis. Nevertheless, the R2CH system is very complicated, and developing robust and accurate numerical schemes is a highly nontrivial task.

To achieve these goals, we have already learned that the invariant-preserving property is an important index to judge the merits of an algorithm, see e.g., \cite{FQ1991,HPR2022,LV1995}. Therefore, we first ensure that the constructed difference scheme preserves as many invariants as possible. In addition, the fully implicit nonlinear numerical schemes usually have a better invariant-preserving ability and long-time numerical simulation capability. As a consequence, we start to derive the implicit time discretization, and then establish a class of fully implicit nonlinear numerical schemes for the R2CH system. Unfortunately, nonlinear difference schemes are usually computationally expensive in solving nonlinear system. As a possible compromise to balance computational effort and to preserve all invariants, we propose a well-designed two-level linearized iterative method to solve the nonlinear difference scheme of the R2CH system, see \eqref{Ite} in Section \ref{section5}.

Additionally, numerical analysis for most of the shallow water waves is often intractable due to complex high-order nonlinear derivative terms and coupled nonlinear terms. As a result, much previous research work focuses on the numerical schemes and numerical simulation, but ignore rigorous convergence analysis, see e.g., \cite{COR2008,CKL2020,GG2021,LP2016,LX2016}. Moreover, convergence analysis for nonlinear difference schemes is usually tougher compared with linear ones since their error systems would encounter more unknown information. To overcome these difficulties appeared in the nonlinear scheme for the R2CH system, by means of a more refined observation on the bilinear operators, and technical energy analysis, we find that some troublesome terms in the error system can be coincidentally cancelled out by taking a new inner product. At the same time, we find a new recursion relation, which could skillfully deal with the discrete time difference quotient. These, together with the nice properties of the bilinear operators and several key techniques, make the desired convergence possible.

The main contribution of the present work is to analyze a coupled nonlinear difference scheme for solving the R2CH system, and  dedicated to establishing a complete convergence theory. In particular:
\begin{itemize}
  \item The nonlinear difference scheme not only preserves the energy and the mass of the original problem naturally, but also preserves the momentum explicitly;
  \item When $\Omega=0$, the nonlinear difference scheme converges unconditionally (without time step restriction), which allows large time steps in calculation;
  \item When $\Omega\neq0$, it is worth noting that the nonlinear difference scheme converges unconditionally provided the initial energy is suitably small (or the initial values are small). However, it does not need a time step restriction and any boundedness hypothesis of numerical solutions;
  \item The difference scheme has better resolution in long-time numerical simulation on a large area even with a coarse grid comparing with that in \cite{CKL2020,CMR2014,LP2016,YFS2018,ZLZ2022} under the smooth/nonsmooth initial values.
  \item {The analyzing techniques based on the framework of the bilinear operator open the door of the difference methods to solve other types of the shallow water problems.}
\end{itemize}

We begin in Section \ref{section2} by introducing several useful notations and lemmas. In Section \ref{section3}, we derive a fully implicit nonlinear difference scheme for the R2CH system followed by detailed analyses for the invariant-preserving properties and a rigorous proof for the existence of a numerical solution. Section \ref{section4} is the main body of the paper, which focuses on the convergence analysis involving two scenarios: $\Omega=0$ and $\Omega\neq 0$. In Section \ref{section5}, we carry out two specific numerical examples including several different cases to test the theoretical results by designing an iterated scheme.
We end the paper by comparing the pros and cons of the present difference scheme with that in the literature in Section \ref{section6}. Some conclusions and outlooks are summarized in the last Section.

\section{Notations and lemmas}\label{section2}
\setcounter {equation}{0}
To define our finite difference method, we introduce some notations.
Once the positive integers $M$ and $N$ are selected, we let $h=L/M$ and $\tau=T/N$. Denote
 $x_i=ih,\;  i\in \mathbb{Z}$,
 $t_n=n\tau,\; 0\leqslant n\leqslant N$,
 $\Omega_{h}=\{x_i\,|\, x_i=ih,\;i \in \mathbb{Z}\}$,
 $\Omega_{\tau}=\{t_n\,|\,t_n=n\tau,\; 0\leqslant n\leqslant N\}$, $\Omega_{h\tau}=\Omega_{h}\times\Omega_{\tau}$.
For arbitrary grid functions $v=\{v_{i}^n\,|\,i\in \mathbb{Z},\; 0\leqslant n\leqslant N\}$, and $u=\{u_{i}^n\,|\,i\in \mathbb{Z},\; 0\leqslant n\leqslant N\}$ defined on $\Omega_{h\tau}$, we bring in the notations below:
\begin{align*}
& v_{i}^{n+\frac12}=\frac12(v_{i}^{n}+v_{i}^{n+1}),\quad
  \delta_tv_{i}^{n+\frac12}=\frac1\tau(v_{i}^{n+1}-v_{i}^n),\quad
   \delta_xv_{i-\frac12}^{n}=\frac1{h}(v_{i}^{n}-v_{i-1}^{n}),\\
& \delta_x^2v_{i}^n=\frac1{h}(\delta_xv_{i+\frac12}-\delta_xv_{i-\frac12}),\quad
  \Delta_xv_{i}^{n}=\frac1{2h}(v_{i+1}^{n}-v_{i-1}^{n}),\quad
  (uv)_{i}^{n+\frac12}=\frac{u_i^{n}v_i^{n}+u_i^{n+1}v_i^{n+1}}{2}.
\end{align*}
The grid function spaces in space and time are, respectively, denoted by
$$\mathcal{V}_h=\big\{v\,|\,v=\{v_{i}\},v_{i+M}=v_{i},\; i\in \mathbb{Z}\big\},$$
and
$$\mathcal{S}_{\tau}=\big\{w\,|\,w=(w^0,w^1,\cdots,w^N)~{\rm is~the~grid~function~defined~on}~\Omega_{\tau} \big\}.$$
For the integers $i$ and $n$, define the index sets $\mathbb{I}_M=\{i\,|\,1\leqslant i\leqslant M\}$, $\mathbb{K}_N=\{n\,|\,0\leqslant n\leqslant N\}$ and $\mathbb{K}_{N}^0=\{n\,|\,0\leqslant n\leqslant N-1\}$.
For arbitrary grid functions $u$, $v\in \mathcal{V}_h$, we define the discrete $L^2$-inner products by
\begin{align*}
  (u,v)=h\sum_{i\in \mathbb{I}_M}u_{i}v_{i},\quad
\langle \delta_xu,\delta_xv\rangle =h\sum_{i\in \mathbb{I}_M}(\delta_xu_{{i-\frac{1}{2}}})(\delta_x v_{{i-\frac{1}{2}}}),
\end{align*}
and the corresponding norms (seminorm) by
\begin{align*}
 \|v\|=\sqrt{(v,v)}, \quad |v|_1=\sqrt{\langle \delta_xv,\delta_xv\rangle}, \quad \|v\|_{\infty}=\max_{i\in \mathbb{I}_M}\limits|v_i|.
\end{align*}
Furthermore, denote $(uv)_{i}=u_{i}v_{i}$ and define a bilinear operator (see e.g., \cite{Guo1974,Guo1981}) to facilitate numerical analysis later as follows.
\begin{align}
\psi(u,v)_{i}=\frac13[u_{i}\Delta_xv_{i}+\Delta_x(uv)_{i}],\quad i\in \mathbb{I}_M. \label{psi_1}
\end{align}
\begin{lemma}{\rm \cite{ZL2020}}{\rm \label{lemma1}}
For two arbitrary spatial grid functions $u,\,v\in \mathcal{V}_h$, we have
\begin{align*}
&(\psi(u,v),v)=0, \quad (\Delta_xu,u)=0,\quad (\Delta_xu,v)=-(u,\Delta_xv),\quad(\delta_x^2u,v)=-\langle\delta_xu,\delta_xv\rangle.
\end{align*}
\end{lemma}

The discrete form of the classical embedded inequality is introduced as follows, see also \cite{ZL2020}. The continuous edition
is similar to the discrete one, which is omitted here for brevity.
\begin{lemma}\label{lemma2}
	For an arbitrary spatial grid function $v\in \mathcal{{V}}_h$ and $\varepsilon>0$, we have
	\begin{equation*}
		\|v\|_{\infty}^2 \leqslant \varepsilon|v|_1^2 + \Big(\frac{1}{\varepsilon}+\frac{1}{L}\Big)\|v\|^2.
	\end{equation*}
\end{lemma}
In order to simplify the numerical analysis of the temporal direction, we give the following lemma.
\begin{lemma}\label{lemma3}
  For two arbitrary temporal grid functions $u,\;v\in \mathcal{{S}}_{\tau}$, we have
\begin{align*}
&\;  (\delta_t u^{n+\frac12},u^{n+\frac12}v^{n+\frac12})\nonumber\\
=&\;  \frac{1}{2\tau}[(u^{n+1},u^{n+1}v^{n+1})-(u^{n},u^{n}v^{n})]\notag\\
&\; -\frac{1}{4}(u^{n+1}-u^n,(u^{n+1}-u^n)\delta_tv^{n+\frac12})-\frac12(u^{n+1}u^n,\delta_t v^{n+\frac12}).
\end{align*}
\end{lemma}
\begin{proof}
We easily have
\begin{align}
&\;\quad\;  (\delta_tu^{n+\frac12},u^{n+\frac12}v^{n+\frac12})\nonumber\\
&\;= (\delta_tu^{n+\frac12},u^{n+\frac12}v^{n+\frac12}-(uv)^{n+\frac12}+(uv)^{n+\frac12})\nonumber\\
&\;= (\delta_tu^{n+\frac12},u^{n+\frac12}v^{n+\frac12}-(uv)^{n+\frac12})+(\delta_tu^{n+\frac12}, (uv)^{n+\frac12})\nonumber\\
&\;=: A^n+B^n.\label{eq3.41cc}
\end{align}
We calculate each term in \eqref{eq3.41cc}.
\begin{align}
A^n = &\;  (\delta_tu^{n+\frac12},u^{n+\frac12}v^{n+\frac12}-(uv)^{n+\frac12})\nonumber\\
=&\; \frac14(\delta_tu^{n+\frac12},(u^n+u^{n+1})(v^{n+1}+v^n)-2(u^{n+1}v^{n+1}+u^{n}v^{n}))\nonumber\\
=&\; \frac14(\delta_tu^{n+\frac12},u^nv^{n+1}-u^nv^n+u^{n+1}v^n-u^{n+1}v^{n+1})\nonumber\\
=&\; \frac1{4\tau}(u^{n+1}-u^n,u^nv^{n+1}-u^nv^n+u^{n+1}v^n-u^{n+1}v^{n+1})\nonumber\\
=&\; \frac{1}{4}(u^{n+1}-u^n,u^n\delta_tv^{n+\frac12}-u^{n+1}\delta_tv^{n+\frac12})\nonumber\\
=&\; -\frac{1}{4}(u^{n+1}-u^n,(u^{n+1}-u^n)\delta_tv^{n+\frac12}),\label{eq3.41bb}
\end{align}
and
\begin{align}
B^n=&\;  (\delta_tu^{n+\frac12},(uv)^{n+\frac12})\nonumber\\
=&\; \frac{1}{2\tau}(u^{n+1}-u^n,u^{n+1}v^{n+1}+u^{n}v^{n})\nonumber\\
=&\; \frac{1}{2\tau}[(u^{n+1},u^{n+1}v^{n+1})-(u^{n},u^{n}v^{n})]+\frac{1}{2\tau}[(u^{n+1},u^{n}v^{n})-(u^{n},u^{n+1}v^{n+1})]\nonumber\\
=&\; \frac{1}{2\tau}[(u^{n+1},u^{n+1}v^{n+1})-(u^{n},u^{n}v^{n})]-\frac12(u^nu^{n+1},\delta_tv^{n+\frac12}).\label{eq3.41aa}
\end{align}
Plugging \eqref{eq3.41bb} and \eqref{eq3.41aa} into \eqref{eq3.41cc}, it completes the proof.
\end{proof}

\section{Numerical analysis}\label{section3}
\setcounter {equation}{0}
Assume the exact solutions to the problem \eqref{equa1.1}--\eqref{eq1.4} satisfy
$$u(x,t)\in \mathbb{C}^{5,3}(\mathbb{R}\times[0,T]), \quad \rho(x,t)\in \mathbb{C}^{3,3}(\mathbb{R}\times[0,T])$$
and denote
\begin{align}
& c_0 = \max_{0\leqslant x\leqslant L, 0\leqslant t\leqslant T} \{|u_x(x,t)|,\,|u_{xx}(x,t)|,\,|u_t(x,t)|,\,|\rho(x,t)|,\,|\rho_x(x,t)|\},\label{eq3.0a}\\
& c_{\rm max} = \max_{0\leqslant x\leqslant L, 0\leqslant t\leqslant T}|u(x,t)|.\label{eq3.0c}
\end{align}
\subsection{The derivation of the difference scheme}
Denote the exact solutions at the grid functions in short by
\begin{align*}
U_{i}^n=u(x_i,t_n),\; \Pi_{i}^n=\rho(x_i,t_n),
 \quad i\in\mathbb{I}_M,\; n\in \mathbb{K}_N.
\end{align*}
Considering \eqref{equa1.1} directly at the virtual point $(x_i,t_{n+\frac12})$, and combining Taylor expansion, we have
\begin{subequations}
\label{equa3.1}
\begin{numcases}{}
\delta_tU_i^{n+\frac12}-\delta_t\delta_x^2U_i^{n+\frac12}-\kappa \Delta_xU_i^{n+\frac12}+3\psi(U^{n+\frac12},U^{n+\frac12})_i\nonumber\\
   =3\sigma\psi(\delta_x^2U^{n+\frac12},U^{n+\frac12})_i-\mu\Delta_x\delta_x^2U_i^{n+\frac12}-(1-2\Omega \kappa)\Pi_i^{n+\frac12}\Delta_x\Pi_i^{n+\frac12}\nonumber\\
 \quad+2\Omega\Pi_i^{n+\frac12}\Delta_x(\Pi^{n+\frac12}U^{n+\frac12})_i+Q_i^{n+\frac12}, \quad i\in\mathbb{I}_M,\; n\in \mathbb{K}_N^0,\label{eq3.1}\\
  \delta_t\Pi_i^{n+\frac12}+\Delta_x(\Pi^{n+\frac12}U^{n+\frac12})_i=R_i^{n+\frac12}, \quad i\in\mathbb{I}_M,\; n\in \mathbb{K}_N^0,\label{eq3.2}
\end{numcases}
\end{subequations}
there is a constant $c_1$ independent of $\tau$ and $h$ such that
\begin{subequations}
\label{equa3.3}
\begin{numcases}{}
|Q_{i}^{n+\frac12}|\leqslant c_1(\tau^2+h^2), \quad i\in\mathbb{I}_M,\; n\in \mathbb{K}_N^0,\label{eq3.3}\\
|R_{i}^{n+\frac12}|\leqslant c_1(\tau^2+h^2), \quad i\in\mathbb{I}_M,\; n\in \mathbb{K}_N^0,\label{eq3.4}
\end{numcases}
\end{subequations}
Taking the initial and boundary value conditions \eqref{eq1.3}--\eqref{eq1.4} into account, we have
\begin{subequations}
\label{equa3.5}
\begin{numcases}{}
 U_{i}^0=u^0(x_i),\; \Pi_i^0=\rho^0(x_i),\quad i\in\mathbb{I}_M, \label{eq3.5}\\
 U_{i}^n=U_{i+M}^n,\; \Pi_{i}^n=\Pi_{i+M}^n, \quad i\in\mathbb{I}_M,\; n\in \mathbb{K}_N.\label{eq3.6}
\end{numcases}
\end{subequations}
Negelcting the local truncation errors \eqref{equa3.3}, replacing the functions $U_{i}^n$ and $\Pi_{i}^n$  with their numerical
approximations $u_{i}^n$ and $\rho_{i}^n$  in \eqref{equa3.1} and \eqref{equa3.5}, respectively,
a two-level coupled nonlinear difference scheme reads
\begin{subequations}
\label{equa3.7}
\begin{numcases}{}
 \delta_t u_i^{n+\frac12}-\delta_t\delta_x^2u_i^{n+\frac12}-\kappa \Delta_xu_i^{n+\frac12}+3\psi(u^{n+\frac12},u^{n+\frac12})_i\nonumber\\
   =3\sigma\psi(\delta_x^2u^{n+\frac12},u^{n+\frac12})_i-\mu\Delta_x\delta_x^2u_i^{n+\frac12} -(1-2\Omega \kappa )\rho_i^{n+\frac12}\Delta_x\rho_i^{n+\frac12}\nonumber\\
\quad+2\Omega\rho_i^{n+\frac12}\Delta_x(\rho^{n+\frac12}u^{n+\frac12})_i, \quad i\in\mathbb{I}_M,\; n\in \mathbb{K}_N^0,\label{eq3.7}\\
  \delta_t\rho_i^{n+\frac12}+\Delta_x(\rho^{n+\frac12}u^{n+\frac12})_i=0, \quad i\in\mathbb{I}_M,\; n\in \mathbb{K}_N^0.\label{eq3.8}\\
  u_{i}^0=u^0(x_i),\; \rho_i^0=\rho^0(x_i),\quad i\in\mathbb{I}_M, \label{eq3.9}\\
  u_{i}^n=u_{i+M}^n,\; \rho_{i}^n=\rho_{i+M}^n, \quad i\in\mathbb{I}_M,\; n\in \mathbb{K}_N.\label{eq3.10}
\end{numcases}
\end{subequations}
\begin{remark}
  Note that the scheme \eqref{equa3.7} is exactly the standard Crank-Nicolson scheme once we replace $3\psi(u^{n+\frac12},u^{n+\frac12})_i$ by $u_i^{n+\frac{1}{2}}\Delta_x u_i^{n+\frac12}$ and $3\sigma\psi(\delta_x^2u^{n+\frac12},u^{n+\frac12})_i$ by $2\Delta_x u_i^{n+\frac12}\delta_x^2 u_i^{n+\frac12}+ u_i^{n+\frac12}\Delta_x \delta_x^2 u_i^{n+\frac12}$, respectively. In this case, the numerical analysis is challengeable and deserves further study.
\end{remark}
\subsection{Conservative invariants and boundedness}
\begin{theorem} \rm\label{theorem1_1}
 Suppose $\{u_i^n,\,\rho_i^n\,|\,i\in\mathbb{I}_M,\; n\in \mathbb{K}_N\}$ is the numerical solution of the scheme \eqref{equa3.7}.
 Then for any $n\in \mathbb{K}_N$, we have the discrete invariants:
 \begin{itemize}
   \item \textbf{Energy}:
   \begin{align}
 E^n=E^0,\quad {\rm where}\quad  E^n=\frac12\big[\|u^n\|^2+|u^n|_1^2+(1-2\Omega \kappa)\|\rho^n\|^2\big];\label{D_energy}
  \end{align}
   \item \textbf{Momentum}:
   \begin{align}
 H^n=H^0,\quad {\rm where}\quad  H^n=(u^n,1)+\Omega \|\rho^{n}\|^2;\label{D_momentum}
  \end{align}
   \item \textbf{Mass}:
   \begin{align}
 I^n=I^0,\quad {\rm where}\quad  I^n=(\rho^n,1).\label{D_mass}
  \end{align}
 \end{itemize}
 \end{theorem}
\begin{proof}
 (\textbf{Energy}): Taking the $L^2$-inner product of \eqref{eq3.7} with $u^{n+\frac12}$, we readily have
 \begin{align*}
 &  (\delta_t u^{n+\frac12},u^{n+\frac12})-(\delta_t\delta_x^2u^{n+\frac12},u^{n+\frac12})-\kappa (\Delta_xu^{n+\frac12},u^{n+\frac12})
    +3(\psi(u^{n+\frac12},u^{n+\frac12}),u^{n+\frac12})\\
 =\; &
    3\sigma(\psi(\delta_x^2u^{n+\frac12},u^{n+\frac12}),u^{n+\frac12})-\mu(\Delta_x\delta_x^2u^{n+\frac12},u^{n+\frac12})
    -(1-2\Omega \kappa)\cdot(\rho^{n+\frac12}\Delta_x\rho^{n+\frac12},u^{n+\frac12})\\
&  +2\Omega(\rho^{n+\frac12}\Delta_x(\rho^{n+\frac12}u^{n+\frac12}),u^{n+\frac12}), \quad n\in \mathbb{K}_N^0.
 \end{align*}
 Using Lemma \ref{lemma1}, for $n\in \mathbb{K}_N^0$, the above equality deduces to
 \begin{align}
 \frac1{2\tau}(\|u^{n+1}\|^2-\|u^n\|^2)+\frac1{2\tau}(|u^{n+1}|_1^2-|u^n|_1^2)=(1-2\Omega \kappa)\cdot(\rho^{n+\frac12},\Delta_x(\rho^{n+\frac12}u^{n+\frac12})).\label{eq3.11}
 \end{align}
 Again taking the $L^2$-inner product of \eqref{eq3.8} with $\rho^{n+\frac12}$, we have
 \begin{align}
 (\delta_t\rho^{n+\frac12},\rho^{n+\frac12})+(\Delta_x(\rho^{n+\frac12}u^{n+\frac12}),\rho^{n+\frac12})=0, \quad  n\in \mathbb{K}_N^0.\label{eq3.12}
 \end{align}
 Multiplying \eqref{eq3.12} by $1-2\Omega \kappa$ and summing up with \eqref{eq3.11}, we have
 \begin{align*}
 \frac{1}{2\tau}\big[(\|u^{n+1}\|^2-\|u^n\|^2)+(|u^{n+1}|_1^2-|u^n|_1^2)+(1-2\Omega \kappa)\cdot(\|\rho^{n+1}\|^2-\|\rho^n\|^2)\big]=0,\quad n\in \mathbb{K}_N^0.
 \end{align*}
 Rearranging the equality above, we have
 \begin{align*}
 \frac12\big[\|u^{n+1}\|^2+|u^{n+1}|_1^2+(1-2\Omega \kappa)\cdot\|\rho^{n+1}\|^2\big]=\frac12\big[\|u^{n}\|^2+|u^{n}|_1^2+(1-2\Omega \kappa)\cdot\|\rho^{n}\|^2\big], \quad n\in \mathbb{K}_N^0,
 \end{align*}
 which implies \eqref{D_energy}.\\
 (\textbf{Momentum}): Taking the $L^2$-inner product of \eqref{eq3.7} with $1$, we have
 \begin{align*}
 &     (\delta_t u^{n+\frac12},1)-(\delta_t\delta_x^2u^{n+\frac12},1)-\kappa(\Delta_xu^{n+\frac12},1)+3(\psi(u^{n+\frac12},u^{n+\frac12}),1) \\
 =\;&
        3\sigma(\psi(\delta_x^2u^{n+\frac12},u^{n+\frac12}),1)-\mu(\Delta_x\delta_x^2u^{n+\frac12},1)
        -(1-2\Omega \kappa)\cdot(\rho^{n+\frac12}\Delta_x\rho^{n+\frac12},1)\\
&      +2\Omega(\rho^{n+\frac12}\Delta_x(\rho^{n+\frac12}u^{n+\frac12}),1), \quad n\in \mathbb{K}_N^0.
 \end{align*}
 Noticing Lemma \ref{lemma1}, the equality above becomes
 \begin{align}
 \frac{1}{\tau}\left[(u^{n+1},1)-(u^n,1)\right]-2\Omega(\rho^{n+\frac12}\Delta_x(\rho^{n+\frac12}u^{n+\frac12}),1)=0,\quad n\in \mathbb{K}_N^0.\label{eq3.15}
 \end{align}
 Taking an inner product of \eqref{eq3.8} with $2\Omega\rho^{n+\frac12}$, we have
 \begin{align}
 2\Omega(\delta_t\rho^{n+\frac12},\rho^{n+\frac12})+2\Omega(\Delta_x(\rho^{n+\frac12}u^{n+\frac12}),\rho^{n+\frac12})=0, \quad  n\in \mathbb{K}_N^0.\label{eq3.16}
 \end{align}
 Adding \eqref{eq3.15} and \eqref{eq3.16} together, we have
 \begin{align*}
  \frac{1}{\tau}\left[(u^{n+1},1)-(u^n,1)+\Omega (\|\rho^{n+1}\|^2-\|\rho^n\|^2)\right]=0,\quad  n\in \mathbb{K}_N^0.
 \end{align*}
 Rearranging the above equality, we have
 \begin{align*}
 (u^{n+1},1)+\Omega \|\rho^{n+1}\|^2=(u^n,1)+\Omega \|\rho^{n}\|^2, \quad  n\in \mathbb{K}_N^0.
 \end{align*}
 By reduction, it consequently yields \eqref{D_momentum}.\\
 (\textbf{Mass}): Taking an $L^2$-inner product of \eqref{eq3.8} directly with $1$, we have
 \begin{align}
 (\delta_t\rho^{n+\frac12},1)+(\Delta_x(\rho^{n+\frac12}u^{n+\frac12}),1)=0,\quad  n\in \mathbb{K}_N^0. \label{eq3.17}
 \end{align}
 By the summation by parts, the second term vanishes in \eqref{eq3.17}. Thus we have \eqref{D_mass},
which ends the proof.
\end{proof}
\begin{remark}
By Theorem \ref{theorem1_1}, we have
\begin{subequations}
\label{equa3.13}
\begin{numcases}{}
\|u^n\|^2+|u^n|_1^2\leqslant \|u^0\|^2+|u^0|_1^2+(1-2\Omega \kappa)\|\rho^0\|^2, \quad n\in \mathbb{K}_N,  \label{equa3.13a}\\
\|\rho^n\|^2\leqslant \frac{1}{1-2\Omega \kappa}\left(\|u^0\|^2+|u^0|_1^2+(1-2\Omega \kappa)\|\rho^0\|^2\right), \quad n\in \mathbb{K}_N. \label{equa3.13b}
\end{numcases}
\end{subequations}
In combination of Lemma \ref{lemma2} with \eqref{equa3.13a}, we further deduce that $\|u^n\|_\infty$ is bounded.
\end{remark}
\begin{remark}
   According to the discrete energy \eqref{D_energy}, we easily have
 $(\rho^n-1,1)=(\rho^0-1,1),\; n\in \mathbb{K}_N$,
 which can be viewed as a discrete edition of the continuous conservation of mass in \eqref{eq1.2c}.
 Furthermore, combining with \eqref{D_energy} and \eqref{D_mass},
 we know that the discrete energy satisfies
$\frac12\big[\|u^n\|^2+|u^n|_1^2+(1-2\Omega \kappa)\|\rho^n-1\|^2\big] =  \frac12\big[\|u^0\|^2+|u^0|_1^2+(1-2\Omega \kappa)\|\rho^n-1\|^2\big],\; n\in \mathbb{K}_N$,
 which recovers the energy conservation of the analytical form \eqref{eq1.2a}.
 Analogously, with the help of \eqref{D_momentum} and \eqref{D_mass}, the discrete momentum conservation becomes
 $(u^n,1)+\Omega \|\rho^{n}-1\|^2 = (u^n,1)+\Omega \|\rho^{n}-1\|^2,\; n\in \mathbb{K}_N$.
 \end{remark}

\subsection{Existence}
The following Browder theorem is very useful for the proof of the existence.
\begin{theorem}[Browder theorem {\rm \cite{AK1993}}]\label{Lemm3}
	Suppose $(H,(\cdot,\cdot))$ is an inner product space of finite dimension.
    $\|\cdot\|$ is the corresponding norm, $\Theta : H \rightarrow H$ is a continuous map. Further assume that
	$$\exists \,\alpha>0,\quad \forall \, z\in H,\quad \|z\|=\alpha,\quad (\Theta(z),z)\geqslant 0.$$
	Then there is an element $z^*\in H$ such that $\Theta(z^*)=0$ when $\|z^*\|\leqslant \alpha$.
\end{theorem}

\begin{theorem}[Existence]\label{theorem2.1}
	The difference scheme \eqref{equa3.7} has a solution.
\end{theorem}
\begin{proof}
Suppose $\{u^n,\,\rho^n\}$ has been determined.
Denote $u_i=u_i^{n+\frac12}$ and $\rho_i=\rho_i^{n+\frac12}$. The difference scheme \eqref{eq3.7}--\eqref{eq3.8} could be rewritten as
\begin{subequations}
\label{Exeq1}
\begin{numcases}{}
 \frac{2}{\tau}(u_i-u_i^n)-\frac{2}{\tau}(\delta_x^2 u_i - \delta_x^2 u_i^n)-\kappa\Delta_xu_i+3\psi(u,u)_i
   -3\sigma\psi(\delta_x^2u,u)_i\nonumber\\
  \quad +\mu\Delta_x\delta_x^2u_i+(1-2\Omega \kappa)\rho_i\Delta_x\rho_i-2\Omega\rho_i \Delta_x(\rho u)_i=0, \quad i\in\mathbb{I}_M,\label{Exeq1.1}\\
  \frac{2}{\tau}(\rho_i-\rho_i^n)+\Delta_x(\rho u)_i=0, \quad i\in\mathbb{I}_M.\label{Exeq1.2}
\end{numcases}
\end{subequations}

Suppose $u\in \mathcal{U}_h$, we define an operator $\Theta: \mathcal{U}_h \rightarrow \mathcal{U}_h$ by
\begin{align}
&  \Theta(u)_i = \frac{2}{\tau}(u_i-u_i^n) - \frac{2}{\tau}(\delta_x^2 u_i-\delta_x^2 u_i^n) -\kappa \Delta_xu_i +3\psi(u,u)_i
   -3\sigma\psi(\delta_x^2u,u)_i\nonumber\\
&  \qquad \qquad +\mu\Delta_x\delta_x^2u_i+(1-2\Omega \kappa)\rho_i\Delta_x\rho_i-2\Omega\rho_i \Delta_x(\rho u)_i, \quad i\in\mathbb{I}_M,\label{Exeq1.3}
\end{align}
where $\rho$ is determined by \eqref{Exeq1.2} when $u$ is known.

In what follows, we will show the existence of $u_i$. Taking a standard $L^2$-inner product of $\Theta(u)$ in \eqref{Exeq1.3} with $u$, it follows
\begin{align}
& (\Theta(u),u) = \frac{2}{\tau}(\|u\|^2-(u^n,u))+\frac{2}{\tau}(\|\delta_x u\|^2-(\delta_x u^n,\delta_x u))-\kappa (\delta_x u, u)+ 3(\psi(u,u),u) \notag\\ &\qquad \qquad \quad -3\sigma (\psi(\delta_x^2u,u),u) + \mu(\Delta_x \delta_x^2 u,u) + (1-2\Omega \kappa)(\rho\Delta_x\rho,u) -2\Omega(\rho(\rho u),u).\label{Exeq1.4}
\end{align}
With the aid of Lemma \ref{lemma1}, it yields
\begin{align*}
  (\delta_x u, u) = 0,\quad
  (\psi(u,u),u) = 0, \quad
  (\psi(\delta_x^2u,u),u) = 0,\quad
  (\Delta_x \delta_x^2 u,u) = 0.
\end{align*}
Using the summation by parts, we have
\begin{align*}
  (\rho \Delta_x \rho, u) & = (\Delta_x \rho, \rho u)\\
                          & = -(\rho,\Delta_x(\rho u))  \overset{\eqref{Exeq1.2}}{=} \Big(\rho, \frac{2}{\tau}(\rho-\rho^n)\Big)\\
                          & =\frac{2}{\tau}(\|\rho\|^2-(\rho,\rho^n))\\
                          &\geqslant \frac{2}{\tau} (-\frac{1}{4}\|\rho^n\|^2)
\end{align*}
and
\begin{align*}
  (\rho \Delta_x(\rho u),u) = (\Delta_x(\rho u),\rho u) = 0.
\end{align*}
In addition,
\begin{align*}
&  \|u\|^2-(u,u^n) + \|\delta_x u\|^2 - (\delta_x u, \delta_x u^n)\\
\geqslant & \|u\|^2 -\|u\|\cdot\|u^n\| + \|\delta_x u\|^2 - \|\delta_x u\|\cdot\|\delta_x u^n\|\\
\geqslant & \|u\|^2 -(\frac{1}{2}\|u\|^2 + \frac12\|u^n\|^2) + \|\delta_x u\|^2 - (\|\delta_x u\|^2+\frac14\|\delta_x u^n\|^2)\\
= & \frac12 \|u\|^2 -\frac12\|u^n\|^2 -\frac14\|\delta_x u^n\|^2,
\end{align*}
in which the Cauchy Schwarz inequality and $ab\leqslant \varepsilon a^2+\frac{1}{4\varepsilon}b^2$ are utilized.

Substituting the above formulas into \eqref{Exeq1.4}, we have
\begin{align*}
  (\Theta(u),u) &\geqslant \frac{2}{\tau}\Big(\frac12 \|u\|^2 -\frac12\|u^n\|^2 -\frac14\|\delta_x u^n\|^2\Big) + (1-2\Omega \kappa)\cdot\frac{2}{\tau} (-\frac{1}{4}\|\rho^n\|^2)\\
                & = \frac{1}{\tau}\Big(\|u\|^2 -\|u^n\|^2 -\frac12\|\delta_x u^n\|^2 -\frac{1-2\Omega \kappa}{2}\|\rho^n\|^2 \Big).
\end{align*}
When $\|u\| = \Big(\|u^n\|^2 +\frac12\|\delta_x u^n\|^2 +\frac{1-2\Omega \kappa}{2}\|\rho^n\|^2 \Big)^{1/2}$, we have $(\Theta(u),u)\geqslant 0$.
By the Browder theorem \ref{Lemm3}, there is a $u^\ast\in \mathcal{U}_h$ satisfying
$$\|u^\ast\| \leqslant \Big(\|u^n\|^2 + \frac12\|\delta_x u^n\|^2 +\frac{1-2\Omega \kappa}{2}\|\rho^n\|^2\Big)^{1/2}$$
such that $\Theta(u^\ast) = 0$.

This finishes the proof.

\end{proof}
\begin{remark}
  In fact, once $u$ has been determined, we can calculate $\rho$ by the homogeneous system
\begin{align*}
  \frac{2}{\tau}\rho_i+\Delta_x(\rho u)_i = 0,\quad i\in\mathbb{I}_M.
\end{align*}
Taking an inner product of the above equality with $\rho$, we have
\begin{align*}
  \|\rho\|^2 + \frac{\tau}{2}(\rho,\Delta_x(\rho u)) = 0.
\end{align*}
Rearranging the above inequality and by means of the Cauchy-Schwarz inequality, it follows
\begin{align*}
  \|\rho\|^2 = -\frac{\tau}{2}(\rho,\Delta_x(\rho u)) = \frac{\tau}{2}(\rho u, \Delta_x \rho)\leqslant \frac{\tau}{2}\|u\|_\infty\cdot\|\rho\|\cdot\|\Delta_x \rho\| \leqslant \frac{\tau}{2h} \|u\|_\infty \cdot\|\rho\|^2.
\end{align*}
Therefore, there is a unique solution $\rho$ to the difference scheme \eqref{Exeq1.2} when $\frac{\tau}{2h}\|u\|_\infty<1$.
\end{remark}

\section{Convergence} \label{section4}
\subsection{Convergence for the difference scheme with zero rotation parameter ($\Omega = 0$)}
In this case, the problem \eqref{equa1.1} deduces to the familiar generalized two-component Dullin-Gottwald-Holm system, see e.g., \cite{HGG2013}.
 Denote
 \begin{align*}
& e_i^n=U_i^n-u_i^n,\quad f_i^n=\Pi_i^n-\rho_i^n, \quad i\in\mathbb{I}_M,\; n\in \mathbb{K}_N,\\
& c_2=\frac12\max\{1+c_{\rm max}+3c_0+\sigma c_0,\; c_{\rm max}+4\sigma c_0\},\quad c_3=c_1\exp\Big(\frac32 c_2T\Big)\sqrt{\frac{L}{c_2}},
  \end{align*}
then the convergence with $\Omega = 0$ is followed.
\begin{theorem} \rm\label{theorem3 1}
 Suppose $\{U_i^n,\,\Pi_i^n\,|\,i\in\mathbb{I}_M,\; n\in \mathbb{K}_N\}$ is the solution of the problem \eqref{equa1.1}--\eqref{eq1.4}, $\{u_i^n,\,\rho_i^n\,|\,i\in\mathbb{I}_M,\; n\in \mathbb{K}_N\}$ is the approximation solution of the difference scheme \eqref{equa3.7}. When $\Omega=0$ and $3c_2\tau \leq 1$,
we have
 \begin{align*}
 \|e^n\| \leqslant c_3(\tau^2+h^2), \quad |e^n|_1 \leqslant c_3(\tau^2+h^2),\quad \|f^n\| \leqslant c_3(\tau^2+h^2),  \quad n\in \mathbb{K}_N.
 \end{align*}
 \end{theorem}
\begin{proof}
When $\Omega=0$, subtracting \eqref{equa3.7} from \eqref{equa3.1} and \eqref{equa3.5}, we get the following error system
\begin{subequations}
\label{equa3.19}
\begin{numcases}{}
  \delta_t e_i^{n+\frac12}-\delta_t\delta_x^2e_i^{n+\frac12}-\kappa\Delta_xe_i^{n+\frac12}+3[\psi(U^{n+\frac12},U^{n+\frac12})_i-\psi(u^{n+\frac12},u^{n+\frac12})_i]
   \nonumber\\
= 3\sigma[\psi(\delta_x^2U^{n+\frac12},U^{n+\frac12})_i-\psi(\delta_x^2u^{n+\frac12},u^{n+\frac12})_i]-\mu\Delta_x\delta_x^2e_i^{n+\frac12}\nonumber\\
 \quad -[\Pi_i^{n+\frac12}\Delta_x\Pi_i^{n+\frac12}-\rho_i^{n+\frac12}\Delta_x\rho_i^{n+\frac12}]+Q_i^{n+\frac12}, \quad i\in\mathbb{I}_M,\; n\in \mathbb{K}_N^0,\label{eq3.19}\\
 \delta_t f_i^{n+\frac12}+[\Delta_x(\Pi^{n+\frac12}U^{n+\frac12})_i-\Delta_x(\rho^{n+\frac12}u^{n+\frac12})_i]=R_i^{n+\frac12},\quad i\in\mathbb{I}_M,\; n\in \mathbb{K}_N^0,\label{eq3.20}\\
 e_i^0=0,\; f_i^0=0, \quad i\in\mathbb{I}_M,   \label{eq3.21}\\
 e_i^n=e_{i+M}^{n},\; f_i^n=f_{i+M}^n,\quad i\in\mathbb{I}_M,\; n\in \mathbb{K}_N.\label{eq3.22}
\end{numcases}
\end{subequations}
An inner product of \eqref{eq3.19} with $e^{n+\frac12}$ is carried out, we have
\begin{align*}
&\;  (\delta_t e^{n+\frac12},e^{n+\frac12})-(\delta_t\delta_x^2e^{n+\frac12},e^{n+\frac12})
   -\kappa(\Delta_xe^{n+\frac12},e^{n+\frac12})\nonumber\\
&\; +3(\psi(U^{n+\frac12},U^{n+\frac12})-\psi(u^{n+\frac12},u^{n+\frac12}),e^{n+\frac12})  \\
= &\; 3\sigma(\psi(\delta_x^2U^{n+\frac12},U^{n+\frac12})-\psi(\delta_x^2u^{n+\frac12},u^{n+\frac12}),e^{n+\frac12})
   -\mu(\Delta_x\delta_x^2e^{n+\frac12},e^{n+\frac12})\nonumber\\
& -(\Pi^{n+\frac12}\Delta_x\Pi^{n+\frac12}-\rho^{n+\frac12}\Delta_x\rho^{n+\frac12},e^{n+\frac12})+(Q^{n+\frac12},e^{n+\frac12}), \quad i\in\mathbb{I}_M,\; n\in \mathbb{K}_N^0.
\end{align*}
Employing Lemma \ref{lemma1}, the equality above becomes
\begin{align}
 & \frac{1}{2\tau}\left[(\|e^{n+1}\|^2-\|e^n\|^2)+(|e^{n+\frac12}|_1^2-|e^n|_1^2)\right]
  +3(\psi(U^{n+\frac12},U^{n+\frac12})-\psi(u^{n+\frac12},u^{n+\frac12}),e^{n+\frac12}) \nonumber\\
=&\; 3\sigma(\psi(\delta_x^2U^{n+\frac12},U^{n+\frac12})-\psi(\delta_x^2u^{n+\frac12},u^{n+\frac12}),e^{n+\frac12})
  -(\Pi^{n+\frac12}\Delta_x\Pi^{n+\frac12}-\rho^{n+\frac12}\Delta_x\rho^{n+\frac12},e^{n+\frac12})\nonumber\\
 & +(Q^{n+\frac12},e^{n+\frac12}),\quad n\in \mathbb{K}_N^0.\label{eq3.23}
\end{align}
With the application of Lemma \ref{lemma1}, we have
\begin{align}
 &-(\psi(U^{n+\frac12},U^{n+\frac12})-\psi(u^{n+\frac12},u^{n+\frac12}),e^{n+\frac12})\nonumber\\
= &-(\psi(U^{n+\frac12},U^{n+\frac12})-\psi(U^{n+\frac12}-e^{n+\frac12},U^{n+\frac{1}{2}}-e^{n+\frac12}),e^{n+\frac12})\nonumber\\
= &-(\psi(U^{n+\frac12},e^{n+\frac12})+\psi(e^{n+\frac12},U^{n+\frac12})-\psi(e^{n+\frac12},e^{n+\frac12}), e^{n+\frac12})\notag\\
= &-(\psi(e^{n+\frac12},U^{n+\frac12}), e^{n+\frac12})\notag\\
=& -\frac13(e^{n+\frac12}\Delta_xU^{n+\frac12},e^{n+\frac12})+\frac13(e^{n+\frac12}U^{n+\frac12},\Delta_xe^{n+\frac12}). \label{eq3.24}
\end{align}
Analogously, the first term and second term on the right-hand side of \eqref{eq3.23} becomes
\begin{align}
  &(\psi(\delta_x^2U^{n+\frac12},U^{n+\frac12})-\psi(\delta_x^2u^{n+\frac12},u^{n+\frac12}),e^{n+\frac12})\nonumber\\
= &( \psi(\delta_x^2U^{n+\frac12},e^{n+\frac12})+\psi(\delta_x^2e^{n+\frac12},U^{n+\frac12})-\psi(\delta_x^2e^{n+\frac12},e^{n+\frac12}),
e^{n+\frac12})\notag\\
= &(\psi(\delta_x^2e^{n+\frac12},U^{n+\frac12}),e^{n+\frac12})\notag\\
= &\frac13(\delta_x^2e^{n+\frac12},\Delta_xU^{n+\frac12}\cdot e^{n+\frac12})-\frac13(\delta_x^2e^{n+\frac12}\cdot
   U^{n+\frac12},\Delta_xe^{n+\frac12}),\label{eq3.25}
\end{align}
and
\begin{align}
  &-(\Pi^{n+\frac12}\Delta_x\Pi^{n+\frac12}-\rho^{n+\frac12}\Delta_x\rho^{n+\frac12},e^{n+\frac12})\nonumber\\
=&-(\Pi^{n+\frac12}\Delta_x\Pi^{n+\frac12}-(\Pi^{n+\frac12}-f^{n+\frac12})\Delta_x(\Pi^{n+\frac12}-f^{n+\frac12}),e^{n+\frac12})\nonumber\\
=& -(\Pi^{n+\frac12}\Delta_xf^{n+\frac12}+f^{n+\frac12}\Delta_x\Pi^{n+\frac12}-f^{n+\frac12}\Delta_xf^{n+\frac12},e^{n+\frac12})\notag\\
=&-(\Pi^{n+\frac12}\Delta_xf^{n+\frac12},e^{n+\frac12})-(f^{n+\frac12}\Delta_x\Pi^{n+\frac12},e^{n+\frac12})
 +(f^{n+\frac12}\Delta_xf^{n+\frac12},e^{n+\frac12})\nonumber\\
=&(f^{n+\frac12},\Delta_x(\Pi^{n+\frac12}e^{n+\frac12}))-(f^{n+\frac12}\Delta_x\Pi^{n+\frac12},e^{n+\frac12})
 -(f^{n+\frac12},\Delta_x(f^{n+\frac12}e^{n+\frac12})).\label{eq3.26}
\end{align}
Substituting \eqref{eq3.24}--\eqref{eq3.26} into \eqref{eq3.23}, we have
\begin{align}
&    \frac{1}{2\tau}\left[(\|e^{n+1}\|^2-\|e^n\|^2)+(|e^{n+\frac12}|_1^2-|e^n|_1^2)\right]\nonumber\\
= & \; -(e^{n+\frac12}\Delta_xU^{n+\frac12},e^{n+\frac12})
     +(e^{n+\frac12}U^{n+\frac12},\Delta_xe^{n+\frac12})+\sigma(\delta_x^2e^{n+\frac12},\Delta_xU^{n+\frac12}\cdot e^{n+\frac12})\nonumber\\
&   -\sigma(\delta_x^2e^{n+\frac12}\cdot U^{n+\frac12},\Delta_xe^{n+\frac12})
   +(f^{n+\frac12},\Delta_x(\Pi^{n+\frac12}e^{n+\frac12}))-(f^{n+\frac12}\Delta_x\Pi^{n+\frac12},e^{n+\frac12})\nonumber\\
&  -(f^{n+\frac12},\Delta_x(f^{n+\frac12}e^{n+\frac12}))+(Q^{n+\frac12},e^{n+\frac12}),\quad  n\in \mathbb{K}_N^0. \label{eq3.26b}
\end{align}
A similar inner product of \eqref{eq3.20} with $f^{n+\frac12}$ is made, we have
\begin{align}
\frac{1}{2\tau}(\|f^{n+1}\|^2-\|f^n\|^2)+(\Delta_x(\Pi^{n+\frac12}U^{n+\frac12})-\Delta_x(\rho^{n+\frac12}u^{n+\frac12}),f^{n+\frac12})
=(R^{n+\frac12},f^{n+\frac12}).\label{eq3.27}
\end{align}
Similar to the derivation in \eqref{eq3.24}, we have
\begin{align}
&     -(\Delta_x(\Pi^{n+\frac12}U^{n+\frac12})-\Delta_x(\rho^{n+\frac12}u^{n+\frac12}),f^{n+\frac12})\nonumber\\
=&\;  -(\Delta_x(\Pi^{n+\frac12}U^{n+\frac12})-\Delta_x(\Pi^{n+\frac12}-f^{n+\frac{1}{2}})(U^{n+\frac12}-e^{n+\frac12}),f^{n+\frac12})\nonumber\\
=&\;- (\Delta_x(\Pi^{n+\frac12}e^{n+\frac12})+\Delta_x(f^{n+\frac12}U^{n+\frac12})-\Delta_x(f^{n+\frac12}e^{n+\frac12}),f^{n+\frac12})\notag\\
=&\;  -(\Delta_x(\Pi^{n+\frac12}e^{n+\frac12}),f^{n+\frac12})-(\Delta_x(f^{n+\frac12}U^{n+\frac12}),f^{n+\frac12})
      +(\Delta_x(f^{n+\frac12}e^{n+\frac12}),f^{n+\frac12}).\label{eq3.28}
\end{align}
Plugging \eqref{eq3.28} into \eqref{eq3.27}, and simple calculation yields
\begin{align}
  &\; \frac{1}{2\tau}(\|f^{n+1}\|^2-\|f^n\|^2)\notag\\
=&\; -(\Delta_x(\Pi^{n+\frac12}e^{n+\frac12}),f^{n+\frac12})-(\Delta_x(f^{n+\frac12}U^{n+\frac12}),f^{n+\frac12})\nonumber\\
&\;      +(\Delta_x(f^{n+\frac12}e^{n+\frac12}),f^{n+\frac12})+(R^{n+\frac12},f^{n+\frac12}).\label{eq3.29}
\end{align}
Summing up \eqref{eq3.26b} with \eqref{eq3.29}, we have
\begin{align}
 &\;  \frac{1}{2\tau}\left[\|e^{n+1}\|^2-\|e^n\|^2)+(|e^{n+\frac12}|_1^2-|e^n|_1^2)+(\|f^{n+1}\|^2-\|f^n\|^2)\right] \nonumber\\
=&\; -(e^{n+\frac12}\Delta_xU^{n+\frac12},e^{n+\frac12})
   +(e^{n+\frac12}U^{n+\frac12},\Delta_xe^{n+\frac12})+\sigma(\delta_x^2e^{n+\frac12},\Delta_xU^{n+\frac12}\cdot e^{n+\frac12})\nonumber\\
&\;  -\sigma(\delta_x^2e^{n+\frac12}\cdot U^{n+\frac12},\Delta_xe^{n+\frac12})
   -(f^{n+\frac12}\Delta_x\Pi^{n+\frac12},e^{n+\frac12})-(\Delta_x(f^{n+\frac12}U^{n+\frac12}),f^{n+\frac12})\nonumber\\
&\;  +(Q^{n+\frac12},e^{n+\frac12})+(R^{n+\frac12},f^{n+\frac12}),\nonumber\\
=&\; \sum_{i=1}^8 J_i, \quad  n\in \mathbb{K}_N^0.\label{eq3.30}
\end{align}
With the help of the Cauchy-Schwarz inequality and combining assumption conditions \eqref{eq3.0a} with  \eqref{eq3.0c}, we easily have the estimates as
\begin{align*}
&J_1=-(e^{n+\frac12}\Delta_xU^{n+\frac12},e^{n+\frac12})\leqslant c_0\|e^{n+\frac12}\|^2,\\
&J_2=(e^{n+\frac12}U^{n+\frac12},\Delta_xe^{n+\frac12})\leqslant c_{\rm max}\|e^{n+\frac12}\|\cdot|e^{n+\frac12}|_1,\\
&J_5=-(f^{n+\frac12}\Delta_x\Pi^{n+\frac12},e^{n+\frac12})\leqslant c_0\|f^{n+\frac12}\|\cdot \|e^{n+\frac12}\|,\\
&J_7=(Q^{n+\frac12},e^{n+\frac12})\leqslant \|Q^{n+\frac12}\|\cdot\|e^{n+\frac12}\|,\\
&J_8=(R^{n+\frac12},f^{n+\frac12})\leqslant \|R^{n+\frac12}\|\cdot\|f^{n+\frac12}\|.
\end{align*}
The remaining terms ($J_3$, $J_4$ and $J_6$) will be estimated respectively. Firstly, we estimate $J_3$ as
\begin{align*}
J_3 =&\; \sigma(\delta_x^2e^{n+\frac12},\Delta_xU^{n+\frac12}\cdot e^{n+\frac12})\\
    =&\; -\sigma \langle \delta_xe^{n+\frac12},\delta_x(\Delta_xU^{n+\frac12}\cdot e^{n+\frac12})\rangle \\
    =&\; -\sigma\sum_{i\in\mathbb{I}_M}\delta_xe_i^{n+\frac12}(\Delta_xU_{i+1}^{n+\frac12}\cdot e_{i+1}^{n+\frac12}-\Delta_xU_{i}^{n+\frac12}\cdot e_{i}^{n+\frac12})\\
    =&\; -\sigma\sum_{i\in\mathbb{I}_M}\delta_xe_i^{n+\frac12}\Big[\Delta_xU_{i+1}^{n+\frac12}\cdot (e_{i+1}^{n+\frac12}-e_i^{n+\frac12})
         +(\Delta_xU_{i+1}^{n+\frac12}-\Delta_xU_{i}^{n+\frac12})\cdot e_{i}^{n+\frac12}\Big]\\
    =&\; -\sigma h\sum_{i\in\mathbb{I}_M}\delta_xe_i^{n+\frac12} \Big[\Delta_xU_{i+1}^{n+\frac12}\cdot\delta_xe_{i+\frac12}^{n+\frac12}
         +\delta_x(\Delta_xU_{i+\frac12}^{n+\frac12})\cdot e_{i}^{n+\frac12}\Big]\\
   \leqslant &\;
          \sigma |e^{n+1}|_1^2\cdot\|\Delta_xU^{n+\frac12}\|_{\infty}+\sigma
            |e^{n+\frac12}|_1\cdot\|\delta_x(\Delta_xU^{n+\frac12})\|_{\infty}\cdot\|e^{n+\frac12}\|\\
    \leqslant &\;
         \sigma c_0 |e^{n+1}|_1^2+\sigma  c_0|e^{n+\frac12}|_1 \cdot\|e^{n+\frac12}\|.
\end{align*}
Then we directly estimate $J_4$ as
\begin{align*}
J_4 =&\; -\sigma(\delta_x^2e^{n+\frac12}\cdot U^{n+\frac12} ,\Delta_xe^{n+\frac12})\\
    =&\; -\sigma h\sum_{i\in\mathbb{I}_M}(\delta_x^2e_i^{n+\frac12})U_i^{n+\frac12}\cdot\Delta_xe_i^{n+\frac12}\\
    =&\; \frac{\sigma}{2}\sum_{i\in\mathbb{I}_M}\Big[(\delta_xe_{i-\frac12}^{n+\frac12})^2-(\delta_xe_{i+\frac12}^{n+\frac12})^2\Big]U_i^{n+\frac12}\\
    =&\; \frac{\sigma h}{2}\sum_{i\in\mathbb{I}_M}(\delta_xe_{i+\frac12}^{n+\frac12})^2\delta_xU_{i+\frac12}^{n+\frac12}\\
    \leqslant &\;
         \frac{\sigma c_0}{2}|e^{n+\frac12}|_1^2.
\end{align*}
In light of the summation by parts, $J_6$ is calculated as
\begin{align*}
J_6 = &\; -(\Delta_x(f^{n+\frac12}U^{n+\frac12}),f^{n+\frac12})\\
    = &\;  (f^{n+\frac12}U^{n+\frac12},\Delta_xf^{n+\frac12})\\
    = &\;  h\sum_{i\in\mathbb{I}_M} f_i^{n+\frac12}U_i^{n+\frac12}\Delta_xf_i^{n+\frac12} \\
    = &\;  \frac12 \sum_{i\in\mathbb{I}_M}f_i^{n+\frac12}U_i^{n+\frac12}(f_{i+1}^{n+\frac12}-f_{i-1}^{n+\frac12})\\
    = &\;  \frac12 \sum_{i\in\mathbb{I}_M}(f_i^{n+\frac12}f_{i+1}^{n+\frac12}U_i^{n+\frac12}-f_i^{n+\frac12}f_{i-1}^{n+\frac12}U_i^{n+\frac12})\\
    = &\;  \frac{h}{2} \sum_{i\in\mathbb{I}_M}f_i^{n+\frac12}f_{i+1}^{n+\frac12}\cdot\frac{U_i^{n+\frac12}-U_{i+1}^{n+\frac12}}{h}\\
    \leqslant &\;
           \frac{c_0}2\|f^{n+\frac12}\|^2.
\end{align*}
Substituting $J_i\;(1\leqslant i \leqslant 8)$ into \eqref{eq3.30}, we have
\begin{align}
 &\;  \frac{1}{2\tau}\left[\|e^{n+1}\|^2-\|e^n\|^2)+(|e^{n+\frac12}|_1^2-|e^n|_1^2)+(\|f^{n+1}\|^2-\|f^n\|^2)\right] \nonumber\\
 \leqslant &\;
      c_0\|e^{n+\frac12}\|^2+c_{\rm max}\|e^{n+\frac12}\|\cdot|e^{n+\frac12}|_1+\sigma c_0 |e^{n+1}|_1^2+\sigma  c_0|e^{n+\frac12}|_1 \cdot\|e^{n+\frac12}\|+\frac{1}{2}\sigma c_0|e^{n+\frac12}|_1^2\nonumber\\
&\;   +c_0\|f^{n+\frac12}\|\cdot \|e^{n+\frac12}\|+\frac{c_0}2\|f^{n+\frac12}\|^2+\|Q^{n+\frac12}\|\cdot\|e^{n+\frac12}\|
      +\|R^{n+\frac12}\|\cdot\|f^{n+\frac12}\|\nonumber\\
\leqslant &\;
       c_0\|e^{n+\frac12}\|^2+\frac{1}2c_{\rm max}\|e^{n+\frac12}\|^2+\frac{1}2c_{\rm max}|e^{n+\frac12}|_1^2+\sigma c_0 |e^{n+\frac12}|_1^2+\frac{1}2\sigma c_0|e^{n+\frac12}|_1^2 +\frac{1}2\sigma c_0\|e^{n+\frac12}\|^2\nonumber\\
&\;   +\frac{1}{2}\sigma c_0|e^{n+\frac12}|_1^2+\frac{c_0}2\|f^{n+\frac12}\|^2+\frac{c_0}2\|e^{n+\frac12}\|^2
      +\frac{c_0}2\|f^{n+\frac12}\|^2+\frac12\|Q^{n+\frac12}\|^2\nonumber\\
&\;   +\frac12\|e^{n+\frac12}\|^2+\frac12\|R^{n+\frac12}\|^2+\frac12\|f^{n+\frac12}\|^2\nonumber\\
=&\;  (\frac12 +\frac{c_{\rm max}}2+\frac32c_0+\frac{\sigma}{2}c_0)\|e^{n+\frac12}\|^2+(\frac{c_{\rm max}}{2}+2\sigma c_0)|e^{n+\frac12}|_1^2
      +(c_0+\frac12)\|f^{n+\frac12}\|^2\nonumber\\
&\;   +\frac12(\|Q^{n+\frac12}\|^2+\|R^{n+\frac12}\|^2)\nonumber\\
\leqslant &\;
      c_2(\|e^{n+\frac12}\|^2+|e^{n+\frac12}|_1^2+\|f^{n+\frac12}\|^2)+\frac12(\|Q^{n+\frac12}\|^2+\|R^{n+\frac12}\|^2)\nonumber\\
\leqslant &\;
      \frac{c_2}{2}(\|e^{n+1}\|^2+\|e^{n}\|^2+|e^{n+1}|_1^2+|e^{n}|_1^2+\|f^{n+1}\|^2+\|f^{n}\|^2)\nonumber\\
&\;   +\frac12(\|Q^{n+\frac12}\|^2+\|R^{n+\frac12}\|^2), \quad n\in \mathbb{K}_N^0.    \label{eq3.31}
\end{align}
Denote
\begin{align*}
F^{n}= \|e^{n}\|^2+|e^{n}|_1^2+\|f^{n}\|^2, \quad n\in \mathbb{K}_N.
\end{align*}
Therefore, \eqref{eq3.31} becomes
\begin{align*}
\frac{1}{2\tau}(F^{n+1}-F^n)\leqslant \frac{c_2}{2}(F^{n+1}+F^n)+\frac12(\|Q^{n+\frac12}\|^2+\|R^{n+\frac12}\|^2), \quad n\in \mathbb{K}_N^0.
\end{align*}
Multiplying the inequality above by $2\tau$ and then noticing \eqref{equa3.3}, we have
\begin{align*}
(1-c_2\tau)F^{n+1}\leqslant (1+c_2\tau)F^n+2Lc_1^2\tau(\tau^2+h^2)^2,  \quad n\in \mathbb{K}_N^0.
\end{align*}
When $c_2\tau\leqslant \frac13$, we have
\begin{align*}
F^{n+1}\leqslant(1+3c_2\tau)F^n+3Lc_1^2\tau(\tau^2+h^2)^2,  \quad n\in \mathbb{K}_N^0.
\end{align*}
Employing the discrete Gronwall inequality, we have
\begin{align*}
F^{n}\leqslant \exp(3c_2T)\cdot\frac{Lc_1^2}{c_2}(\tau^2+h^2)^2 =c_3^2(\tau^2+h^2)^2,  \quad n\in \mathbb{K}_N.
\end{align*}
According to the definition of $F^n$, Theorem \ref{theorem3 1} holds, which completes the proof.
\end{proof}

\subsection{Convergence for the difference scheme with nonzero rotation parameter ($\Omega\neq 0$)}
Next we further analyze the convergence of the difference scheme \eqref{equa3.7} when $\Omega\neq0$
by taking an inner product of \eqref{eq3.8} twice but with different quantities, and based on the technical energy analysis.
For this purpose, denote
\begin{align*}
&c_4=\frac12(3c_0+\sigma c_0+1), \qquad c_5=\frac{c_0}{2}(4\sigma+1-2\Omega \kappa+c_0),\\
&c_6=\frac12\big[(2c_0+1)(1-2\Omega \kappa)+2(c_0+c_{\rm max})\Omega\big], \qquad c_7=\frac12\big(2-\Omega \kappa+2\Omega c_0\big),\\
&c_8=\frac12\max\left\{c_4,\,c_5,\,\frac{c_6+3c_0\Omega}{1-2\Omega( \kappa+ c_0)}\right\},\qquad  c_9=c_1\exp(3c_8T)\sqrt{\frac{L}{2}\cdot\frac{c_7}{c_8}},\\
&E(0)=E(u^0(x),\rho^0(x)),\qquad  c_{10} = \sqrt{\frac{1+\sqrt{1+4L^2}}{L}\cdot E(0)}.
\end{align*}
Then the convergence under the restriction of small initial energy is listed as follows.
\begin{theorem} \rm\label{theorem3 2}
 Suppose $\{U_i^n,\,\Pi_i^n\,|\,i\in\mathbb{I}_M,\; n\in \mathbb{K}_N\}$ is the solution of the problem  \eqref{equa1.1}--\eqref{eq1.4}, $\{u_i^n,\,\rho_i^n\,|\,i\in\mathbb{I}_M,\; n\in \mathbb{K}_N\}$ is the solution of the difference scheme \eqref{equa3.7}. When $c_{10}\leqslant \frac{1}{2\Omega}-\kappa$ and $6c_8\tau<1$, we have the error estimates
 \begin{align*}
 \|e^n\| \leqslant c_9(\tau^2+h^2), \quad |e^n|_1 \leqslant c_9(\tau^2+h^2),\quad \|f^n\| \leqslant \frac{c_9}{\sqrt{1-2\Omega( \kappa+ c_{10})}}(\tau^2+h^2),  \quad n\in \mathbb{K}_N.
 \end{align*}
 \end{theorem}
\begin{proof}
When $\Omega\neq0$, subtracting \eqref{equa3.1} and \eqref{equa3.5} from \eqref{equa3.7} directly, we get the following error system of equations
\begin{subequations}
\label{equa3.32}
\begin{numcases}{}
  \delta_t e_i^{n+\frac12}-\delta_t\delta_x^2e_i^{n+\frac12}-\kappa\Delta_xe_i^{n+\frac12}+3[\psi(U^{n+\frac12},U^{n+\frac12})_i-\psi(u^{n+\frac12},u^{n+\frac12})_i]
   \nonumber\\
= 3\sigma[\psi(\delta_x^2U^{n+\frac12},U^{n+\frac12})_i-\psi(\delta_x^2u^{n+\frac12},u^{n+\frac12})_i]-\mu\Delta_x\delta_x^2e_i^{n+\frac12}\nonumber\\
 \quad  -(1-2\Omega \kappa)[\Pi_i^{n+\frac12}\Delta_x\Pi_i^{n+\frac12}-\rho_i^{n+\frac12}\Delta_x\rho_i^{n+\frac12}]\nonumber\\
  \quad +2\Omega[\Pi_i^{n+\frac12}\Delta_x(\Pi^{n+\frac12}U^{n+\frac12})_i-\rho_i^{n+\frac12}\Delta_x(\rho^{n+\frac12}u^{n+\frac12})_i]+Q_i^{n+\frac12}, \notag\\
  \qquad \qquad \qquad \qquad \qquad  i\in\mathbb{I}_M,\; n\in \mathbb{K}_N^0,\label{eq3.32}\\
 \delta_t f_i^{n+\frac12}+[\Delta_x(\Pi^{n+\frac12}U^{n+\frac12})_i-\Delta_x(\rho^{n+\frac12}u^{n+\frac12})_i]=R_i^{n+\frac12},\quad i\in\mathbb{I}_M,\; n\in \mathbb{K}_N^0,\label{eq3.33}\\
 e_i^0=0,\; f_i^0=0, \quad i\in\mathbb{I}_M,   \label{eq3.34}\\
 e_i^n=e_{i+M}^{n},\; f_i^n=f_{i+M}^n,\quad i\in\mathbb{I}_M,\; n\in \mathbb{K}_N.\label{eq3.35}
\end{numcases}
\end{subequations}
Denote
\begin{align}
G^n=\|e^n\|^2+|e^n|_1^2+(1-2\Omega \kappa)\|f^n\|^2-2\Omega(f^n,f^nU^n).\label{eq3.35b}
\end{align}
By employing Lemma \ref{lemma2} in continuous counterpart and the energy conservation \eqref{eq1.2a}, and then taking $\varepsilon = \frac{1+\sqrt{1+4L^2}}{2L}$, we have
\begin{align*}
  c^2_{\rm max}=\|u(\cdot,t)\|_\infty^2 & \leqslant \varepsilon |u(\cdot,t)|_1^2 + \Big(\frac{1}{L}+\frac{1}{\varepsilon}\Big)\|u(\cdot,t)\|^2\\
  & =\varepsilon ( |u(\cdot,t)|_1^2 + \|u(\cdot,t)\|^2)\\
  & \leqslant 2\varepsilon E(u,\rho) = 2\varepsilon E(0) = c_{10}^2.
\end{align*}
When $c_{10}=\sqrt{2\varepsilon E(0)} \leqslant \frac{1}{2\Omega}-\kappa$, we have $1-2\Omega \kappa-2\Omega c_{10}\geqslant 0$. The nonnegativity of $G^n$  at the moment is guaranteed because of
\begin{align}
G^n \geqslant & \|e^n\|^2+|e^n|_1^2+(1-2\Omega \kappa-2\Omega c_{\rm max})\|f^n\|^2\notag\\
        \geqslant & \|e^n\|^2+|e^n|_1^2+(1-2\Omega \kappa-2\Omega c_{10})\|f^n\|^2\geqslant 0.
\end{align}
Taking a $L^2$-inner product of \eqref{eq3.32} with $e^{n+\frac12}$ and noticing Lemma \ref{lemma1}, we have
\begin{align}
 & \frac{1}{2\tau}\left[(\|e^{n+1}\|^2-\|e^n\|^2)+(|e^{n+\frac12}|_1^2-|e^n|_1^2)\right]
  +3(\psi(U^{n+\frac12},U^{n+\frac12})-\psi(u^{n+\frac12},u^{n+\frac12}),e^{n+\frac12}) \nonumber\\
=&\; 3\sigma(\psi(\delta_x^2U^{n+\frac12},U^{n+\frac12})-\psi(\delta_x^2u^{n+\frac12},u^{n+\frac12}),e^{n+\frac12})
  \nonumber\\
 & -(1-2\Omega \kappa)(\Pi^{n+\frac12}\Delta_x\Pi^{n+\frac12}-\rho^{n+\frac12}\Delta_x\rho^{n+\frac12},e^{n+\frac12})+(Q^{n+\frac12},e^{n+\frac12})\nonumber\\
 &  +2\Omega(\Pi^{n+\frac12}\Delta_x(\Pi^{n+\frac12}U^{n+\frac12})-\rho^{n+\frac12}\Delta_x(\rho^{n+\frac12}u^{n+\frac12}),e^{n+\frac12})
  ,\quad i\in\mathbb{I}_M,\; n\in \mathbb{K}_N^0.\label{eq3.36}
\end{align}
Noticing that expression in the second term on the right-hand side of the equation above can be written as
\begin{align*}
&\; \Pi_i^{n+\frac12}\Delta_x(\Pi^{n+\frac12}U^{n+\frac12})_i-\rho^{n+\frac12}\Delta_x(\rho^{n+\frac12}u^{n+\frac12})_i\\
= &\;
   \Pi_i^{n+\frac12}\Delta_x(\Pi^{n+\frac12}U^{n+\frac12}-\rho^{n+\frac12}u^{n+\frac12})_i+ f_i^{n+\frac12}\Delta_x(\rho^{n+\frac12}u^{n+\frac12})_i \\
= &\;
   \Pi_i^{n+\frac12}\Delta_x(\Pi^{n+\frac12}e^{n+\frac12}+f^{n+\frac12}u^{n+\frac12})_i+ f_i^{n+\frac12}\Delta_x(\rho^{n+\frac12}u^{n+\frac12})_i,
\end{align*}
and using Lemma \ref{lemma1} again, the second term in the right-hand side of \eqref{eq3.36} becomes
\begin{align}
&\;  (\Pi^{n+\frac12}\Delta_x(\Pi^{n+\frac12}U^{n+\frac12})-\rho^{n+\frac12}\Delta_x(\rho^{n+\frac12}u^{n+\frac12}),e^{n+\frac12})\nonumber\\
= &\;
     (\Pi^{n+\frac12}\Delta_x(f^{n+\frac12}u^{n+\frac12}),e^{n+\frac12})+(f^{n+\frac12}\Delta_x(\rho^{n+\frac12}u^{n+\frac12}),e^{n+\frac12})\nonumber\\
= &\;
     (\Pi^{n+\frac12}\Delta_x(f^{n+\frac12}(U^{n+\frac12}-e^{n+\frac12})),e^{n+\frac12})\nonumber\\
&\;  +(f^{n+\frac12}\Delta_x[(\Pi^{n+\frac12}-f^{n+\frac12})(U^{n+\frac12}-e^{n+\frac12})],e^{n+\frac12})\nonumber\\
= &\;
     (\Pi^{n+\frac12}\Delta_x(f^{n+\frac12}U^{n+\frac12}),e^{n+\frac12})-(\Pi^{n+\frac12}\Delta_x(f^{n+\frac12}e^{n+\frac12}),e^{n+\frac12})\nonumber\\
&\; +(f^{n+\frac12}\Delta_x(\Pi^{n+\frac12}U^{n+\frac12}),e^{n+\frac12})-(f^{n+\frac12}\Delta_x(\Pi^{n+\frac12}e^{n+\frac12}),e^{n+\frac12})\nonumber\\
&\; -(f^{n+\frac12}\Delta_x(f^{n+\frac12}U^{n+\frac12}),e^{n+\frac12})+(f^{n+\frac12}\Delta_x(f^{n+\frac12}e^{n+\frac12}),e^{n+\frac12})\nonumber\\
= &\;
     (\Pi^{n+\frac12}\Delta_x(f^{n+\frac12}U^{n+\frac12}),e^{n+\frac12})+(f^{n+\frac12}\Delta_x(\Pi^{n+\frac12}U^{n+\frac12}),e^{n+\frac12})\nonumber\\
&\;  -(f^{n+\frac12}\Delta_x(f^{n+\frac12}U^{n+\frac12}),e^{n+\frac12}).  \label{eq3.37}
\end{align}
Substituting \eqref{eq3.24}--\eqref{eq3.26} and \eqref{eq3.37} into \eqref{eq3.36}, for any $n\in \mathbb{K}_N^0$, we have
\begin{align}
&\;   \frac{1}{2\tau}\left[(\|e^{n+1}\|^2-\|e^n\|^2)+(|e^{n+\frac12}|_1^2-|e^n|_1^2)\right] \nonumber\\
= & \; -(e^{n+\frac12}\Delta_xU^{n+\frac12},e^{n+\frac12})
     +(e^{n+\frac12}U^{n+\frac12},\Delta_xe^{n+\frac12})+\sigma(\delta_x^2e^{n+\frac12},\Delta_xU^{n+\frac12}\cdot e^{n+\frac12})\nonumber\\
&\;  -\sigma(\delta_x^2e^{n+\frac12}\cdot U^{n+\frac12},\Delta_xe^{n+\frac12})
    +(1-2\Omega \kappa)[(f^{n+\frac12},\Delta_x(\Pi^{n+\frac12}e^{n+\frac12}))\nonumber\\
&\; -(f^{n+\frac12}\Delta_x\Pi^{n+\frac12},e^{n+\frac12})-(f^{n+\frac12},\Delta_x(f^{n+\frac12}e^{n+\frac12}))]
    +2\Omega[(\Pi^{n+\frac12}\Delta_x(f^{n+\frac12}U^{n+\frac12}),e^{n+\frac12})\nonumber\\
&\; +(f^{n+\frac12}\Delta_x(\Pi^{n+\frac12}U^{n+\frac12}),e^{n+\frac12})
    -(f^{n+\frac12}\Delta_x(f^{n+\frac12}U^{n+\frac12}),e^{n+\frac12})]
  +(Q^{n+\frac12},e^{n+\frac12}). \label{eq3.38}
\end{align}
Taking an inner product of \eqref{eq3.20} with $(1-2\Omega \kappa)f^{n+\frac12}$ and noticing \eqref{eq3.29}, we can derive that
\begin{align}
&\;  \frac{(1-2\Omega \kappa)}{2\tau}(\|f^{n+1}\|^2-\|f^n\|^2)\nonumber\\
=&\; (1-2\Omega \kappa) [-(\Delta_x(\Pi^{n+\frac12}e^{n+\frac12}),f^{n+\frac12})-(\Delta_x(f^{n+\frac12}U^{n+\frac12}),f^{n+\frac12})\nonumber\\
&\;      +(\Delta_x(f^{n+\frac12}e^{n+\frac12}),f^{n+\frac12})]+(1-2\Omega \kappa)(R^{n+\frac12},f^{n+\frac12}).\label{eq3.39}
\end{align}
Summing up \eqref{eq3.38} and \eqref{eq3.39}, we have
\begin{align}
&\;   \frac{1}{2\tau}\left[(\|e^{n+1}\|^2-\|e^n\|^2)+(|e^{n+\frac12}|_1^2-|e^n|_1^2)+(1-2\Omega \kappa)(\|f^{n+1}\|^2-\|f^n\|^2)\right] \nonumber\\
=&\;   -(e^{n+\frac12}\Delta_xU^{n+\frac12},e^{n+\frac12})
     +(e^{n+\frac12}U^{n+\frac12},\Delta_xe^{n+\frac12})+\sigma(\delta_x^2e^{n+\frac12},\Delta_xU^{n+\frac12}\cdot e^{n+\frac12})\nonumber\\
&\;  -\sigma(\delta_x^2e^{n+\frac12}\cdot U^{n+\frac12},\Delta_xe^{n+\frac12})
     +(1-2\Omega \kappa)[-(f^{n+\frac12}\Delta_x\Pi^{n+\frac12},e^{n+\frac12})\nonumber\\
&\;  -(\Delta_x(f^{n+\frac12}U^{n+\frac12}),f^{n+\frac12})]+2\Omega[(\Pi^{n+\frac12}\Delta_x(f^{n+\frac12}U^{n+\frac12}),e^{n+\frac12})\nonumber\\
&\; +(f^{n+\frac12}\Delta_x(\Pi^{n+\frac12}U^{n+\frac12}),e^{n+\frac12})
    -(f^{n+\frac12}\Delta_x(f^{n+\frac12}U^{n+\frac12}),e^{n+\frac12})]\nonumber\\
&\; +(Q^{n+\frac12},e^{n+\frac12})+(1-2\Omega \kappa)(R^{n+\frac12},f^{n+\frac12}) .  \label{eq3.40}
\end{align}
In order to eliminate the difficulty of estimating the term
$-(f^{n+\frac12}\Delta_x(f^{n+\frac12}U^{n+\frac12}),e^{n+\frac12})$
in \eqref{eq3.40}, we further taking an inner product of \eqref{eq3.20} with $-2\Omega f^{n+\frac12}U^{n+\frac12}$. Similar to the derivation in \eqref{eq3.29},
we have
\begin{align}
&\;  -2\Omega(\delta_tf^{n+\frac12},f^{n+\frac12}U^{n+\frac12})\nonumber\\
=&\; 2\Omega [(\Delta_x(\Pi^{n+\frac12}e^{n+\frac12}),f^{n+\frac12}U^{n+\frac12})+(\Delta_x(f^{n+\frac12}U^{n+\frac12}),f^{n+\frac12}U^{n+\frac12})\nonumber\\
&\;  -(\Delta_x(f^{n+\frac12}e^{n+\frac12}),f^{n+\frac12}U^{n+\frac12})]-2\Omega (R^{n+\frac12},f^{n+\frac12}U^{n+\frac12}).\label{eq3.41}
\end{align}
Directly applying Lemma \ref{lemma3} to the first term in \eqref{eq3.41}, it becomes
\begin{align}
&\;  -2\Omega(\delta_tf^{n+\frac12},f^{n+\frac12}U^{n+\frac12})\nonumber\\
=&\;  -2\Omega(\delta_tf^{n+\frac12},f^{n+\frac12}U^{n+\frac12}-(fU)^{n+\frac12}+(fU)^{n+\frac12})\nonumber\\
=&\; -\frac{\Omega}{\tau}[(f^{n+1},f^{n+1}U^{n+1})-(f^{n},f^{n}U^{n})]+\Omega(f^nf^{n+1},\delta_tU^{n+\frac12})\nonumber\\
&\;   +\frac{\Omega}{2}(f^{n+1}-f^n,(f^{n+1}-f^n)\delta_tU^{n+\frac12}).\label{eq3.41b}
\end{align}
Substituting \eqref{eq3.41b} into \eqref{eq3.41} and adding with \eqref{eq3.40} together, we have
\begin{align*}
&\;   \frac{1}{2\tau}(G^{n+1}-G^n)\nonumber\\
=&\;  -(e^{n+\frac12}\Delta_xU^{n+\frac12},e^{n+\frac12})
     +(e^{n+\frac12}U^{n+\frac12},\Delta_xe^{n+\frac12})+\sigma(\delta_x^2e^{n+\frac12},\Delta_xU^{n+\frac12}\cdot e^{n+\frac12})\nonumber\\
&\;  -\sigma(\delta_x^2e^{n+\frac12}\cdot U^{n+\frac12},\Delta_xe^{n+\frac12})+(1-2\Omega \kappa)[-(f^{n+\frac12}\Delta_x\Pi^{n+\frac12},e^{n+\frac12})\nonumber\\
&\;   -(\Delta_x(f^{n+\frac12}U^{n+\frac12}),f^{n+\frac12})]
     + (Q^{n+\frac12},e^{n+\frac12}) +(1-2\Omega \kappa)(R^{n+\frac12},f^{n+\frac12})  \nonumber\\
&\; +2\Omega(f^{n+\frac12}\Delta_x(\Pi^{n+\frac12}U^{n+\frac12}),e^{n+\frac12})-2\Omega (R^{n+\frac12},f^{n+\frac12}U^{n+\frac12})\nonumber\\
&\; -\Omega(f^nf^{n+1},\delta_tU^{n+\frac12})-\frac{\Omega}{2}(f^{n+1}-f^n,(f^{n+1}-f^n)\delta_tU^{n+\frac12})\nonumber\\
=&\; \sum_{i=1}^{12}P_i, \quad  n\in \mathbb{K}_N^0,\label{eq3.42}
\end{align*}
in which four terms diminish.
The estimates of $P_i\;(1\leqslant i \leqslant 8)$ are similar to $J_i\;(1\leqslant i \leqslant 8)$. The remaining terms $P_i\;(9\leqslant i \leqslant 12)$ will be estimated respectively as
\begin{align*}
&P_{9}=2\Omega(f^{n+\frac12}\Delta_x(\Pi^{n+\frac12}U^{n+\frac12}),e^{n+\frac12})
\leqslant  2c_0^2\Omega  \|f^{n+\frac12}\|\cdot\|e^{n+\frac12}\|,\nonumber\\
&P_{10}= -2\Omega (R^{n+\frac12},f^{n+\frac12}U^{n+\frac12})
\leqslant 2c_{\rm max}\Omega \|R^{n+\frac12}\|\cdot\|f^{n+\frac12}\|,\nonumber\\
&P_{11}=-\Omega(f^nf^{n+1},\delta_tU^{n+\frac12})\leqslant c_0\Omega\|f^n\|\cdot\|f^{n+1}\| ,\nonumber\\
&P_{12}=-\frac{\Omega}{2}(f^{n+1}-f^n,(f^{n+1}-f^n)\delta_tU^{n+\frac12})\leqslant c_0\Omega(\|f^n\|^2+\|f^{n+1}\|^2) .
\end{align*}
Therefore,
\begin{align*}
&\; \frac{1}{2\tau}(G^{n+1}-G^n)\nonumber\\
\leqslant &\;
   c_0\|e^{n+\frac12}\|^2+c_0\|e^{n+\frac12}\|\cdot|e^{n+\frac12}|_1+c_0\sigma  |e^{n+1}|_1^2+c_0\sigma |e^{n+\frac12}|_1 \cdot\|e^{n+\frac12}\|\nonumber\\
&\;
   +\frac{1}{2}c_0\sigma|e^{n+\frac12}|_1^2 +c_0(1-2\Omega \kappa)\|f^{n+\frac12}\|\cdot \|e^{n+\frac12}\|+\frac{1}2c_0(1-2\Omega \kappa)\|f^{n+\frac12}\|^2\nonumber\\
&\;
   +\|Q^{n+\frac12}\|\cdot\|e^{n+\frac12}\|+(1-2\Omega \kappa)\|R^{n+\frac12}\|\cdot\|f^{n+\frac12}\|+2c_0^2\Omega\|e^{n+\frac12}\|\cdot\|f^{n+\frac12}\|\nonumber\\
&\;
   +2c_{\rm max} \Omega\|R^{n+\frac12}\|\cdot\|f^{n+\frac12}\|+c_0\Omega\|f^n\|\cdot\|f^{n+1}\|+c_0\Omega(\|f^n\|^2+\|f^{n+1}\|^2)  \nonumber\\
\leqslant &\;
   c_0\|e^{n+\frac12}\|^2+\frac{c_0^2}{2}\|e^{n+\frac12}\|^2+\frac{c_0^2}{2}|e^{n+\frac12}|_1^2+ c_0 \sigma|e^{n+1}|_1^2
   +\frac{1}{2}c_0\sigma|e^{n+1}|_1^2+\frac{1}{2}c_0 \sigma\|e^{n+\frac12}\|^2\nonumber\\
&\;
   +\frac{c_0\sigma}{2}|e^{n+\frac12}|_1^2 + \frac{1}2c_0(1-2\Omega \kappa)\|f^{n+\frac12}\|^2+ \frac{1}2c_0(1-2\Omega \kappa)|e^{n+\frac12}|_1^2\nonumber\\
&\;
  + \frac{1}2c_0(1-2\Omega \kappa)\|f^{n+\frac12}\|^2+\frac12\|Q^{n+\frac12}\|^2+\frac12\|e^{n+\frac12}\|^2+\frac{1}2(1-2\Omega \kappa)\|R^{n+\frac12}\|^2\nonumber\\
&\;
  +\frac{1}2(1-2\Omega \kappa)\|f^{n+\frac12}\|^2+c_0^2\Omega\|e^{n+\frac12}\|^2 +c_{\rm max}\Omega\|f^{n+\frac12}\|^2+c_{\rm max}\Omega \|R^{n+\frac12}\|^2\nonumber\\
&\;
  + c_0\Omega\|f^{n+\frac12}\|^2+\frac{1}2c_0\Omega\|f^n\|^2+\frac{1}2c_0\Omega\|f^{n+1}\|^2+c_0\Omega(\|f^n\|^2+\|f^{n+1}\|^2)\nonumber\\
=&\; c_4\|e^{n+\frac12}\|^2+c_5|e^{n+\frac12}|_1^2+c_6\|f^{n+\frac12}\|^2+\frac{1}23c_0\Omega\|f^n\|^2+\frac{1}23c_0\Omega\|f^{n+1}\|^2\nonumber\\
&\;
     +\frac12\|Q^{n+\frac12}\|^2+\frac12(1-2\Omega \kappa+2c_{\rm max}\Omega)\|R^{n+\frac12}\|^2\\
\leqslant &\;
    \frac{c_4}{2}\|e^{n+1}\|^2+\frac{c_4}{2}\|e^n\|^2+\frac{c_5}{2}|e^{n+1}|_1^2 +\frac{c_5}{2}|e^{n}|_1^2 +\frac12(c_6+3c_0\Omega)\|f^{n+1}\|^{2}\\
&\;  +\frac12(c_6+3c_0\Omega)\|f^{n}\|^{2}+c_7(\|Q^{n+\frac12}\|^2+\|R^{n+\frac12}\|^2)\\
\leqslant &\;
     c_8(G^n+G^{n+1})+c_7(\|Q^{n+\frac12}\|^2+\|R^{n+\frac12}\|^2).
\end{align*}
Multiplying the inequality above by $2\tau$ on both sides and noticing \eqref{equa3.3}, we have
\begin{align*}
(1-2c_8\tau)G^{n+1}\leqslant (1+2c_8\tau)G^n+4c_7c_1^2L\tau(\tau^2+h^2)^2, \quad n\in \mathbb{K}_N^0.
\end{align*}
When $2c_8\tau\leqslant \frac13$, we have
\begin{align*}
G^{n+1}\leqslant (1+6c_8\tau)G^n+6c_7c_1^2L\tau(\tau^2+h^2)^2, \quad n\in \mathbb{K}_N^0.
\end{align*}
which implies
\begin{align*}
G^n\leqslant \exp(6c_8T)\cdot \frac{c_7c_1^2L}{c_8}(\tau^2+h^2)^2 \equiv c_9^2(\tau^2+h^2)^2, \quad n\in \mathbb{K}_N
\end{align*}
by using the discrete Gronwall inequality.
According to the definition of $G^n$ in \eqref{eq3.35b}, Theorem \ref{theorem3 2} holds.
We finish the proof.
\end{proof}

\begin{remark}
The restriction condition $c_{10}<\frac{1}{2\Omega}-\kappa$ implies that the initial energy satisfies
\begin{equation}\label{Small_intial1}
  E(0) < \Big(\frac{1-2\Omega \kappa}{2\Omega}\Big)^2 \cdot \frac{\sqrt{1+4L^2}-1}{4L}.
\end{equation}
The inequality \eqref{Small_intial1} further indicates that Theorem \ref{theorem3 2} holds only for small initial value or small initial energy.
In addition, we know from the proof that
\begin{equation}\label{Small_intial2}
  c_{\rm max} < \frac{1+\sqrt{1+4L^2}}{L}\cdot E(0)=\frac{1}{2\Omega}-\kappa.
\end{equation}
In physics, the inequality \eqref{Small_intial2} implies that Theorem \ref{theorem3 2} holds when the maximum fluid velocity is no more than $\frac{1}{2\Omega}-\kappa$.

\end{remark}
\begin{remark}
  Theorem \ref{theorem3 2} states that there is no grid ratio restriction for the spatial and temporal stepsizes. It implies that the numerical scheme \eqref{equa3.7} is unconditionally convergent. Advantages of the difference scheme \eqref{equa3.7} compared with those in literature \cite{ZLZ2022} are shown in Table \ref{Table4}. We see that the current scheme performs much better.
  \begin{table}[tbh!]
 \vspace{-5mm}
\begin{center}
\renewcommand{\arraystretch}{1.25}
  \caption{The pros and cons of the scheme \eqref{equa3.7} at present with that in literature \cite{ZLZ2022}.}\label{Table4}
\def\temptablewidth{1.000\textwidth}
\rule{\temptablewidth}{1pt}
\renewcommand\tabcolsep{1pt}
{\footnotesize
\begin{tabular*}{\temptablewidth}{@{\extracolsep{\fill}}|c|c|c|c|c|}
&\multicolumn{2}{c|}{$\Omega = 0$}
&\multicolumn{2}{c|}{$\Omega \neq 0$}\\
\cline{2-3}\cline{4-5}
      &{\rm Scheme\;in\;\cite{ZLZ2022}}&{\rm Scheme\;\eqref{equa3.7}}&{\rm Scheme\;in\;\cite{ZLZ2022}}&{\rm Scheme\;\eqref{equa3.7}}  \\
\hline
{\rm \;Convergence\;condition\;} & $\tau \lesssim h$  & $\tau \lesssim 1$  & $\tau \lesssim h$? &  $\tau \lesssim 1$, small initial energy \\
{\rm \;Convergence\;order}     &$O(h^2+\tau^2)$             &$O(h^2+\tau^2)$             &$O(h^2+\tau^2)?$  &$O(h^2+\tau^2)\;$ \\
{\rm \;Invariants}           & {\rm Energy,\;Mass\;} &{\rm Energy,\;Momentum,\;Mass\;}     & {\rm Energy,\;Mass\;} &{\rm Energy,\;Momentum,\;Mass\;}\\
{\rm \;Solvability}            &$\checkmark$                &$\checkmark$    &$\checkmark$    &$\checkmark\;$\\
\end{tabular*}}
\rule{\temptablewidth}{1pt}
\end{center}
``$\checkmark$'' denotes that the result has been proved; ``$?$'' denotes that the results are numerically correct, but lack of a theoretical proof.
 \vspace{-5mm}
\end{table}
\end{remark}

\section{Numerical results}\label{section5}
We will present two benchmark problems (includes several cases in each one) to detect numerical theories including the unconditional convergence, invariant-preserving properties (energy/momentum/mass), and long-time simulation on large domain.

To facilitate calculation, we propose a subsequent two-level iteration method to solve the nonlinear numerical scheme \eqref{equa3.7}. Suppose we have known $\{u_i^n,\, \rho_i^n\,|\,i\in\mathbb{I}_M\}$,
the following linear system of equations will be used to approximate the solution of the difference scheme \eqref{equa3.7},
\begin{subequations}
\label{Ite}
\begin{numcases}{}
 \frac{2}{\tau}(u_i^{(l+1)}-u_i^n)-\frac{2}{\tau}(\delta_x^2 u_i^{(l+1)} - \delta_x^2 u_i^n)-\kappa\Delta_xu_i^{(l+1)}+3\psi(u^{(l)},u^{(l+1)})_i
   -3\sigma\psi(\delta_x^2u^{(l)},u^{(l+1)})_i\nonumber\\
  \quad +\mu\Delta_x\delta_x^2u_i^{(l+1)}+(1-2\Omega \kappa)\rho_i^{(l)}\Delta_x\rho_i^{(l+1)}-2\Omega\rho_i^{(l)} \Delta_x(\rho^{(l)} u^{(l+1)})_i=0, \quad i\in\mathbb{I}_M,\label{Ite1.1}\\
  \frac{2}{\tau}(\rho_i^{(l+1)}-\rho_i^n)+\Delta_x(\rho^{(l)} u^{(l+1)})_i=0, \quad i\in\mathbb{I}_M,\label{Ite1.2}\\
  u_{i}^{(0)}=u_i^{n},\; \rho_i^{(0)}=\rho_i^{n},\quad i\in\mathbb{I}_M, \label{Ite1.3}\\
  u_{i}^{(l)}=u_{i+M}^{(l)},\; \rho_{i}^{(l)}=\rho_{i+M}^{(l)}, \quad i\in\mathbb{I}_M,\label{Ite1.4}
\end{numcases}
\end{subequations}
until
$$\max\big\{\max_{i\in\mathbb{I}_M}|u_i^{(l+1)}-u_i^{(l)}|,\max_{i\in\mathbb{I}_M}|\rho_i^{(l+1)}-\rho_i^{(l)}|\big\}\leqslant \varepsilon,\quad l=0,\;1,\;\ldots.$$
Then we have
$$u_i^{n+1} = 2u_i^{(l+1)}-u_i^n,\quad \rho_i^{n+1} = 2\rho_i^{(l+1)}-\rho_i^n,\quad i\in\mathbb{I}_M.$$
In numerical implementation, we fix the tolerance error $\varepsilon=1e-12$ for each iteration unless otherwise specified.
The formulas
\begin{align*}
&\mathrm{Ord_\infty^h}=\mathbf{\log_2}\frac{\|\mathrm{F_u}(2h,\tau)\|_\infty}
{\|\mathrm{F_u}(h,\tau)\|_\infty},\qquad\quad
 \mathrm{Ord_{2}^h}=\mathbf{\log_2}\frac{\|\mathrm{F_\rho}(2h,\tau)\|}
{\|\mathrm{F_\rho}(h,\tau)\|},\\
&\mathrm{Ord_\infty^\tau}=\mathbf{\log_2}\frac{\|\mathrm{G_u}(h,2\tau)\|_\infty}
{\|\mathrm{G_u}(h,\tau)\|_\infty},\qquad \quad
\mathrm{Ord_{2}^\tau}=\mathbf{\log_2}\frac{\|\mathrm{G_\rho}(h,2\tau)\|}
{\|\mathrm{G_\rho}(h,\tau)\|}
\end{align*}
are borrowed to test the convergence orders in space and time, where
the $L^{\infty}$-norm
and the $L^{2}$-norm of numerical errors are defined by
\begin{align*}
&\|\mathrm{F_u}(h,\tau)\|_{\infty}=\max_{i\in \mathbb{I}_M,k\in \mathbb{K}_N}\limits\Big|u_{i}^{k}(h,\tau)-u_{2i}^{k}(h/2,\tau)\Big|, \\
&\|\mathrm{G_u}(h,\tau)\|_{\infty}=\max_{i\in \mathbb{I}_M,k\in \mathbb{K}_N}\limits\Big|u_{i}^{k}(h,\tau)-u_{i}^{2k}(h,\tau/2)\Big|, \\
&\|\mathrm{F_\rho}(h,\tau)\|=\max_{k\in \mathbb{K}_N}\limits\Big({h\sum_{i\in \mathbb{I}_M} |\rho_{i}^{k}(h,\tau)-\rho_{2i}^{k}(h/2,\tau)|^2}\Big)^{1/2},\\
&\|\mathrm{G_\rho}(h,\tau)\|=\max_{k\in \mathbb{K}_N}\limits\Big({h\sum_{i\in \mathbb{I}_M} |\rho_{i}^{k}(h,\tau)-\rho_{i}^{2k}(h,\tau/2)|^2}\Big)^{1/2},
\end{align*}
 respectively.

\begin{example}\label{exam1}
  \textbf{(Dam break problem)} 
  The initial fluid velocity and free surface elevation from equilibrium are taken, respectively, as
\begin{align*}
u^0(x)=0,\quad \rho^0(x)=1+\tanh(x+\emph{a})-\tanh(x-\emph{a}),
\end{align*}
with $\emph{a}$ being a dam-breaking parameter. Six groups of parameters are selected in Table \ref{tabparameter}.
\end{example}
\begin{table}[tbh!]
 \vspace{-5mm}
\begin{center}
\renewcommand{\arraystretch}{1.25}
\tabcolsep 0pt \caption{Selected parameters and the calculated domain in the numerical tests.}\label{tabparameter}
\def\temptablewidth{0.75\textwidth}
\rule{\temptablewidth}{1pt}
{\footnotesize
\begin{tabular*}{\temptablewidth}{@{\extracolsep{\fill}}l|l l}
 & ${\rm ~~~~~\qquad \qquad Parameters}$ &${\rm Domain}~{\rm of}~(x,t)$  \\
\hline
{\;\rm \textbf{Case A}\quad}  &$\emph{a}=0.1$, $\kappa=\mu=\Omega=0$, $\sigma=1$, \cite{CMR2014,HI2010,YFS2018} &$[-6,6]\times[0,20];\;$         \\
{\;\rm \textbf{Case B}\quad}  &$\emph{a}=4$, $\kappa=\mu=\Omega=0$, $\sigma=1$, \cite{ZLZ2022}  &$[-12\pi,12\pi]\times[0,2];\;$  \\
{\;\rm \textbf{Case C}\quad}  &$\emph{a}=0.2$, $\kappa=0,\, \mu=\sigma=1,\,\Omega=73\times 10^{-6}$, \cite{ZLZ2022}  &$[-8,8]\times[0,1];\;$  \\
{\;\rm \textbf{Case D}\quad}  &$\emph{a}=\kappa=\mu=\sigma=1,\, \Omega=73\times 10^{-6}$, \cite{ZLZ2022}  &$[-8,8]\times[0,1];\;$\\
{\;\rm \textbf{Case E}\quad}  &$\emph{a}=\mu=\sigma=1,\kappa=0,\, \Omega=73\times 10^{-6}$,\cite{ZLZ2022}  &$[-12\pi,12\pi]\times[0,50];\;$\\
{\;\rm \textbf{Case F}\quad}  &$\emph{a}=\kappa=\mu=\sigma=1,\, \Omega=73\times 10^{-6}$,  &$[-100,100]\times[0,1000].\;$\\
\end{tabular*}}
\rule{\temptablewidth}{1pt}
\end{center}
 \vspace{-5mm}
\end{table}
We report the numerical data in Tables \ref{Table1}--\ref{Table3} and display the numerical behavior in Figures \ref{fig:1}--\ref{fig:3}.
\begin{table}[tbh!]
\begin{center}
\renewcommand{\arraystretch}{1.25}
\tabcolsep 6pt \caption{Numerical errors against $h$-grid size reduction with the fixed temporal stepsizes.}\label{Table1}
\def\temptablewidth{1.0\textwidth}
\rule{\temptablewidth}{1pt}
{\footnotesize
\begin{tabular*}{\temptablewidth}{@{\extracolsep{\fill}}|c|cccc|cccc|}
&\multicolumn{4}{c|}{\textbf{Case A}, $\tau=1/50$}
&\multicolumn{4}{c|}{\textbf{Case B}, $\tau=1/50$}\\
\cline{2-5}\cline{6-9}
$h$& $\|\mathrm{F_u}(h,\tau)\|_\infty$ &$\mathrm{Ord_{\infty}^h}$ &$\|\mathrm{F_\rho}(h,\tau)\|$ &$\mathrm{Ord_{2}^h}$\;\;
&$\|\mathrm{F_u}(h,\tau)\|_\infty$&$\mathrm{Ord_{\infty}^h}$ &$\|\mathrm{F_\rho}(h,\tau)\|$ &$\mathrm{Ord_{2}^h}$ \\
\hline
$\;0.6\;$    &$3.1656{\rm e}-02$  &$*$        &$8.2588{\rm e}-02$  &$*$\;\;
        &$1.1779{\rm e}-01$  &$*$        &$3.1404{\rm e}-01$  &$*\;$         \\
$\;0.3\;$    &$8.0761{\rm e}-03$  &$1.9707$   &$3.1629{\rm e}-02$  &$1.3847$\;\;
        &$4.6146{\rm e}-02$  &$1.3519$   &$1.2761{\rm e}-01$  &$1.2992\;$    \\
$\;0.15\;$   &$2.2533{\rm e}-03$  &$1.8416$   &$7.2516{\rm e}-03$  &$2.1249$\;\;
        &$1.3660{\rm e}-02$  &$1.7562$   &$4.1642{\rm e}-02$  &$1.6156\;$    \\
$\;0.075\;$  &$5.7025{\rm e}-04$  &$1.9824$   &$1.8181{\rm e}-03$  &$1.9959$\;\;
        &$3.5951{\rm e}-03$  &$1.9259$   &$1.1627{\rm e}-02$  &$1.8406\;$    \\
$\;0.0375\;$ &$1.4320{\rm e}-04$  &$1.9936$   &$4.5442{\rm e}-04$  &$2.0003$\;\;
        &$9.0983{\rm e}-04$  &$1.9824$   &$3.0100{\rm e}-03$  &$1.9496\;$\\
\hline
&\multicolumn{4}{c|}{\textbf{Case C}, $\tau=1/1000$}
&\multicolumn{4}{c|}{\textbf{Case D}, $\tau=1/1000$}\\
\cline{2-5}\cline{6-9}
$h$& $\|\mathrm{F_u}(h,\tau)\|_\infty$ &$\mathrm{Ord_{\infty}^h}$ &$\|\mathrm{F_\rho}(h,\tau)\|$ &$\mathrm{Ord_{2}^h}$\;\;
&$\|\mathrm{F_u}(h,\tau)\|_\infty$&$\mathrm{Ord_{\infty}^h}$ &$\|\mathrm{F_\rho}(h,\tau)\|$ &$\mathrm{Ord_{2}^h}$ \\
\hline
$\;0.4\;$   &$3.7012{\rm e}-03$  &$*$        &$8.9008{\rm e}-03$  &$*$\;\;
        &$1.6694{\rm e}-02$  &$*$        &$3.8179{\rm e}-02$  &$*\;$         \\
$\;0.2\;$   &$1.0564{\rm e}-03$  &$1.8088$   &$2.1868{\rm e}-03$  &$2.0251$\;\;
        &$4.9151{\rm e}-03$  &$1.7640$   &$9.8260{\rm e}-03$  &$1.9581\;$    \\
$\;0.1\;$   &$2.7672{\rm e}-04$  &$1.9327$   &$5.4452{\rm e}-04$  &$2.0058$\;\;
        &$1.3122{\rm e}-03$  &$1.9053$   &$2.4850{\rm e}-03$  &$1.9834\;$    \\
$\;0.05\;$  &$7.0471{\rm e}-05$  &$1.9733$   &$1.3601{\rm e}-04$  &$2.0012$\;\;
        &$3.3245{\rm e}-04$  &$1.9808$   &$6.2326{\rm e}-04$  &$1.9953\;$    \\
$\;0.025\;$ &$1.7670{\rm e}-05$  &$1.9957$   &$3.3996{\rm e}-05$  &$2.0003$\;\;
        &$8.3398{\rm e}-05$  &$1.9951$   &$1.5595{\rm e}-04$  &$1.9988\;$
\end{tabular*}
}
\rule{\temptablewidth}{1pt}
\end{center}
\end{table}

\begin{table}[tbh!]
 \vspace{-5mm}
\begin{center}
\renewcommand{\arraystretch}{1.25}
  \caption{Numerical errors against $\tau$-grid size reduction with the fixed spatial stepsizes.}\label{Table2}
\def\temptablewidth{1.000\textwidth}
\rule{\temptablewidth}{1pt}
\renewcommand\tabcolsep{1pt}
{\footnotesize
\begin{tabular*}{\temptablewidth}{@{\extracolsep{\fill}}|c|cccc|cccc|}
&\multicolumn{4}{c|}{\textbf{Case A}, $h=6/25$}
&\multicolumn{4}{c|}{\textbf{Case B}, $h=6/25$}\\
\cline{2-5}\cline{6-9}
$\tau$& $\|\mathrm{F_u}(h,\tau)\|_\infty$ &$\mathrm{Ord_{\infty}^\tau}$ &$\|\mathrm{F_\rho}(h,\tau)\|$ &$\mathrm{Ord_{2}^\tau}$\;\;
&$\|\mathrm{F_u}(h,\tau)\|_\infty$&$\mathrm{Ord_{\infty}^\tau}$ &$\|\mathrm{F_\rho}(h,\tau)\|$ &$\mathrm{Ord_{2}^\tau}$ \\
\hline
$\;0.25\;$     &$1.2391{\rm e}-03$  &$*$        &$4.2968{\rm e}-03$  &$*$\;\;
        &$5.3429{\rm e}-02$  &$*$        &$1.2532{\rm e}-01$  &$*\;$         \\
$\;0.125\;$    &$3.1403{\rm e}-04$  &$1.9803$   &$1.0730{\rm e}-03$  &$2.0016$\;\;
        &$1.6057{\rm e}-02$  &$1.7345$   &$3.8919{\rm e}-02$  &$1.6871\;$    \\
$\;0.0625\;$   &$7.8767{\rm e}-05$  &$1.9952$   &$2.6815{\rm e}-04$  &$2.0006$\;\;
        &$4.1766{\rm e}-03$  &$1.9428$   &$1.0397{\rm e}-02$  &$1.9043\;$    \\
$\;0.03125\;$  &$1.9708{\rm e}-05$  &$1.9988$   &$6.7032{\rm e}-05$  &$2.0001$\;\;
        &$1.0538{\rm e}-03$  &$1.9867$   &$2.6455{\rm e}-03$  &$1.9746\;$    \\
$\;0.015625\;$ &$4.9280{\rm e}-06$  &$1.9997$   &$1.6757{\rm e}-05$  &$2.0001$\;\;
        &$2.6404{\rm e}-04$  &$1.9968$   &$6.6432{\rm e}-04$  &$1.9936\;$\\
        \hline
&\multicolumn{4}{c|}{\textbf{Case C},  $h=4/25$}
&\multicolumn{4}{c|}{\textbf{Case D},  $h=4/25$}\\
\cline{2-5}\cline{6-9}
$\tau$& $\|\mathrm{F_u}(h,\tau)\|_\infty$ &$\mathrm{Ord_{\infty}^\tau}$ &$\|\mathrm{F_\rho}(h,\tau)\|$ &$\mathrm{Ord_{2}^\tau}$\;\;
&$\|\mathrm{F_u}(h,\tau)\|_\infty$&$\mathrm{Ord_{\infty}^\tau}$ &$\|\mathrm{F_\rho}(h,\tau)\|$ &$\mathrm{Ord_{2}^\tau}$ \\
\hline
$\;0.0125\;$     &$2.7005{\rm e}-06$  &$*$        &$3.4715{\rm e}-06$  &$*$\;\;
        &$3.4991{\rm e}-05$  &$*$        &$4.1005{\rm e}-05$  &$*\;$         \\
$\;0.00625\;$    &$6.7526{\rm e}-07$  &$1.9997$   &$8.6798{\rm e}-07$  &$1.9998$\;\;
        &$8.7511{\rm e}-06$  &$1.9994$   &$1.0254{\rm e}-05$  &$1.9996\;$    \\
$\;0.003125\;$   &$1.6883{\rm e}-07$  &$1.9999$   &$2.1700{\rm e}-07$  &$2.0000$\;\;
        &$2.1878{\rm e}-06$  &$2.0000$   &$2.5638{\rm e}-06$  &$1.9999\;$    \\
$\;0.0015625\;$  &$4.2309{\rm e}-08$  &$1.9965$   &$5.4154{\rm e}-08$  &$2.0025$\;\;
        &$5.4701{\rm e}-07$  &$1.9998$   &$6.4092{\rm e}-07$  &$2.0000\;$    \\
$\;0.00078125\;$ &$1.0528{\rm e}-08$  &$2.0066$   &$1.3626{\rm e}-08$  &$1.9907$\;\;
        &$1.3677{\rm e}-07$  &$1.9998$   &$1.6025{\rm e}-07$  &$1.9998\;$
\end{tabular*}}
\rule{\temptablewidth}{1pt}
\end{center}
\end{table}

\begin{table}[tbh!]
\begin{center}
\renewcommand{\arraystretch}{1.12}
\tabcolsep 8pt \caption{Numerical conservative invariants of $E^n$, $H^n$, and $I^n$ at time $t_n$.}\label{Table3}
\def\temptablewidth{0.8\textwidth}
\rule{\temptablewidth}{1pt}
{\footnotesize
\begin{tabular*}{\temptablewidth}{@{\extracolsep{\fill}}|c|c|c|c|}
&\multicolumn{3}{c|}{\textbf{Case A}, $(h,\tau) = (1/5,1/256)$}  \\
\cline{2-4}
\quad$t_n\quad$  &$\qquad E^n\qquad$&$\qquad H^n\qquad$&$\qquad I^n\qquad$\\
\hline
 \quad$0\quad$   &$6.426590811396586\quad$   &$0                \quad$ &$12.39999498602724\quad$\\
 \quad$2\quad$   &$6.426590811396584\quad$   &$0.000000000000006\quad$ &$12.39999498602724\quad$\\
 \quad$4\quad$   &$6.426590811396586\quad$   &$0.000000000000024\quad$ &$12.39999498602725\quad$\\
 \quad$6\quad$   &$6.426590811396584\quad$   &$0.000000000000055\quad$ &$12.39999498602725\quad$\\
 \quad$8\quad$   &$6.426590811396582\quad$   &$0.000000000000073\quad$ &$12.39999498602724\quad$\\
 \quad$10\quad$  &$6.426590811396588\quad$   &$0.000000000000023\quad$ &$12.39999498602725\quad$\\
 \hline
&\multicolumn{3}{c|}{\textbf{Case B}, $(h,\tau) = (1/5,1/256)$}\\
\cline{2-4}
\quad$t_n\quad$  &$\qquad E^n\qquad$&$\qquad H^n\qquad$&$\qquad I^n\qquad$\\
\hline
 \quad$0\quad$   &$67.70071037403311\quad$   &$0                 \quad$ &$91.40037694550523\quad$\\
 \quad$2\quad$   &$67.70067103780045\quad$   &$-0.000000000000133\quad$ &$91.40037694550517\quad$\\
 \quad$4\quad$   &$67.70064579547363\quad$   &$-0.000000000000383\quad$ &$91.40037694550520\quad$\\
 \quad$6\quad$   &$67.70062829372516\quad$   &$-0.000000000000820\quad$ &$91.40037694550522\quad$\\
 \quad$8\quad$   &$67.70061909000481\quad$   &$-0.000000000000772\quad$ &$91.40037694550519\quad$\\
 \quad$10\quad$  &$67.70061445300891\quad$   &$-0.000000000000737\quad$ &$91.40037694550520\quad$\\
 \hline
 &\multicolumn{3}{c|}{\textbf{Case C}, $(h,\tau) = (1/10,1/256)$} \\
\cline{2-4}
\quad$t_n\quad$  &$\qquad E^n\qquad$&$\qquad H^n\qquad$&$\qquad I^n\qquad$\\
\hline
 \quad$0\quad$    &$8.905545767953516\quad$   &$0.001300209682121\quad$ &$16.79999981448777\quad$\\
 \quad$2\quad$    &$8.905545767953521\quad$   &$0.001300209690680\quad$ &$16.79999981448777\quad$\\
 \quad$4\quad$    &$8.905545767953528\quad$   &$0.001300209671917\quad$ &$16.79999981448778\quad$\\
 \quad$6\quad$    &$8.905545767953543\quad$   &$0.001300209652087\quad$ &$16.79999981448778\quad$\\
 \quad$8\quad$    &$8.905545767953550\quad$   &$0.001300209648438\quad$ &$16.79999981448778\quad$\\
 \quad$10\quad$   &$8.905545767953560\quad$   &$0.001300209621099\quad$ &$16.79999981448778\quad$\\
 \hline
&\multicolumn{3}{c|}{\textbf{Case D}, $(h,\tau) = (1/5,1/256)$}\\
\cline{2-4}
\quad$t_n\quad$  &$\qquad E^n\qquad$&$\qquad H^n\qquad$&$\qquad I^n\qquad$\\
\hline
 \quad$0\quad$    &$14.14719145331662\quad$   &$0.002065791557752\quad$ &$19.99999836196486\quad$\\
 \quad$2\quad$    &$14.14719145331671\quad$   &$0.002065792012461\quad$ &$19.99999836196486\quad$\\
 \quad$4\quad$    &$14.14719145331683\quad$   &$0.002065792010409\quad$ &$19.99999836196486\quad$\\
 \quad$6\quad$    &$14.14719145331692\quad$   &$0.002065792072498\quad$ &$19.99999836196486\quad$\\
 \quad$8\quad$    &$14.14719145331738\quad$   &$0.002065794300180\quad$ &$19.99999836196487\quad$\\
 \quad$10\quad$   &$14.14719145331774\quad$   &$0.002065796011144\quad$ &$19.99999836196487\quad$
\end{tabular*}}
\rule{\temptablewidth}{1pt}
\end{center}
\end{table}

\begin{figure}[htbp]
 \vspace{-5mm}
\subfigtopskip=2pt
\subfigbottomskip=2pt
\subfigcapskip=-3pt
\subfigure[Case A-$u$]{ \centering
\includegraphics[width=0.5\textwidth]{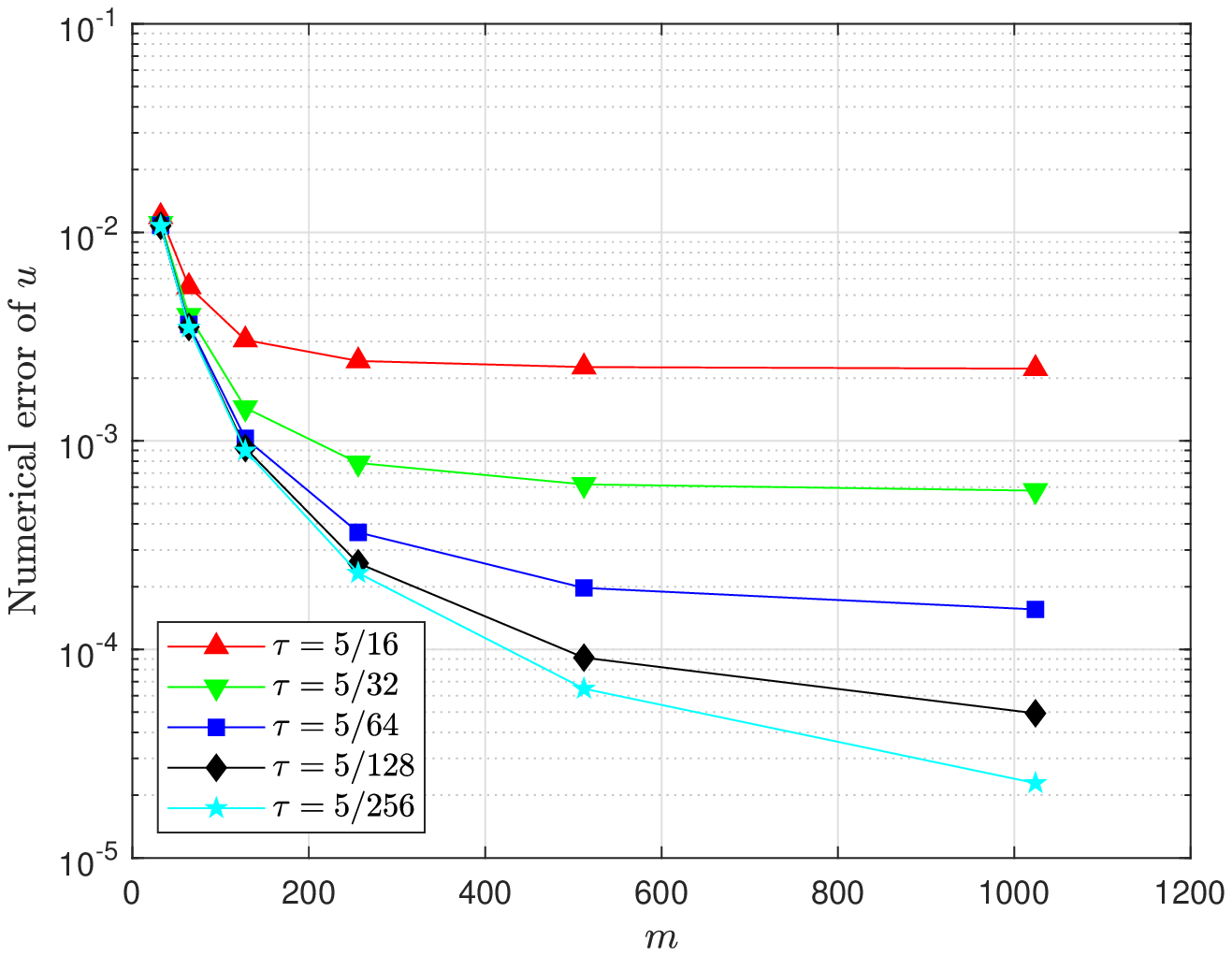}}
\hspace{-26pt}
\subfigure[Case A-$\rho$]{ \centering
\includegraphics[width=0.5\textwidth]{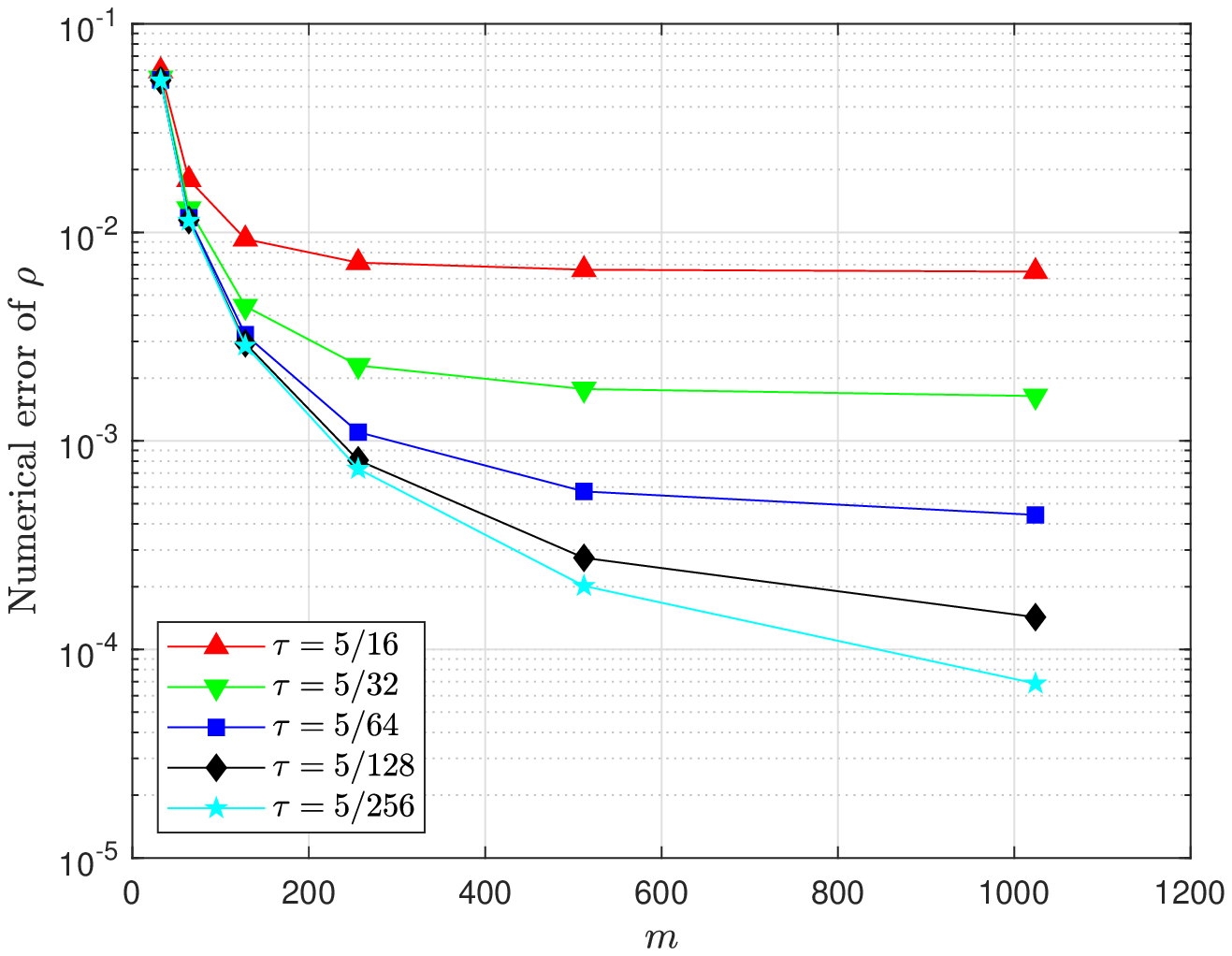}
}
\subfigure[Case D-$u$]{ \centering
\includegraphics[width=0.5\textwidth]{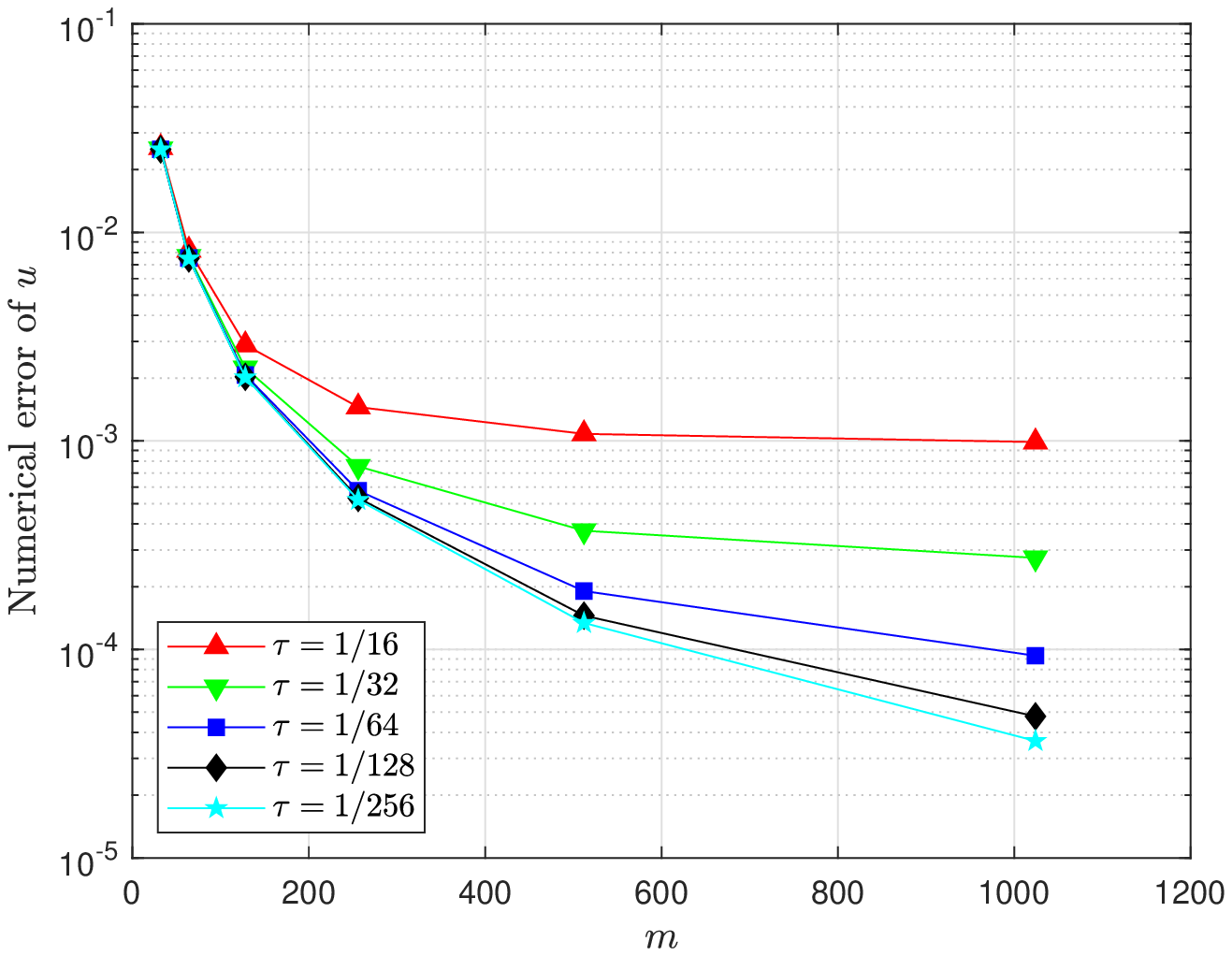}}
\hspace{-15pt}
\subfigure[Case D-$\rho$]{ \centering
\includegraphics[width=0.5\textwidth]{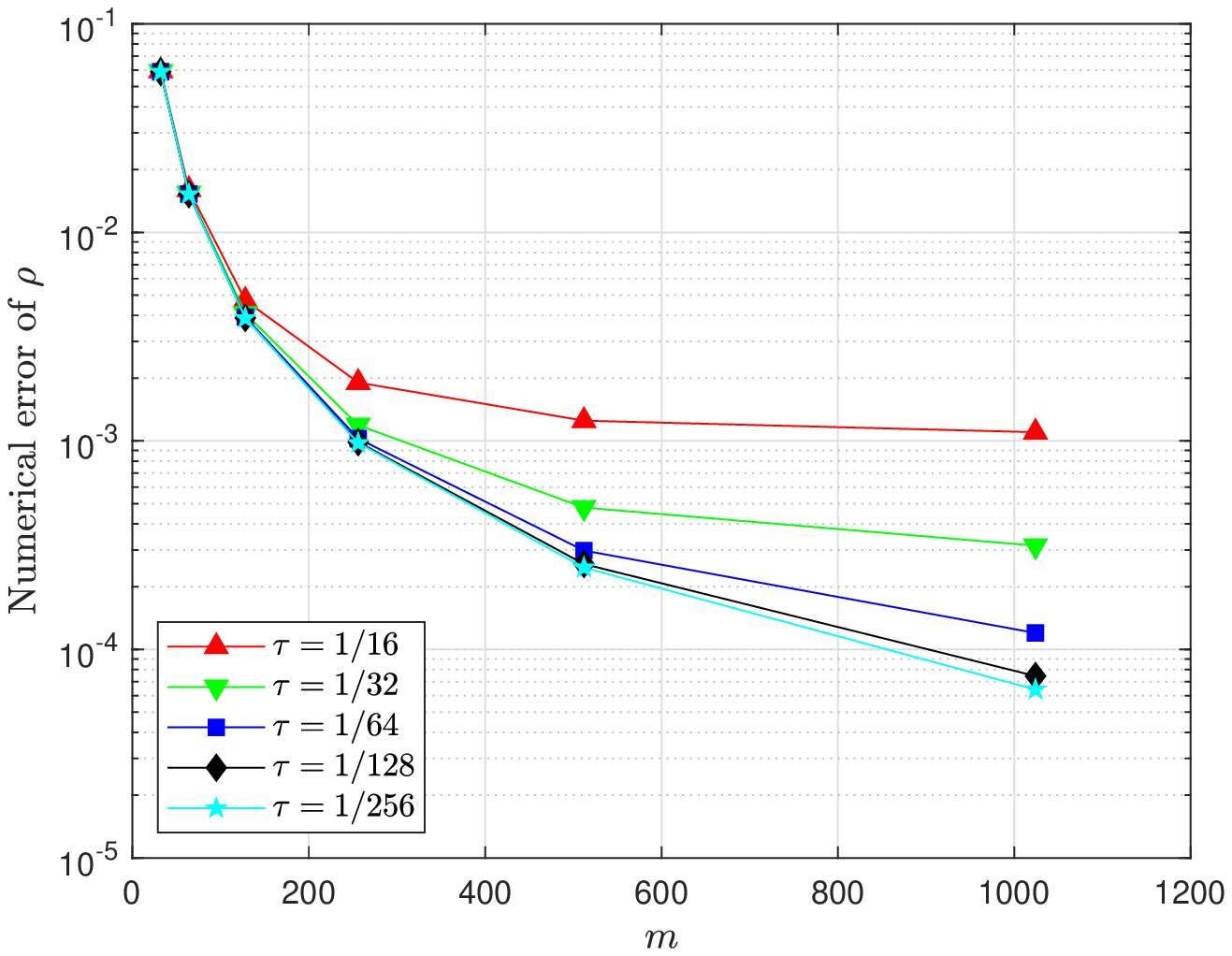}}
\caption{Unconditional convergence tests against spatial grid with the fixed temporal stepsize $\tau$.
\textbf{Case A}: (a) and (b); \textbf{Case D}: (c) and (d).} \label{fig:1}
 \vspace{-5mm}
\end{figure}

\begin{figure}[htbp]
\subfigtopskip=2pt
\subfigbottomskip=2pt
\subfigcapskip=-3pt
\subfigure[Case E-$u$, view(45,70)\,]{ \centering
\includegraphics[width=0.25\textwidth]{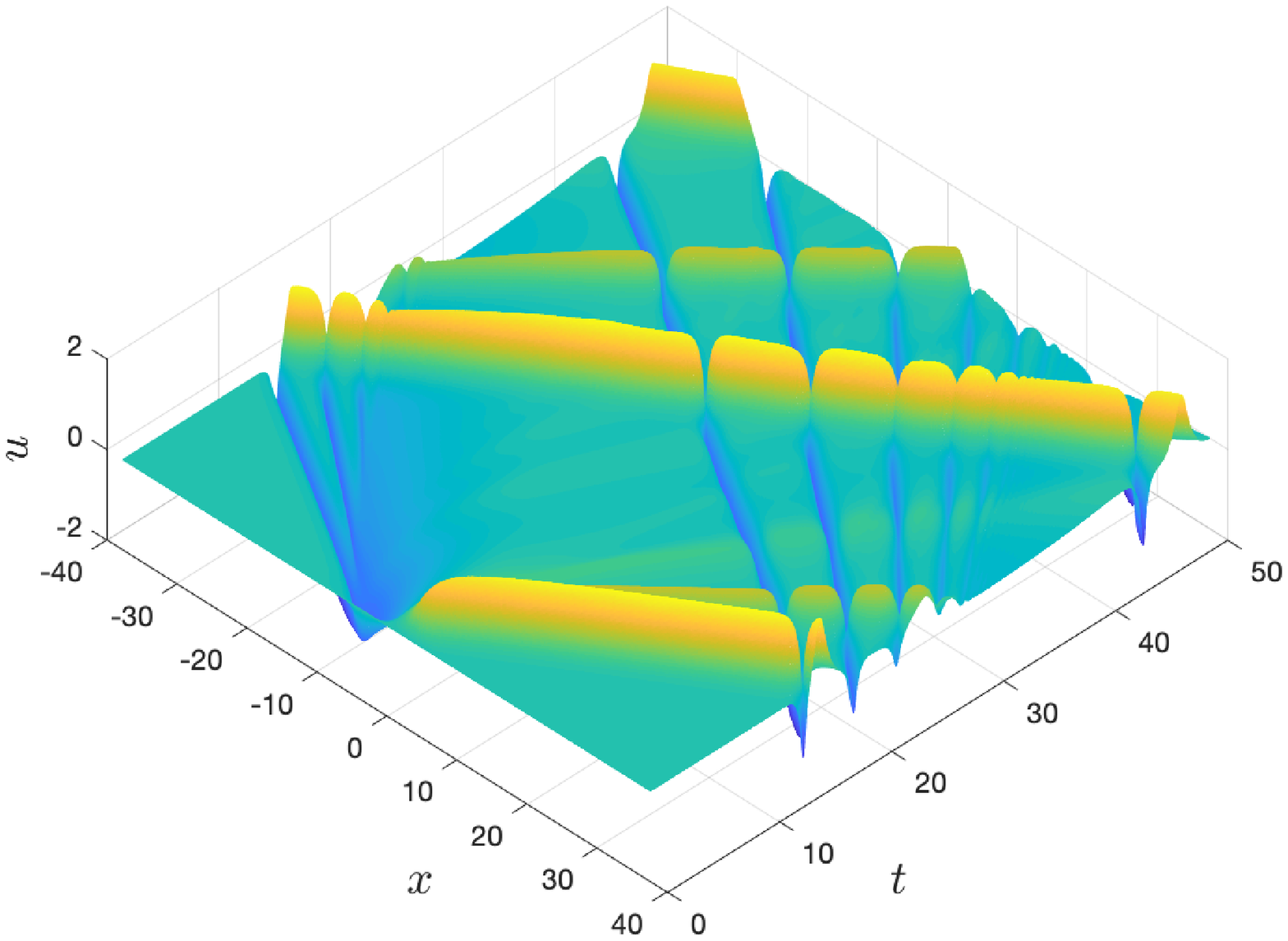}}
\hspace{-20pt}
\subfigure[Case E-$\rho$, view(45,70)\,]{ \centering
\includegraphics[width=0.25\textwidth]{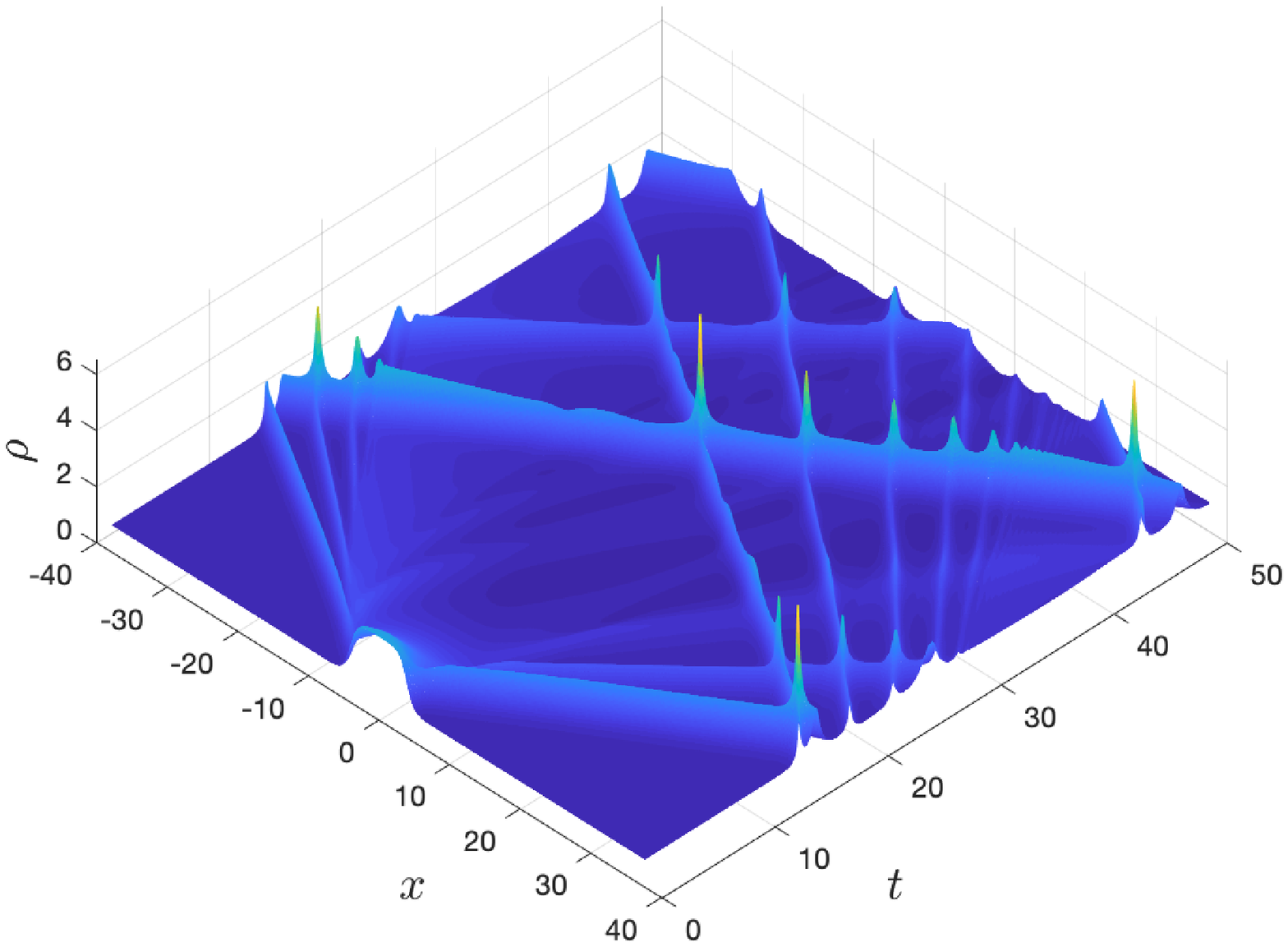}
}\hspace{-18pt}
\subfigure[Case E-$u$, view(45,70)\,]{ \centering
\includegraphics[width=0.25\textwidth]{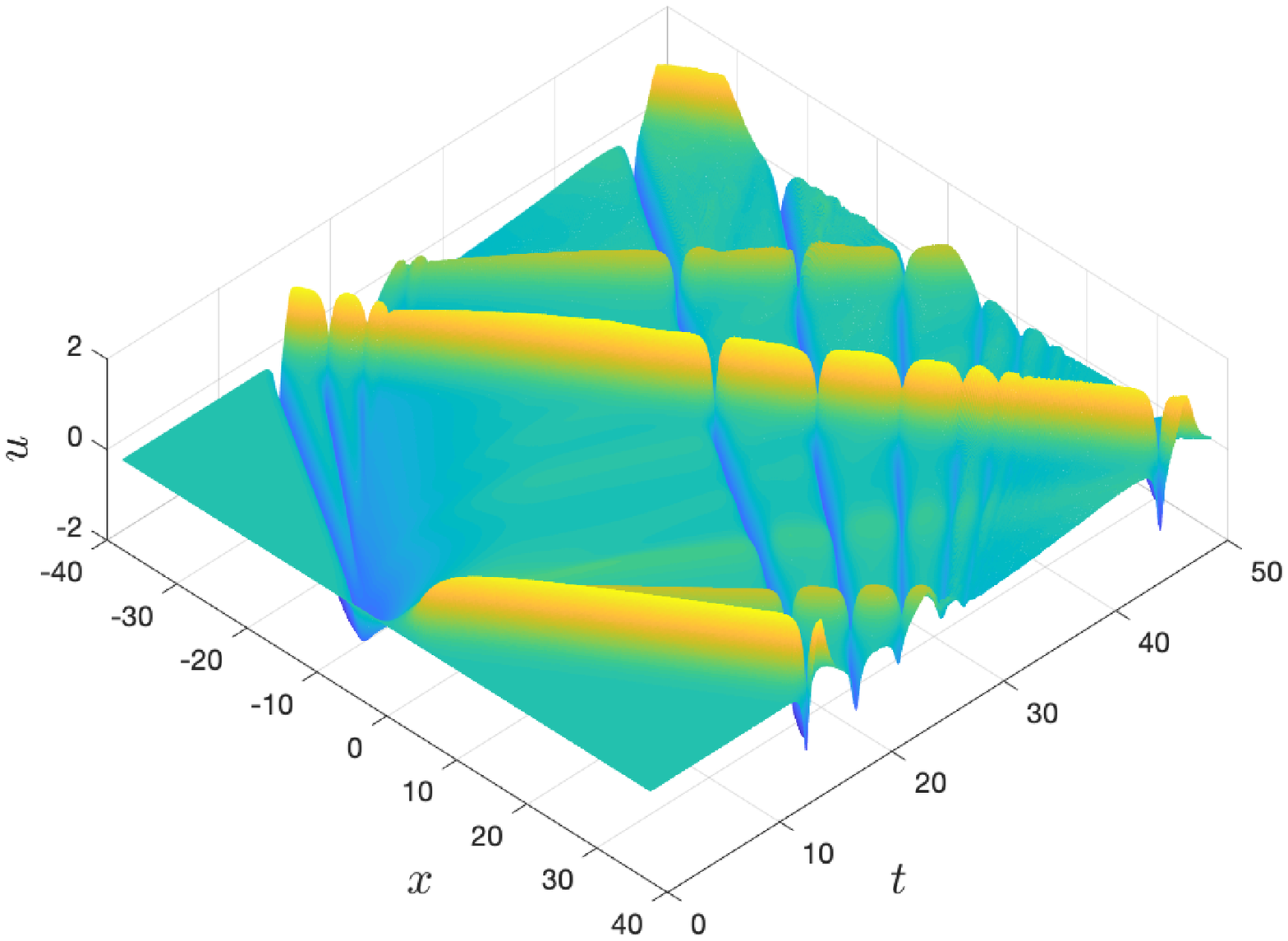}}
\hspace{-20pt}
\subfigure[Case E-$\rho$, view(45,70)\,]{ \centering
\includegraphics[width=0.25\textwidth]{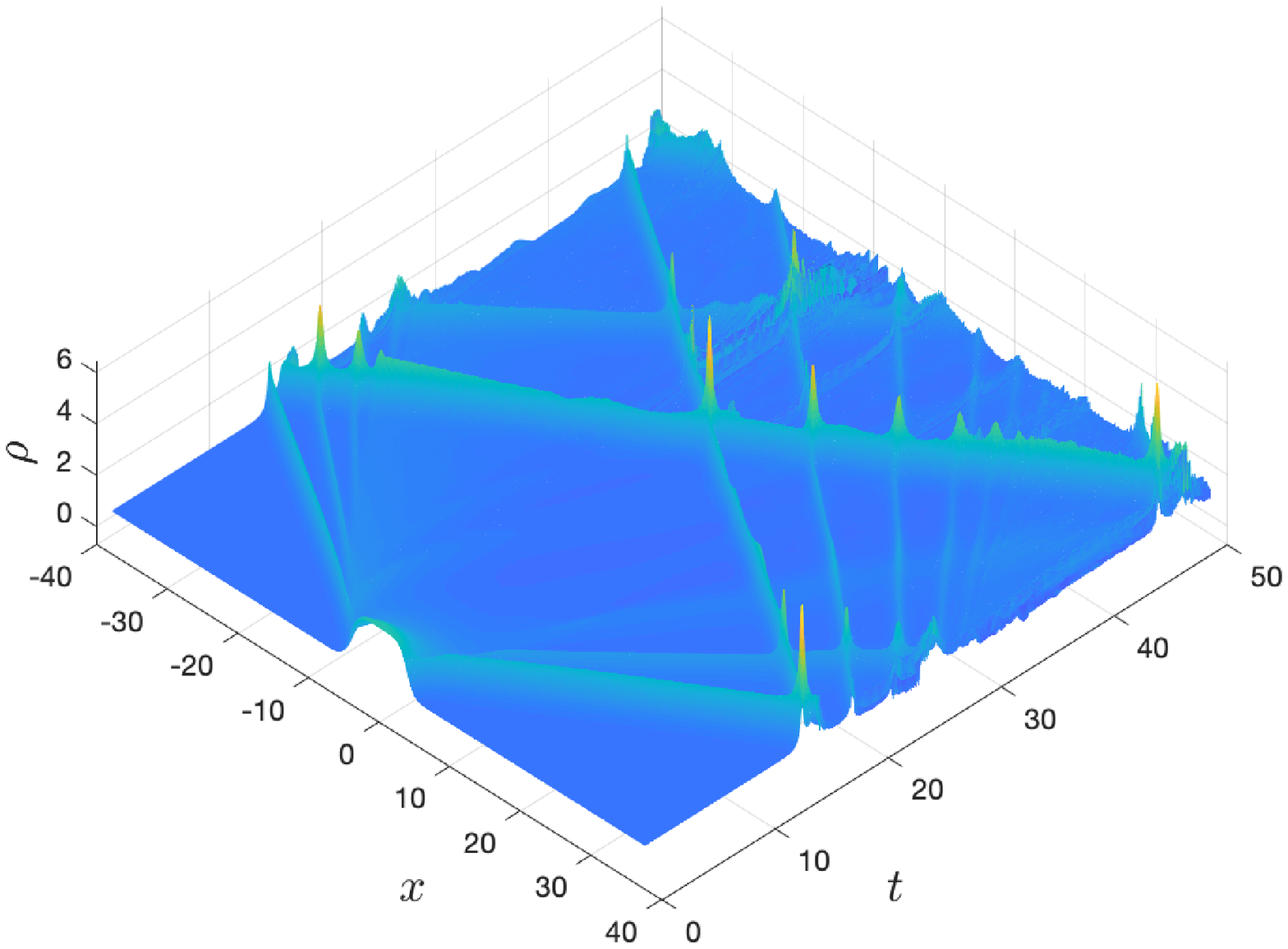}
}\\
\subfigure[Case E-$u$, contour\,]{ \centering
\includegraphics[width=0.25\textwidth]{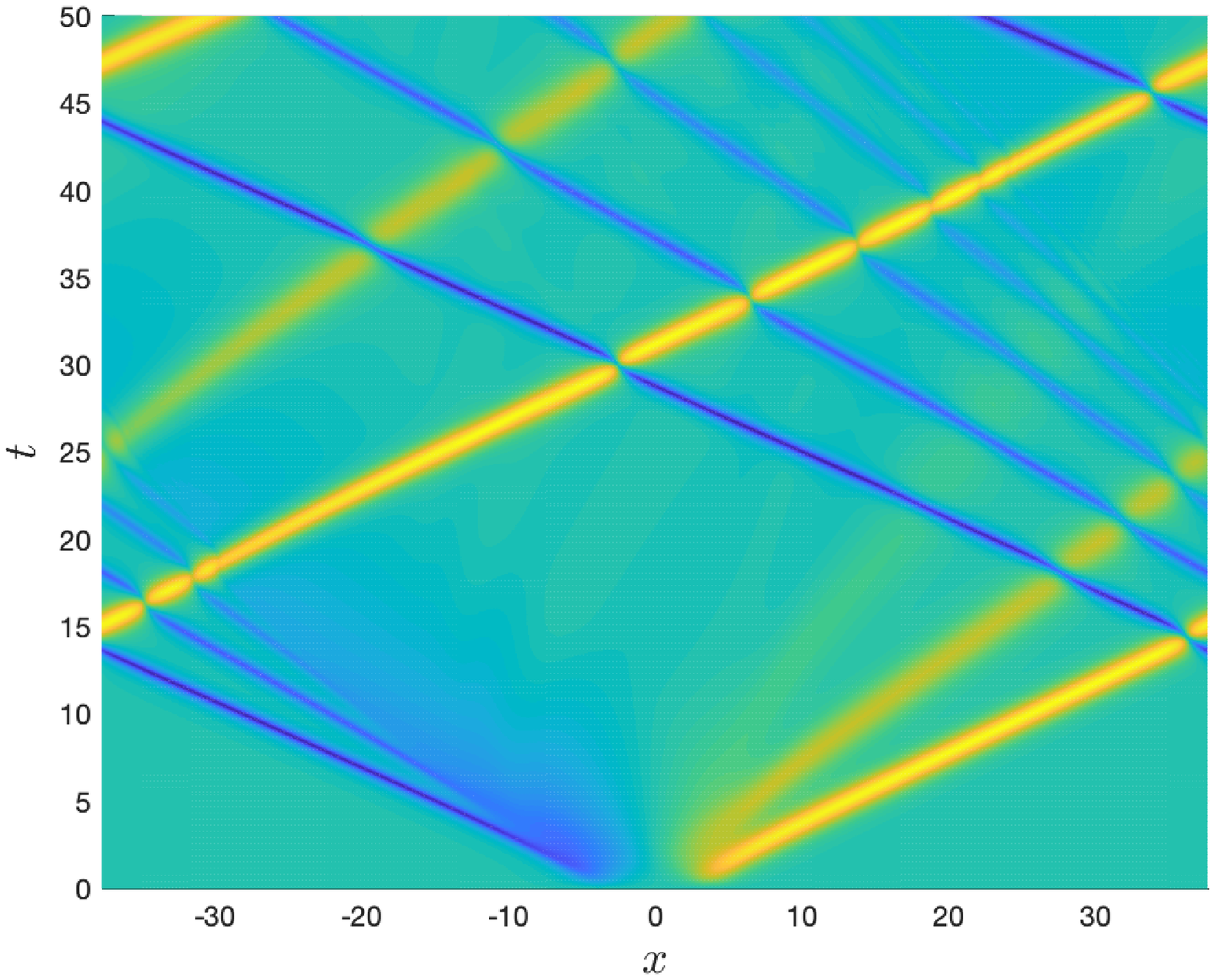}}
\hspace{-20pt}
\subfigure[Case E-$\rho$, contour\,]{ \centering
\includegraphics[width=0.25\textwidth]{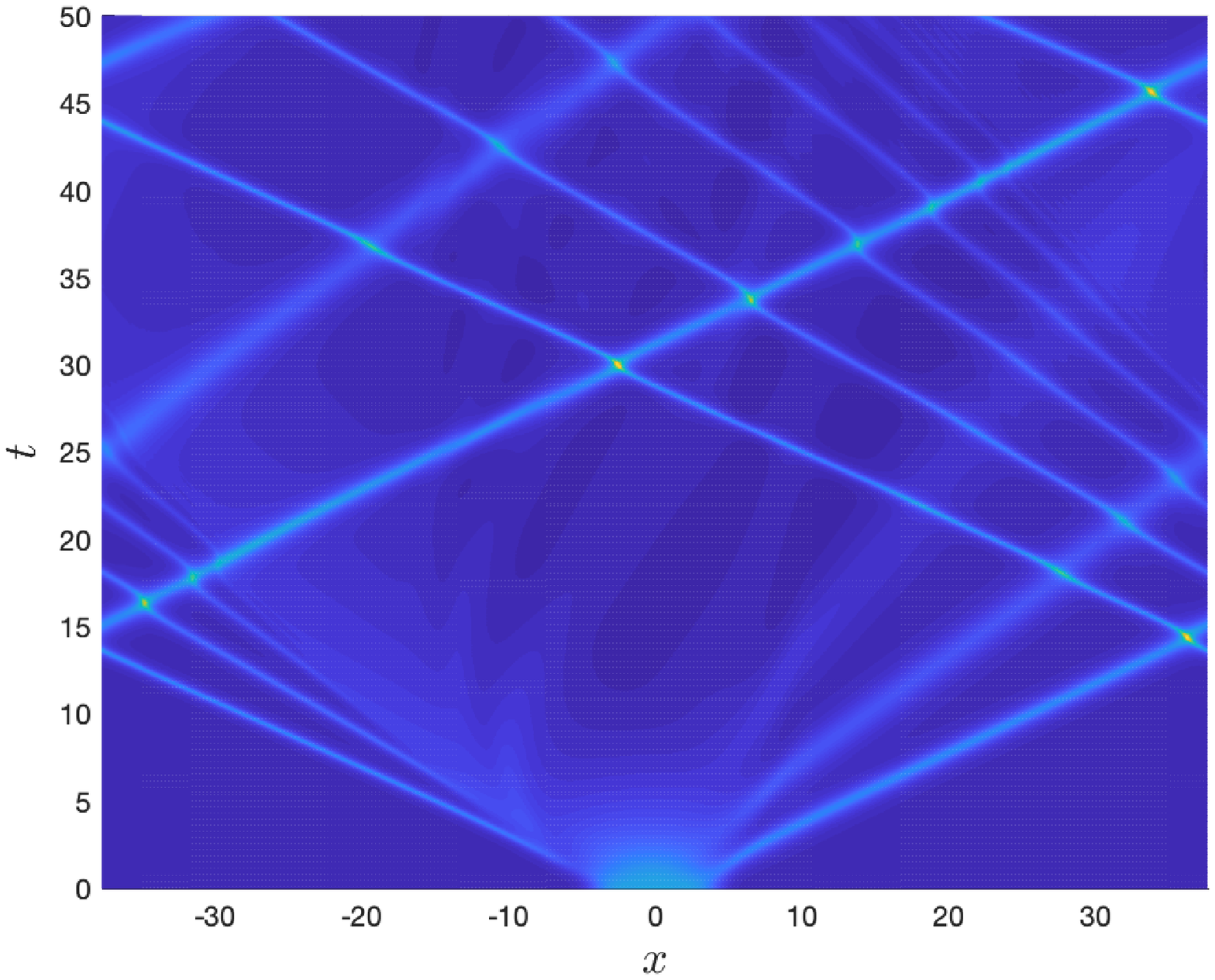}
}\hspace{-18pt}
\subfigure[Case E-$u$, contour\,]{ \centering
\includegraphics[width=0.25\textwidth]{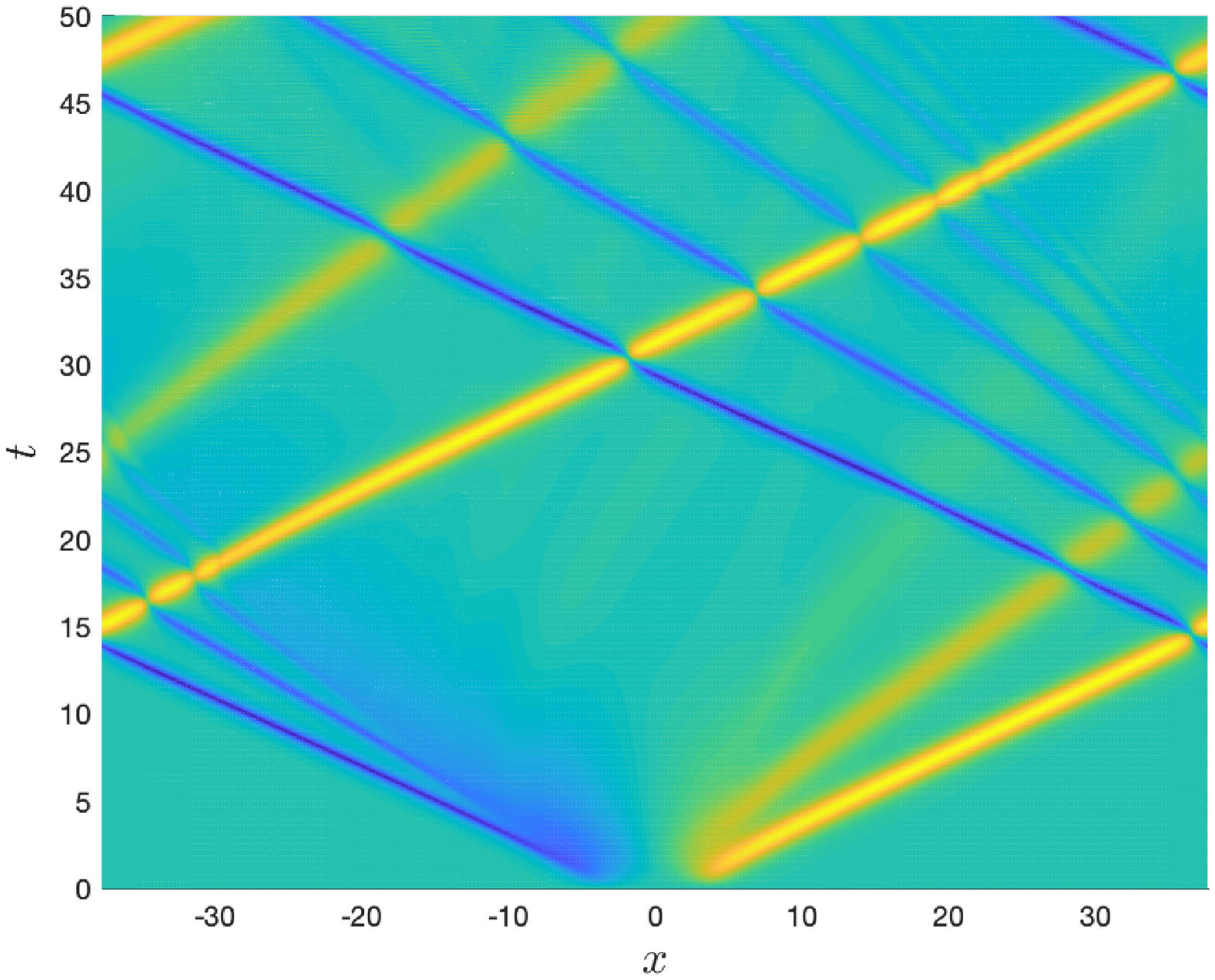}}
\hspace{-20pt}
\subfigure[Case E-$\rho$, contour\,]{ \centering
\includegraphics[width=0.25\textwidth]{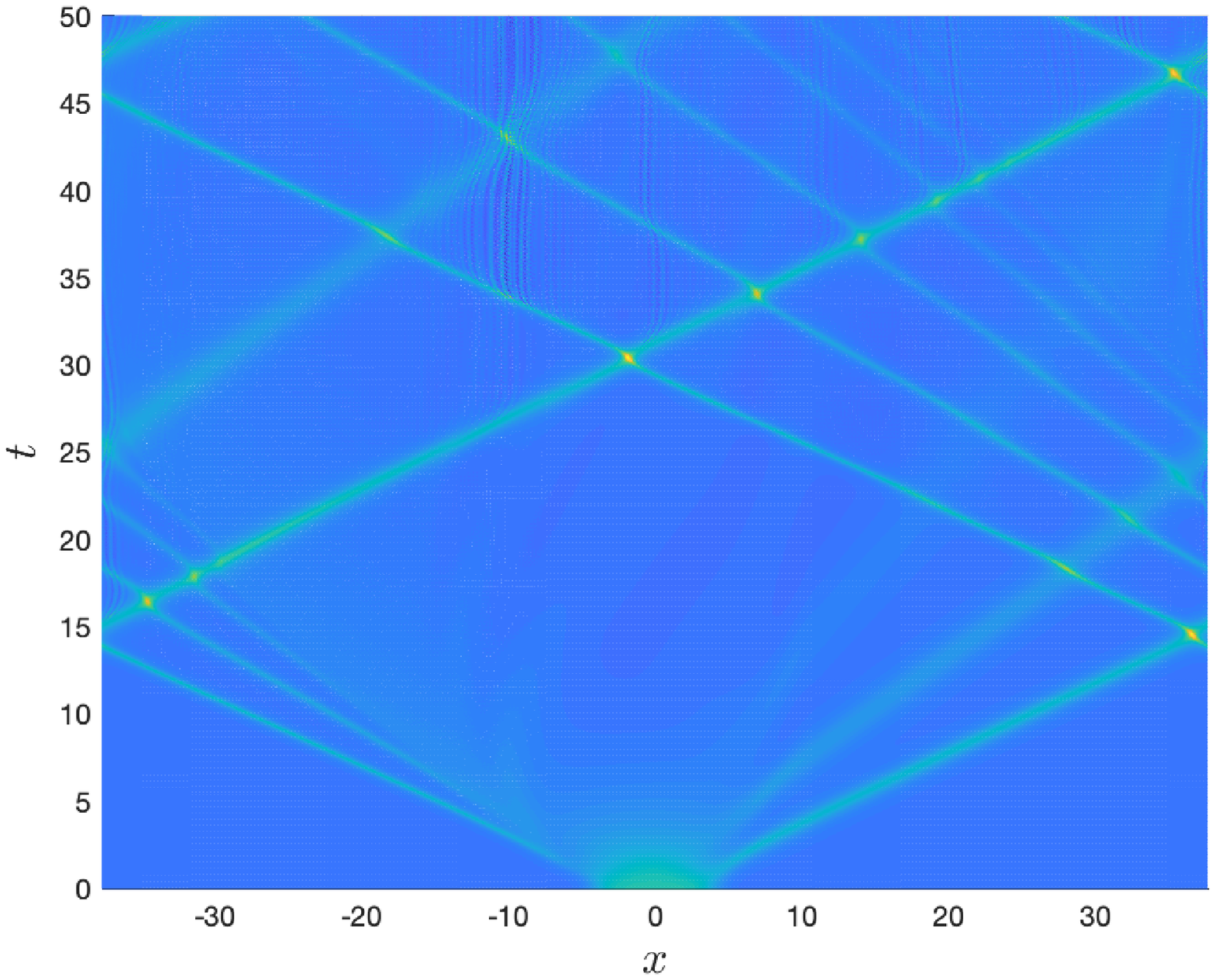}
}
\caption{The portraits of the velocity and altitude in long-time simulation with same stepsizes $h=1/16$ and $\tau=1/20$ in \textbf{Case E}. Upper: (a) and (b) calculated by the difference scheme \eqref{equa3.7}, (c) and (d) calculated by the difference scheme in \cite{ZLZ2022}; Lower: corresponding contours.} \label{fig:2}
\end{figure}

\begin{figure}[htbp]
\subfigtopskip=2pt
\subfigbottomskip=2pt
\subfigcapskip=-3pt
\subfigure[$u$, view(45,70)]{ \centering
\includegraphics[width=0.25\textwidth]{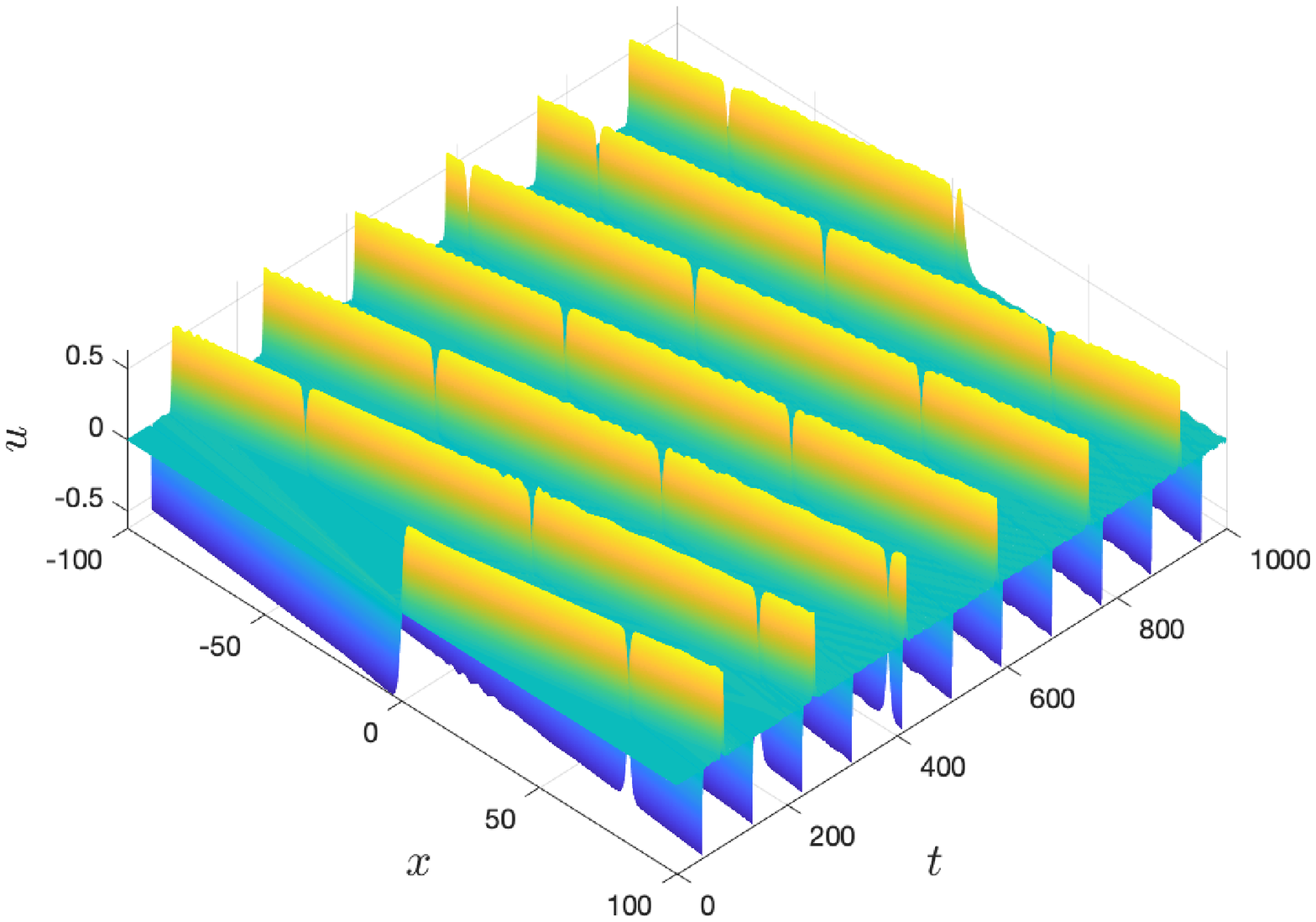}
}\hspace{-20pt}
\subfigure[$\rho$, view(45,70)]{ \centering
\includegraphics[width=0.25\textwidth]{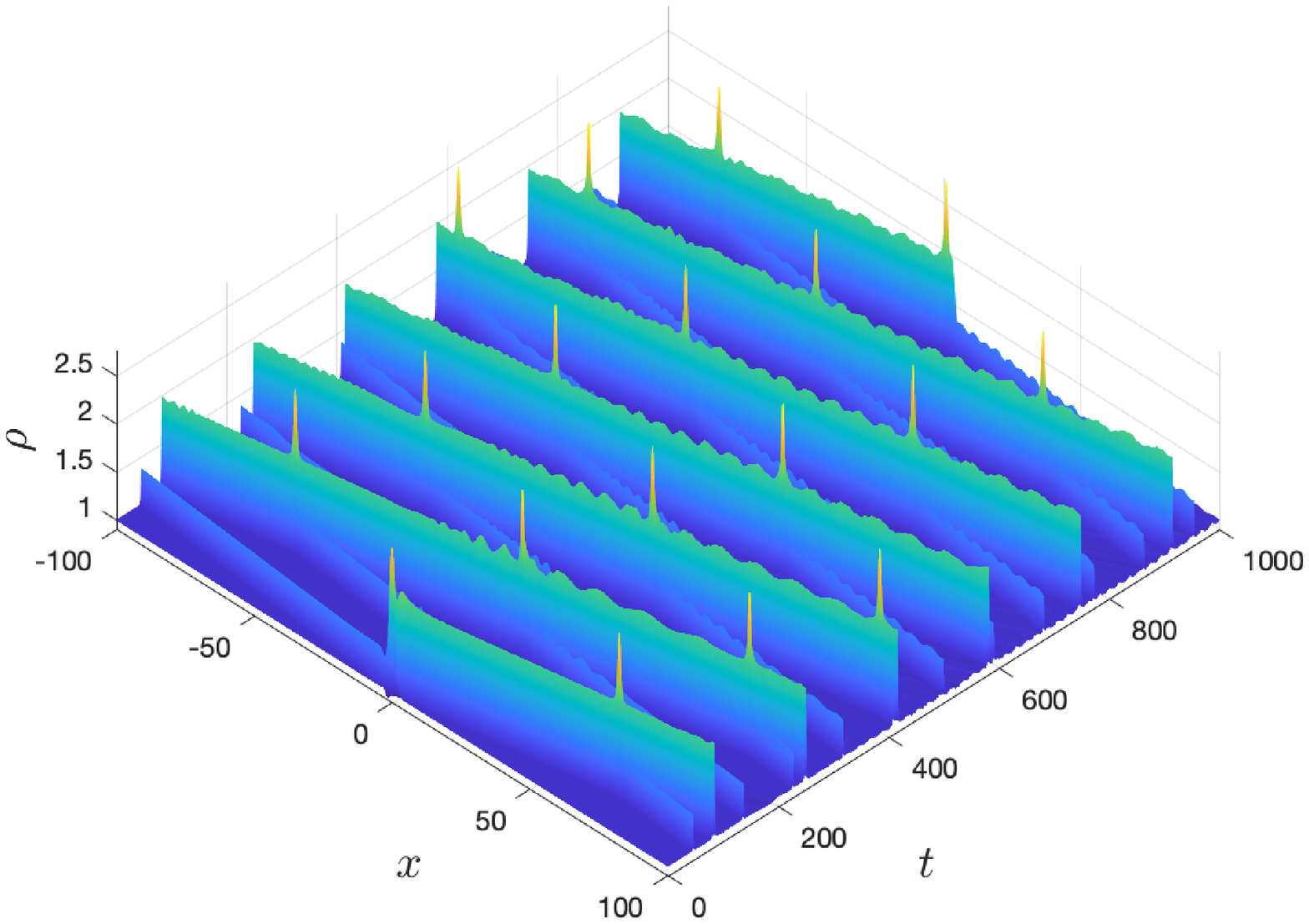}
}\hspace{-20pt}
\subfigure[$u$, view(45,70)]{ \centering
\includegraphics[width=0.25\textwidth]{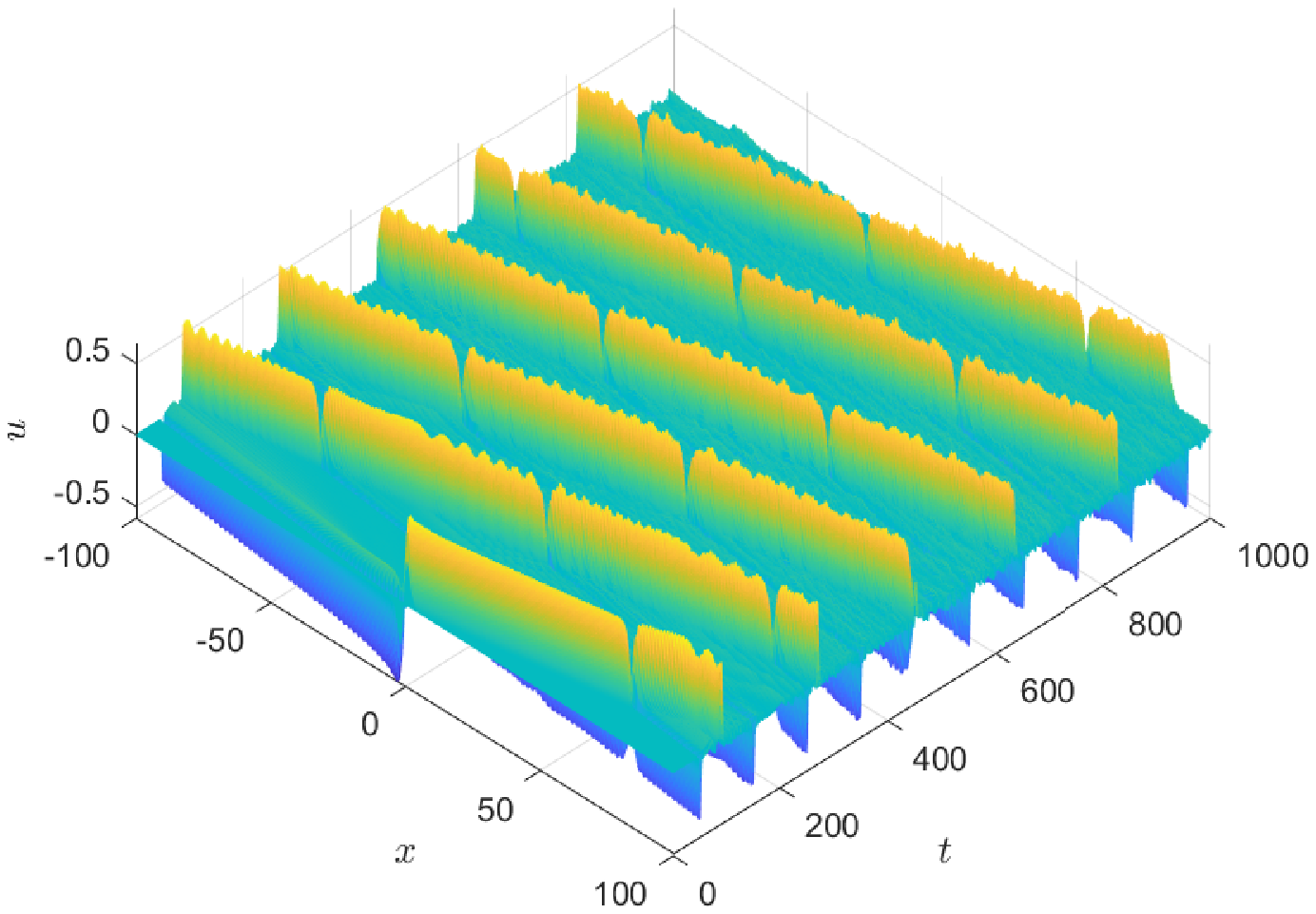}
}\hspace{-20pt}
\subfigure[$\rho$, view(45,70)]{ \centering
\includegraphics[width=0.25\textwidth]{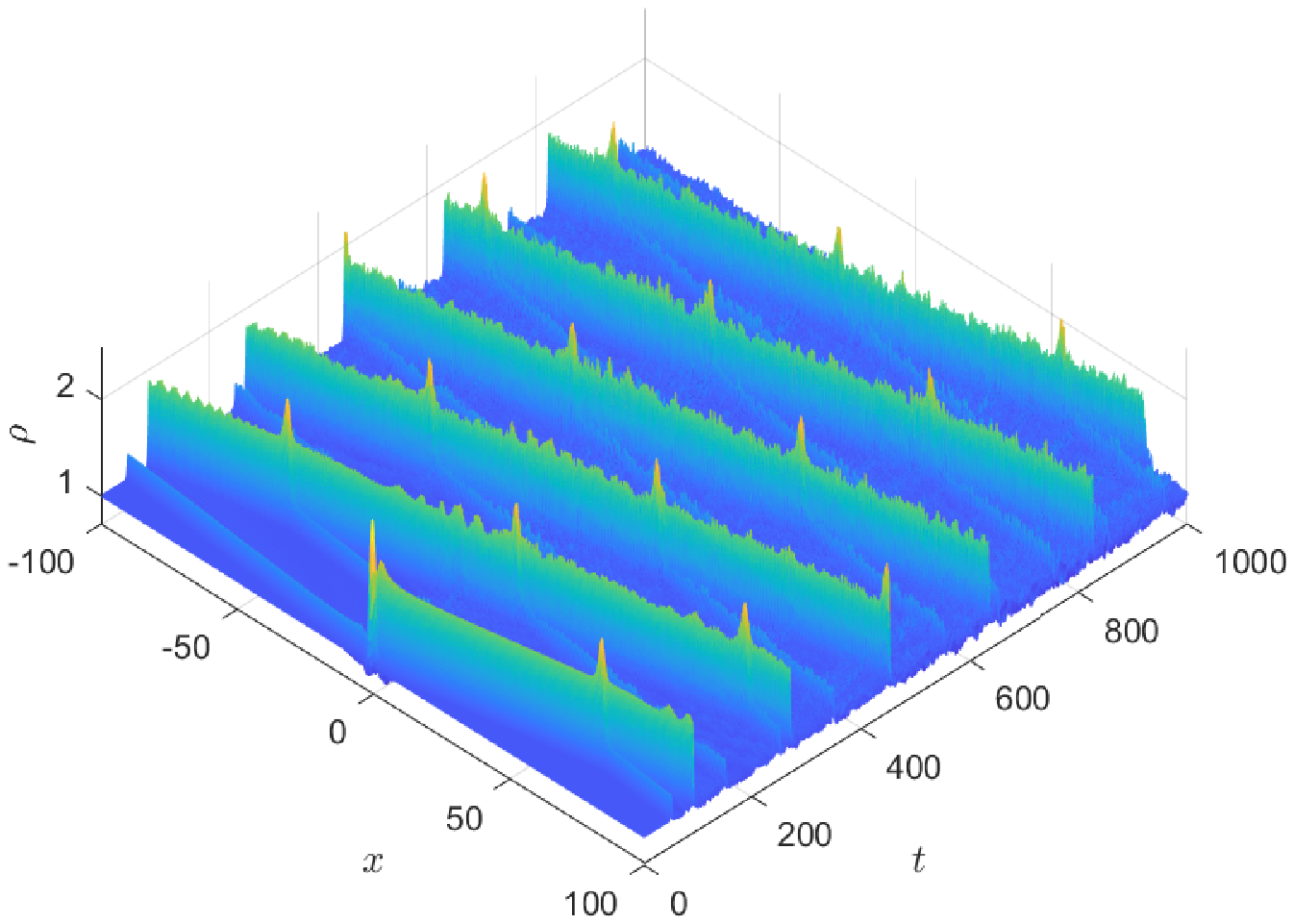}
}
\\
\subfigure[$u$-contour]{ \centering
\includegraphics[width=0.25\textwidth]{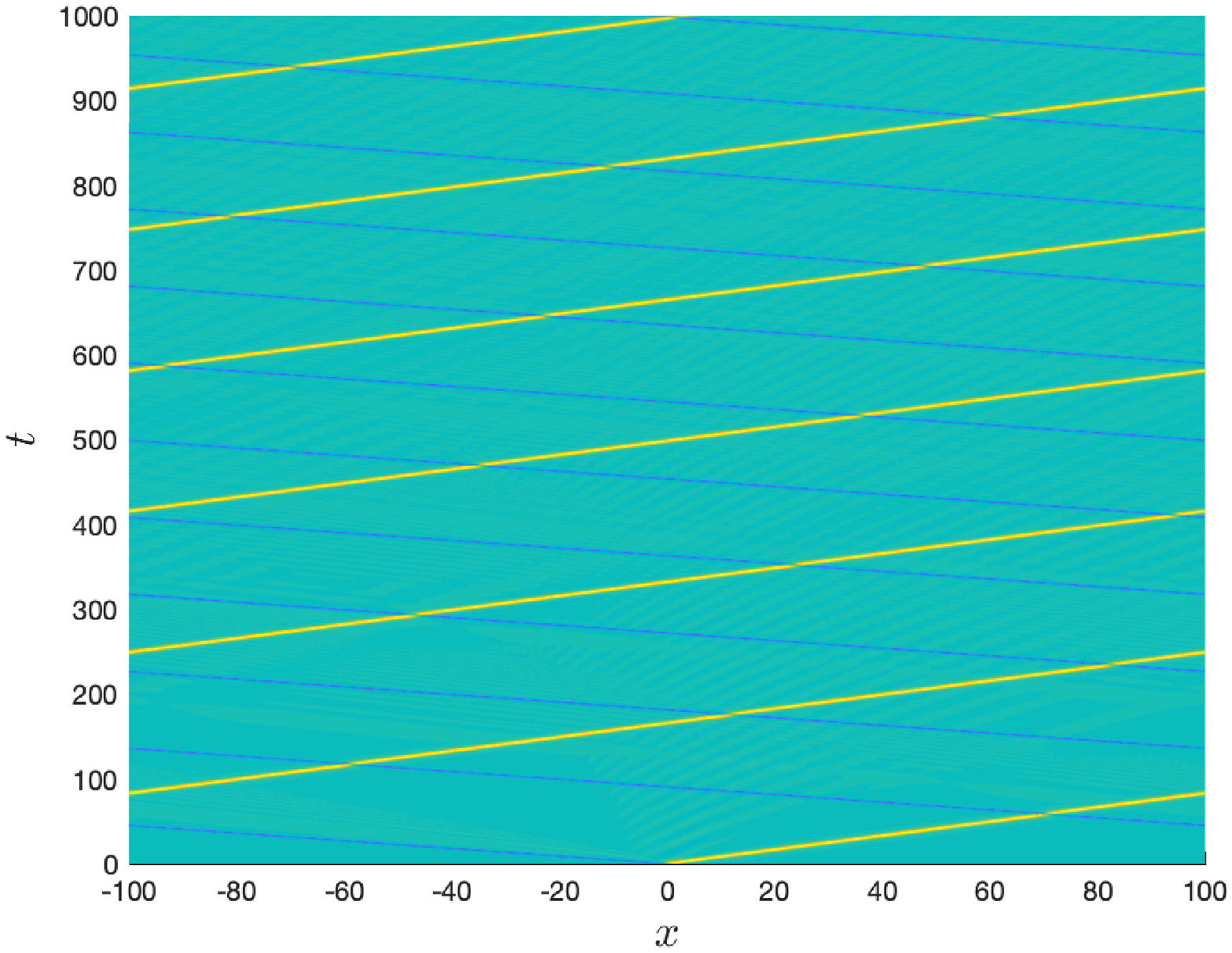}
} \hspace{-20pt}
\subfigure[$\rho$-contour]{ \centering
\includegraphics[width=0.25\textwidth]{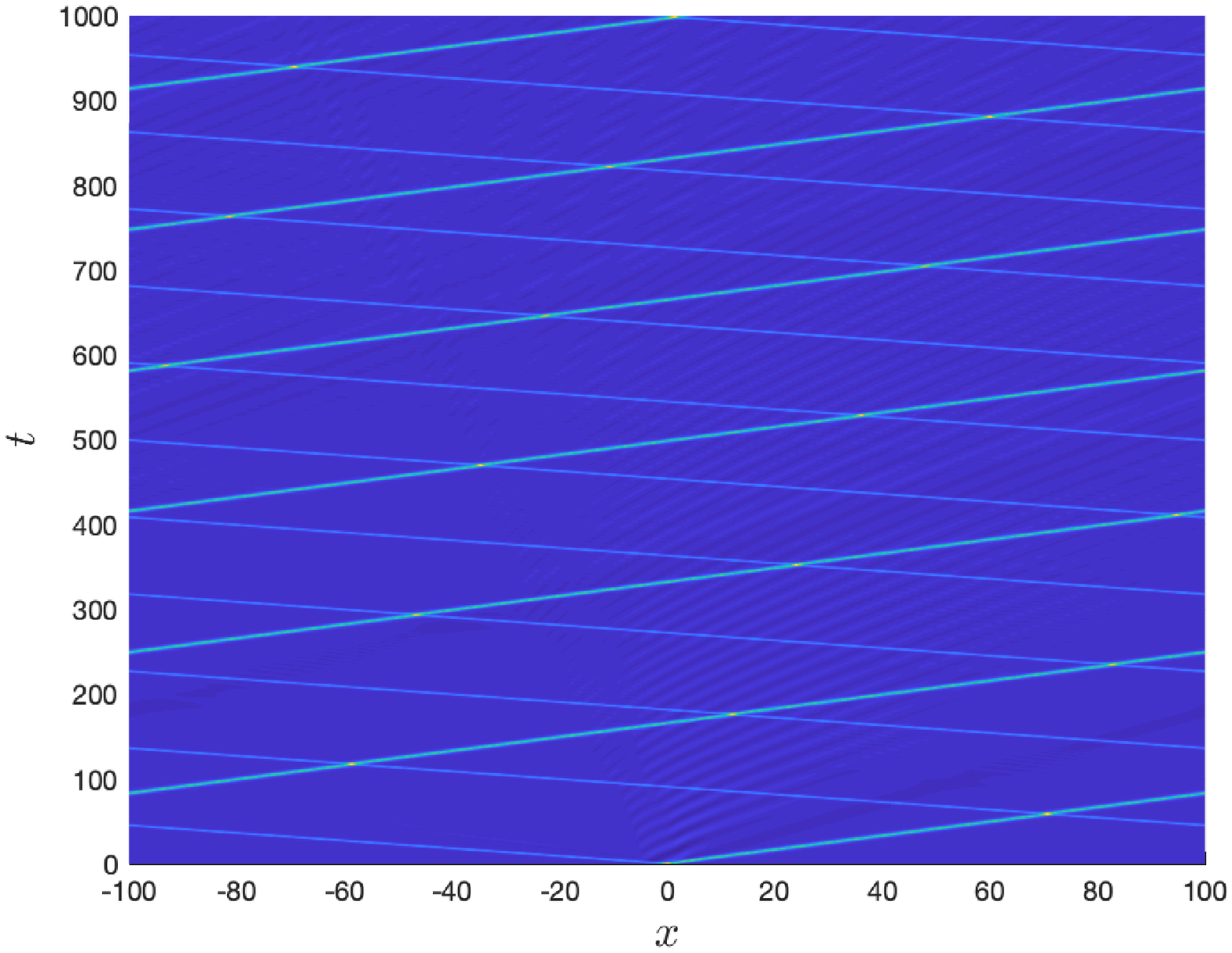}
}\hspace{-20pt}
\subfigure[$u$-contour]{ \centering
\includegraphics[width=0.25\textwidth]{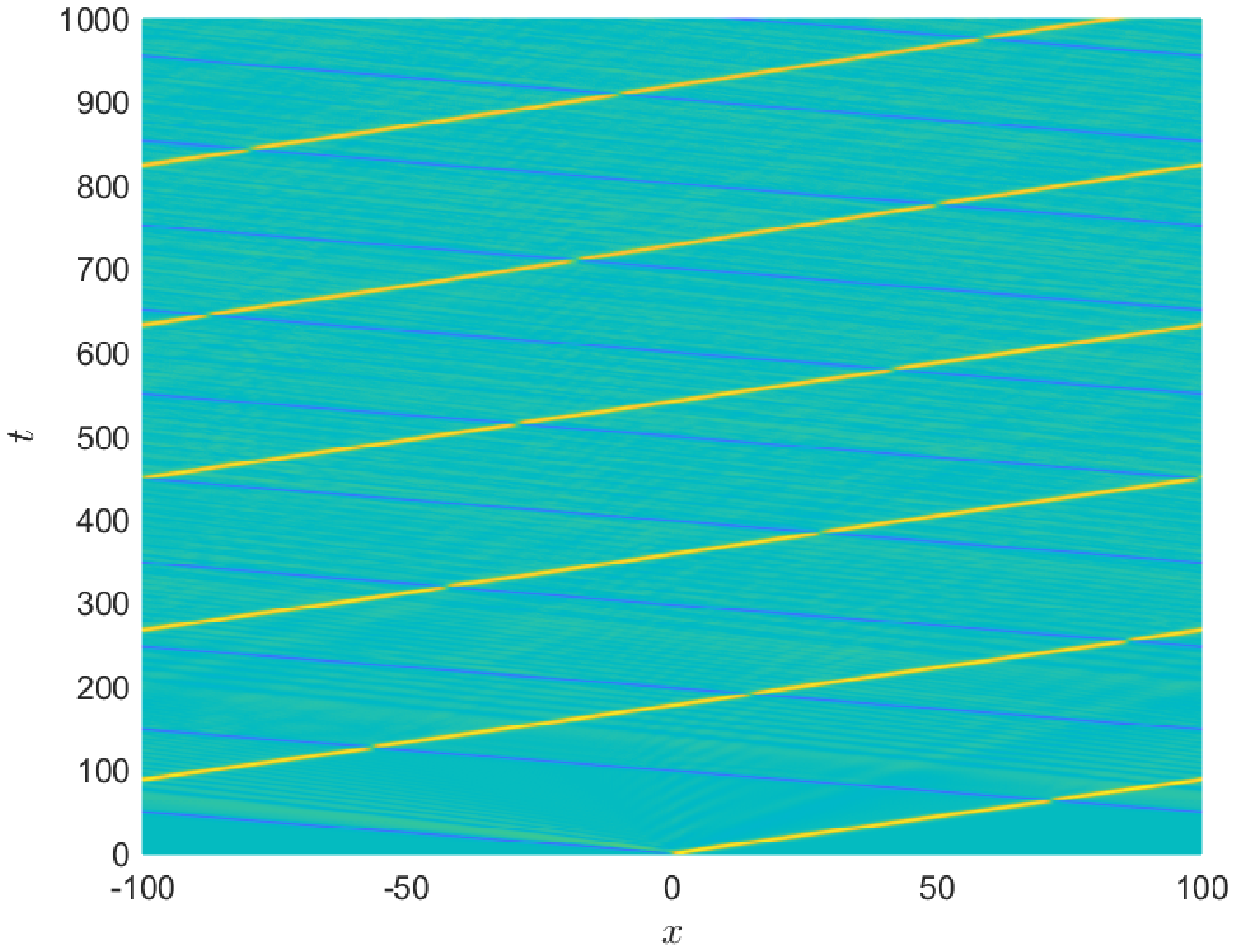}
} \hspace{-20pt}
\subfigure[$\rho$-contour]{ \centering
\includegraphics[width=0.25\textwidth]{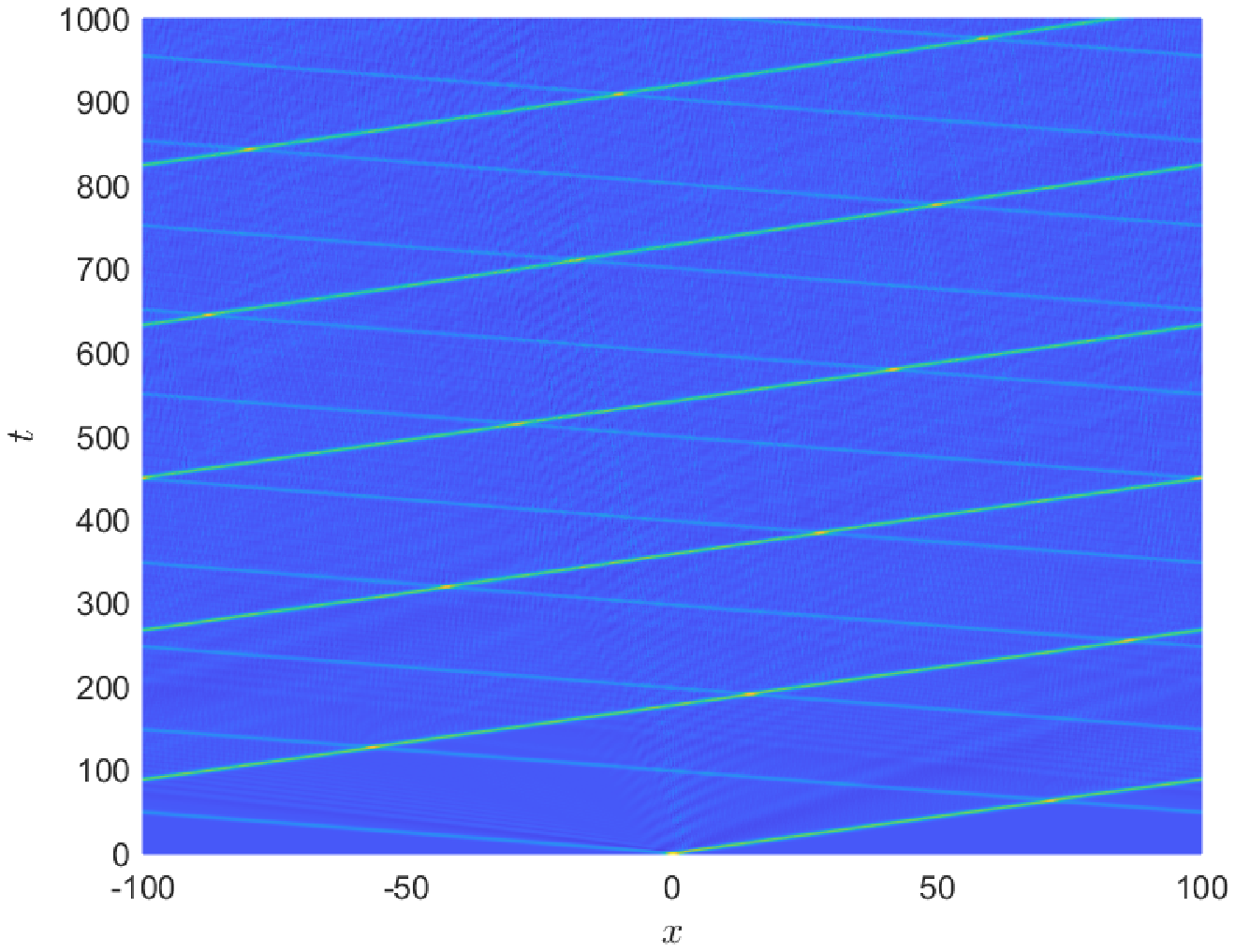}
}
\\
\subfigure[Numerical invariants]{ \centering
\includegraphics[width=0.48\textwidth]{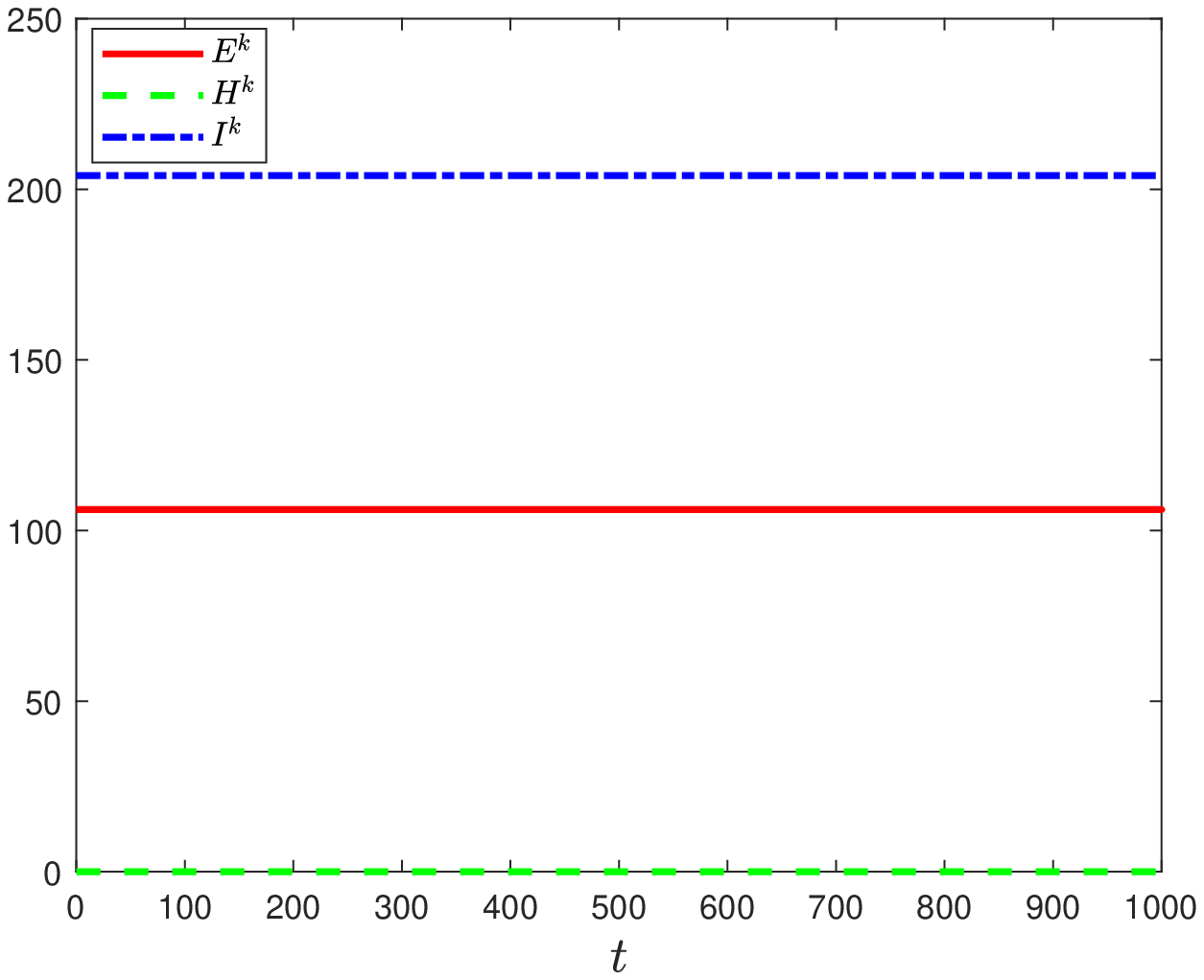}}
\hspace{-25pt}
\subfigure[Errors of the numerical invariants]{ \centering
\includegraphics[width=0.48\textwidth]{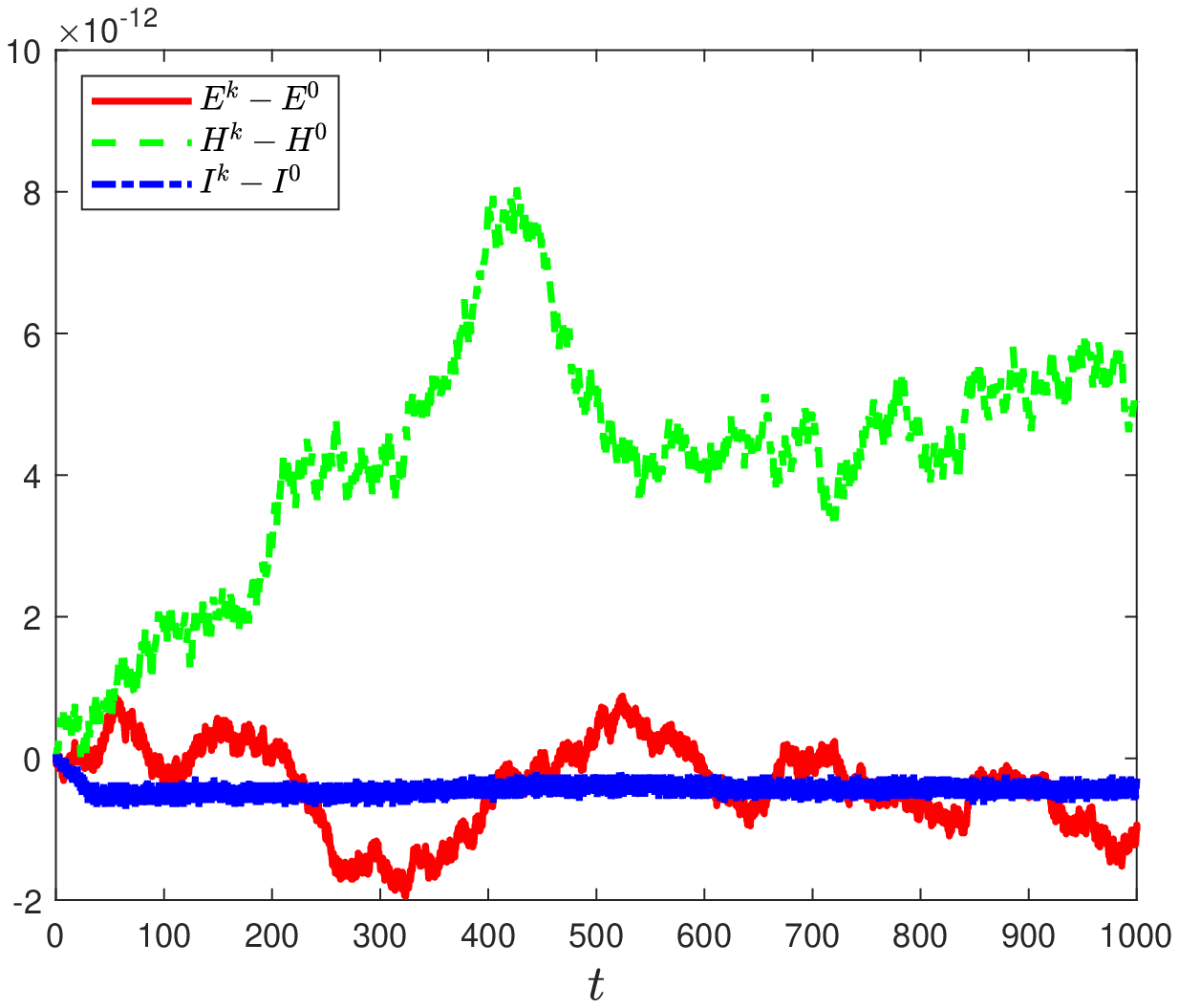}
}
\caption{The portraits of the velocity and altitude in long-time simulation calculated by the difference scheme \eqref{equa3.7} (a)--(b), and three-level linearized difference scheme (c)--(d) with same stepsizes $h=1/10$, $\tau=1/50$ in \textbf{Case F}; (e)--(h) denote the corresponding contours; (i)--(j): Numerical invariants, and relevant error curves with the same grid stepsizes for the difference scheme \eqref{equa3.7}.} \label{fig:3}
\end{figure}

 \textbf{Convergent accuracy test in Cases A--F}.
  Table \ref{Table1} displays the numerical errors and spatial convergence orders when the temporal stepsize $\tau=1/50$ is fixed for \textbf{Case A} and \textbf{Case B}. It is observed that the finite difference scheme is quadratically convergent in the temporal direction. These data are consistent with the theoretical result in Theorem \ref{theorem3 1} for case of $\Omega=0$. The second part of Table \ref{Table1} displays the numerical errors in space and convergence orders in \textbf{Cases C} and \textbf{D} with a fixed and tiny temporal stepsize $\tau=1/1000$ such that the temporal error may be negligible. Similar results are observed, which are consistent with Theorem \ref{theorem3 2} for case of $\Omega\neq 0$.

  Tables \ref{Table3} displays the temporal error behavior and convergence orders with the fixed spatial stepsize $h=6/25$ for \textbf{Cases A} and \textbf{B}, and $h=4/25$ for \textbf{Cases C} and \textbf{D}, respectively. Clearly, the temporal convergence orders approach to two whether $\Omega=0$ or $\Omega\neq 0$. These data are also consistent with Theorems \ref{theorem3 1} and \ref{theorem3 2}.

 \textbf{Invariants test in Cases A--D}. We next study the long-time behaviors of solutions for the difference scheme \eqref{equa3.7}. Table \ref{Table3} shows the discrete energy, momentum, and mass for case of $\Omega=0$, in which the computations are performed with $h=1/5$ and $\tau=1/256$. We observe that all the invariants are conserved very well, even for long-time dynamics. Compared with the results in \cite{ZLZ2022}, the new method preserves one more invariant: momentum. It should be noticed that the computed energy is approximately half of that in \cite{ZLZ2022} due to the different definition of the discrete energy.

 When $\Omega=73\times 10^{-6}$, Table \ref{Table3} also shows similar discrete invariants for \textbf{Cases C} and \textbf{D}, which are consistent with Theorems \ref{theorem1_1}. We also observe from Table \ref{Table3} that the discrete momentum is close to zero when $\Omega=0$, while it is non-zero when $\Omega=73\times 10^{-6}$.

\textbf{Unconditional convergence test in Cases A--D}. To explore the unconditional convergence of \eqref{equa3.7}, we take \textbf{Case A} and \textbf{Case D} as examples for illustration.  Figures \ref{fig:1}(a) and \ref{fig:1}(b) demonstrate the convergence behavior of the velocity and height in \textbf{Case A} with the spatial grid size reduction when the fixed temporal stepsizes $\tau = 5/16$, $5/32$, $5/64$, $5/128$, $5/256$ are used, respectively. We see that no matter how the spatial stepsize varies, the numerical errors always approach fixed values, which confirms that the difference scheme \eqref{equa3.7} converges unconditionally (namely, no grid ratio restriction. Otherwise, the numerical error would increase steeply with the refined spatial grid). The similar phenomenon except minor differences is observed in \textbf{Case D} with the fixed temporal stepsizes $\tau = 1/16$, $1/32$, $1/64$, $1/128$, $1/256$. \textbf{Cases B} and \textbf{C} also show a similar phenomenon, in which unconditionally convergent figures are omitted for brevity.

\textbf{Long time behavior test on large domain in Cases E--F}. We first compare the simulation effect of the difference scheme \eqref{equa3.7} with that in \cite{ZLZ2022} under the same coarser grids $h=1/16,\;\tau = 1/20$ (at this point, ${\tau}/{h} = 0.8<1$, see Figure \ref{fig:2}), and we see that the solution portraits in the present paper are better than that generated by the difference scheme in \cite{ZLZ2022}. To further study the performance of \eqref{equa3.7} and that in \cite{ZLZ2022}, we conduct numerical simulations for \textbf{Case F} on a larger area for a long-time (at the moment, $L=200$ and $T=1000$). As illustrated in Figure \ref{fig:3}, the numerical results show amazing periodicity for both schemes. However, the difference scheme \eqref{equa3.7} owns better resolution even using refined stepsizes $h=1/10,\;\tau = 1/50$ (${\tau}/{h} = 0.2$). Though the smooth initial data are utilized, we observe that the occurrence of the bounded peakon solutions. Figure \ref{fig:3}(i) and Figure \ref{fig:3}(j) denote numerical invariants of the difference scheme \eqref{equa3.7} and their numerical errors, which demonstrates superior simulation performance.

\begin{example}\label{exam2}
  \textbf{(Peakon anti-peakon interaction)} We consider a R2CH system with nonsmooth initial values defined by
  \begin{equation*}
    u^0(x) = {\rm e}^{-|x-5|}-{\rm e}^{-|x+5|},\quad \rho^0(x)= 0.5.
  \end{equation*}
   The problem depicts the interaction of peakons and anti-peakons. Special cases of the problem have been studied in the early literature on the  computational domain $[-20,20]$ with $\Omega = 0$, see e.g., \cite{CKL2020,CMR2014,LP2016,YFS2018}. In the circumstances, we have $L=40$. We will solve the underlying problem based on the difference scheme \eqref{equa3.7}
    for the following five cases, especially the case of $\Omega \neq 0$, see Table \ref{tabparameter} for the details. The numerical results are shown in Figure \ref{fig:4} and Figure \ref{fig:5}.
\begin{table}[tbh!]
\begin{center}
\renewcommand{\arraystretch}{1.25}
\tabcolsep 6pt \caption{The selected parameters in the numerical tests.}\label{tabparameter}
\def\temptablewidth{0.53\textwidth}
\rule{\temptablewidth}{1pt}
{\footnotesize
\begin{tabular*}{\temptablewidth}{l|l}
 & ${\rm  \qquad \qquad \qquad Parameters}$\\
\hline
{\;\rm \textbf{Case A}\quad}     &\qquad$\sigma=1$, $\kappa=\mu=\Omega=0;\qquad$\cite{CKL2020,CMR2014,LP2016,YFS2018}        \\
{\;\rm \textbf{Case B}\quad}    &\qquad$\sigma=1$, $\kappa=\mu=0$, $\Omega=0.2;\qquad$\\
{\;\rm \textbf{Case C}\quad}   &\qquad$\sigma=1$, $\kappa=\mu=0$, $\Omega=73\times 10^{-6};\qquad$ \\
{\;\rm \textbf{Case D}\quad}    &\qquad$\sigma=1$, $\kappa=1$, $\mu=0$, $\Omega=73\times 10^{-6};\qquad$\\
{\;\rm \textbf{Case E}\quad}     &\qquad$\sigma=1$, $\kappa=1$, $\mu=1$, $\Omega=73\times 10^{-6}.\qquad$\\
\end{tabular*}}
\rule{\temptablewidth}{1pt}
\end{center}
 \vspace{-5mm}
\end{table}
\end{example}

\begin{figure}[htbp]
\subfigtopskip=2pt
\subfigbottomskip=2pt
\subfigcapskip=-3pt
\subfigure[Case A, $t=1$]{ \centering
\includegraphics[width=0.26\textwidth]{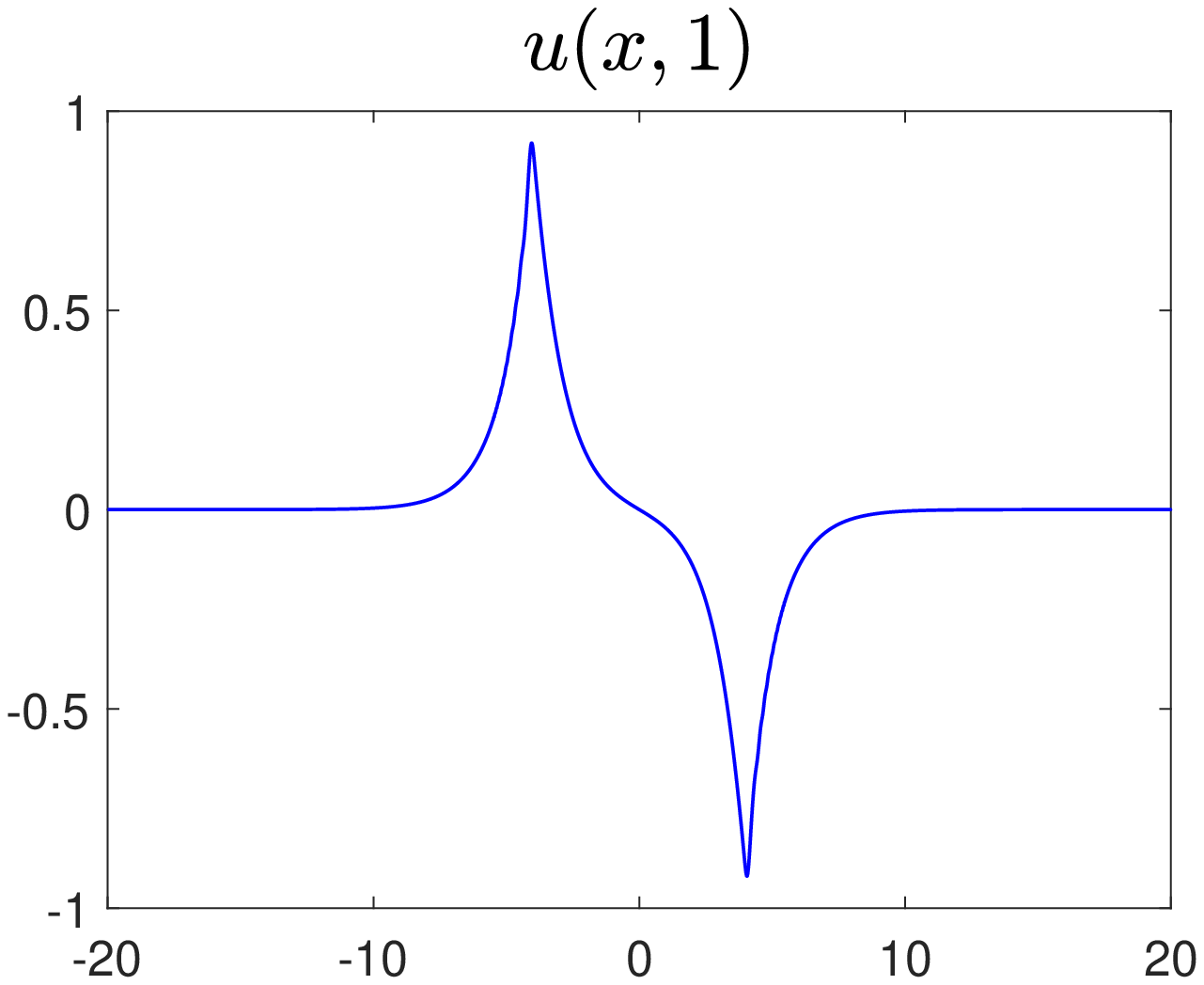}}
\hspace{-15pt}
\subfigure[Case A, $t=3$]{ \centering
\includegraphics[width=0.26\textwidth]{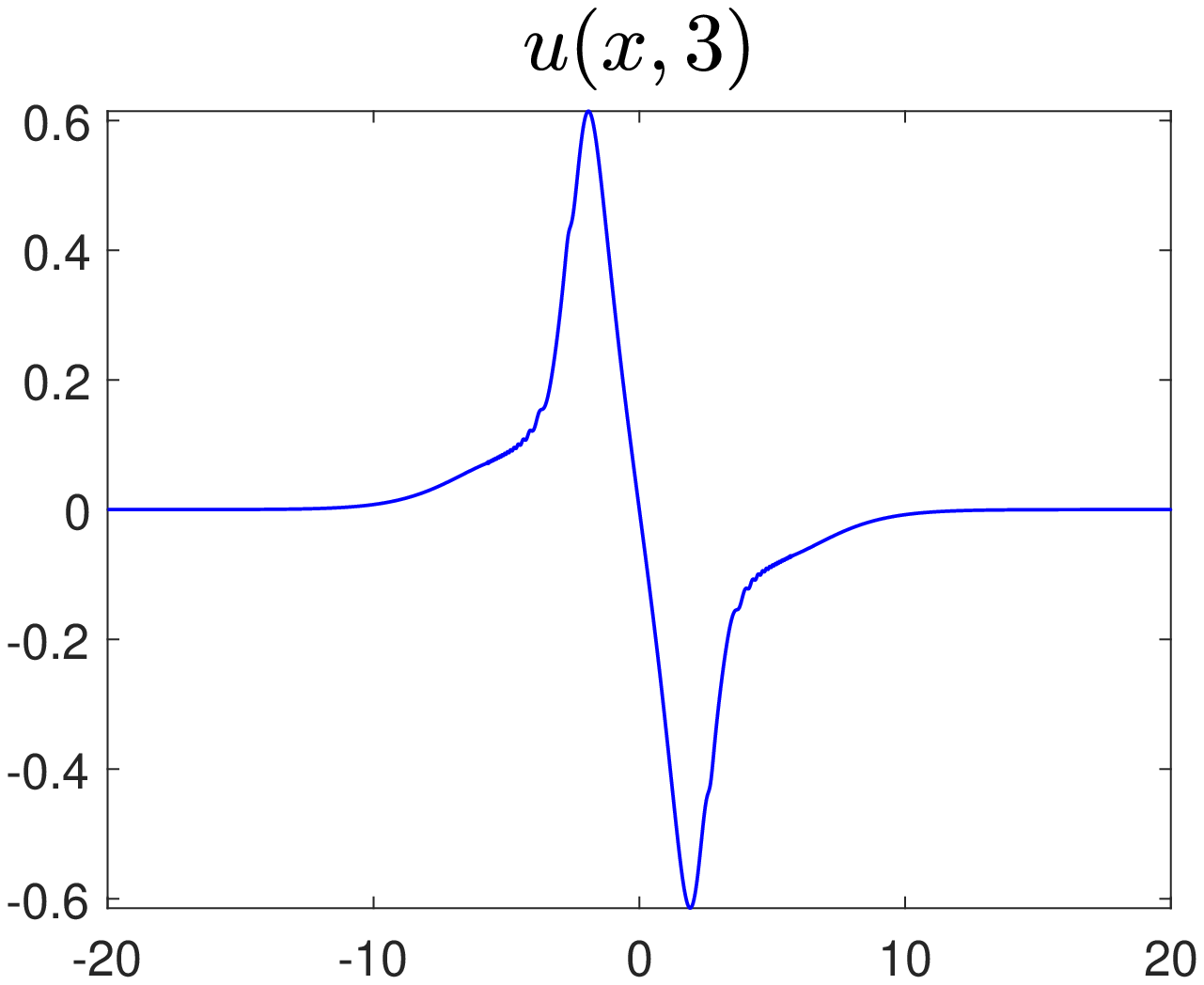}}
\hspace{-15pt}
\subfigure[Case A, $t=6$]{ \centering
\includegraphics[width=0.26\textwidth]{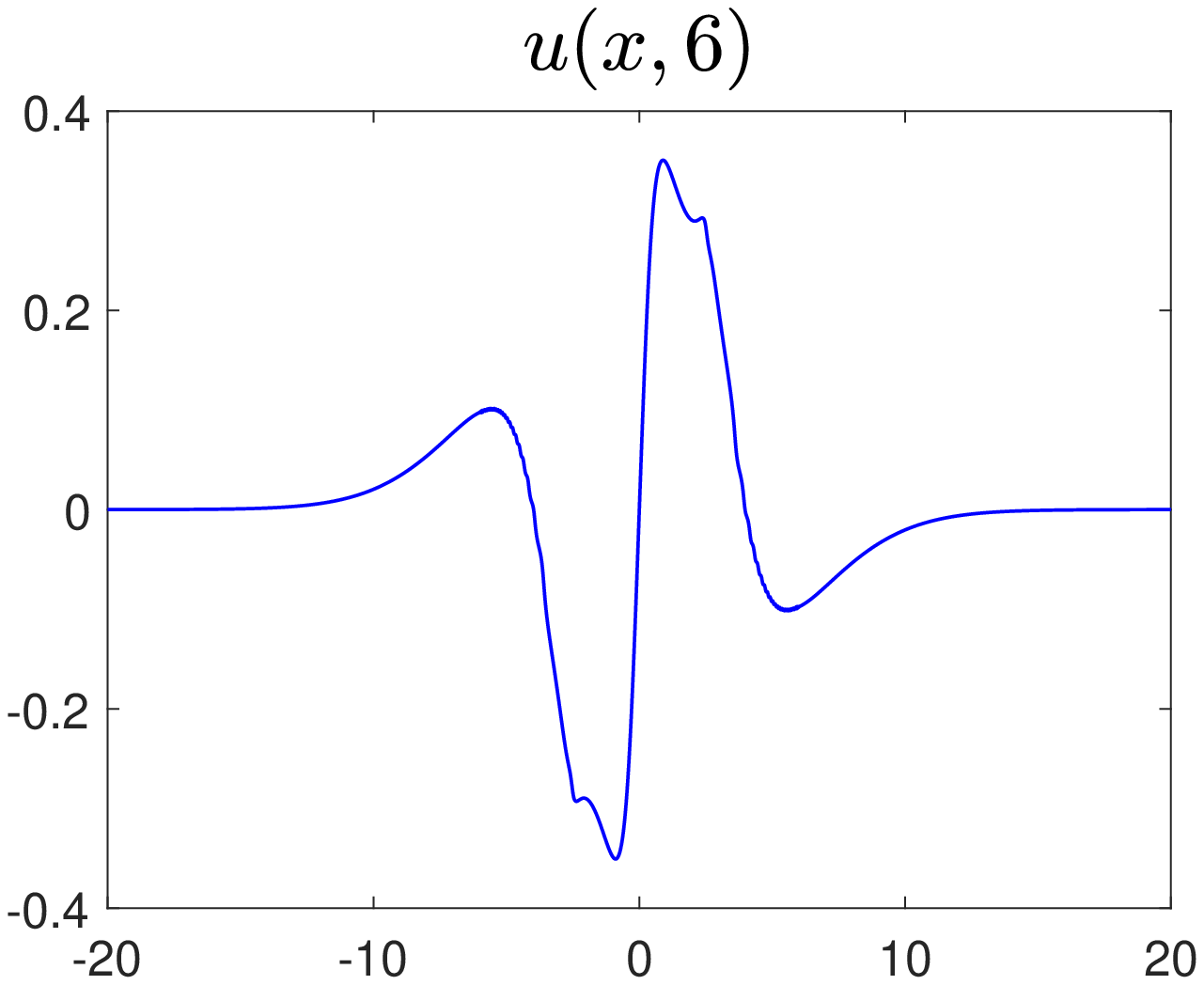}}
\hspace{-15pt}
\subfigure[Case A, $t=8$]{ \centering
\includegraphics[width=0.26\textwidth]{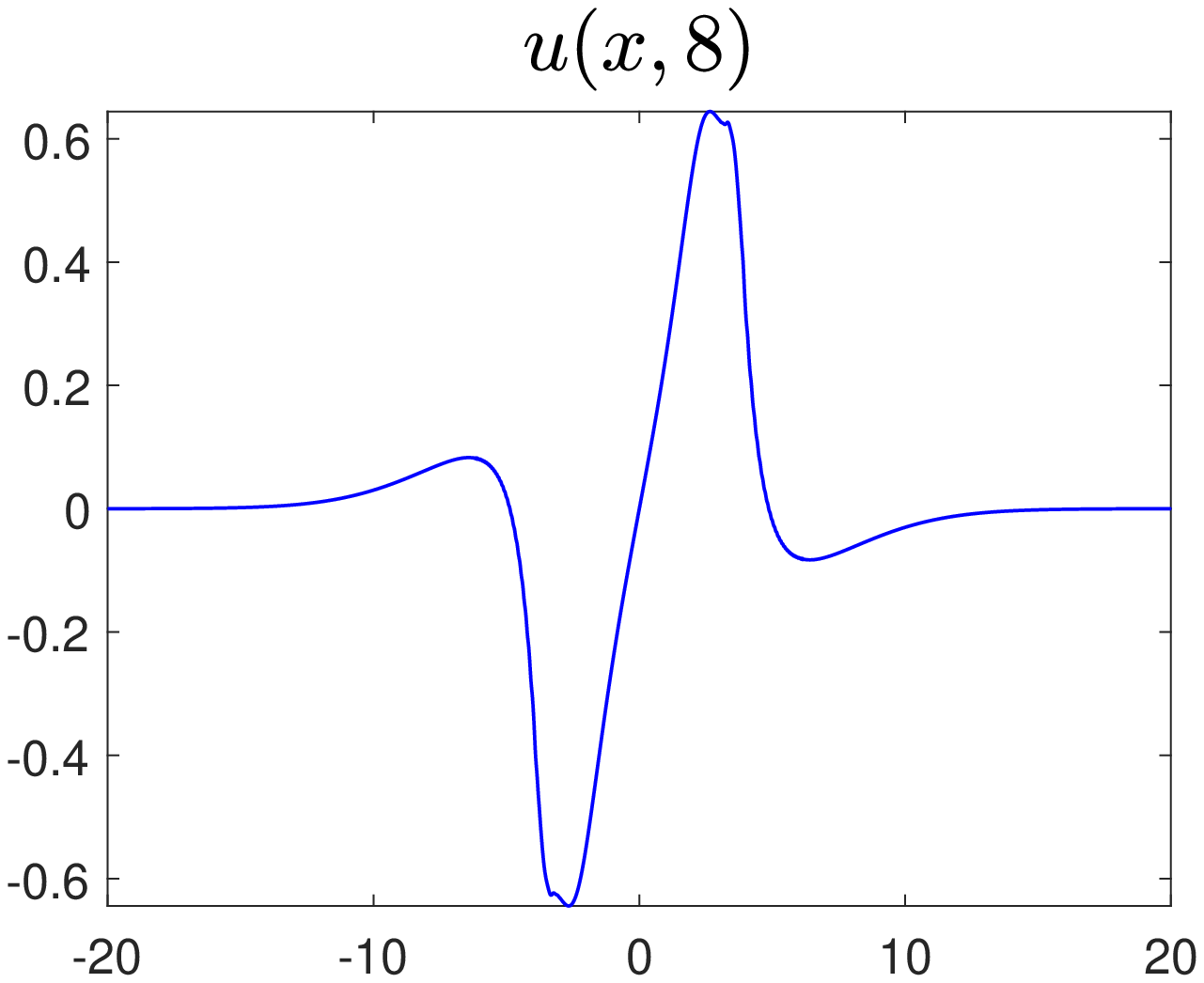}}
\\
\subfigure[Case B, $t=1$]{ \centering
\includegraphics[width=0.26\textwidth]{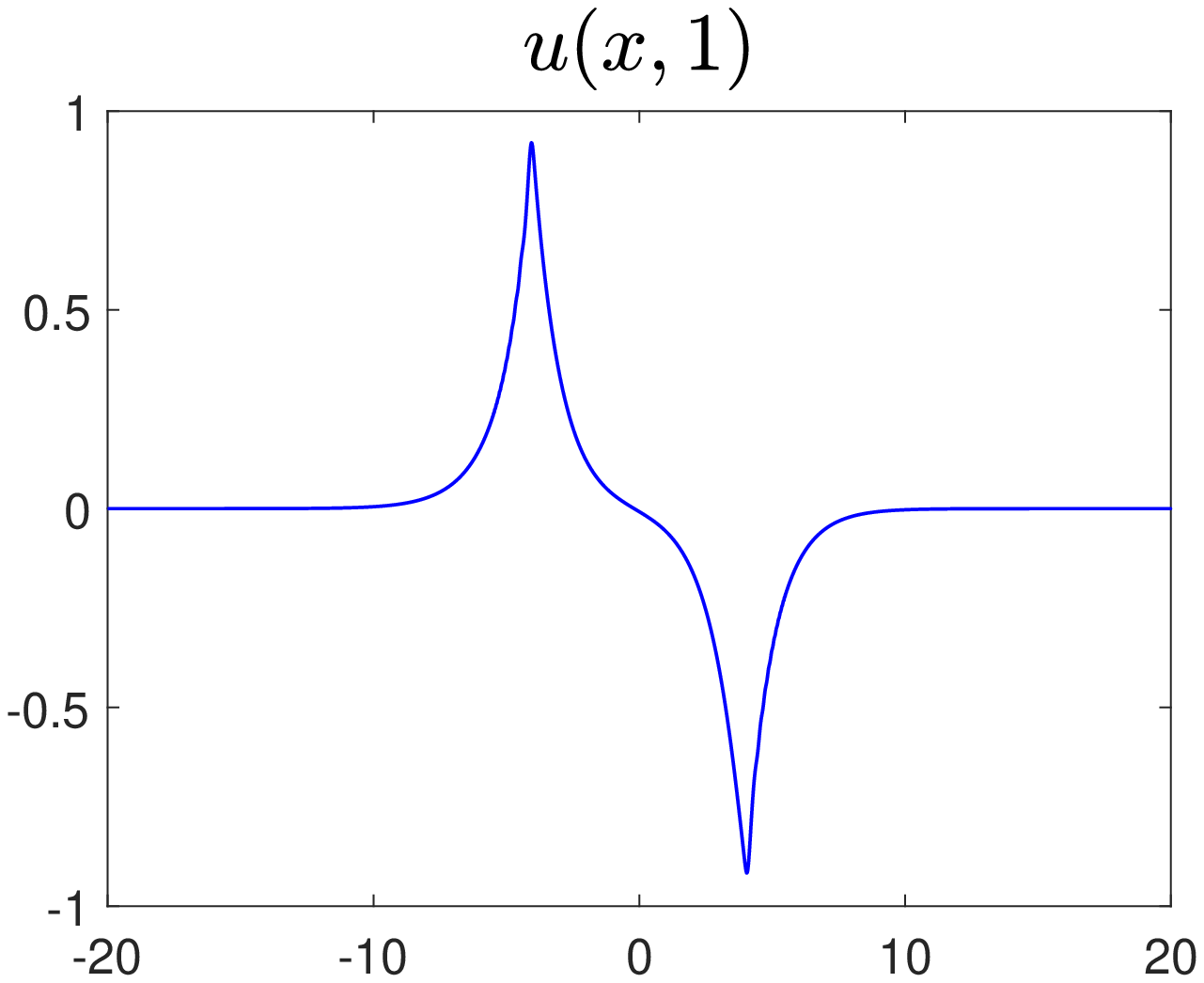}}
\hspace{-15pt}
\subfigure[Case B, $t=3$]{ \centering
\includegraphics[width=0.26\textwidth]{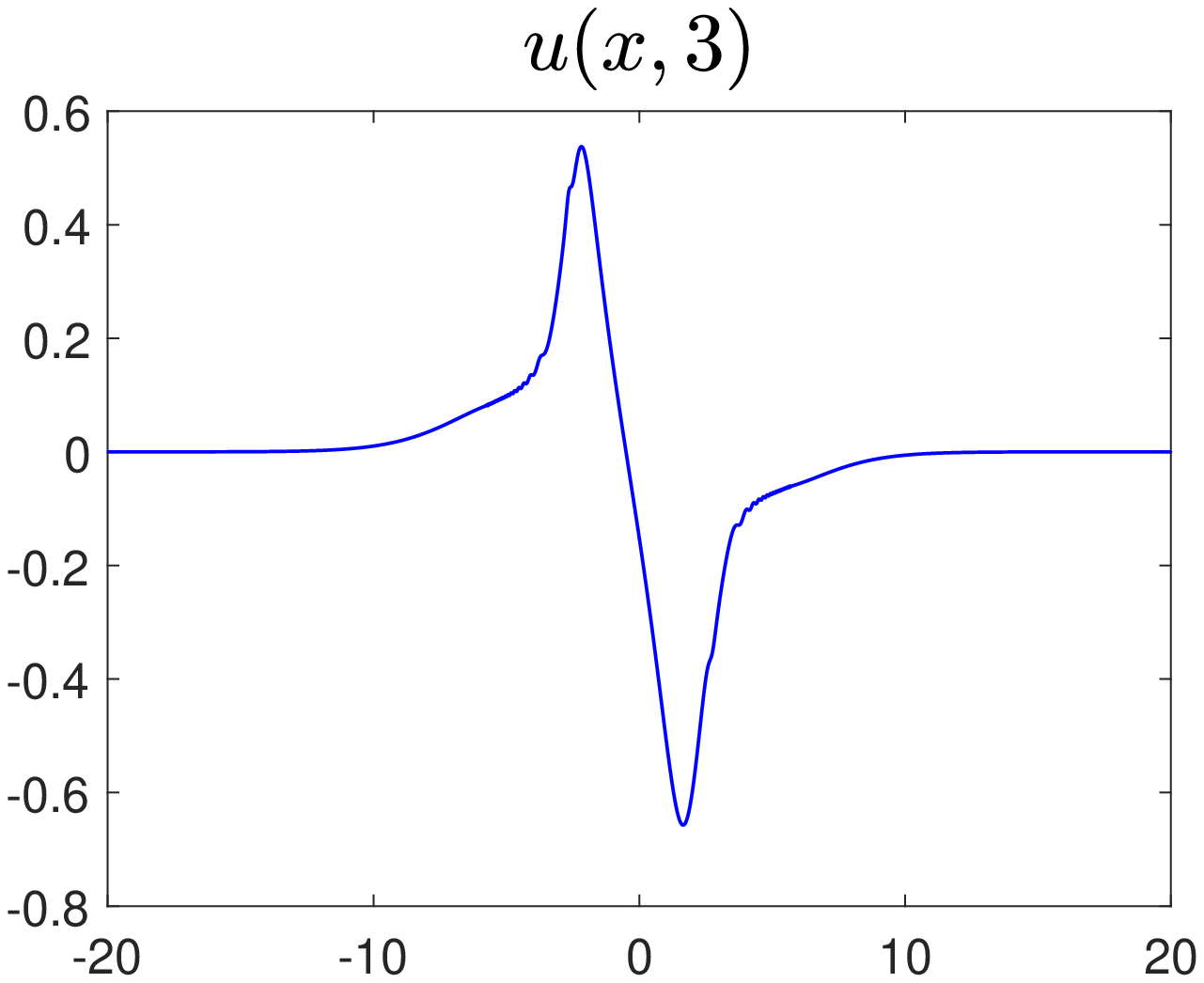}}
\hspace{-15pt}
\subfigure[Case B, $t=6$]{ \centering
\includegraphics[width=0.26\textwidth]{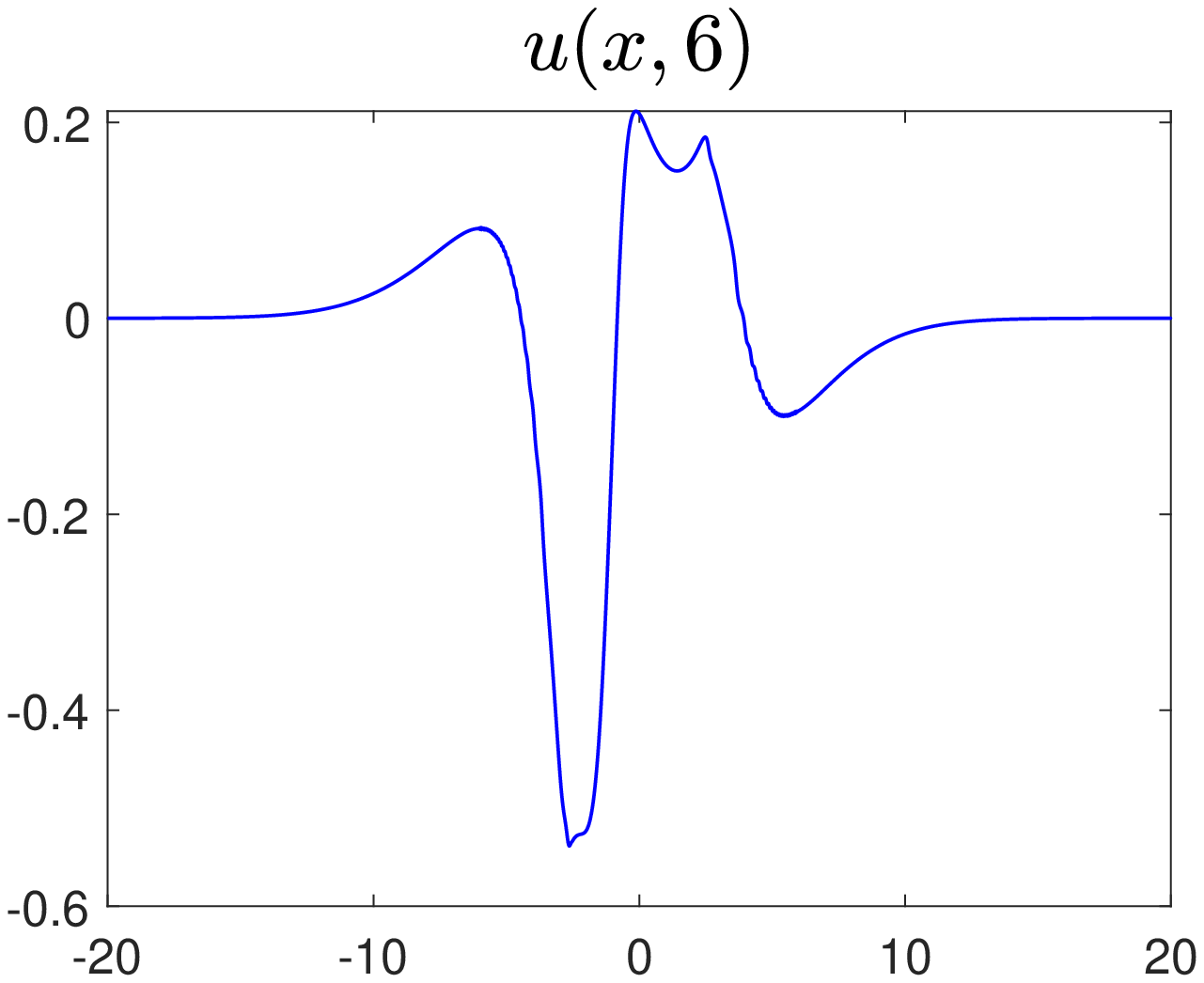}}
\hspace{-15pt}
\subfigure[Case B, $t=8$]{ \centering
\includegraphics[width=0.26\textwidth]{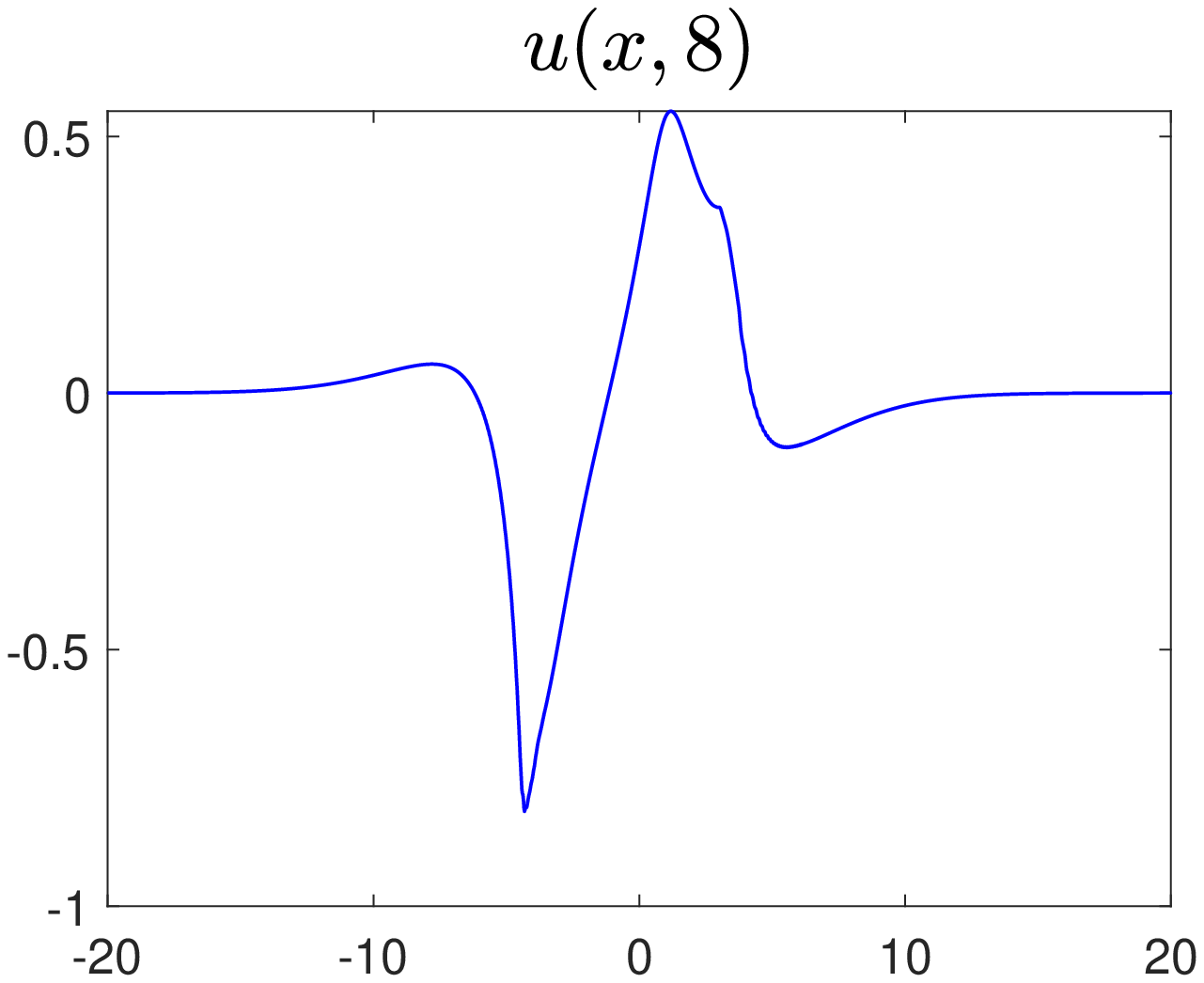}}
\\
\subfigure[Case C, $t=1$]{ \centering
\includegraphics[width=0.26\textwidth]{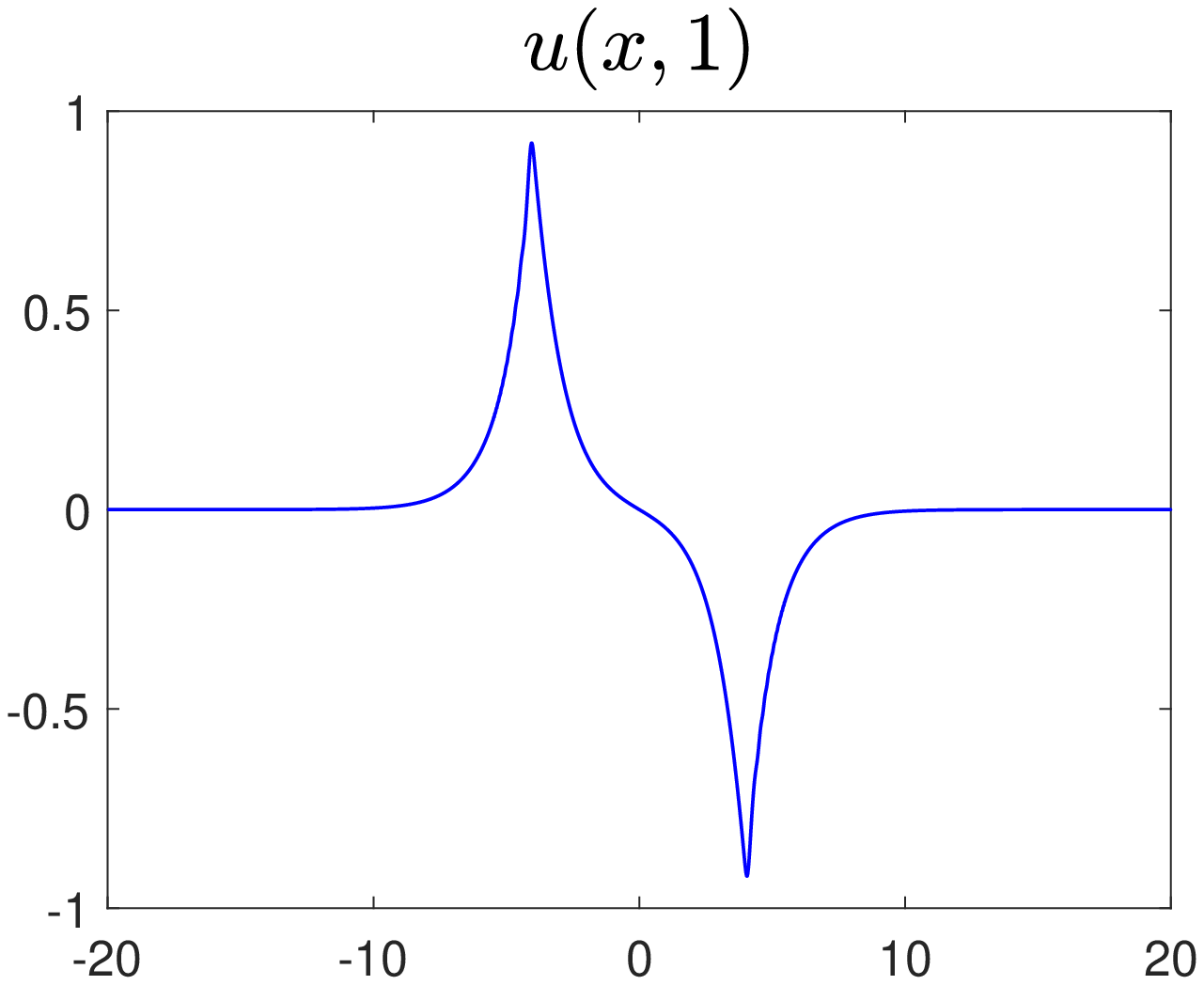}}
\hspace{-15pt}
\subfigure[Case C, $t=3$]{ \centering
\includegraphics[width=0.26\textwidth]{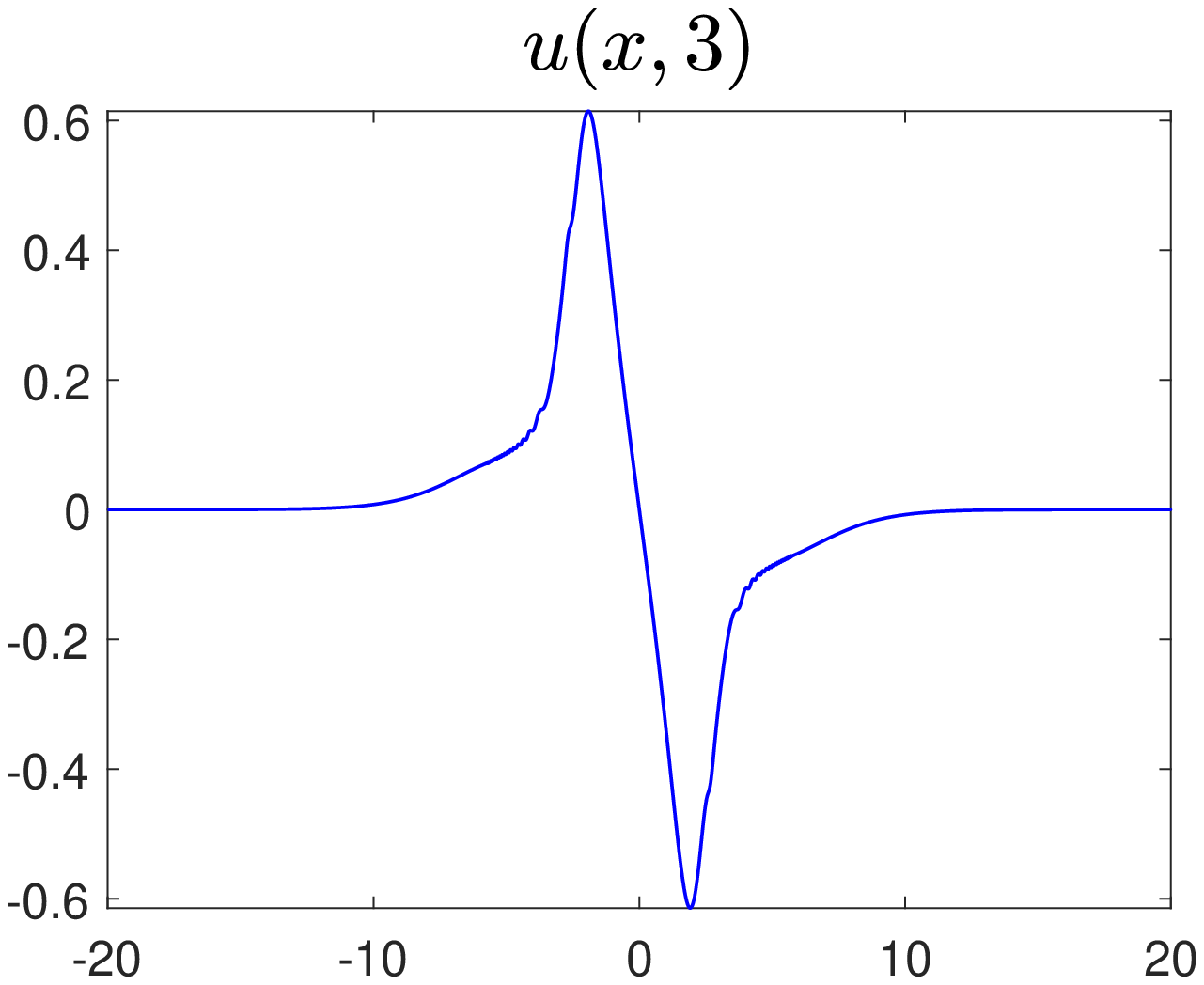}}
\hspace{-15pt}
\subfigure[Case C, $t=6$]{ \centering
\includegraphics[width=0.26\textwidth]{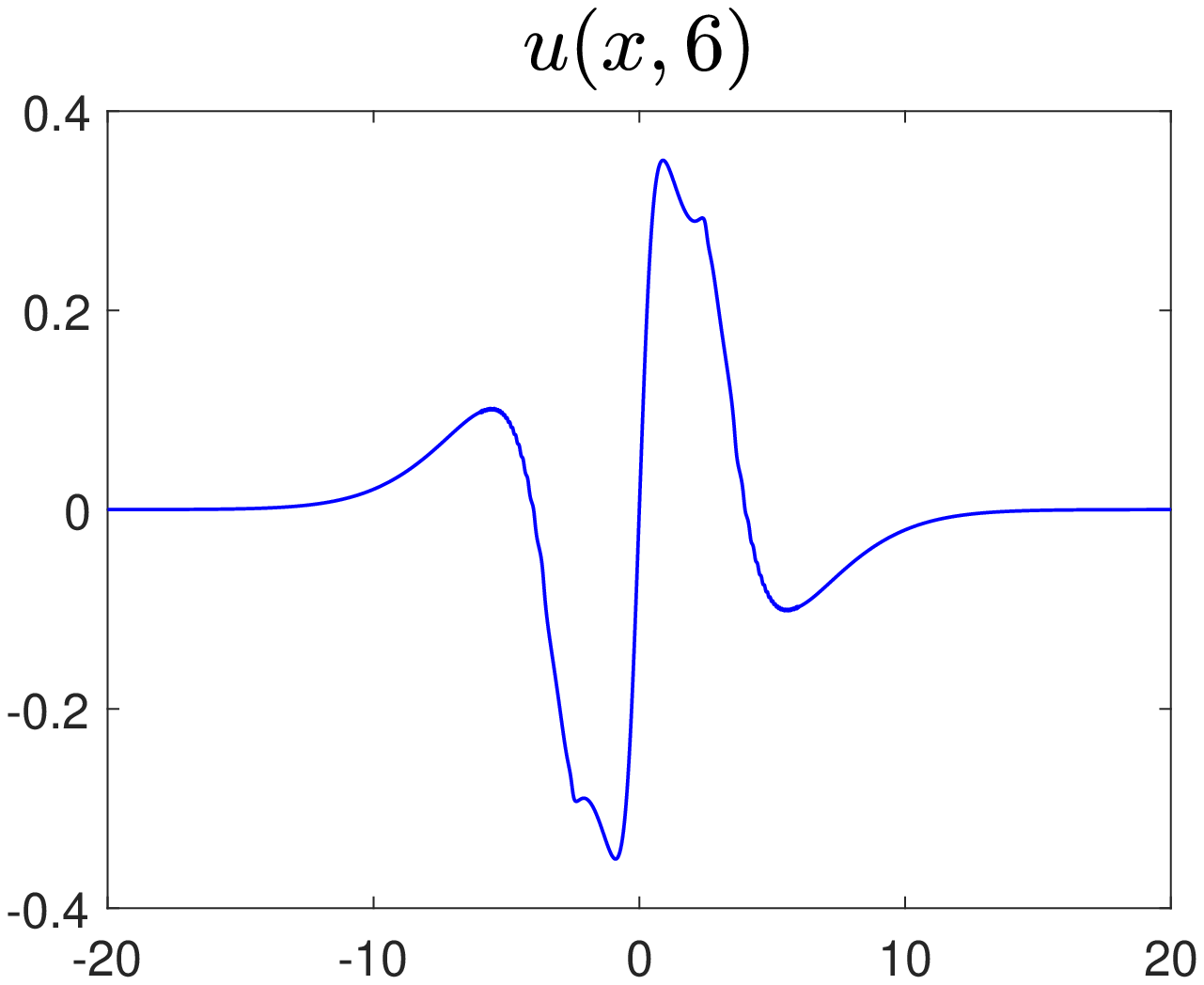}}
\hspace{-15pt}
\subfigure[Case C, $t=8$]{ \centering
\includegraphics[width=0.26\textwidth]{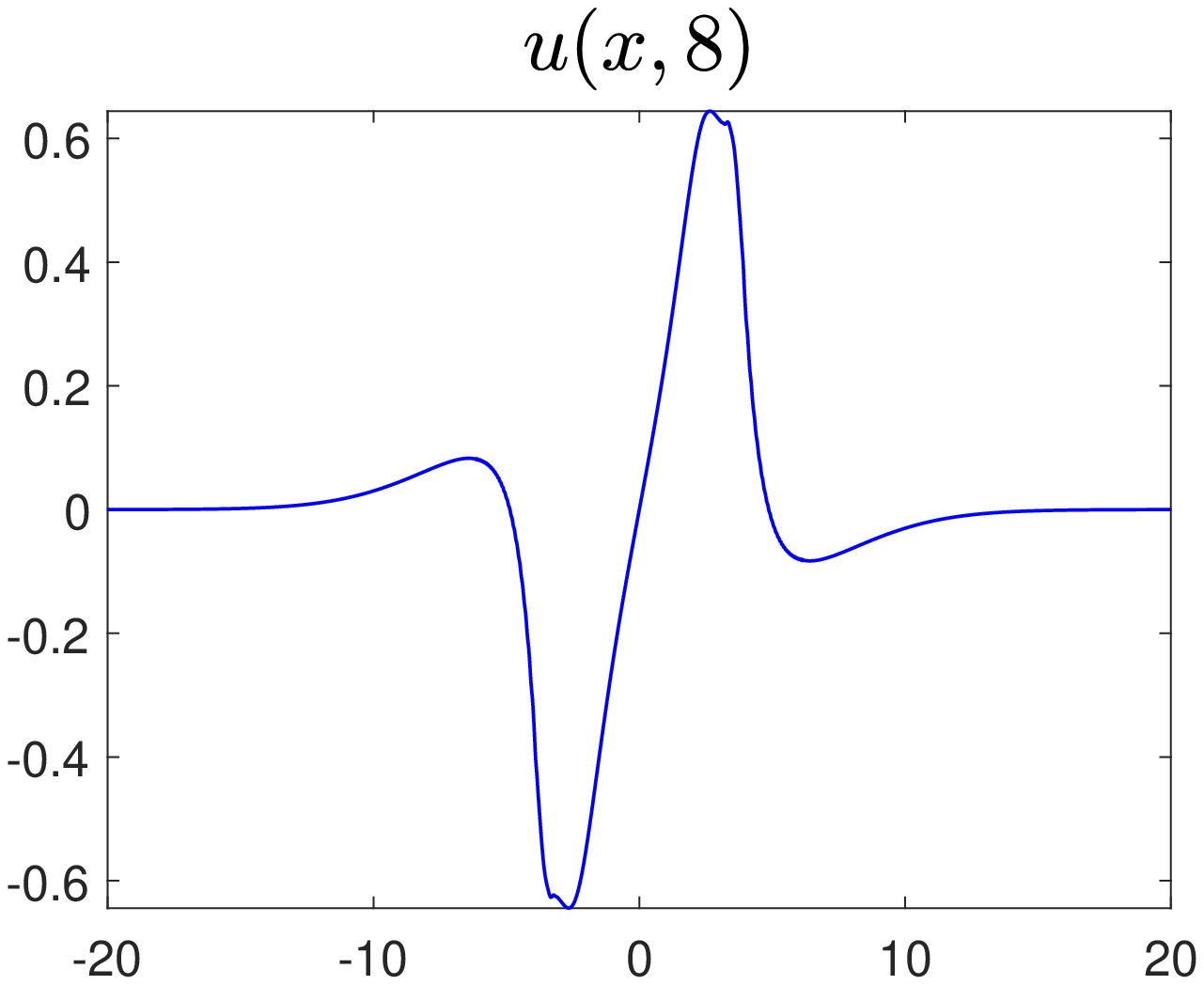}}
\\
\subfigure[Case D, $t=1$]{ \centering
\includegraphics[width=0.26\textwidth]{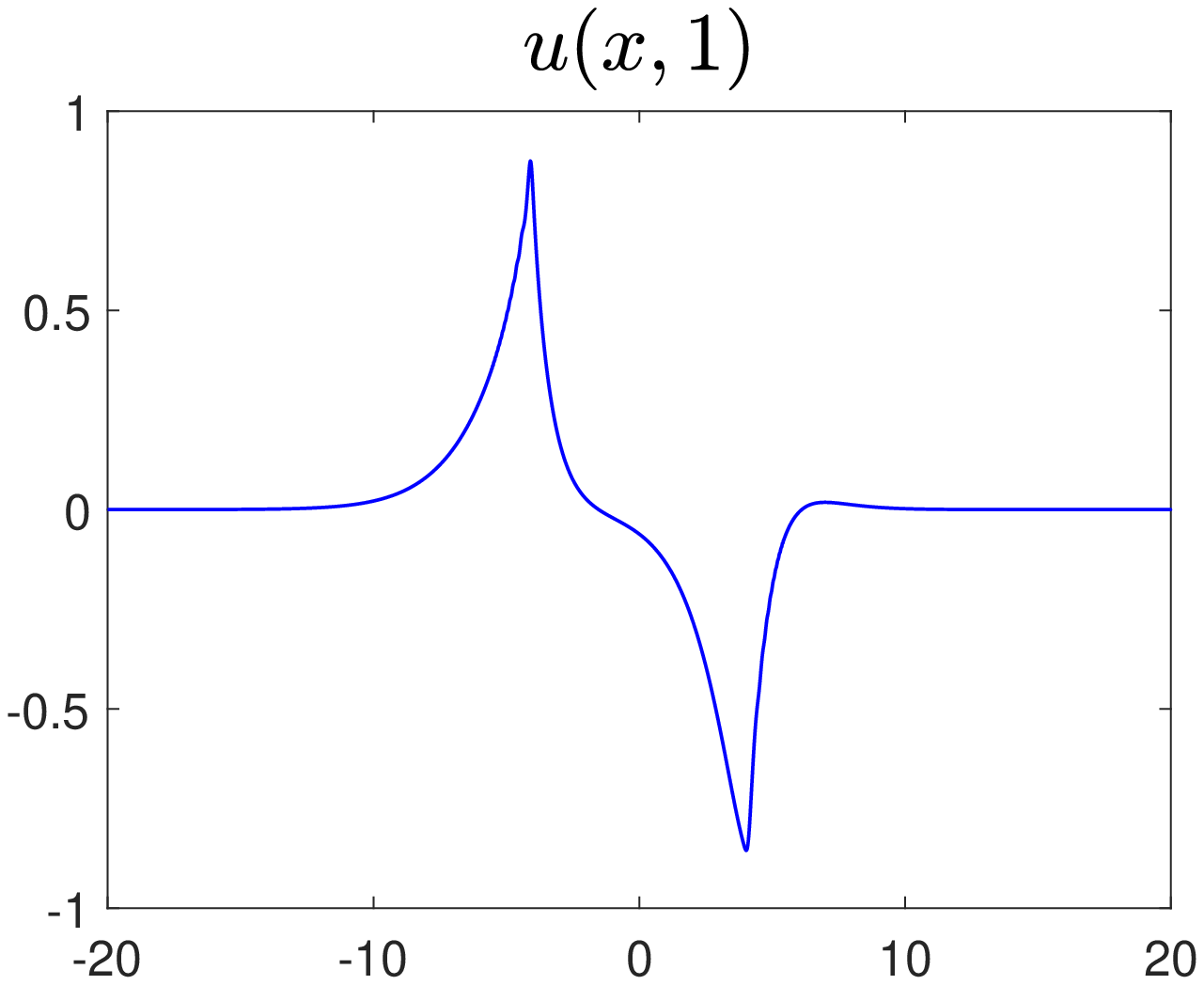}}
\hspace{-15pt}
\subfigure[Case D, $t=3$]{ \centering
\includegraphics[width=0.26\textwidth]{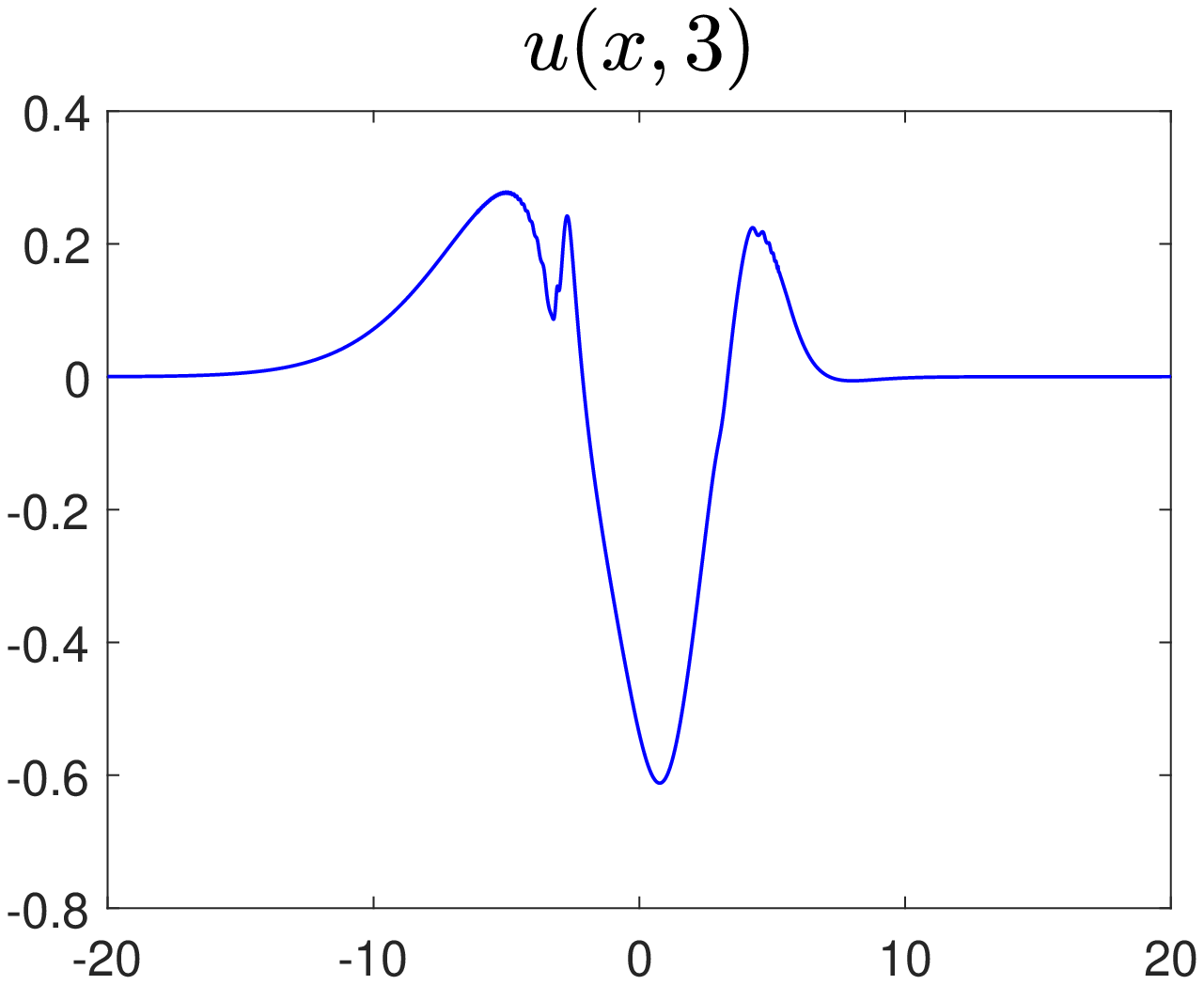}}
\hspace{-15pt}
\subfigure[Case D, $t=6$]{ \centering
\includegraphics[width=0.26\textwidth]{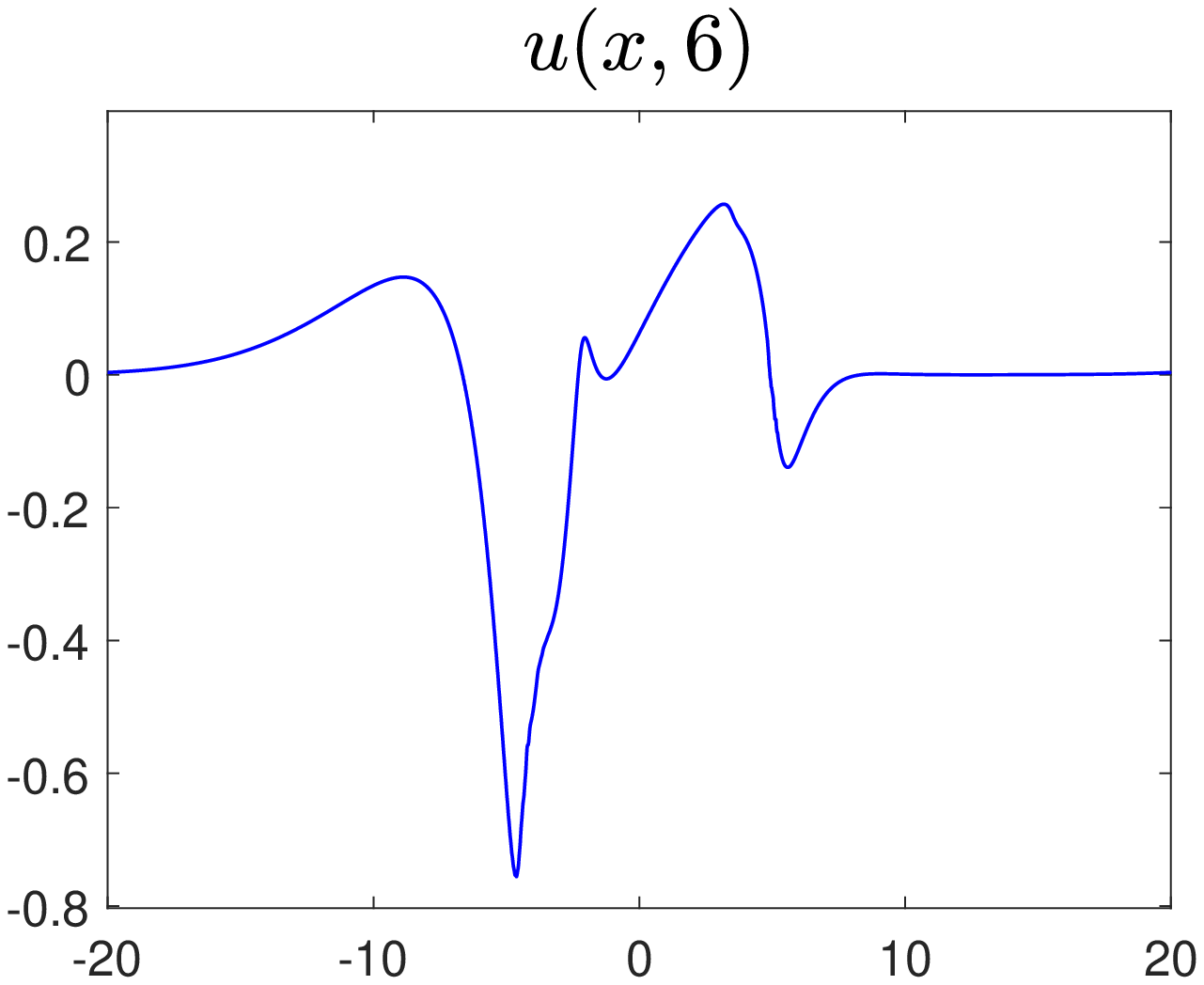}}
\hspace{-15pt}
\subfigure[Case D, $t=8$]{ \centering
\includegraphics[width=0.26\textwidth]{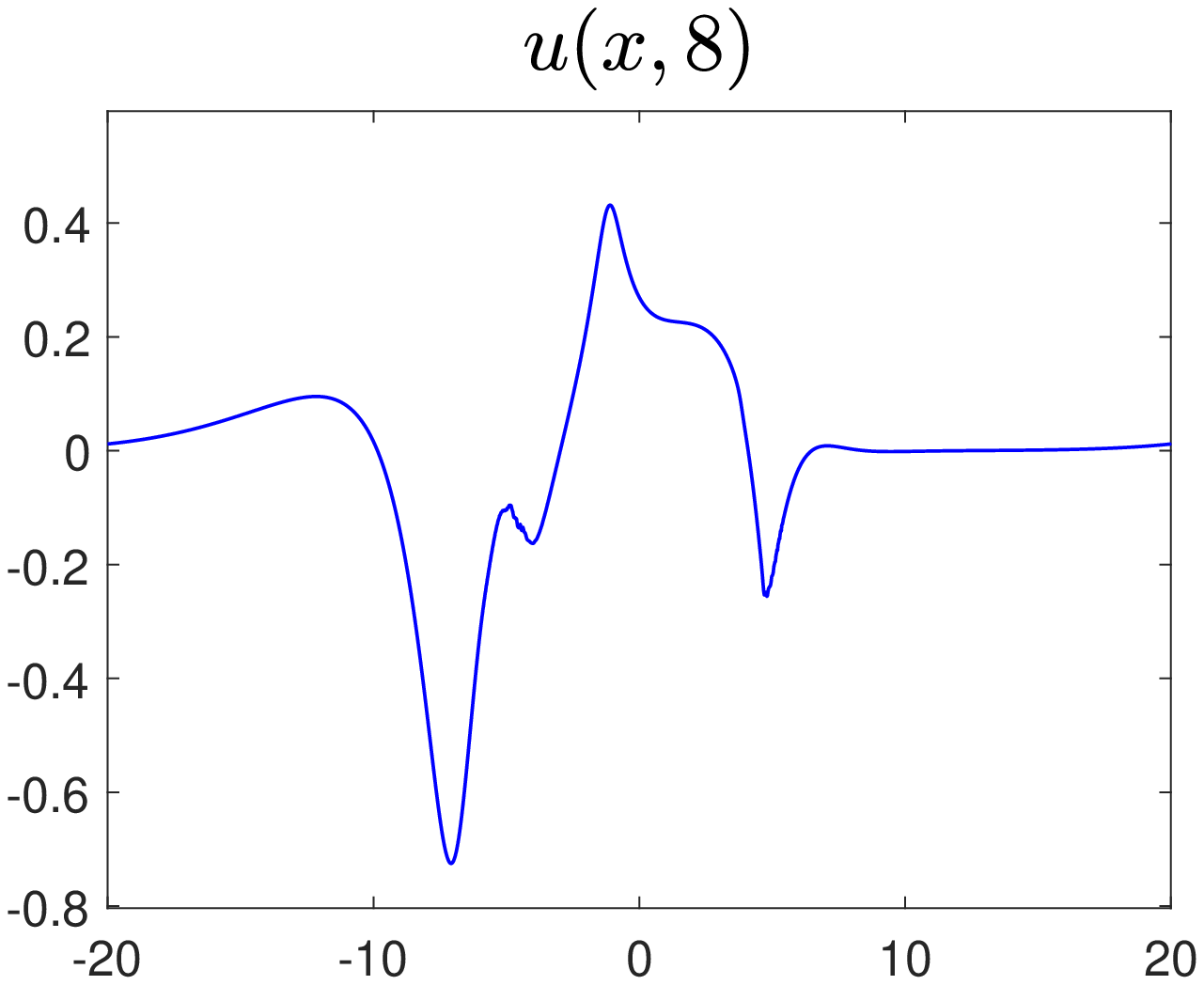}}
\\
\subfigure[Case E, $t=1$]{ \centering
\includegraphics[width=0.26\textwidth]{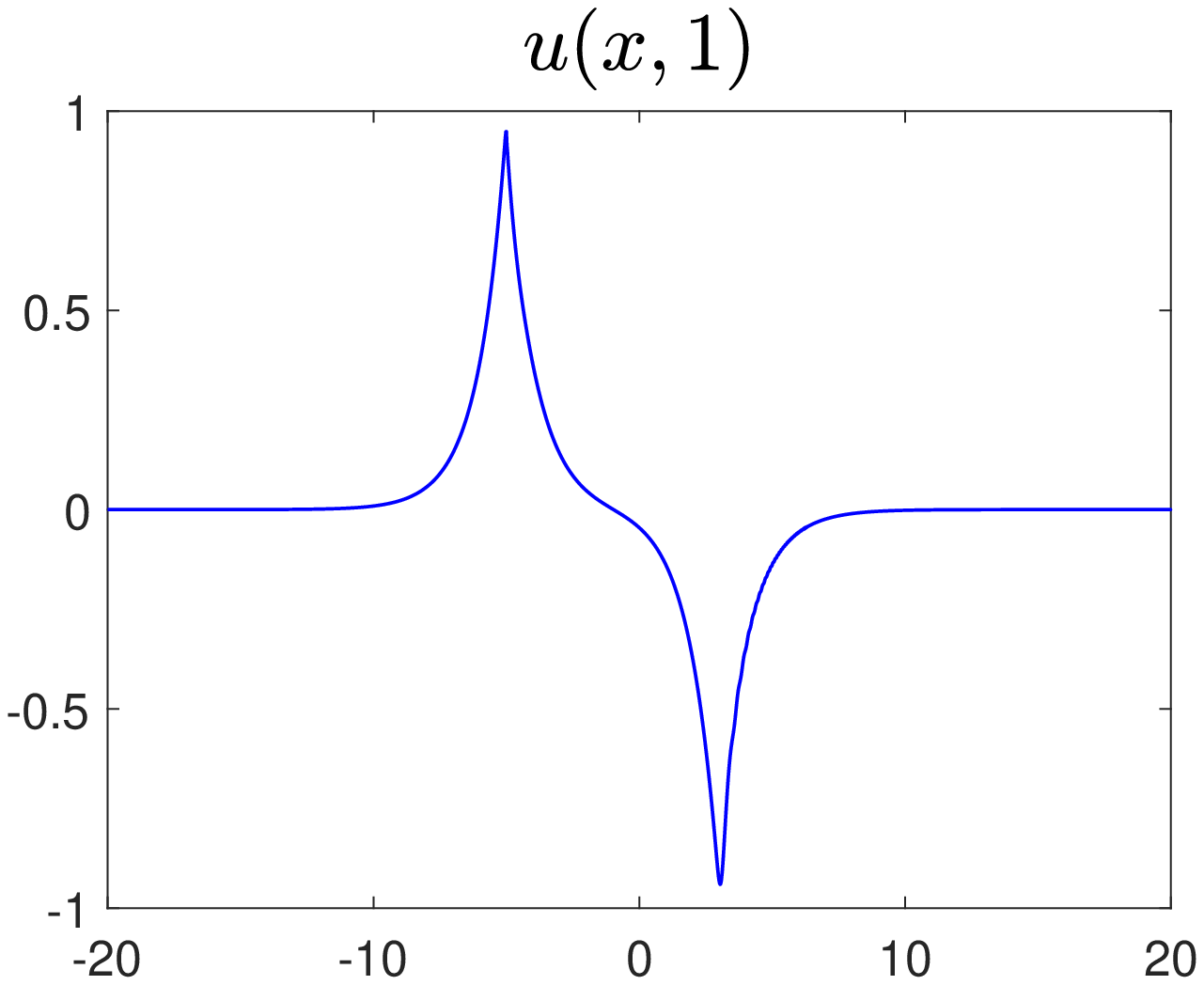}}
\hspace{-15pt}
\subfigure[Case E, $t=3$]{ \centering
\includegraphics[width=0.26\textwidth]{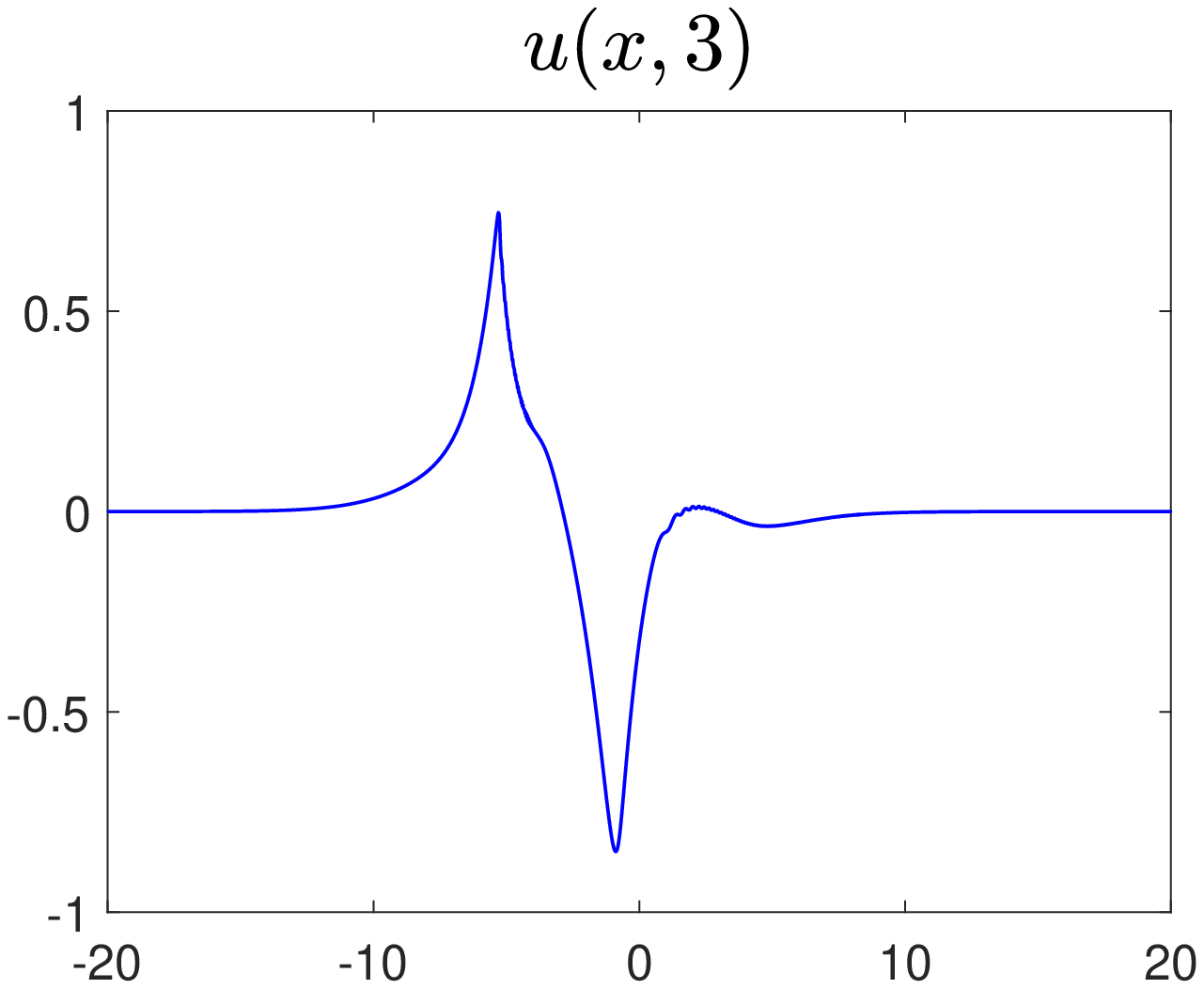}}
\hspace{-15pt}
\subfigure[Case E, $t=6$]{ \centering
\includegraphics[width=0.26\textwidth]{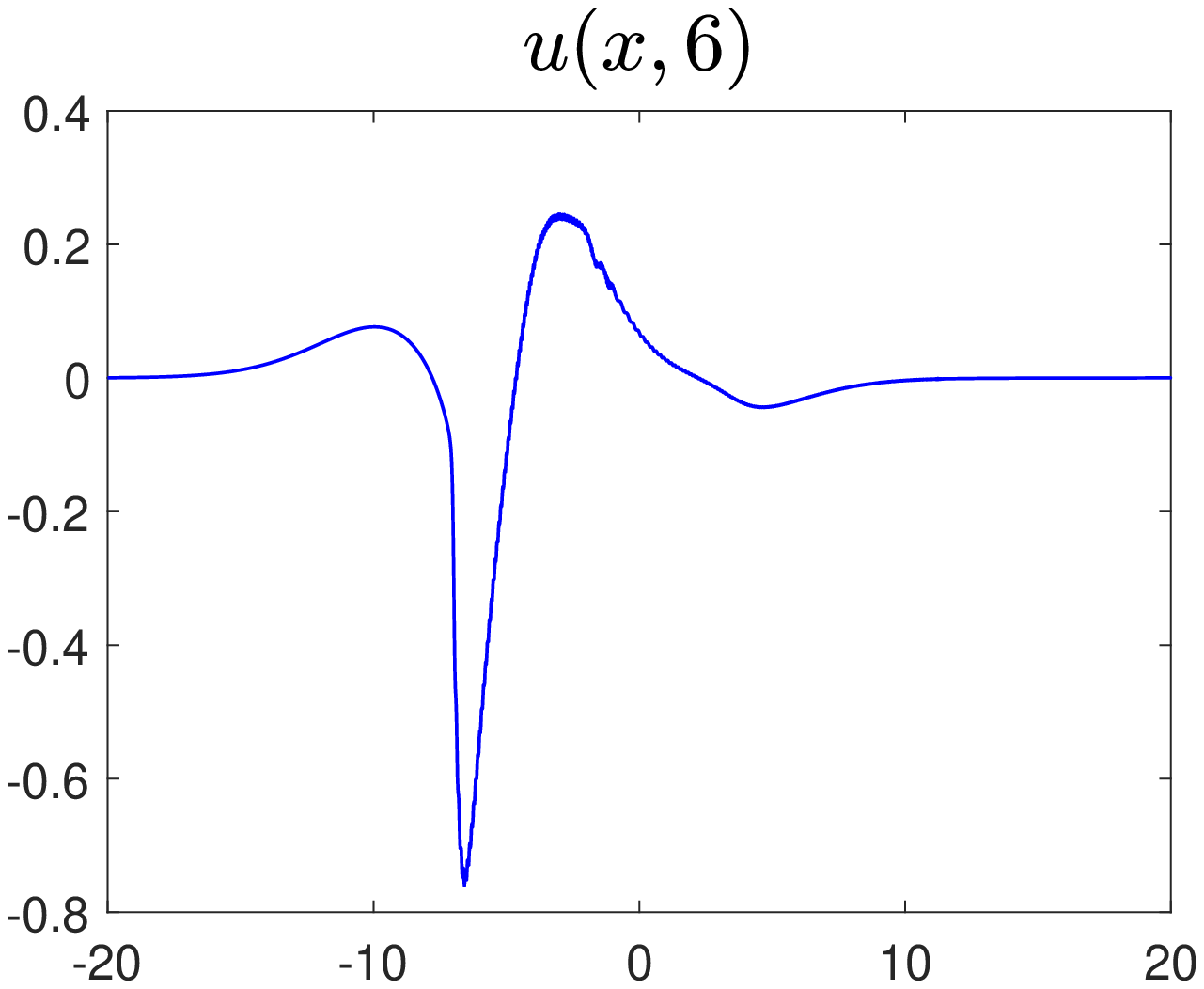}}
\hspace{-15pt}
\subfigure[Case E, $t=8$]{ \centering
\includegraphics[width=0.26\textwidth]{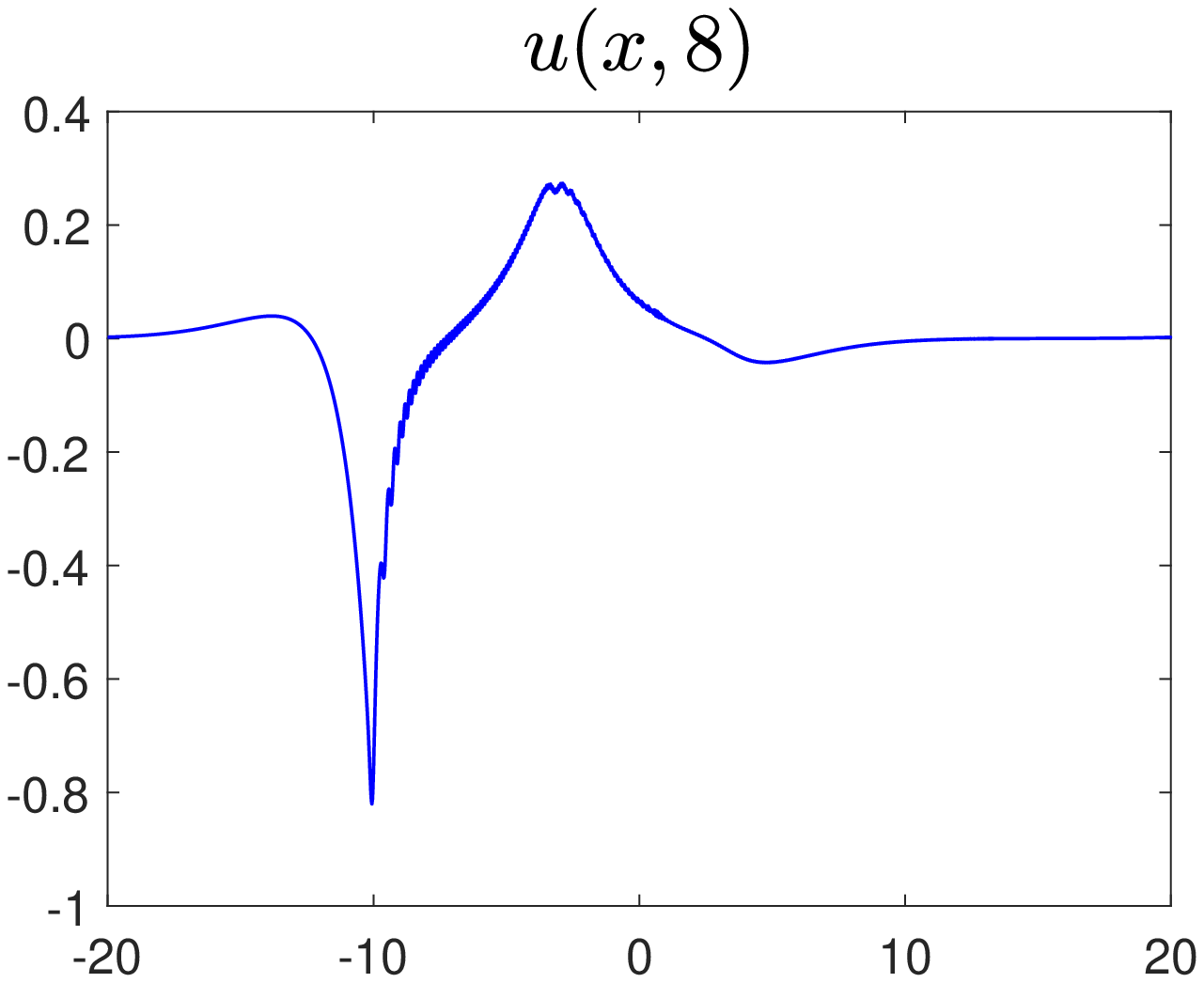}}
\caption{Velocities of the waves at different instants of times calculated by the difference scheme \eqref{equa3.7} for five different Cases;
the spatial grid stepsize is fixed as $h = 0.02$.} \label{fig:4}
\end{figure}

\begin{figure}[htbp]
\subfigtopskip=2pt
\subfigbottomskip=2pt
\subfigcapskip=-3pt
\subfigure[Case A, $t=1$]{ \centering
\includegraphics[width=0.26\textwidth]{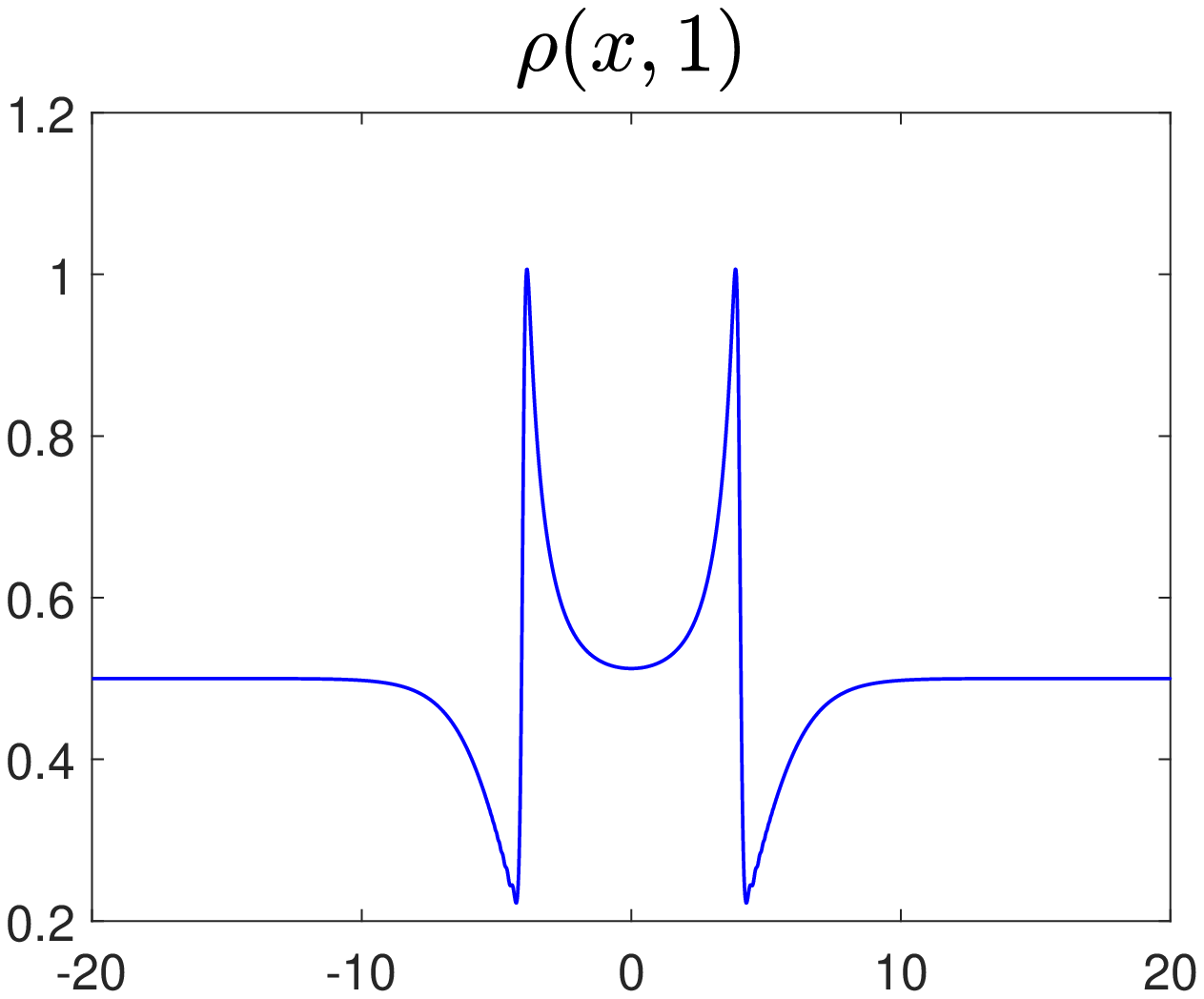}}
\hspace{-15pt}
\subfigure[Case A, $t=3$]{ \centering
\includegraphics[width=0.26\textwidth]{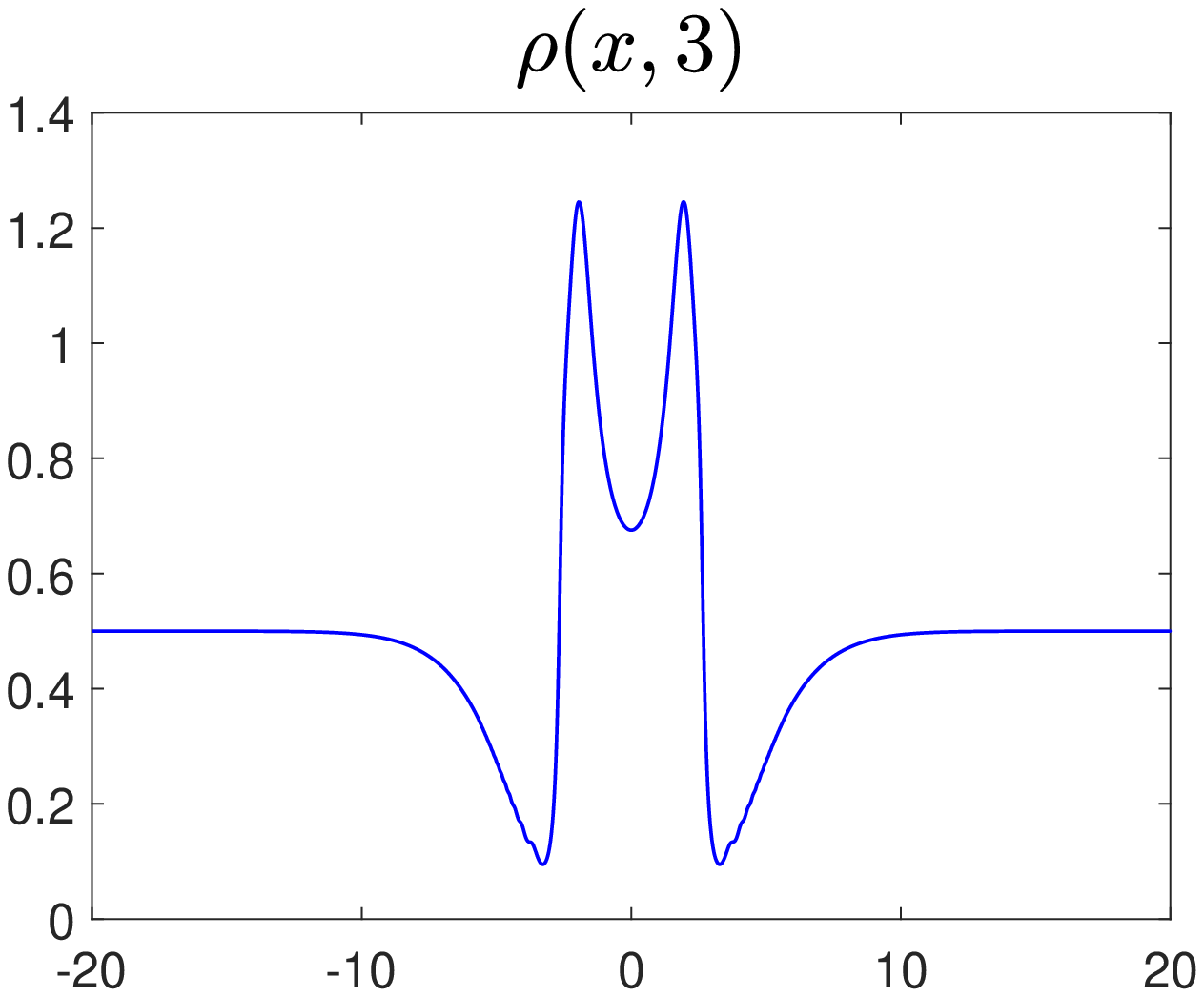}}
\hspace{-15pt}
\subfigure[Case A, $t=6$]{ \centering
\includegraphics[width=0.26\textwidth]{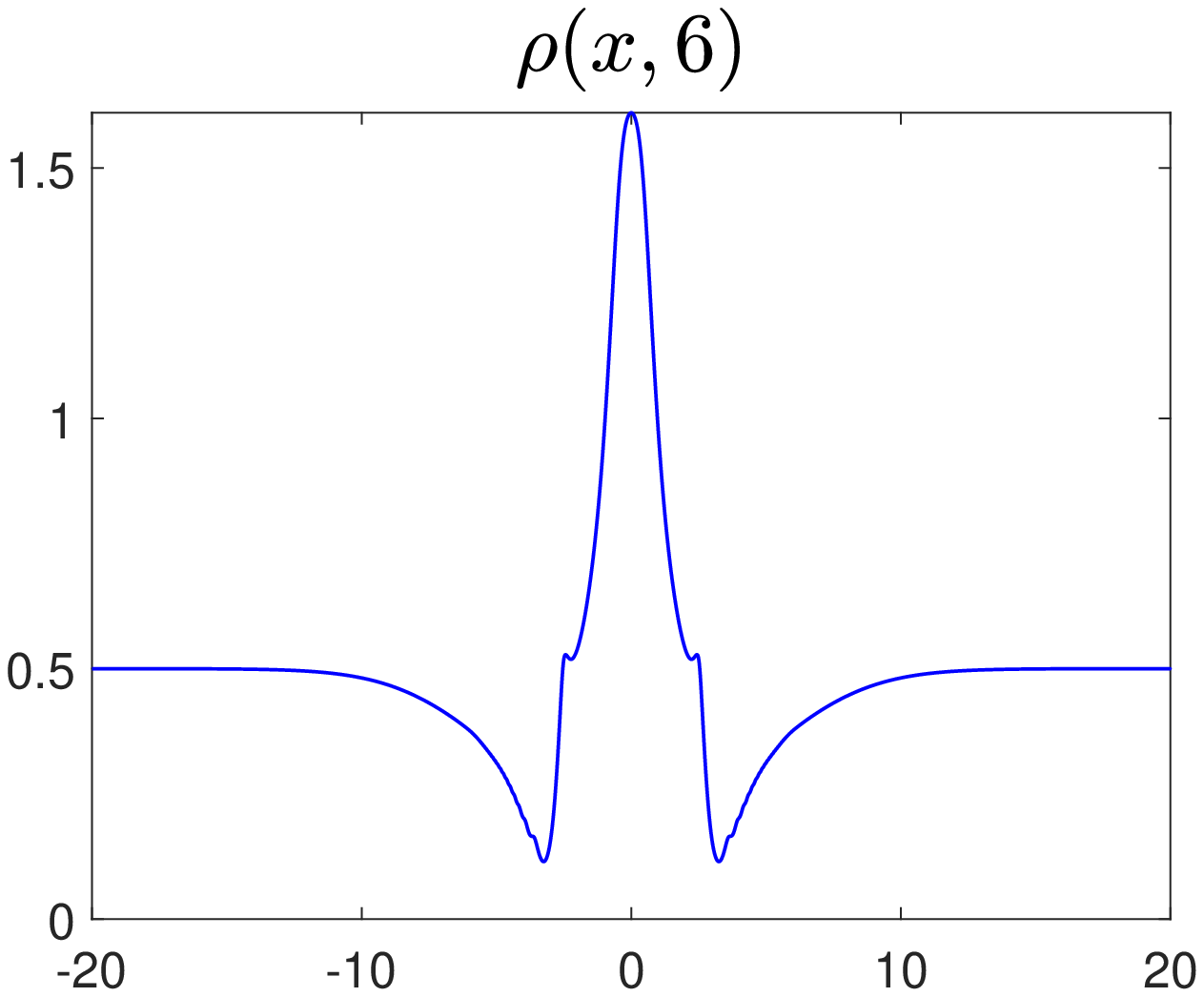}}
\hspace{-15pt}
\subfigure[Case A, $t=8$]{ \centering
\includegraphics[width=0.26\textwidth]{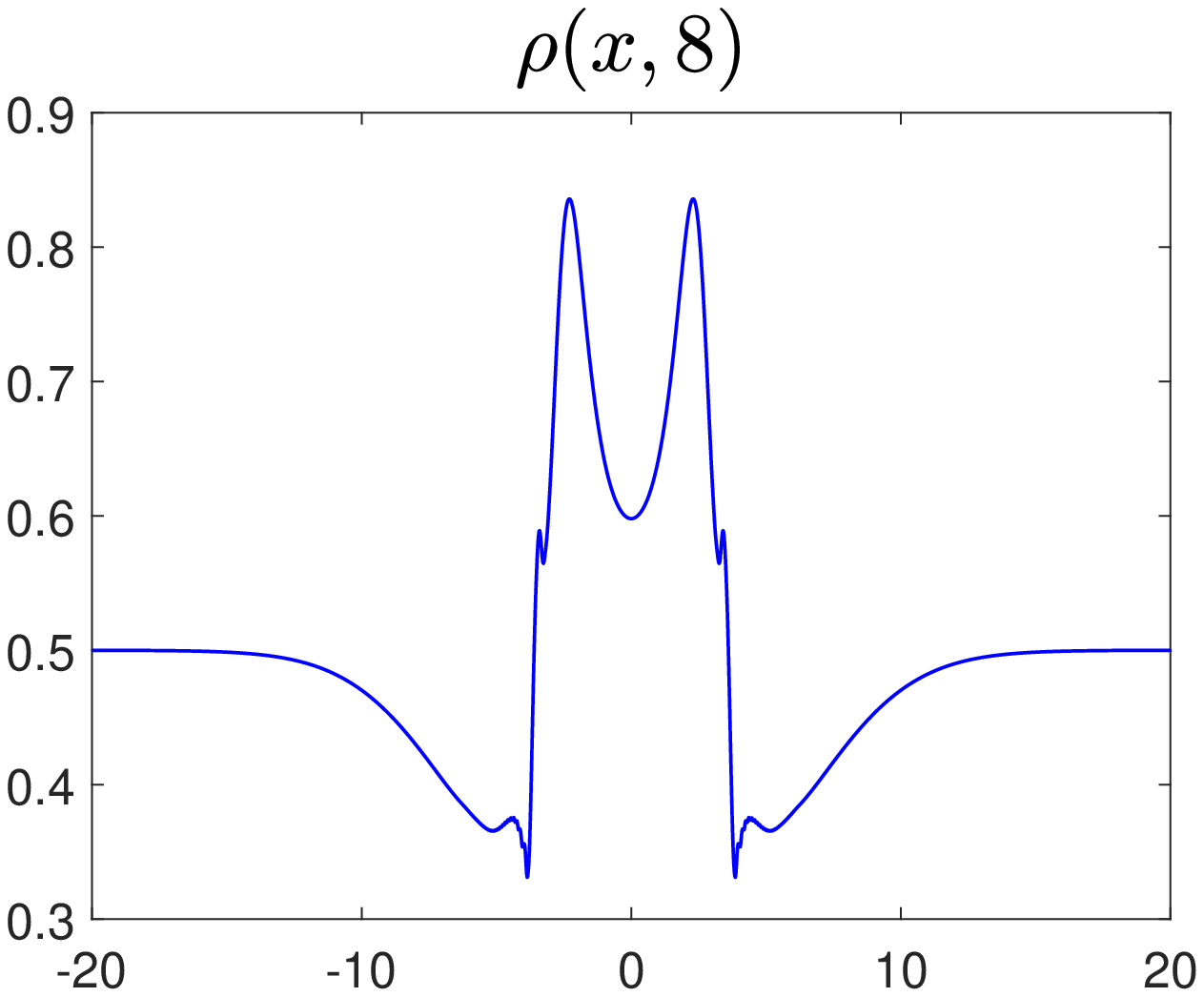}}
\\
\subfigure[Case B, $t=1$]{ \centering
\includegraphics[width=0.26\textwidth]{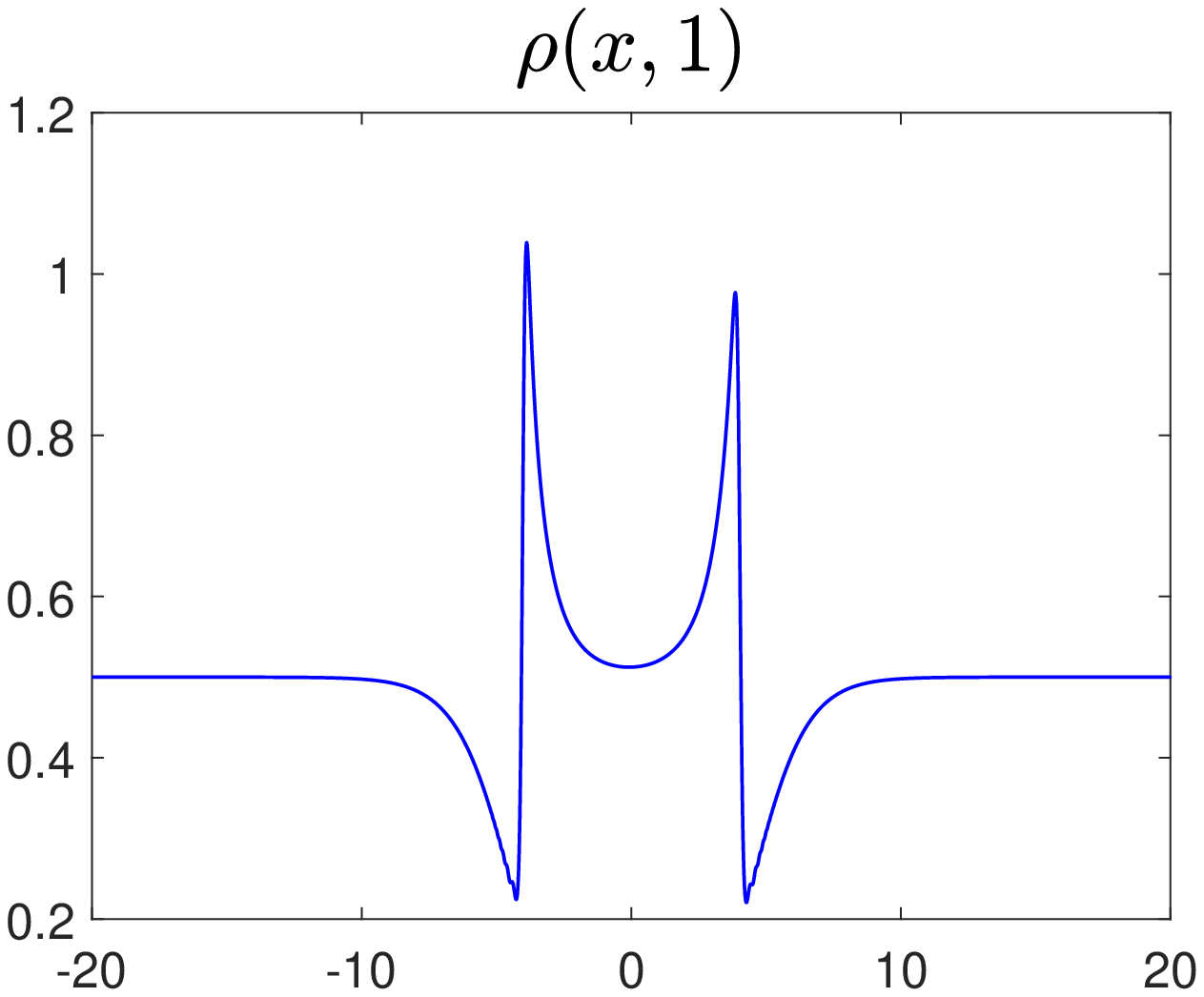}}
\hspace{-15pt}
\subfigure[Case B, $t=3$]{ \centering
\includegraphics[width=0.26\textwidth]{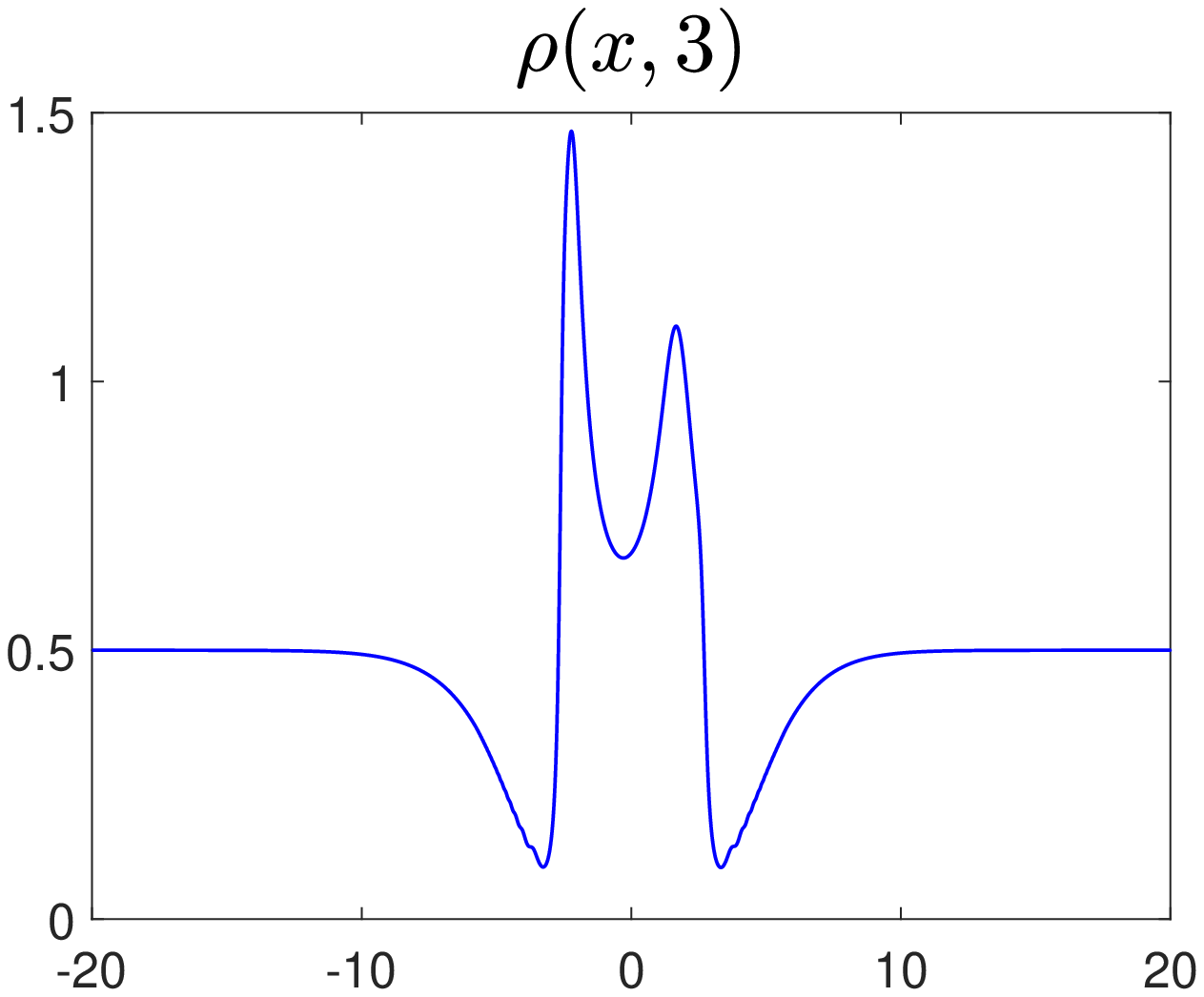}}
\hspace{-15pt}
\subfigure[Case B, $t=6$]{ \centering
\includegraphics[width=0.26\textwidth]{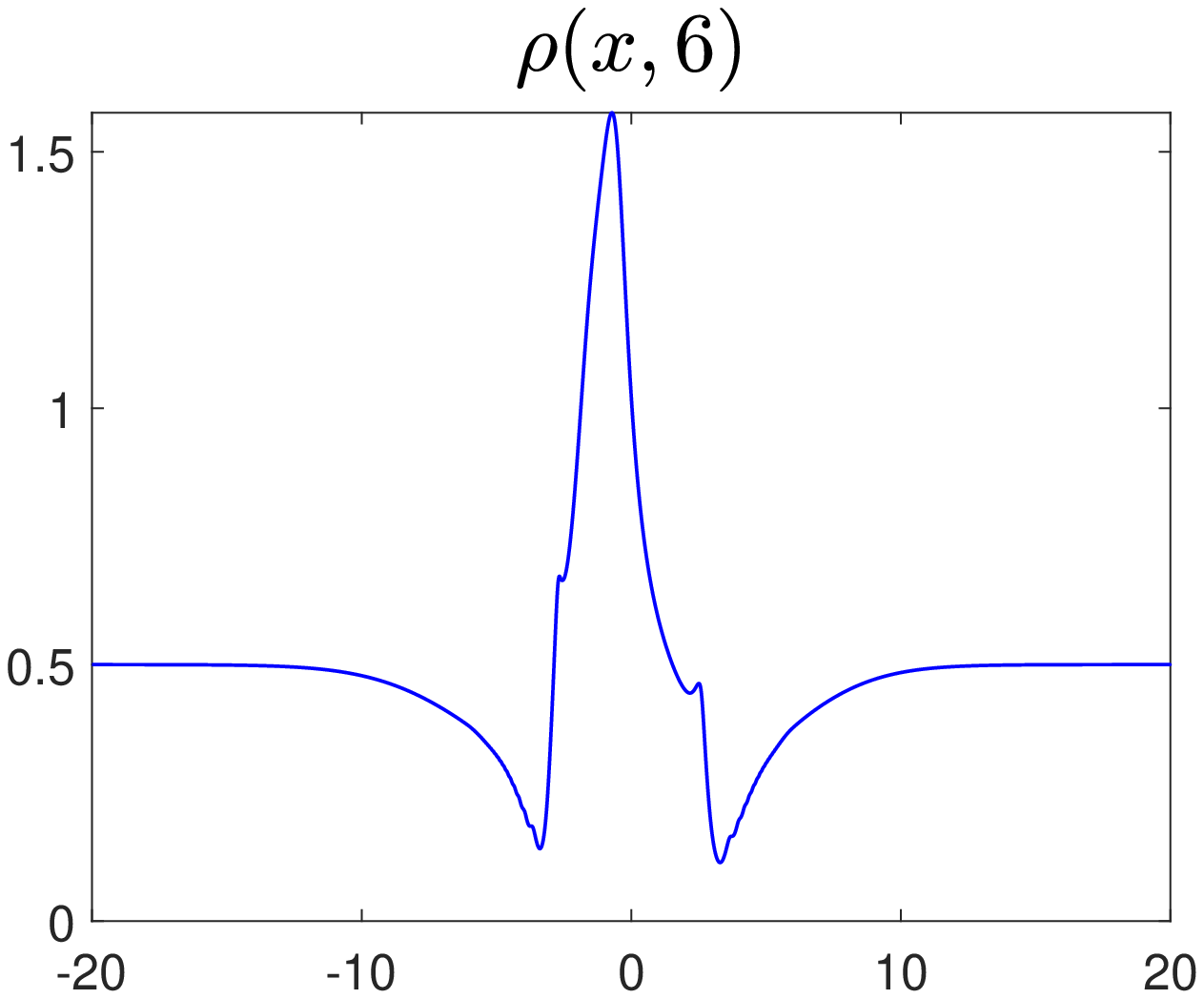}}
\hspace{-15pt}
\subfigure[Case B, $t=8$]{ \centering
\includegraphics[width=0.26\textwidth]{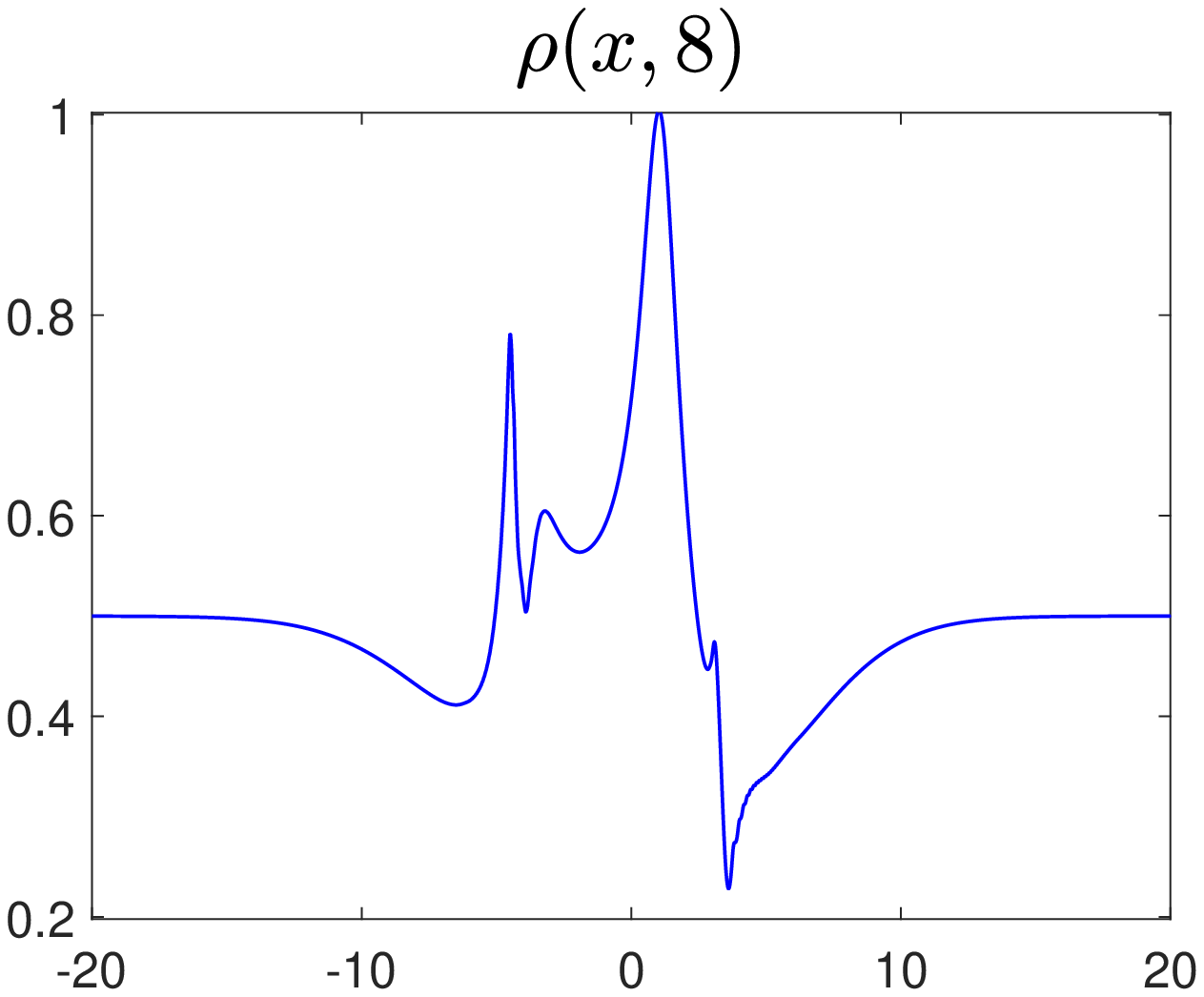}}
\\
\subfigure[Case C, $t=1$]{ \centering
\includegraphics[width=0.26\textwidth]{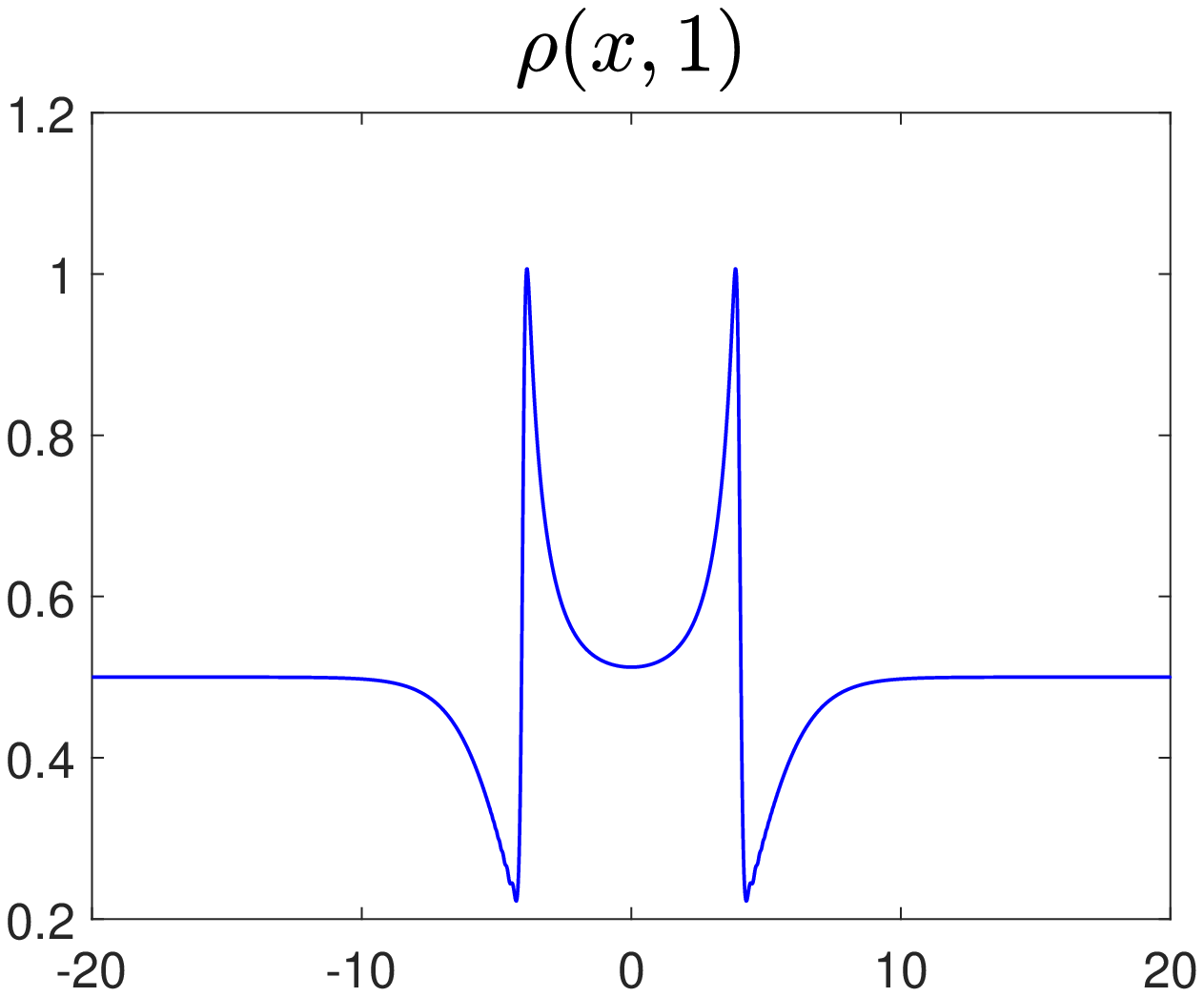}}
\hspace{-15pt}
\subfigure[Case C, $t=3$]{ \centering
\includegraphics[width=0.26\textwidth]{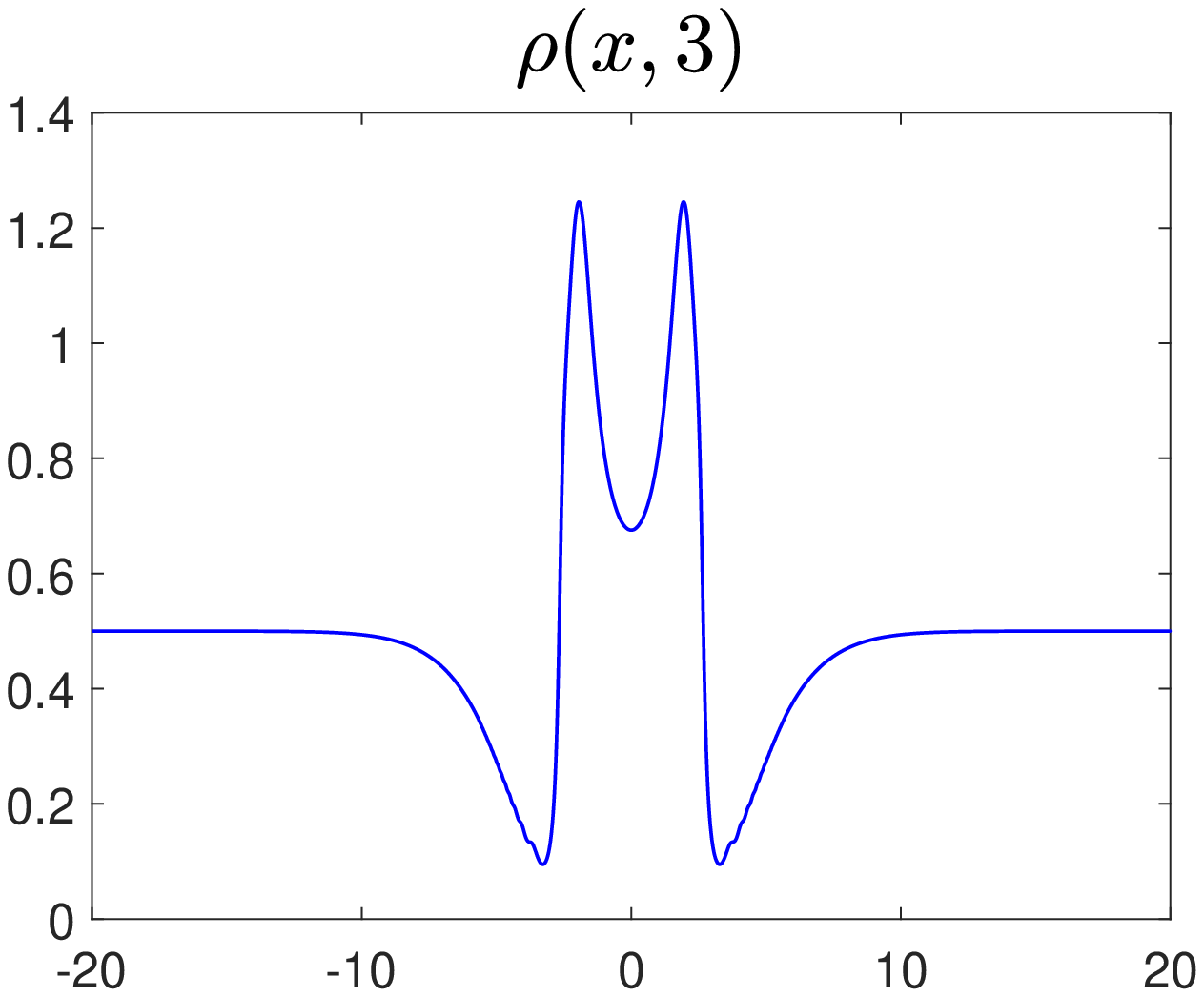}}
\hspace{-15pt}
\subfigure[Case C, $t=6$]{ \centering
\includegraphics[width=0.26\textwidth]{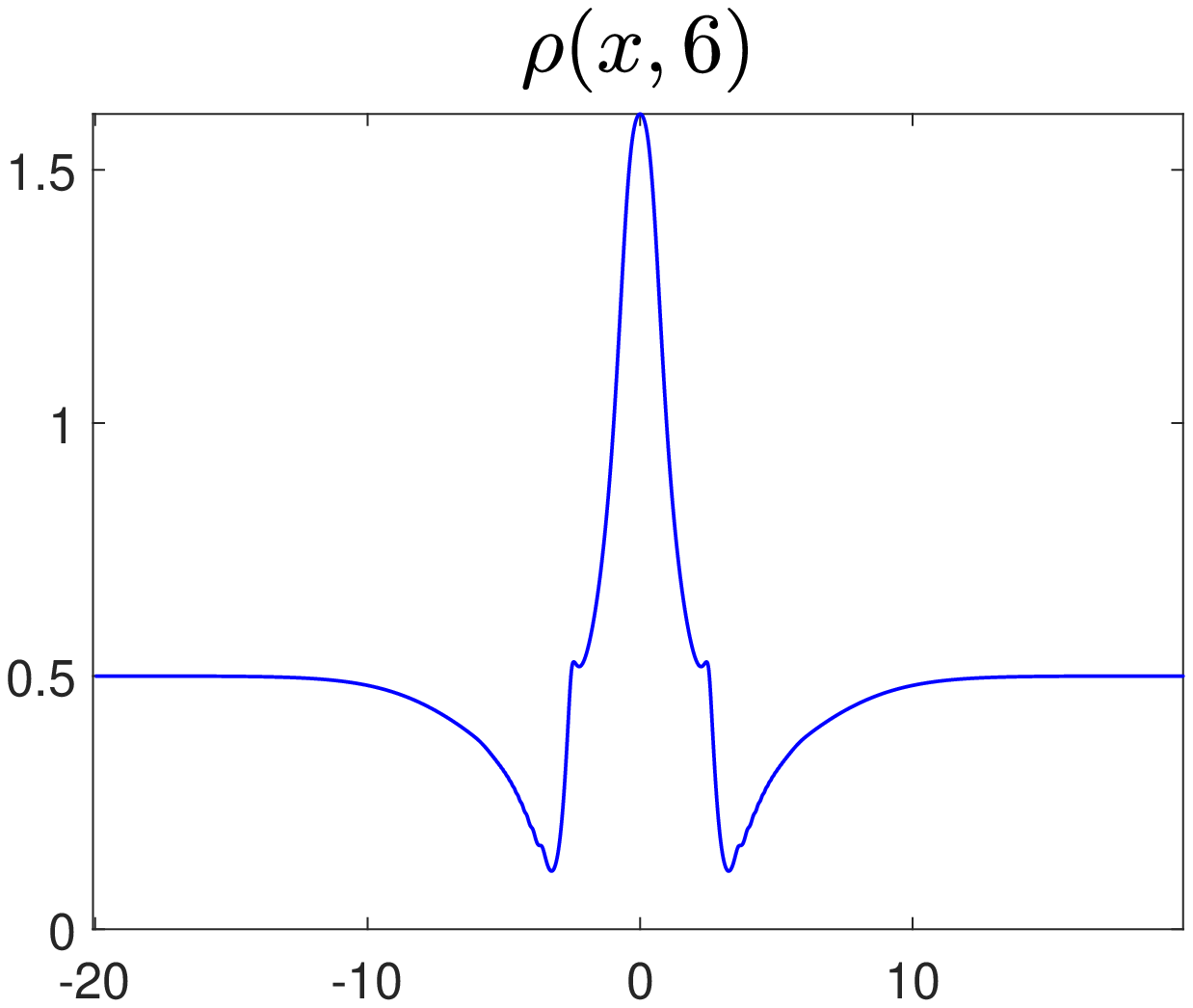}}
\hspace{-15pt}
\subfigure[Case C, $t=8$]{ \centering
\includegraphics[width=0.26\textwidth]{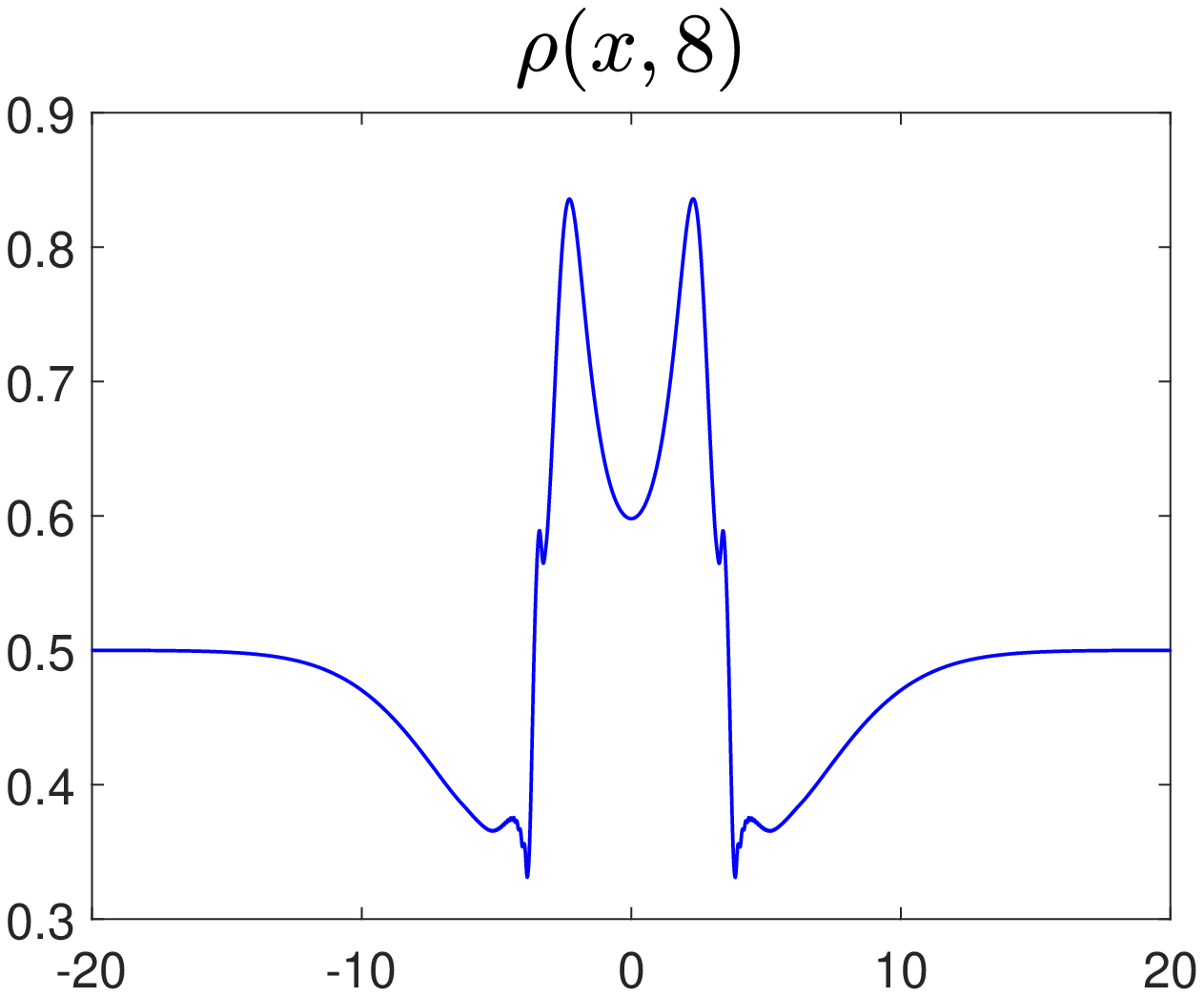}}
\\
\subfigure[Case D, $t=1$]{ \centering
\includegraphics[width=0.26\textwidth]{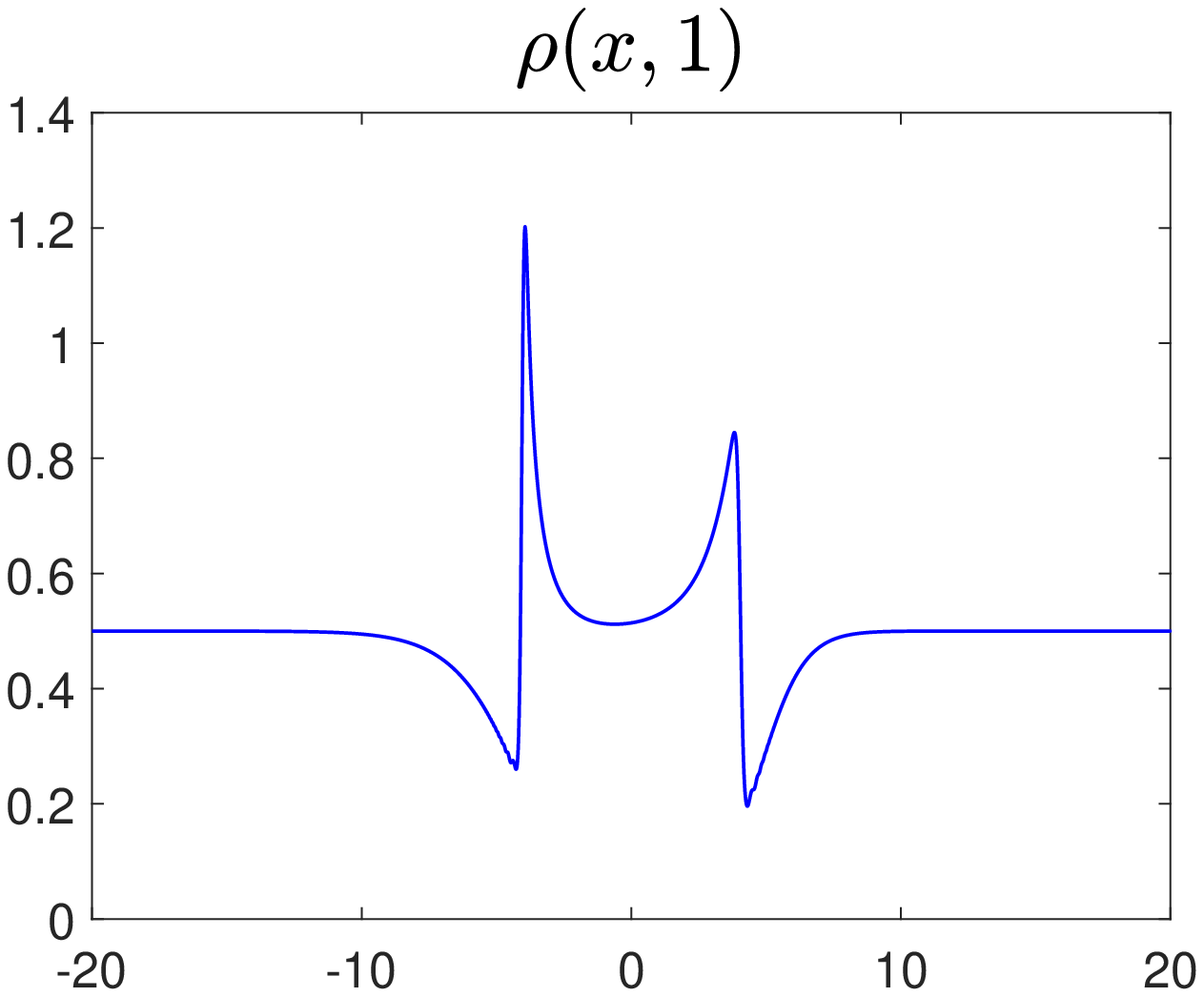}}
\hspace{-15pt}
\subfigure[Case D, $t=3$]{ \centering
\includegraphics[width=0.26\textwidth]{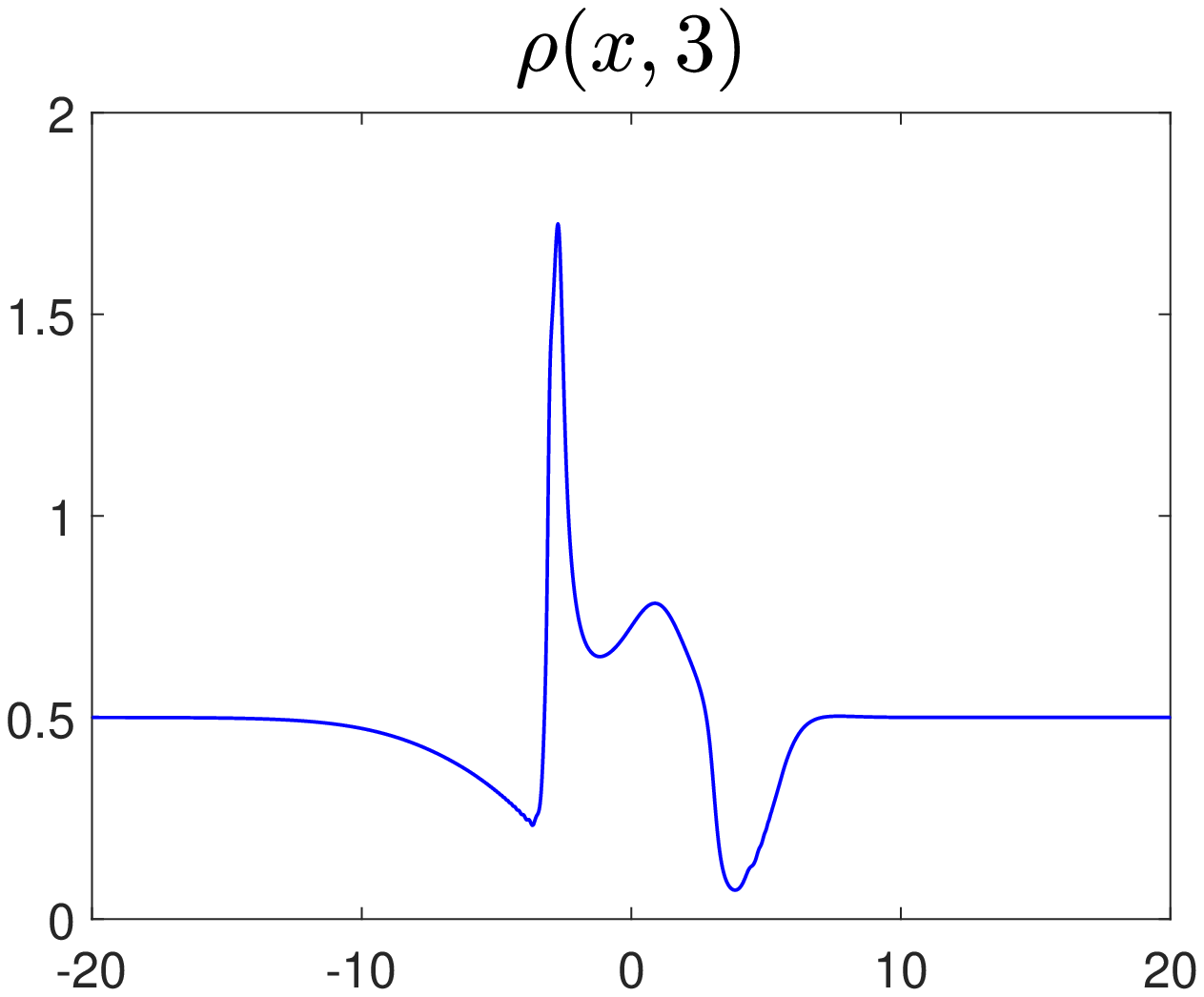}}
\hspace{-15pt}
\subfigure[Case D, $t=6$]{ \centering
\includegraphics[width=0.26\textwidth]{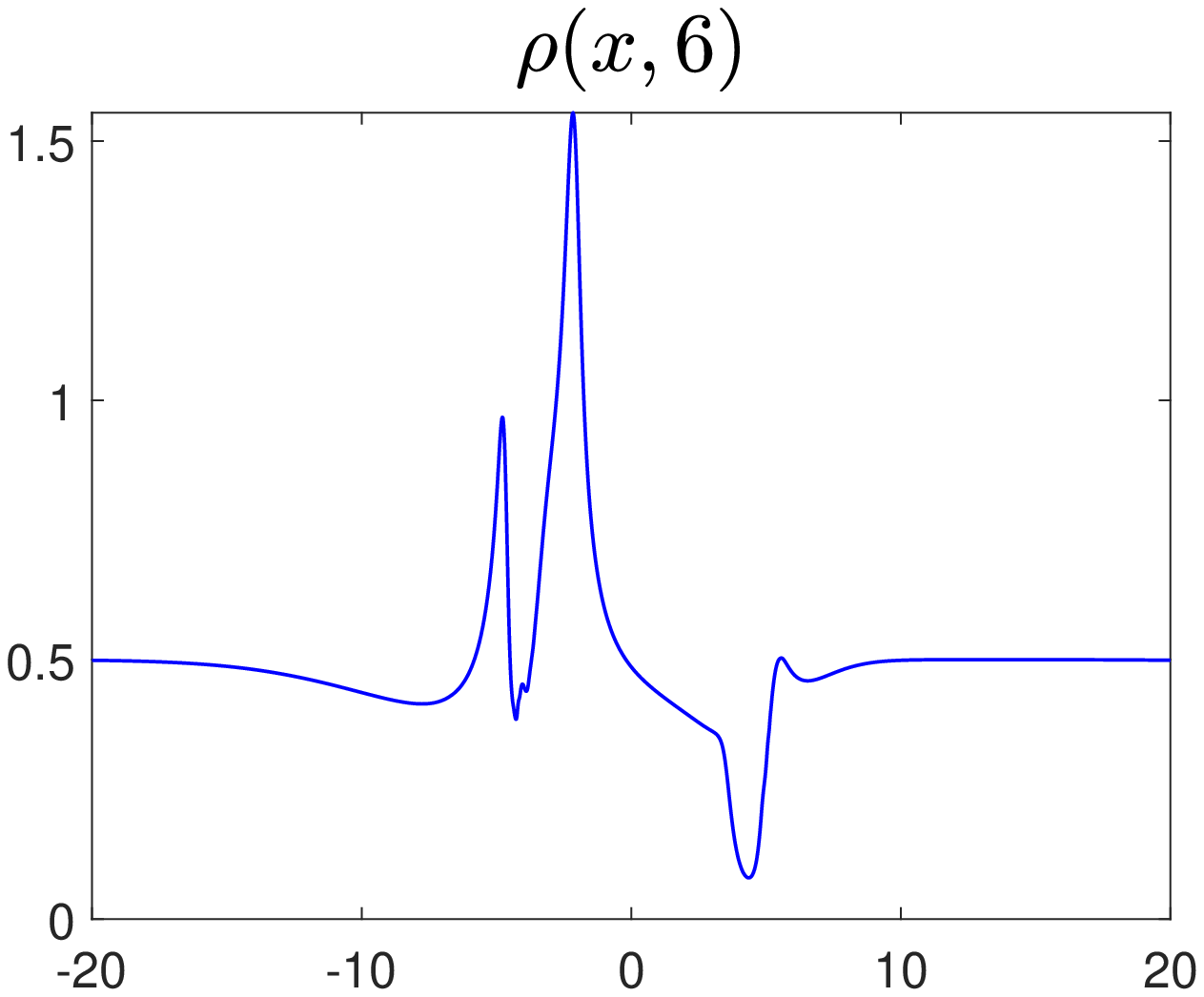}}
\hspace{-15pt}
\subfigure[Case D, $t=8$]{ \centering
\includegraphics[width=0.26\textwidth]{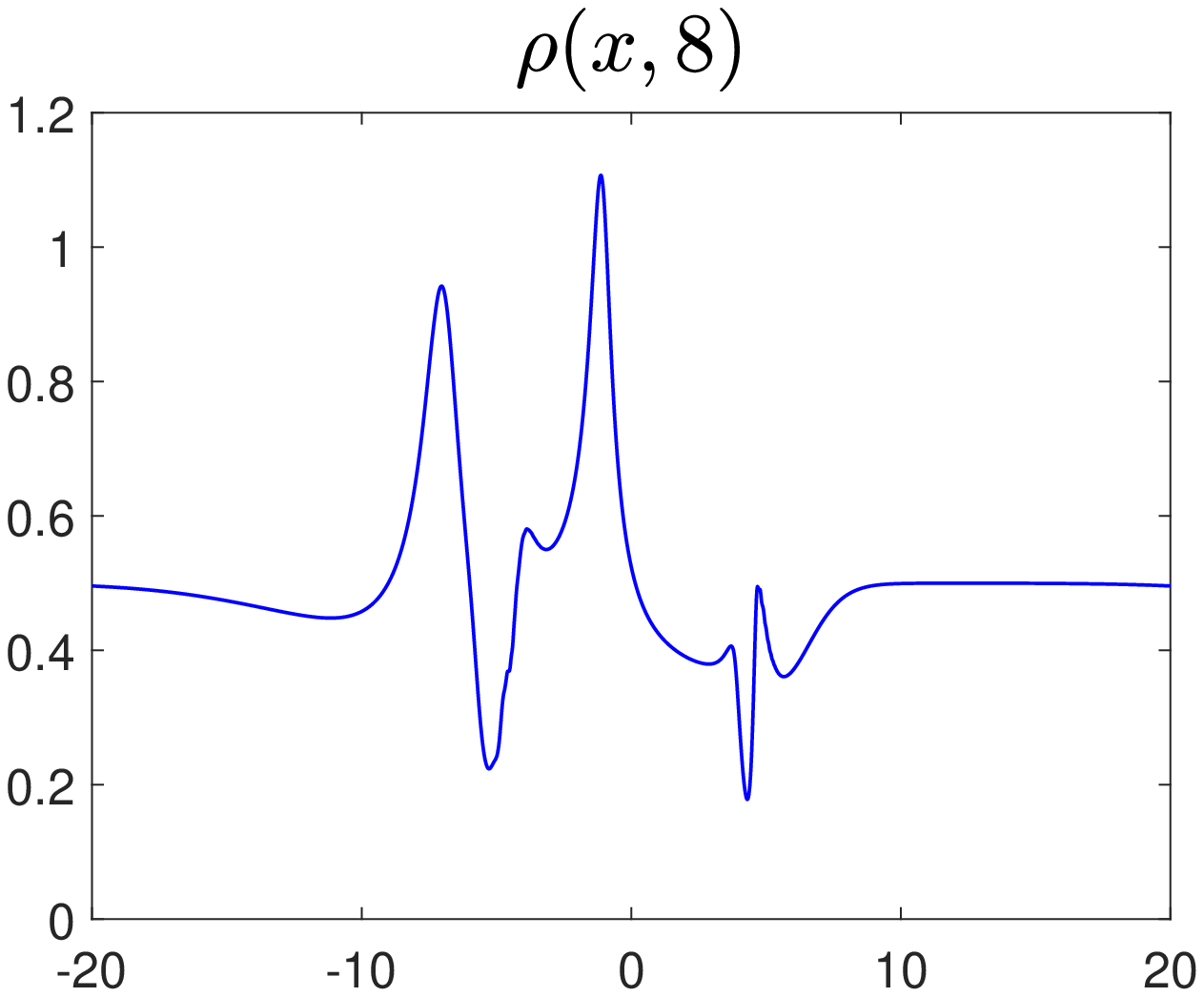}}
\\
\subfigure[Case E, $t=1$]{ \centering
\includegraphics[width=0.26\textwidth]{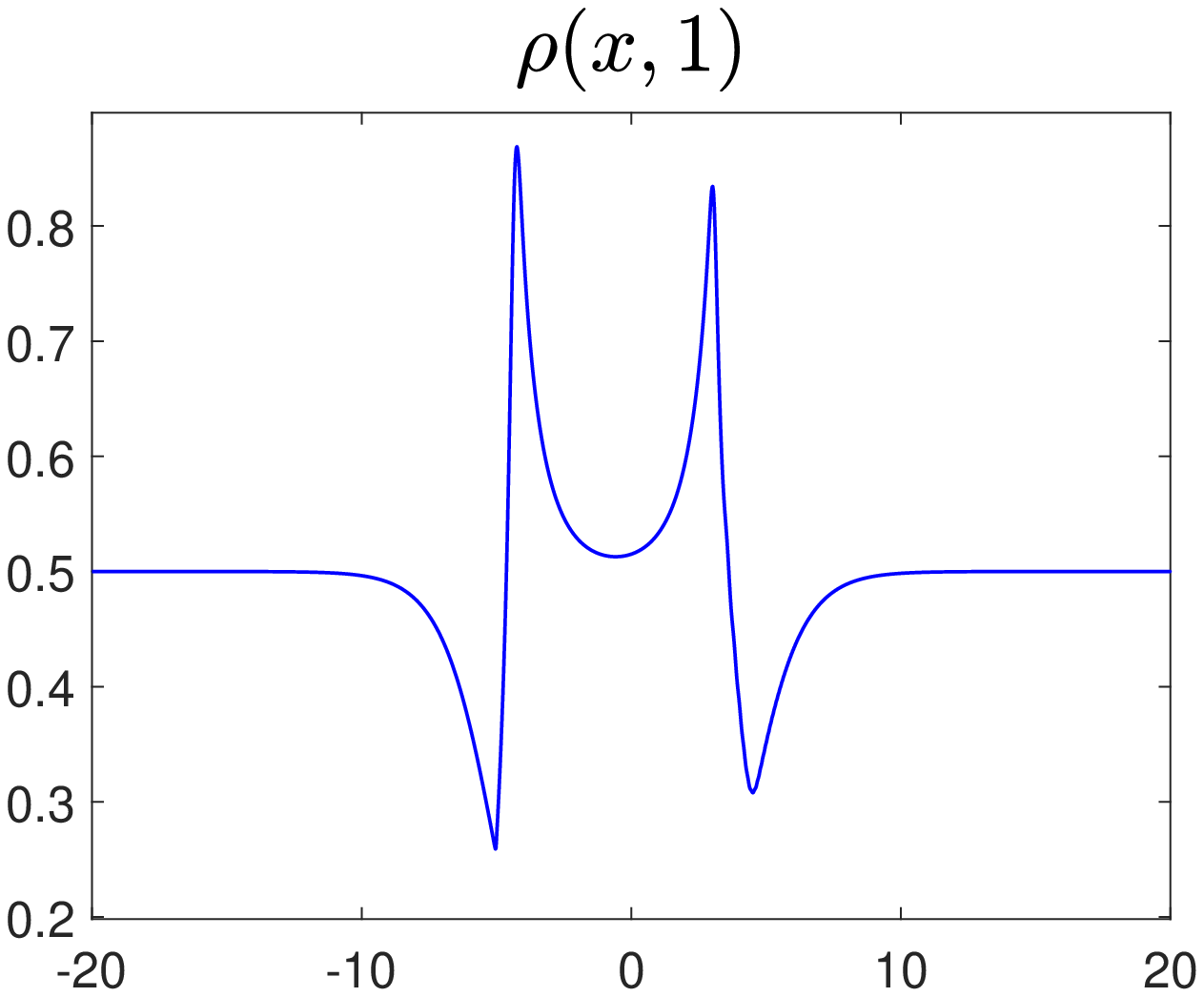}}
\hspace{-15pt}
\subfigure[Case E, $t=3$]{ \centering
\includegraphics[width=0.26\textwidth]{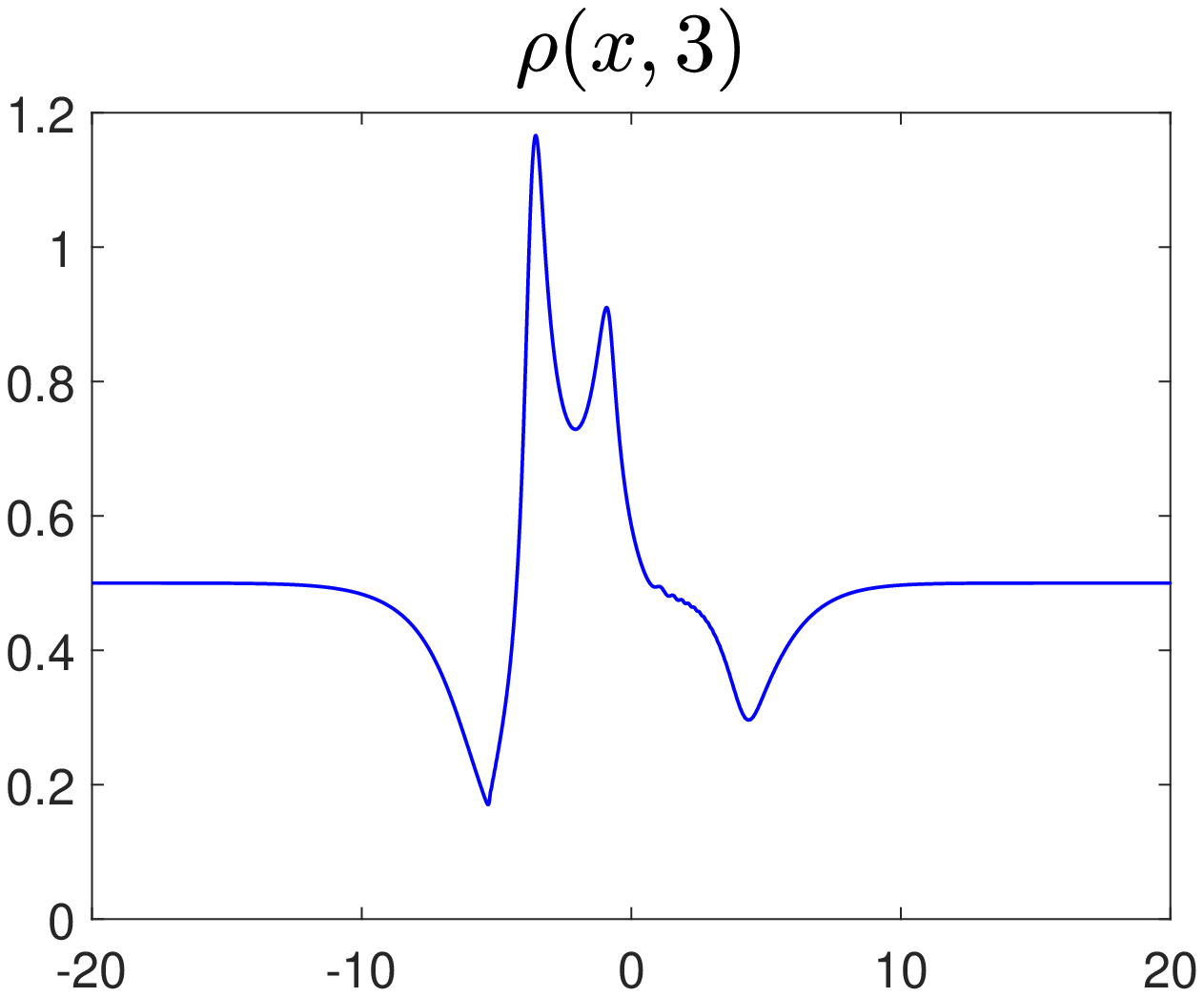}}
\hspace{-15pt}
\subfigure[Case E, $t=6$]{ \centering
\includegraphics[width=0.26\textwidth]{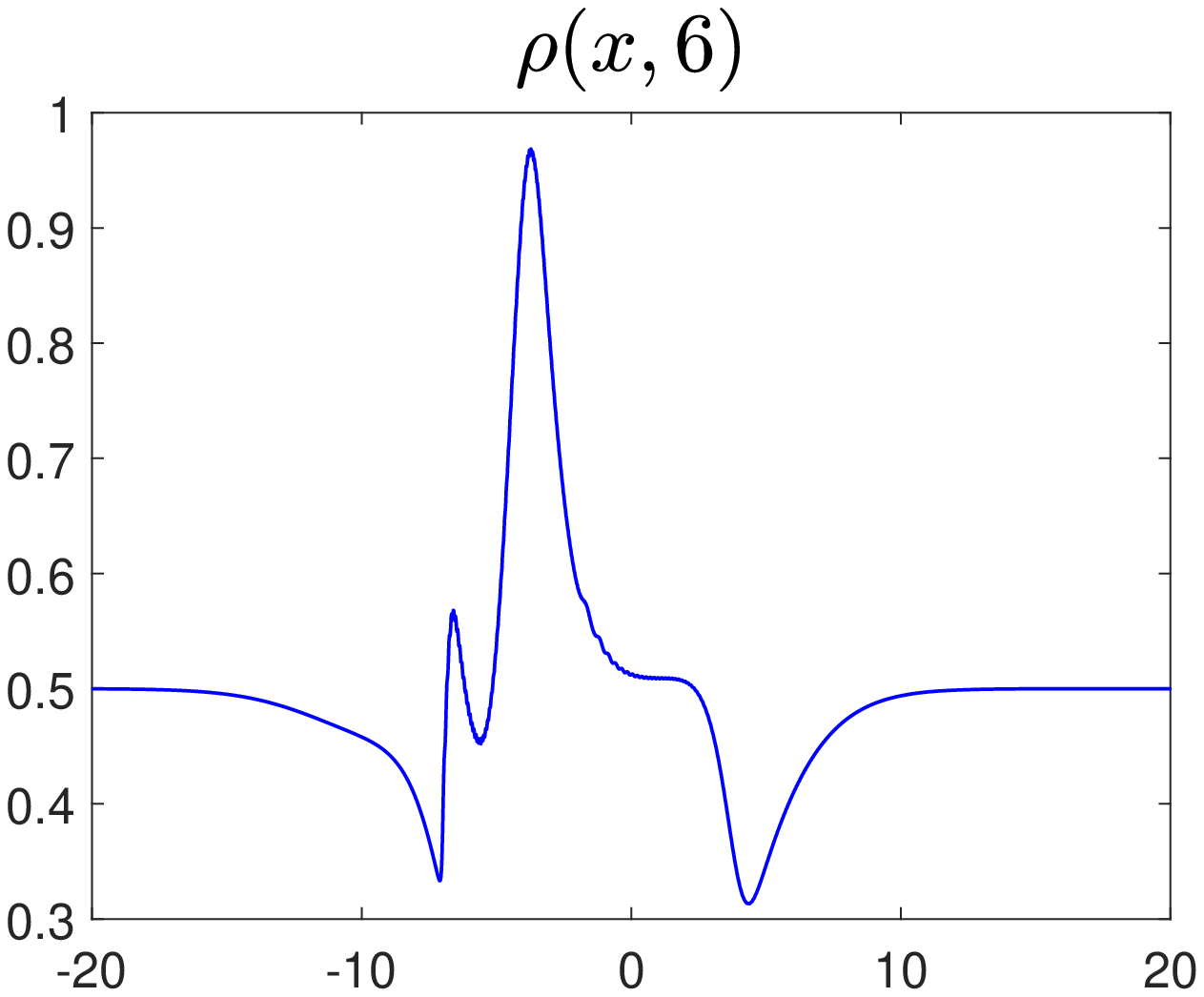}}
\hspace{-15pt}
\subfigure[Case E, $t=8$]{ \centering
\includegraphics[width=0.26\textwidth]{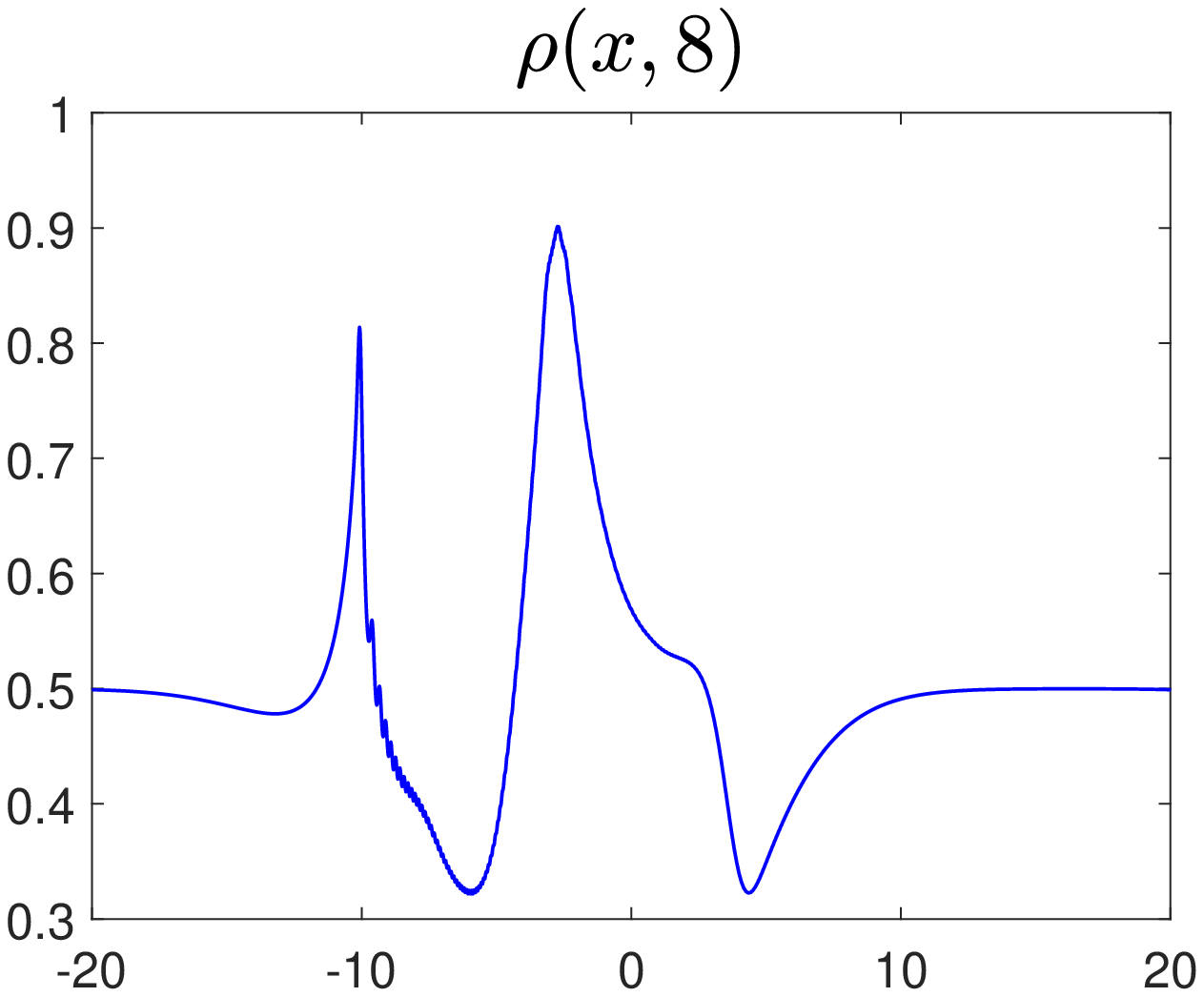}}
\caption{The altitudes of wave propagation at different instance of times calculated by the difference scheme \eqref{equa3.7} in five different Cases;
the spatial grid stepsize is fixed as $h = 0.02$} \label{fig:5}
\end{figure}

\begin{figure}[htbp]
\subfigtopskip=2pt
\subfigbottomskip=2pt
\subfigcapskip=-3pt
\subfigure[Case A-$u$]{ \centering
\includegraphics[width=0.26\textwidth]{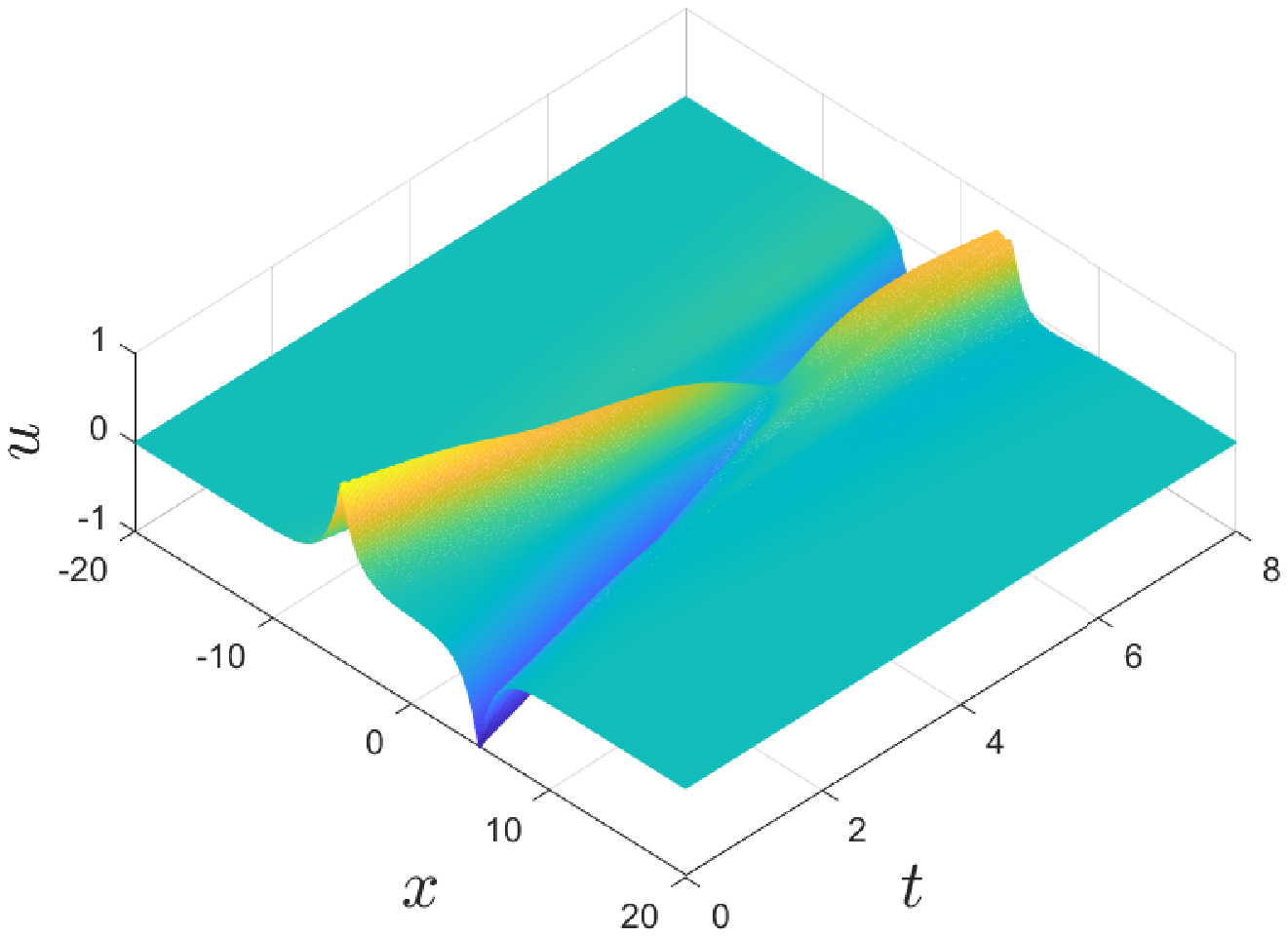}}
\hspace{-15pt}
\subfigure[Case A-$\rho$]{ \centering
\includegraphics[width=0.26\textwidth]{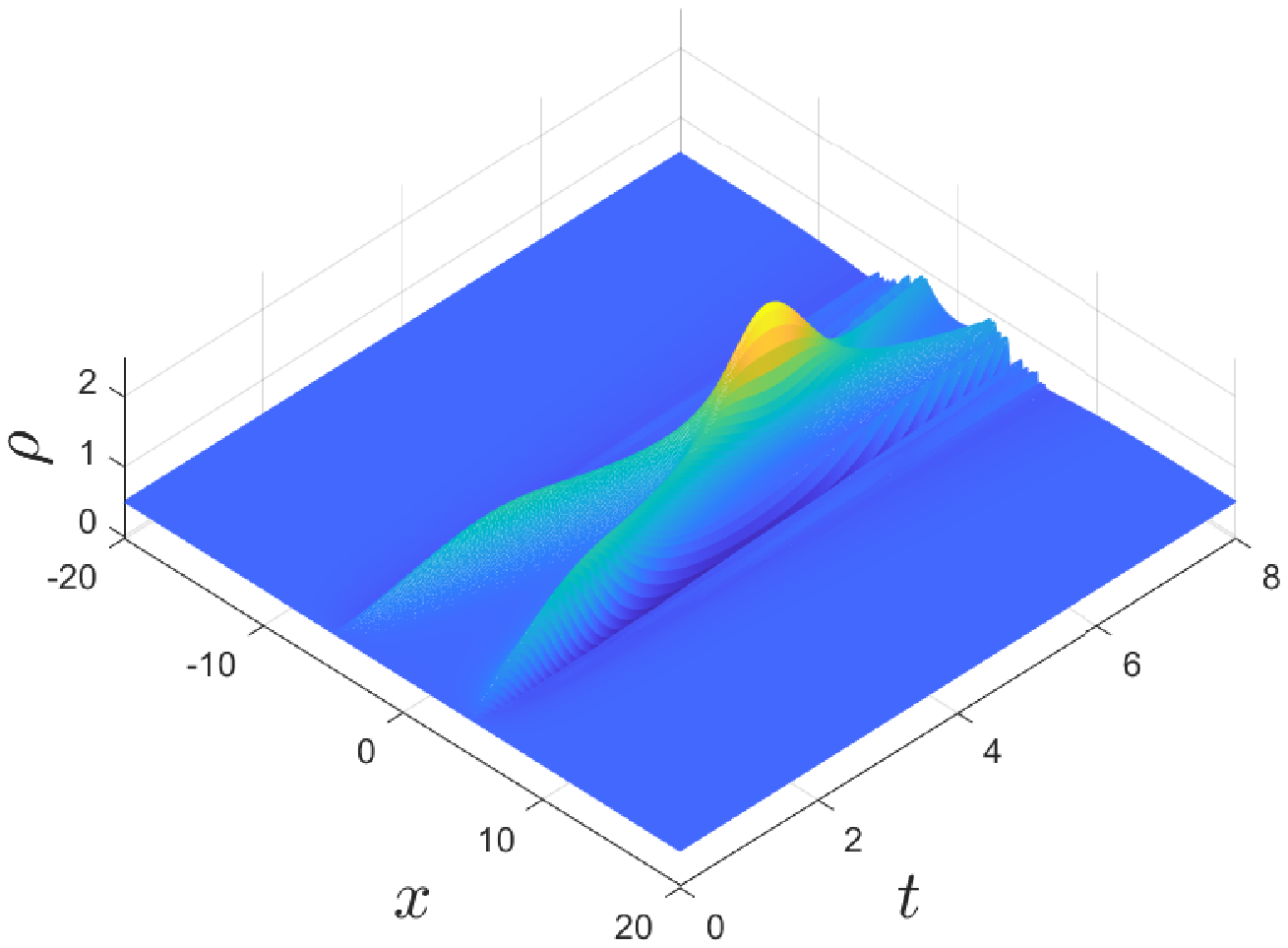}}
\hspace{-15pt}
\subfigure[Case A-Invariants]{ \centering
\includegraphics[width=0.26\textwidth]{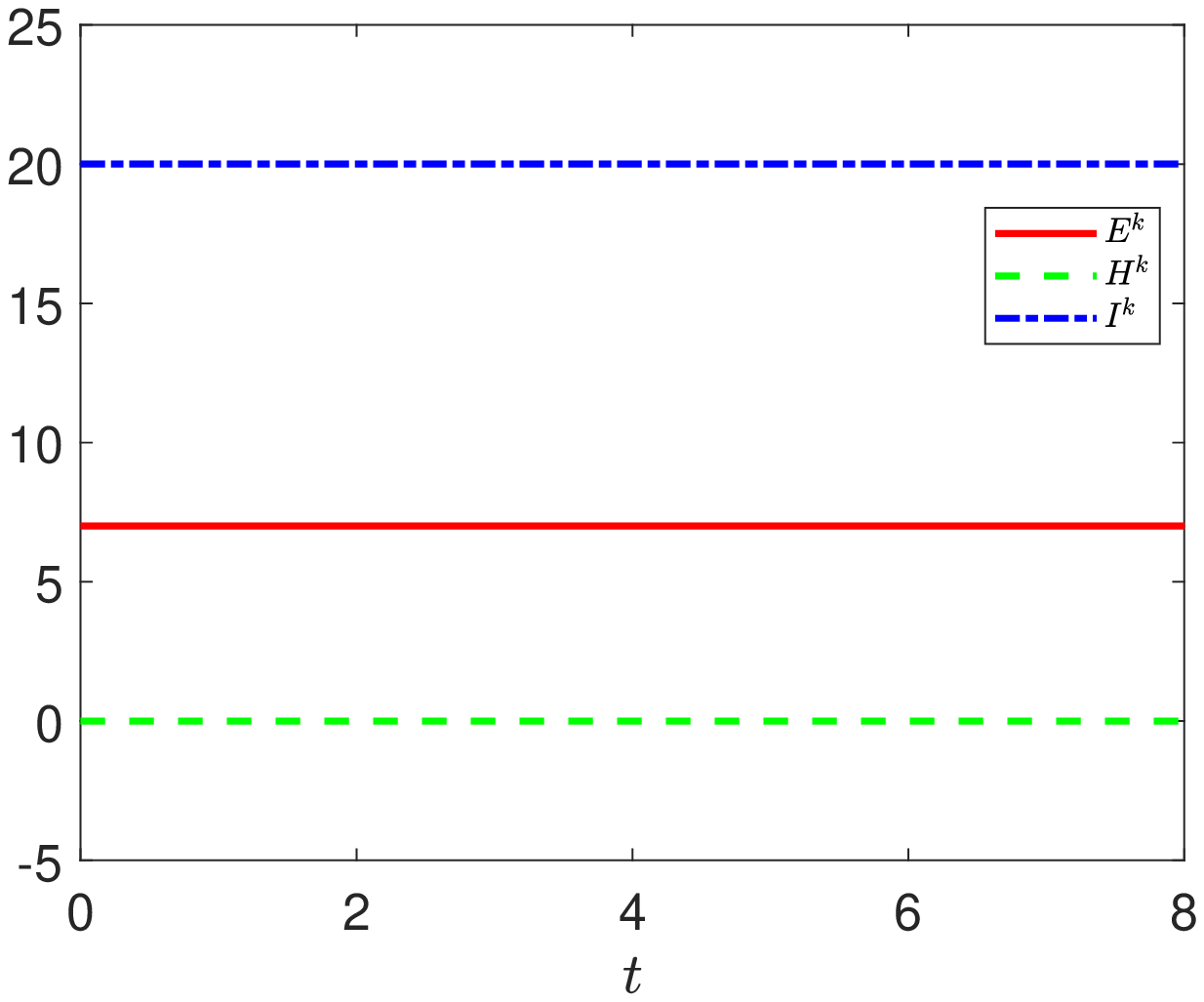}}
\hspace{-15pt}
\subfigure[Errors]{ \centering
\includegraphics[width=0.26\textwidth]{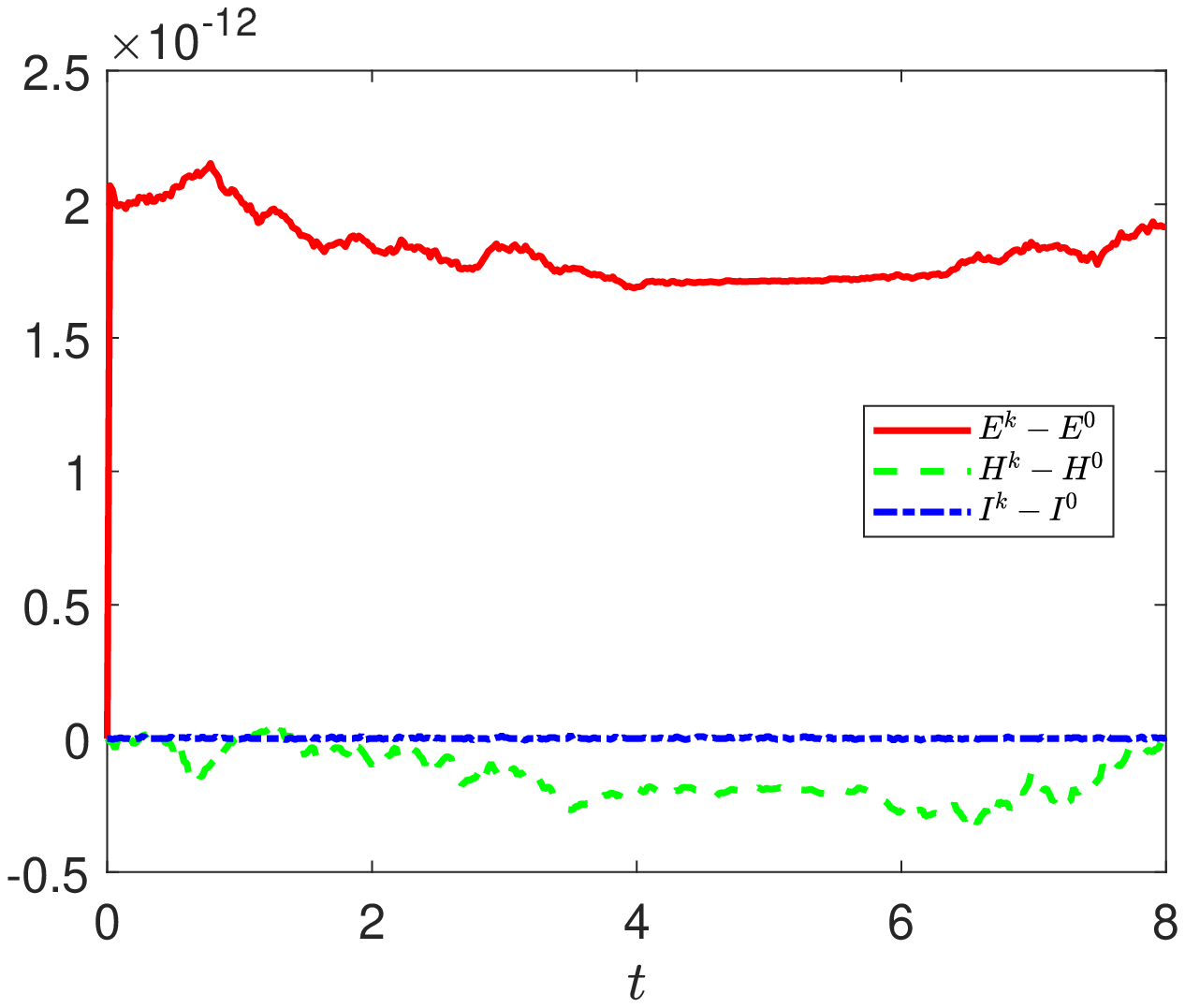}}
\\
\subfigure[Case B-$u$]{ \centering
\includegraphics[width=0.26\textwidth]{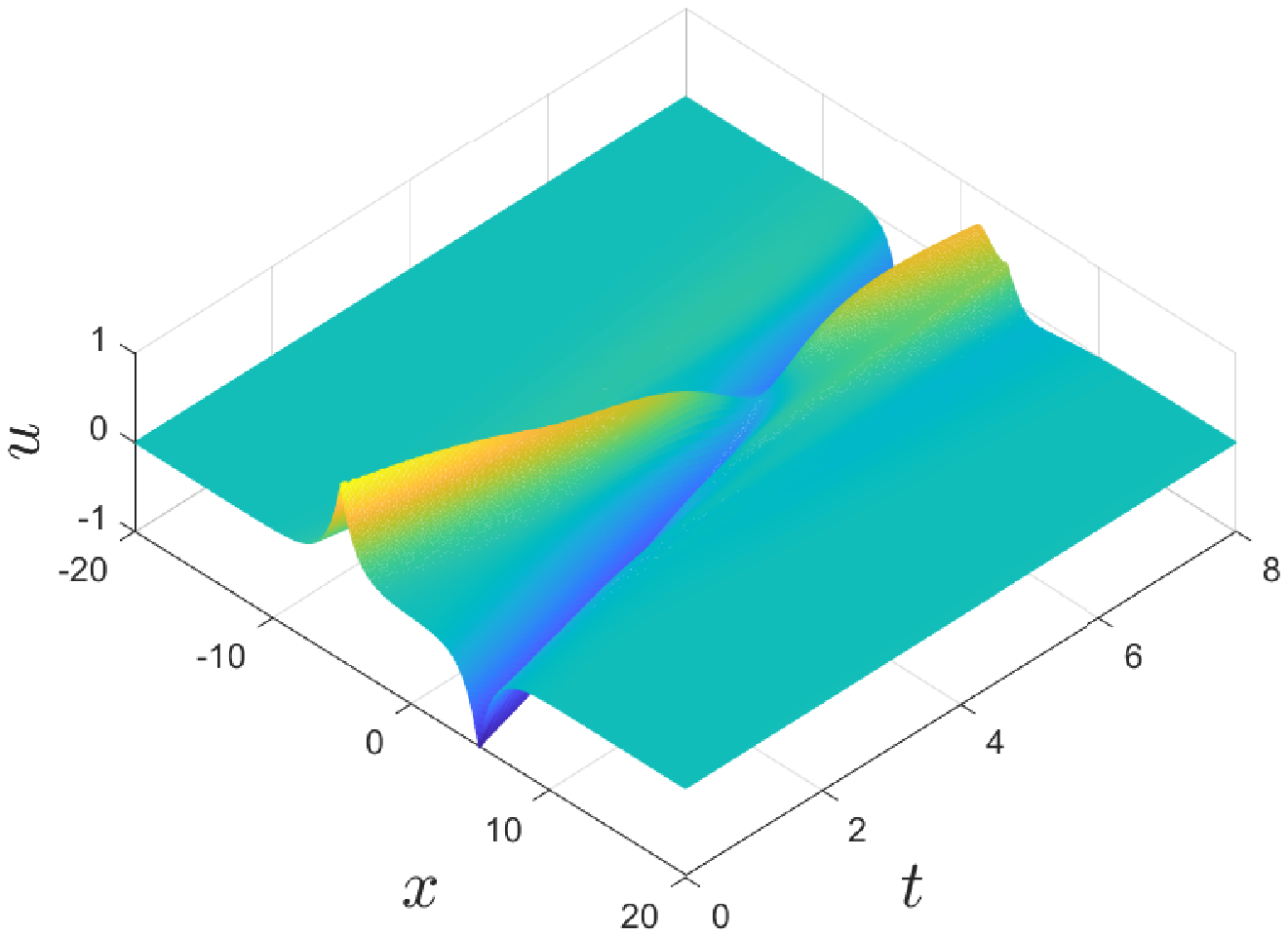}}
\hspace{-15pt}
\subfigure[Case B-$\rho$]{ \centering
\includegraphics[width=0.26\textwidth]{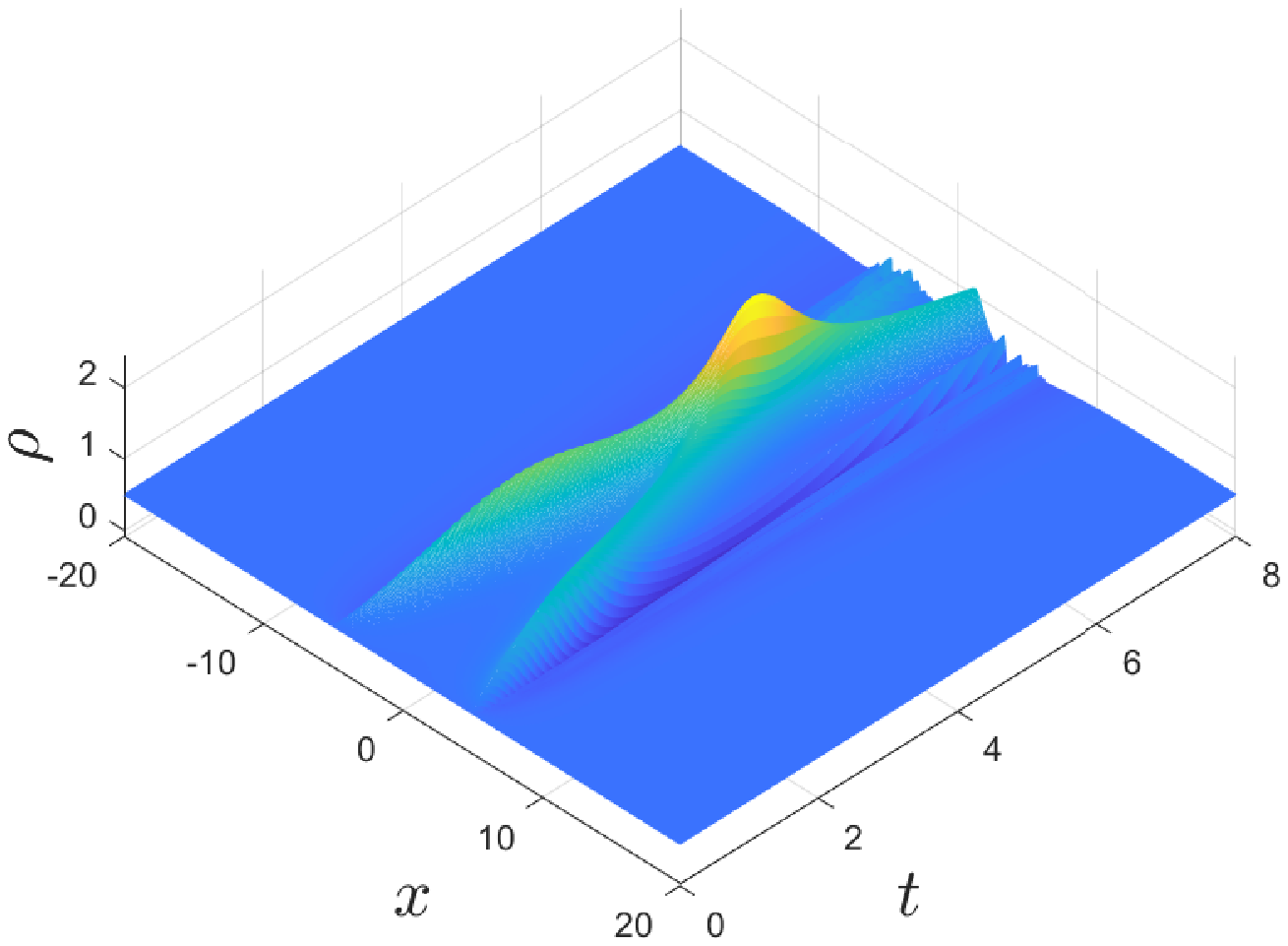}}
\hspace{-15pt}
\subfigure[Case B-Invariants]{ \centering
\includegraphics[width=0.26\textwidth]{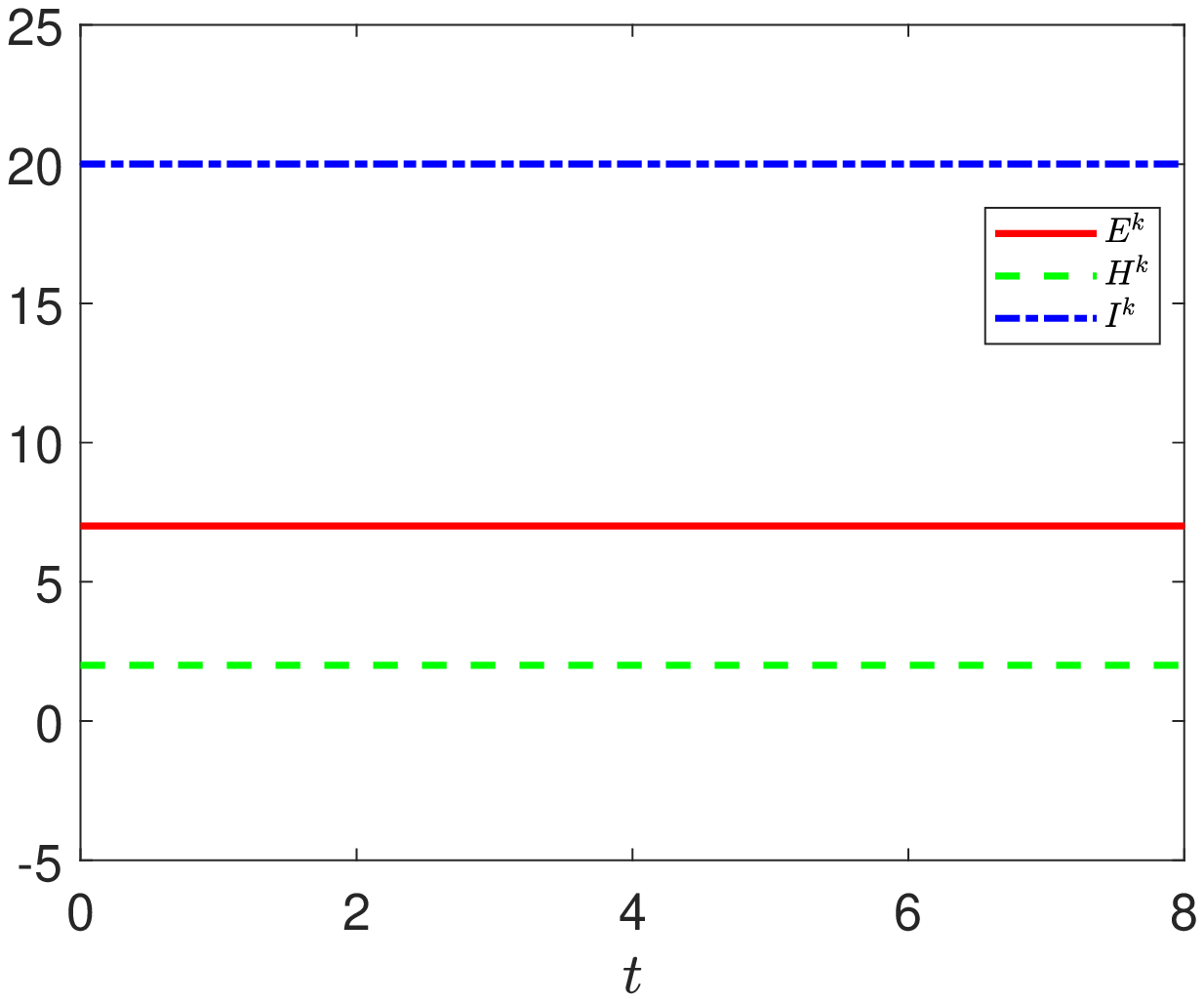}}
\hspace{-15pt}
\subfigure[Errors]{ \centering
\includegraphics[width=0.26\textwidth]{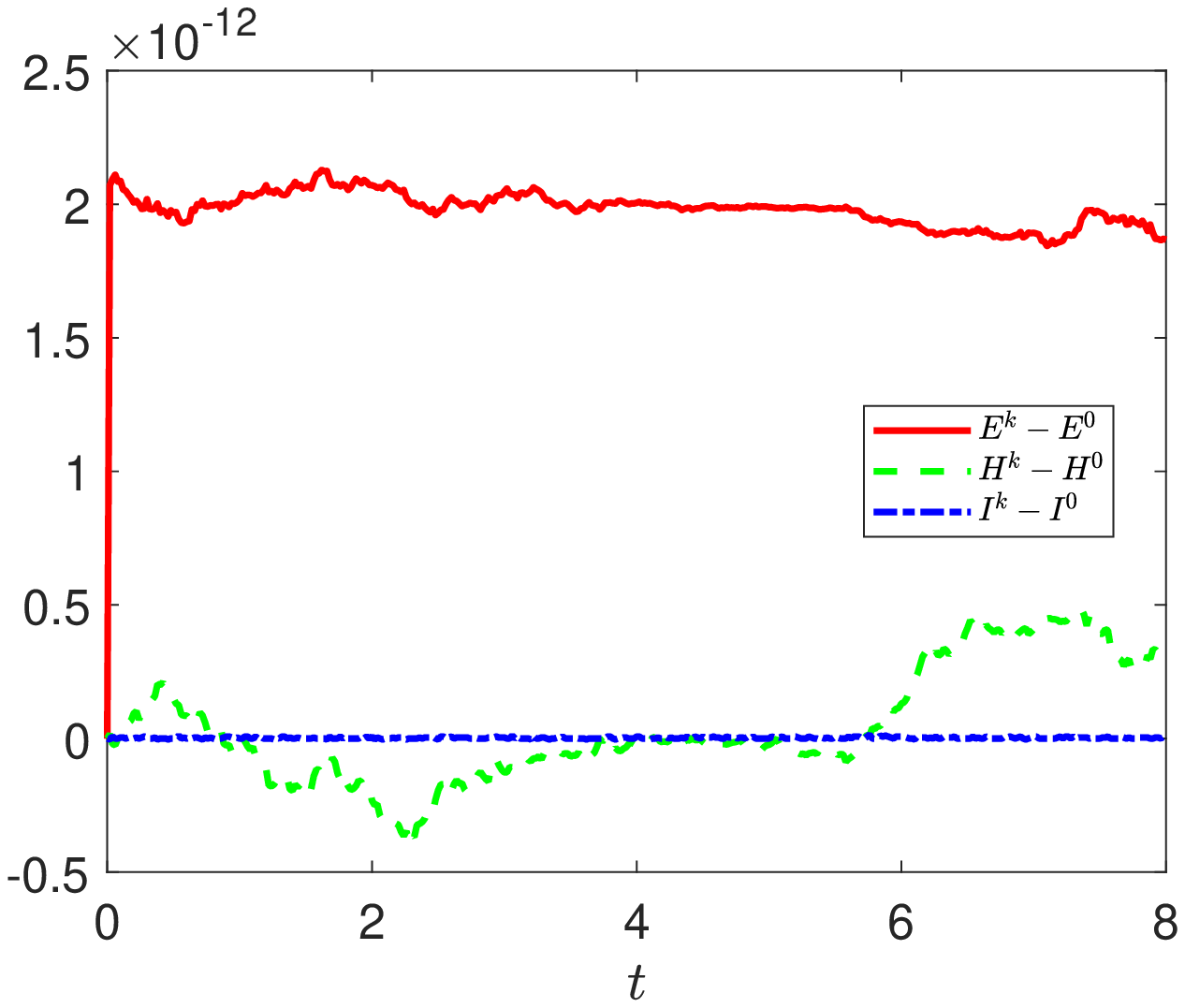}}
\\
\subfigure[Case C-$u$]{ \centering
\includegraphics[width=0.26\textwidth]{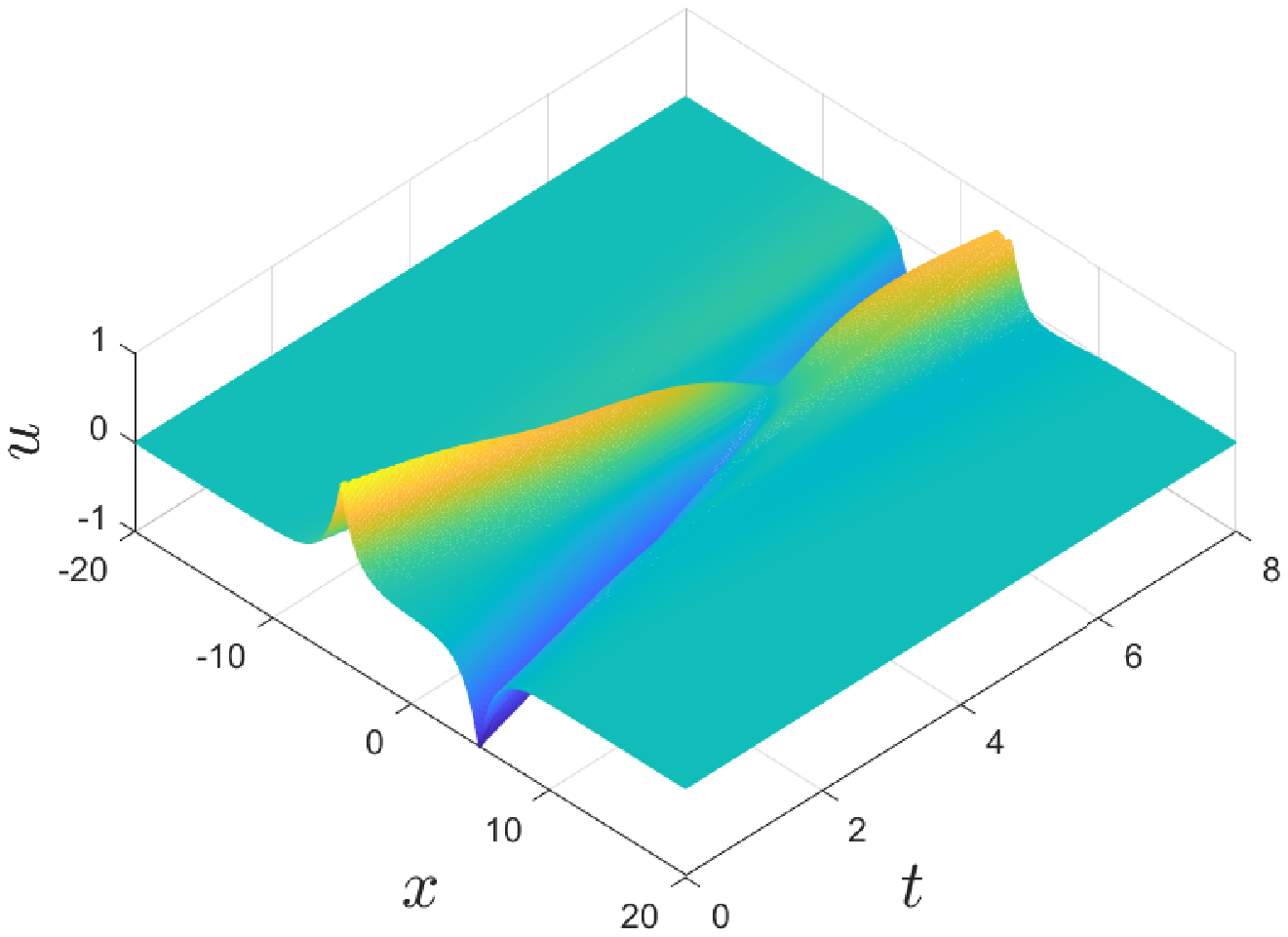}}
\hspace{-15pt}
\subfigure[Case C-$\rho$]{ \centering
\includegraphics[width=0.26\textwidth]{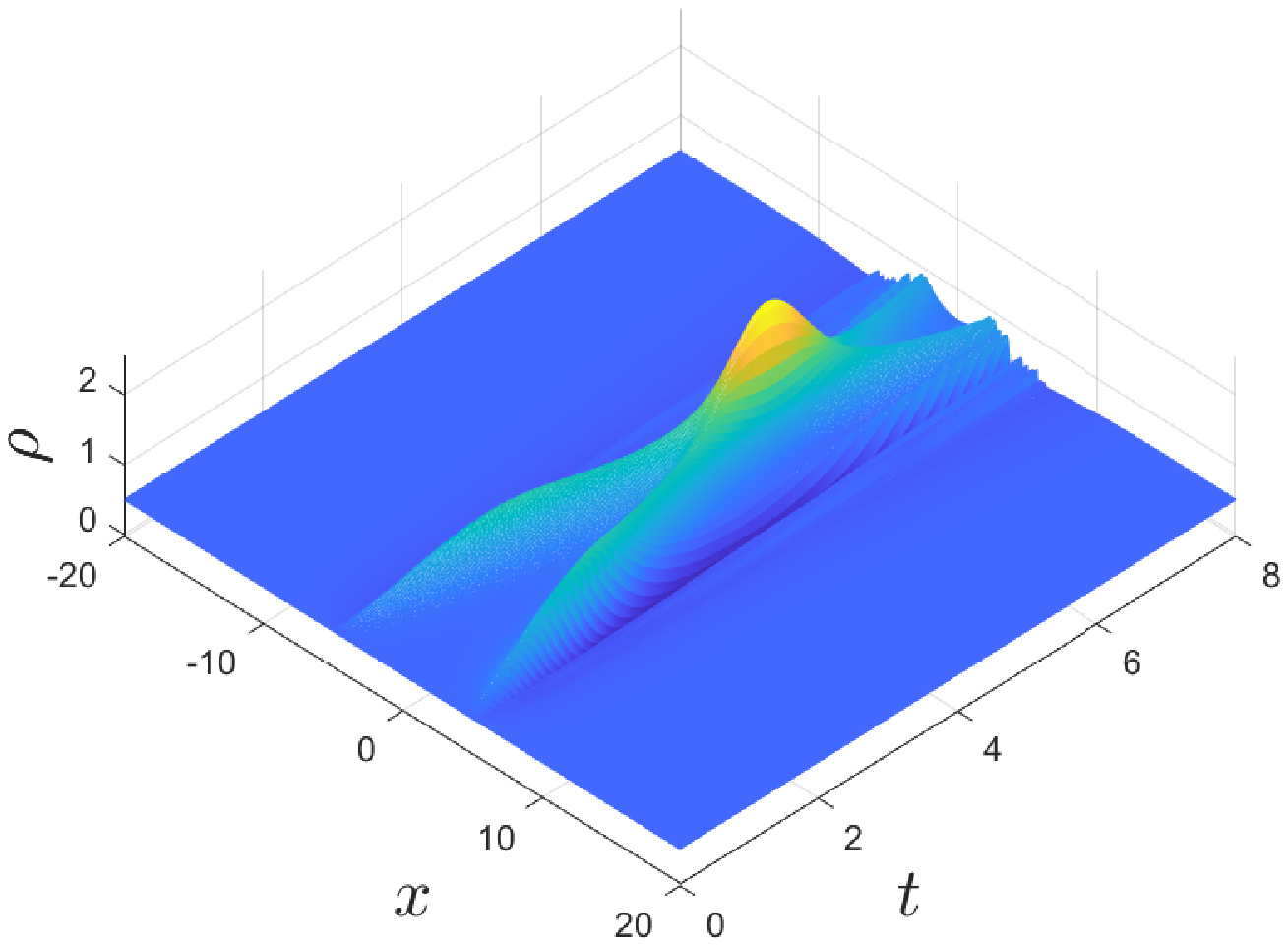}}
\hspace{-15pt}
\subfigure[Case C-Invariants]{ \centering
\includegraphics[width=0.26\textwidth]{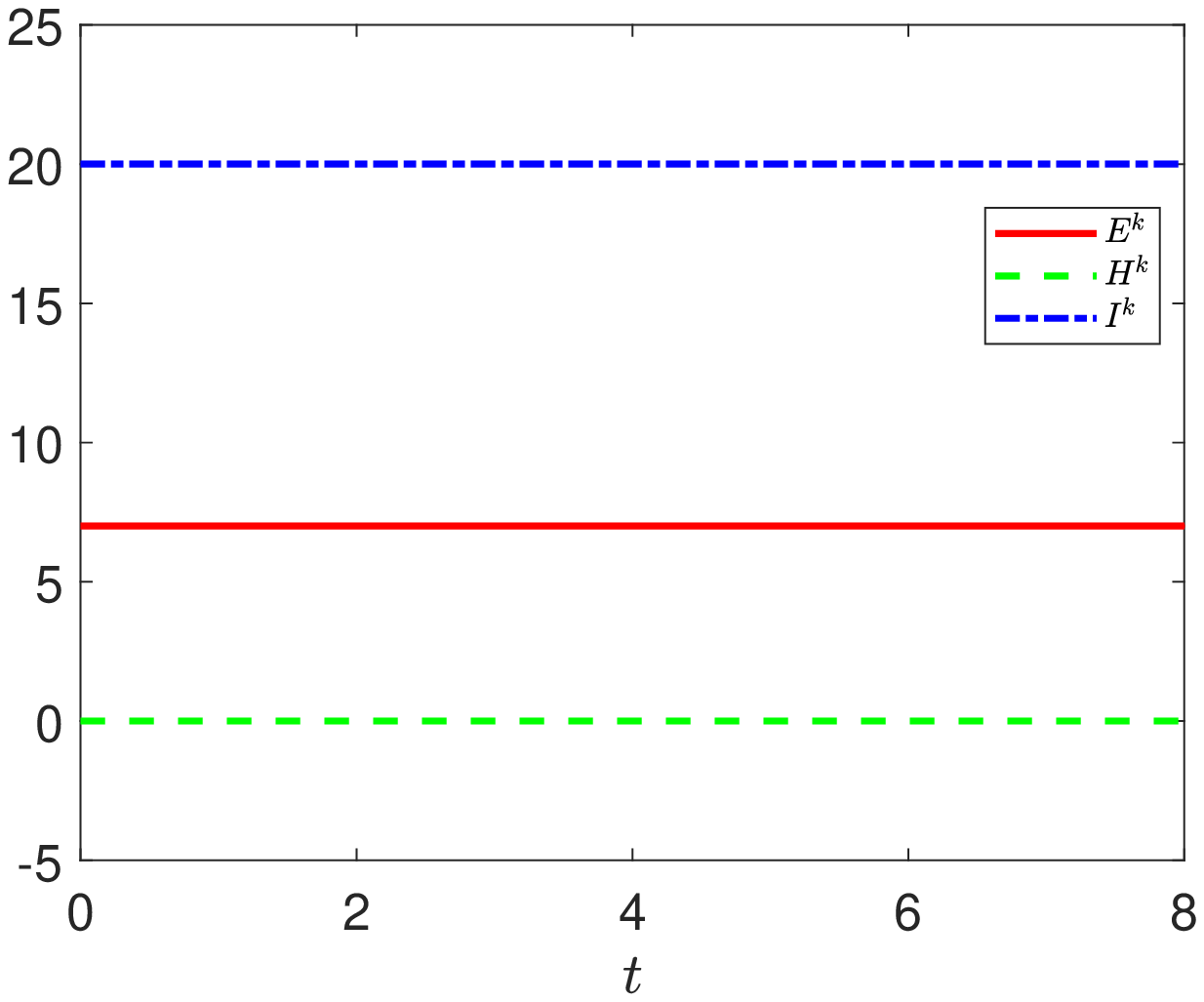}}
\hspace{-15pt}
\subfigure[Errors]{ \centering
\includegraphics[width=0.26\textwidth]{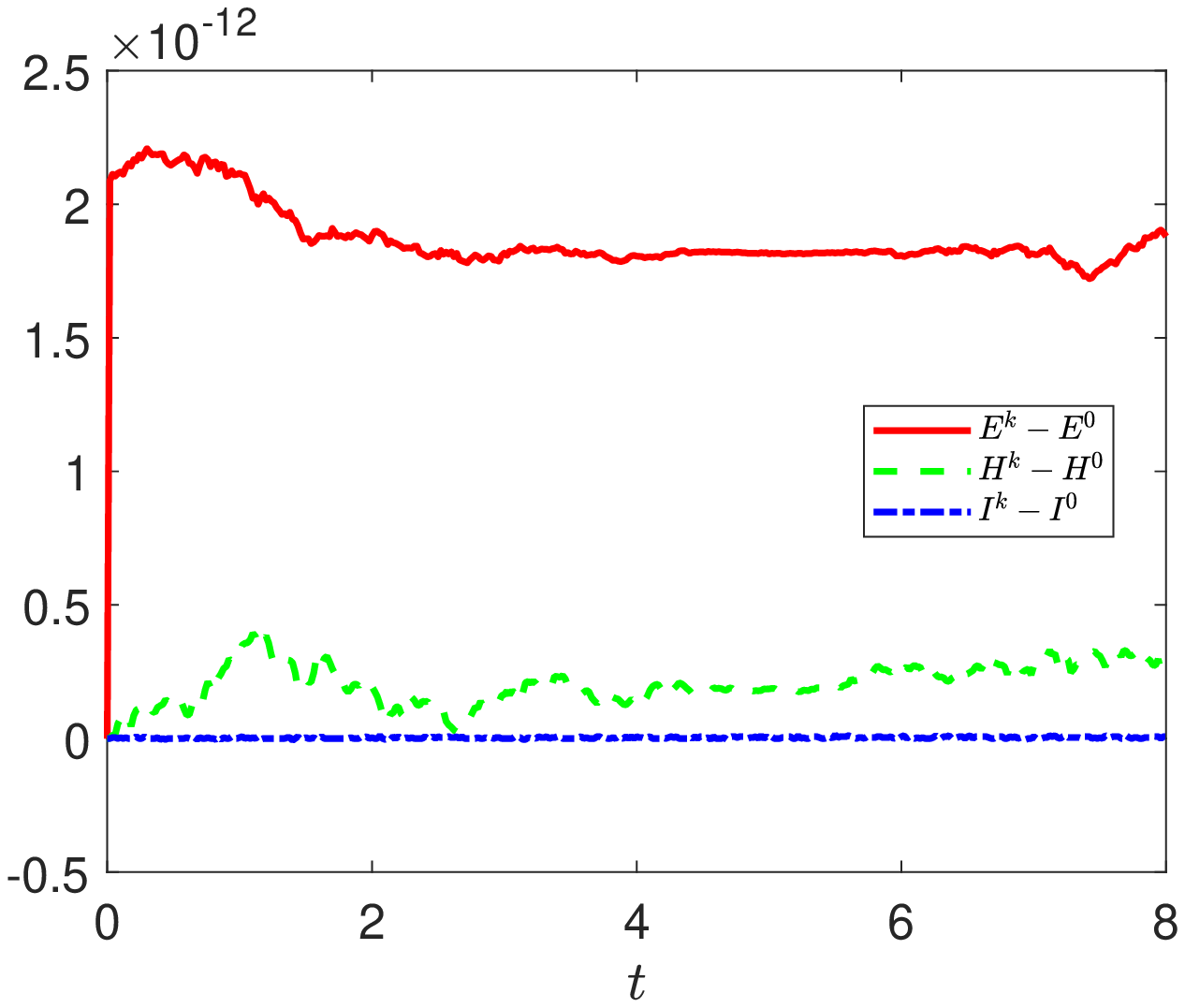}}
\\
\subfigure[Case D-$u$]{ \centering
\includegraphics[width=0.26\textwidth]{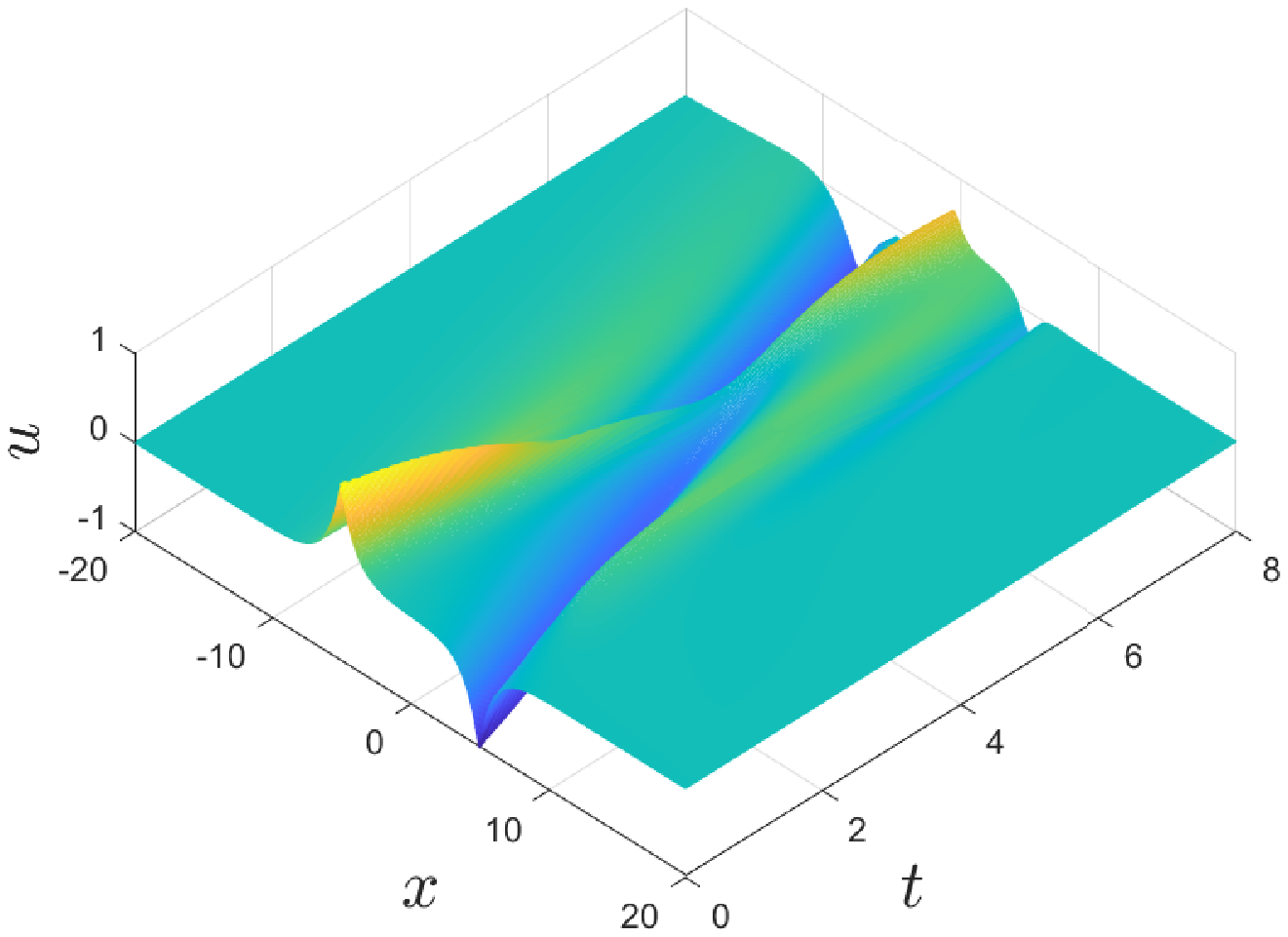}}
\hspace{-15pt}
\subfigure[Case D-$\rho$]{ \centering
\includegraphics[width=0.26\textwidth]{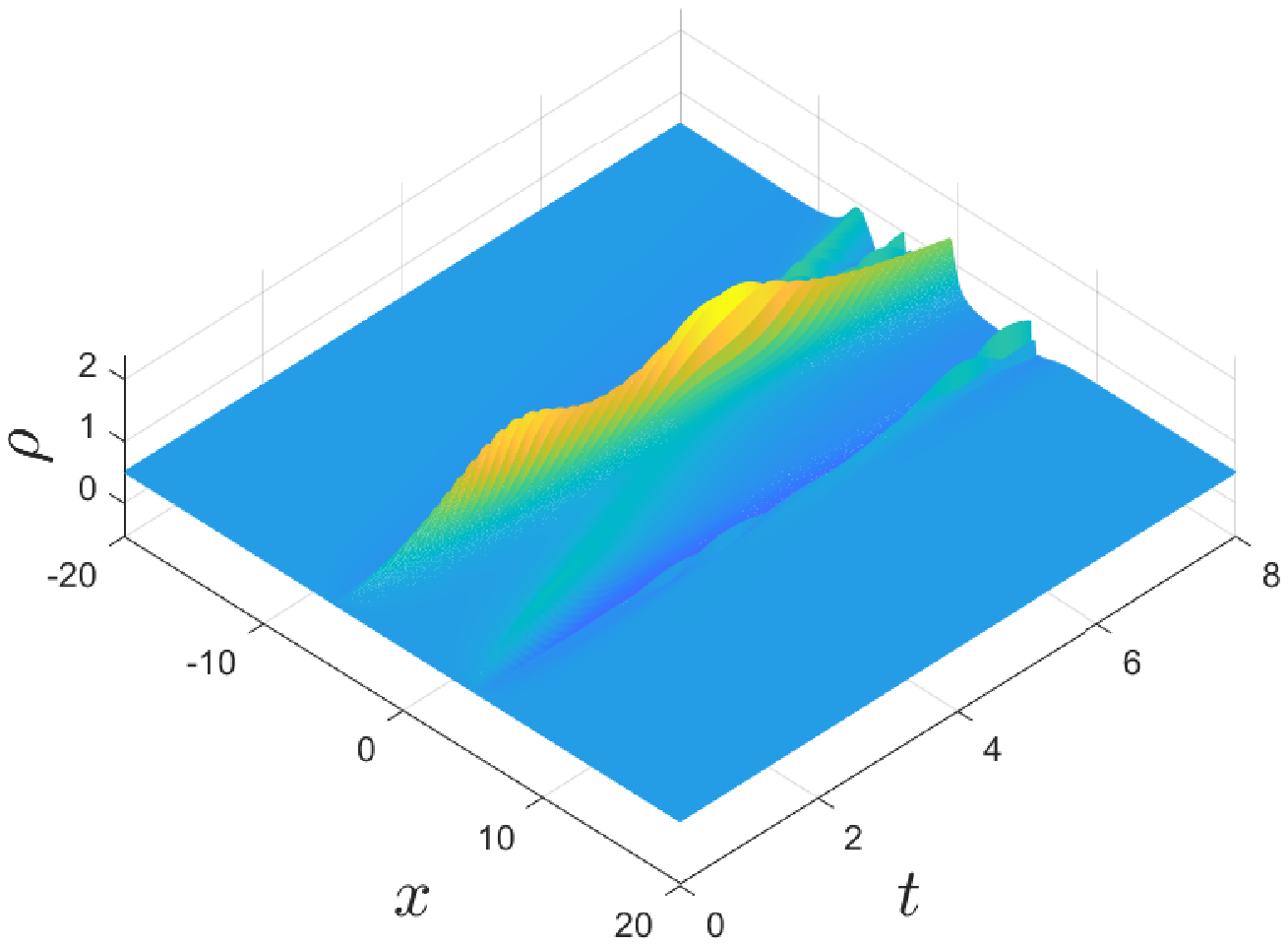}}
\hspace{-15pt}
\subfigure[Case D-Invariants]{ \centering
\includegraphics[width=0.26\textwidth]{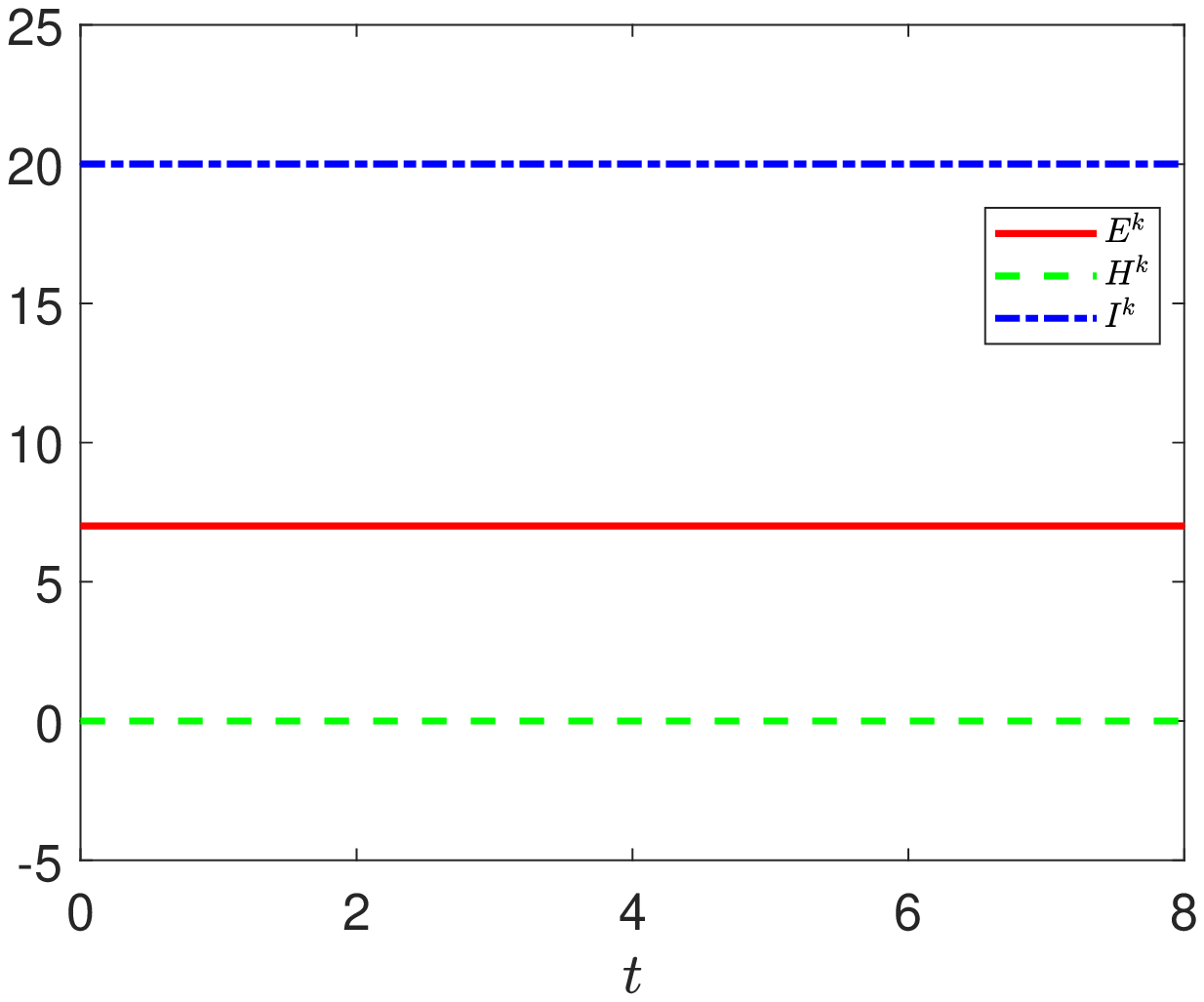}}
\hspace{-15pt}
\subfigure[Errors]{ \centering
\includegraphics[width=0.26\textwidth]{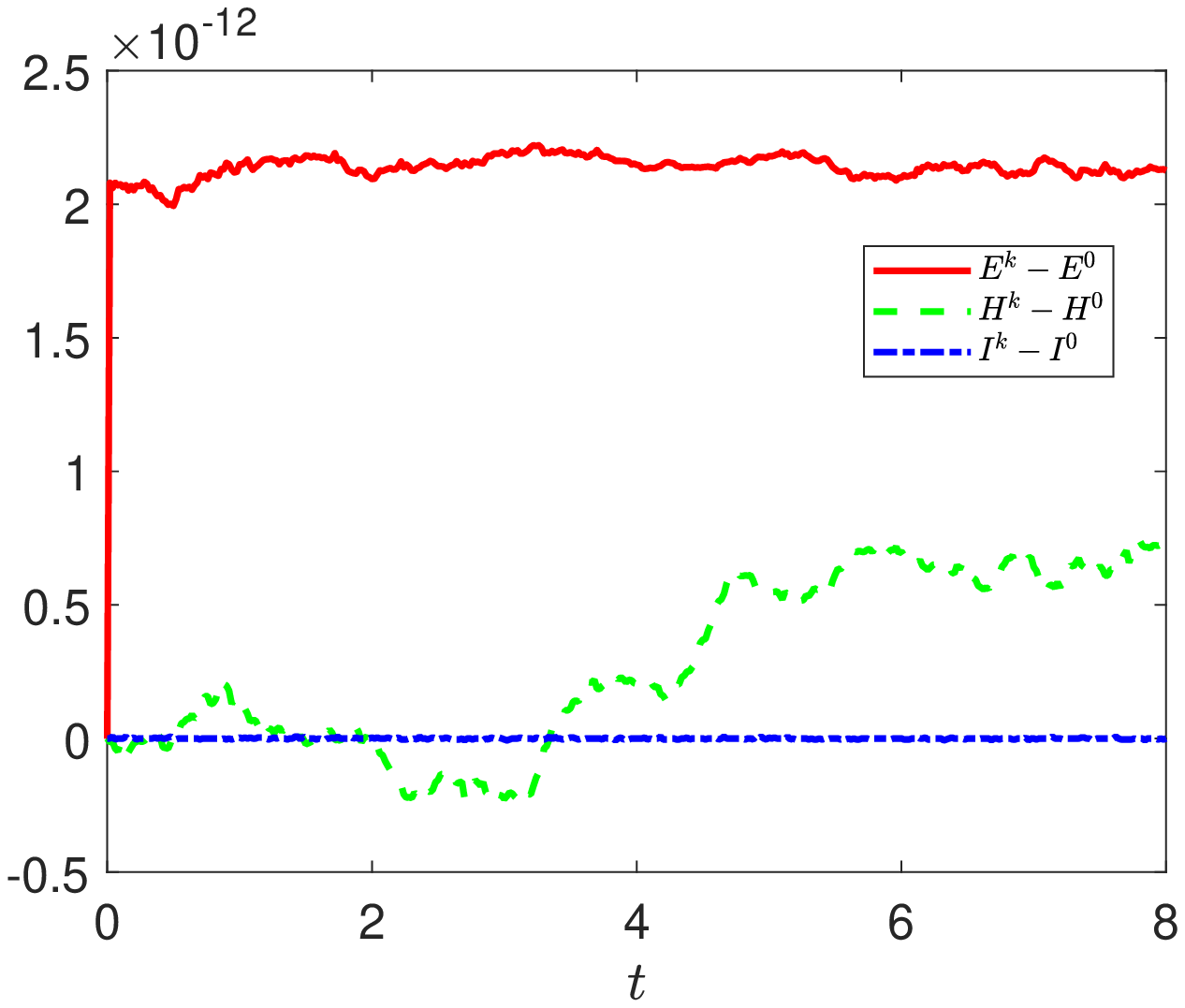}}
\\
\subfigure[Case E-$u$]{ \centering
\includegraphics[width=0.26\textwidth]{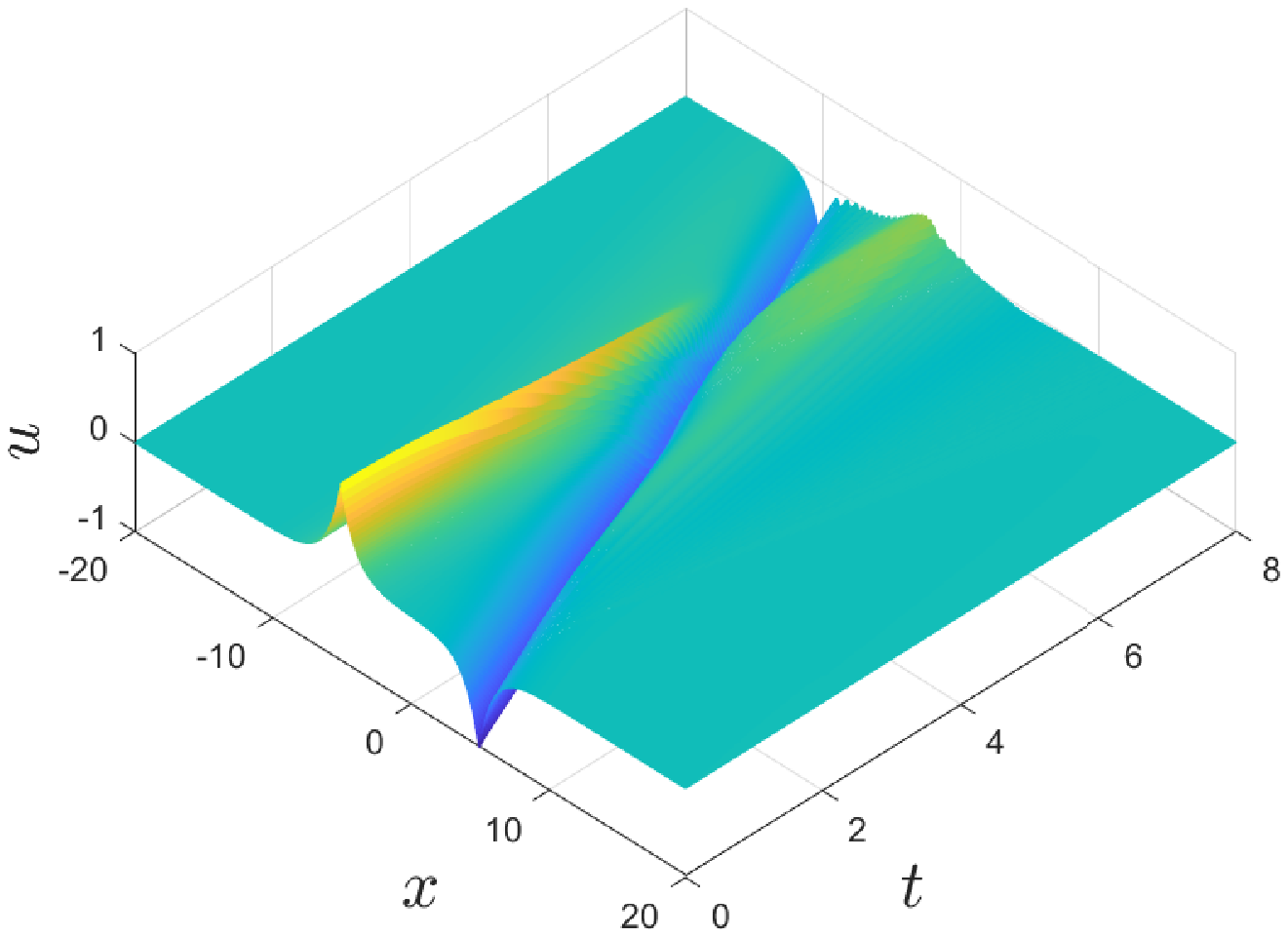}}
\hspace{-15pt}
\subfigure[Case E-$\rho$]{ \centering
\includegraphics[width=0.26\textwidth]{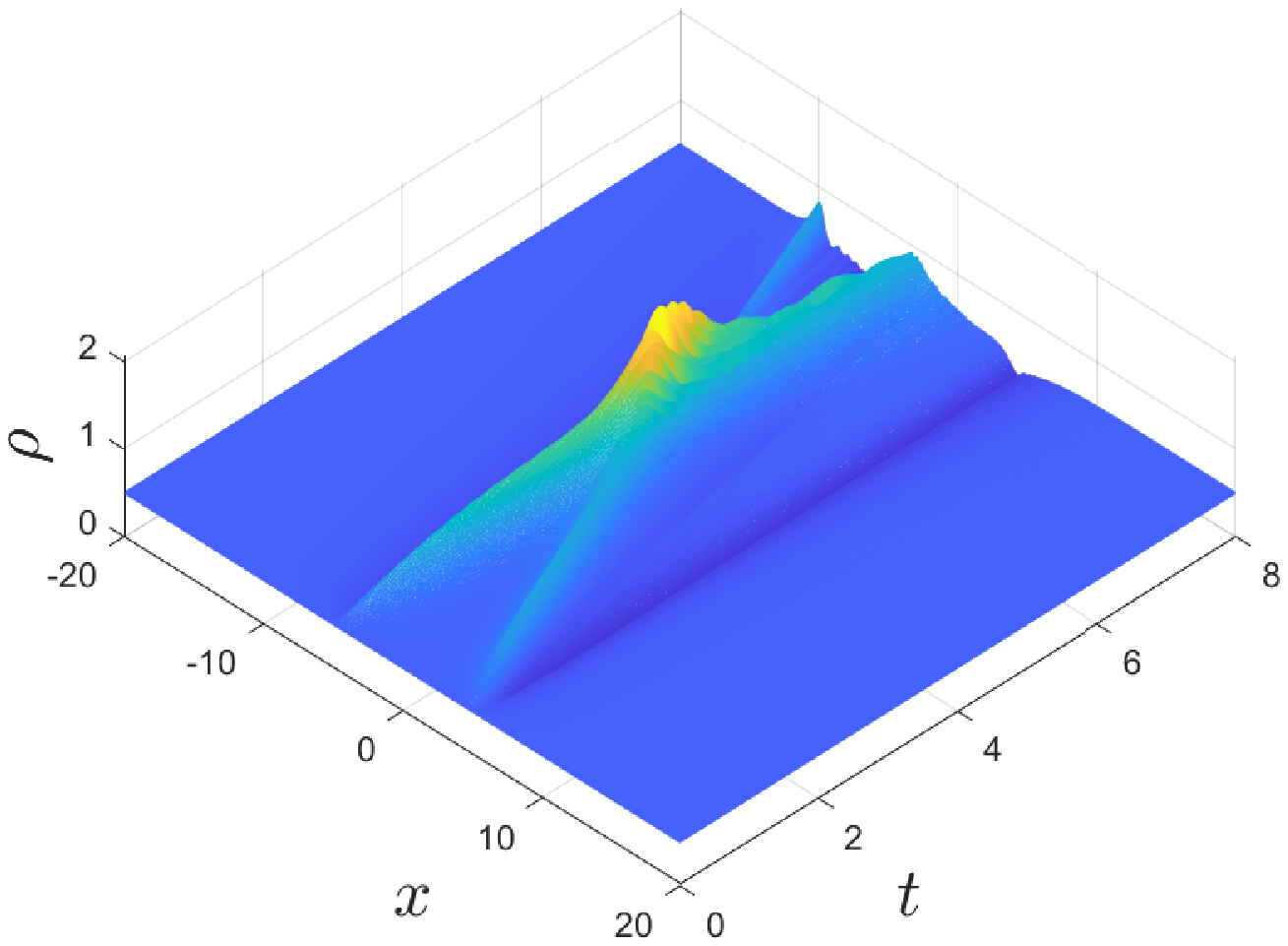}}
\hspace{-15pt}
\subfigure[Case E-Invariants]{ \centering
\includegraphics[width=0.26\textwidth]{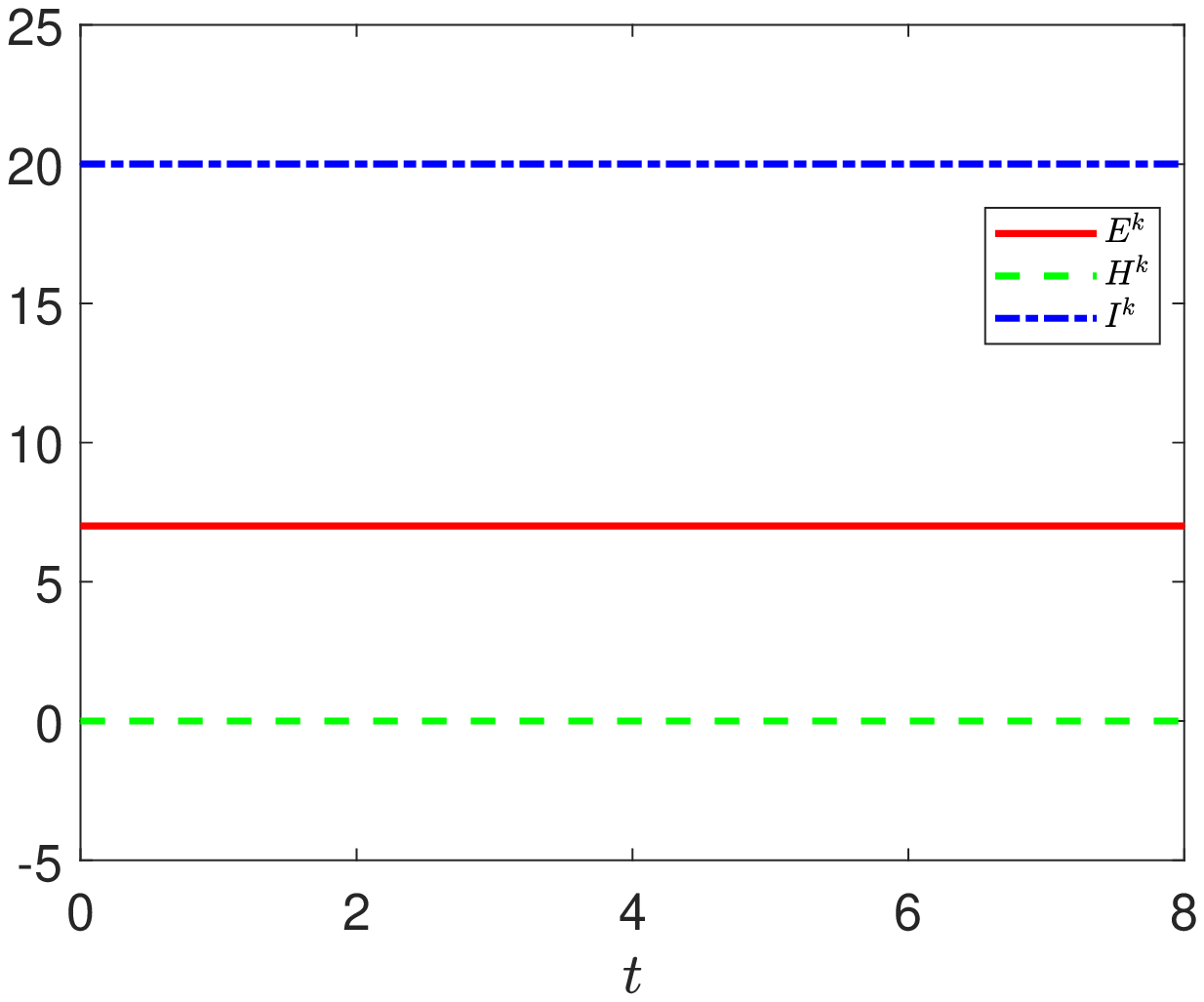}}
\hspace{-15pt}
\subfigure[Errors]{ \centering
\includegraphics[width=0.26\textwidth]{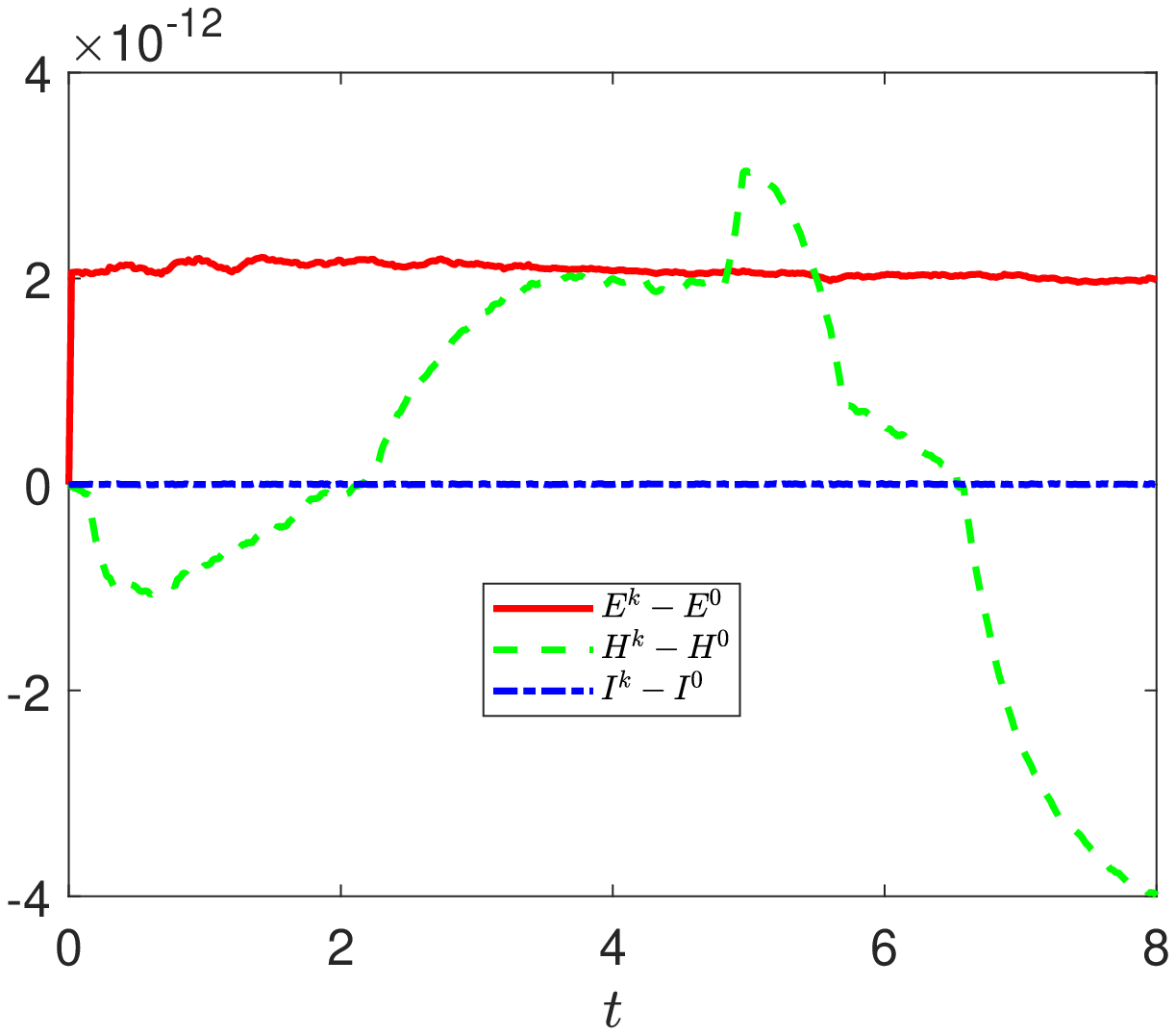}}

\caption{The evolution surfaces of the velocity $u$ (Column one) and the altitude $\rho$ (Column two), and numerical invariants (Column three) and corresponding numerical error curves of the invariants (Column four) for five Cases with the stepsizes $h=1/10$ and $\tau = 1/50$.} \label{fig:6}
\end{figure}

  Figure \ref{fig:4} and Figure \ref{fig:5} for the benchmark problem in {\rm \bf Case A} depict the evolution of the velocity and altitude, respectively, at $t=1,\,3,\,6,\,8$ for $\Omega=0$, which recover the elastic collision, the case studied in the literature, see also \cite{CKL2020,CMR2014,LP2016,YFS2018}.
The second row in Figure \ref{fig:4} and Figure \ref{fig:5} shows the numerical evolution when the rotation parameter $\Omega=0.2$, respectively with the same spatial interval $[-20,20]$ for {\rm \bf Case B}. Compared the results of {\rm \bf Case A} with those of {\rm \bf Case B}, we see that the nonzero rotation parameter has an important effect on the evolution of the solution, i.e., the solution is no longer symmetric.
Next, we decrease the rotation parameter into the practical value $\Omega=73\times 10^{-6}$ in {\rm \bf Case C}, which is a very tiny value.
We observe that the solution is very similar to that of {\rm \bf Case A} (the zero rotation parameter) in such a short time horizon. Moreover, let the linear underlying shear flow $\kappa=1$ in {\rm \bf Case D}, we find that the solution also loses the symmetry compared with the result in  {\rm \bf Case C}. Analogously, we change the dimensionless parameter $\mu$ into unit one in {\rm \bf Case E}, the collision of two waves is also no longer elastic. Furthermore, all the invariants are still preserved very well under the above five cases, see e.g., Figure \ref{fig:6}. In addition, the evolution graphs for the velocity $u$ and the altitude $\rho$ are clearly displayed in the first and second columns in Figure \ref{fig:6}. In a word, these numerical results indicate that if any of the parameters $\Omega$, $\kappa$, and $\mu$ is nonzero, the solution will be asymmetrical, which could be qualitatively verified from the perspective of theoretical analysis.

\section{Conclusions and outlooks}\label{section6}

To summarize, we propose and analyze a nonlinear difference scheme for the R2CH system based on a framework of the bilinear operator, and obtain several new numerical results such as unconditional convergence and invariant-preserving properties including energy, momentum, and mass.
These ensure that the numerical scheme provides an accurate long-time evolution of solitary waves both in smooth and nonsmooth initial values.
Regrettably, the present paper also leaves some loose ends, which are addressed as follows.
\begin{itemize}
  \item  When $\Omega\neq 0$, the convergence result  is valid only for small initial data ($c_{\rm max}<\frac{1}{2\Omega}-\kappa$) due to technical reasons. Considering the practical case $\Omega \approx 73\cdot 10^{-6}$, this is a mild restriction. Nevertheless, it would be desirable to remove this restriction. Indeed, we performed some numerical tests with $c_{\rm max}\geqslant\frac{1}{2\Omega}-\kappa$, and our numerical scheme is still working. On the other hand, the numerical theory does break down for very large $c_{\rm max}$.
  \item From long-time numerical simulation (see e.g., \textbf{Cases E--F} in Example \ref{exam1}), we observe clearly that even if the initial values of the R2CH system are smooth, the solutions may evolve into rough or cuspidal in finite time. This phenomenon is worthy of further study.
  \item Our preliminarily numerical tests in Example \ref{exam2} capture the evolution of the R2CH system with a nonsmooth initial velocity based on the difference scheme \eqref{equa3.7}, the theoretical analysis is necessary to cover this case.
  \item It is worth applying the framework of the bilinear operator to solve and analyze other types of the shallow water wave problems.
\end{itemize}

\section*{Conflict of interest}
The authors of this paper have no conflict of interest to declare.
\section*{Data availability}
Data will be made available on reasonable request.

\small{
\begin{acknowledgements}
The authors would like to thank Prof. Zhi-zhong Sun for most helpful discussions and suggestions.
Part of the work was finished during Qifeng's visit in \'{E}cole Polytechnique F\'{e}d\'{e}rale de Lausanne,
and he would like to thank Prof. Jan S. Hesthaven for his hospitality in 2021-2022.

\end{acknowledgements}}

\begin{thebibliography}{}
\bibitem{AK1993}
{Akrivis, G.D.}:
{Finite difference discretization of the cubic Schr\"{o}dinger equation}.
{IMA J. Numer. Anal.} {\textbf{13}}, 115--124 (1993)

 \bibitem{ADM2019}
 {Antonopoulos, D.C., Dougalis, V.A., Mitsotakis, D.E.}:
 {Error estimates for Galerkin finite element methods for the Camassa--Holm equation}.
 {Numer. Math.} {\textbf{142}}, 833--862 (2019)

 \bibitem{BC2007}
 {Bressan, A., Constantin, A.}:
 {Global conservative solutions of the Camassa-Holm equation}.
 {Arch. Rational Mech. Anal.} {\textbf{183}}, 215--239 (2007)

 \bibitem{CH1993}
 {Camassa, R., Holm, D.}:
 {An integrable shallow water equation with peaked solitons}.
 {Phys. Rev. Lett.} {\textbf{71}}, 1661--1664 (1993)

 \bibitem{CL2008}
 {Camassa, R., Lee, L.}:
 {Complete integrable particle methods and the recurrence of initial states for a nonlinear shallow-water wave equation}.
 {J. Comput. Phys.} {\textbf{227}}, 7206--7221 (2008)

  \bibitem{CFGL2017}
 {Chen. R., Fan, L., Gao, H., Liu, Y.}:
 {Breaking waves and solitary waves to the rotation-two-component Camassa-Holm system}.
 {SIAM J. Math. Anal.} {\textbf{49}}, 3573--3602 (2017)

  \bibitem{CL2011}
 {Chen. R., Liu, Y.}:
 {Wave breaking and global existence for a generalized two-component Camassa-Holm system}.
 {Int. Math. Res. Notices} {\textbf{268}}, 45--66 (2011)

  \bibitem{CKL2020}
 {Chertock, A., Kurganov, A., Liu, Y.}:
 {Finite-volume-particle methods for the two-component Camassa-Holm system}.
 {Commun. Comput. Phys.} {\textbf{27}}, 480--502 (2020)

 \bibitem{CKR2008}
 {Coclite, G., Karlsen, K., Risebro, N.}:
 {A convergent finite difference scheme for the Camassa-Holm equation with general $H^1$ initial data}.
 {SIAM J. Numer. Anal.} {\textbf{46}}, 1554--1579 (2008)

 \bibitem{CMR2014}
 {Cohen, D., Matsuo, T., Raynaud, X.}:
 {A multi-symplectic numerical integrator for the two-component Camassa-Holm equation}.
 {J. Nonlinear Math. Phys.} {\textbf{21}}, 442--453 (2014)

 \bibitem{CR2012}
 {Cohen. D., Raynaud, X.}:
 {Convergent numerical schemes for the compressible hyperelastic rod wave equation}.
 {Numer. Math.} {\textbf{122}}, 1--59 (2012)

 \bibitem{COR2008}
 {Cohen, D., Owren, B., Raynaud, X.}:
 {Multi-symplectic integration of the Camassa-Holm equation}.
 {J. Comput. Phys.} {\textbf{227}}, 5492--5512 (2008)

 \bibitem{CI2008}
 {Constantin, A., Ivanov, R.}:
 {On an integrable two-component Camassa-Holm shallow water system}.
 {Phys. Lett. A.} {\textbf{372}}, 7129--7132 (2008)

 \bibitem{Dan2003}
 {Danchin, R.}:
 {A note on well-posedness for Camassa-Holm equation}.
 {J. Differential Equations} {\textbf{192}}, 429--444 (2003)

 \bibitem{DBX2008}
 {David, C., Brynjulf, O., Xavier, R.}:
 {Multi-symplectic integration of the Camassa-Holm equation}.
 {J. Comput. Phys.} {\textbf{227}}, 5492--5512 (2008)

 \bibitem{FY2019}
 {Fan, E., Yuen, M.}:
 {Peakon weak solutions for the rotation-two-component Camassa-Holm system}.
 {Appl. Math. Lett.} {\textbf{97}}, 53--59 (2019)

 \bibitem{FGL2016}
 {Fan, L., Gao, H., Liu, Y.}:
 {On the rotation-two-component Camassa-Holm system modelling the equatorial water waves}.
 {Adv. Math.} {\textbf{291}}, 59--89 (2016)

 \bibitem{FQ1991}
 Feng, K., Qin, M.:
 Hamiltonian algorithms for Hamiltonian dynamical systems.
 Progress in Natural Science, \textbf{1}(2), 105--116 (1991)

 \bibitem{FF1981}
 {Fuchssteinert, B., Fokas, A.}:
 {Symplectic structures, their backlund transformations and hereditary symmetries}.
 {Phys. D Nonlinear Phenomena} \textbf{4}, 47--66 (1981)

 \bibitem{GG2021}
 {Galtung, S.T., Grunert, K.}:
 {A numerical study of variational discretizations of the Camassa--Holm equation}.
 {BIT Numer. Math.} \textbf{61}, 1271--1309 (2021)

 \bibitem{GX2011}
 {Geng, X., Xue, B.}:
 {A three-component generalization of Camassa-Holm equation with $N$-peakon solutions}.
 {Adv. Math.} {\textbf{226}}, 827--839 (2011)

 \bibitem{GY2011}
 {Guan, C., Yin, Z.}:
 {Global weak solutions for a two-component Camassa-Holm shallow water system}.
 {J. Funct. Anal.} {\textbf{260}}, 1132--1154 (2011)

 \bibitem{GL2011}
 {Gui, G., Liu, Y.}:
 {On the Cauchy problem for the two-component Camassa-Holm system}.
 {Math. Z.} {\textbf{268}}, 45--66 (2011)

 \bibitem{Guo1974}
 {Guo, B.}:
 {A class of difference scheme for two-dimensional vorticity equations with viscous fluids}.
 {Acta Mathematics Sinica} \textbf{17}, 242--258 (1974)

 \bibitem{HGG2013}
 {Han, Y., Guo, F., Gao, H.}:
 {On solitary waves and wave-breaking phenomena for a generalized two-component integrable Dullin-Gottwald-Holm system}.
 {J. Nonlinear Sci.} {\textbf{23}}, 617--656 (2013)

 \bibitem{HPR2022}
 {Hesthaven, J.S., Pagliantini, C., Rozza, G.}:
 {Reduced basis methods for time-dependent problems}.
 {Acta. Numer.} \textbf{31}, 265--345 (2022)

 \bibitem{Henry2009}
 {Henry, D.}:
 {Infinite propagation speed for a two component Camassa-Holm equation}.
 {Discrete Contin. Dyn. Syst. Ser. B} {\textbf{12}}, 597--606 (2009)

 \bibitem{HR2006b}
 {Holden, H., Raynaud, X.}:
 {A convergent numerical scheme for the Camassa-Holm equation based on multipeakons}.
 {Discrete Cont. Dyn. A} {\textbf{14}}, 505--523 (2006)

  \bibitem{HR2006a}
 {Holden, H., Raynaud, X.}:
 {Convergence of a finite difference scheme for the Camassa-Holm equation}.
 {SIAM J. Numer. Anal.} {\textbf{44}}, 1655--1680 (2006)

  \bibitem{HX2008}
 {Holden, H., Raynaud, X.}:
 {Periodic conservative solutions of the Camassa-Holm equation}.
 {Ann. Inst. Fourier (Grenoble)} {\textbf{58}}, 945--988 (2008)

 \bibitem{HI2010}
 {Holm, D., Ivanov, R.}:
 {Multi-component generalizations of the CH equation: geometrical aspects, peakons and numerical examples}.
 {J. Phys. A} {\textbf{43}}, 492001 (2010)

 \bibitem{JGCW2020}
 {Jiang, C., Gong, Y., Cai, W., Wang, Y.}:
 {A linearly implicit structure-preserving scheme for the Camassa-Holm equation based on multiple scalar auxiliary variables approach}.
 {J. Sci. Comput.} {\textbf{83}}, Article 20 (2020)

 \bibitem{JWG2021}
 {Jiang, C., Wang, Y., Gong, Y.}:
 {Arbitrarily high-order energy-preserving schemes for the Camassa--Holm equation}.
 {Appl. Numer. Math.} \textbf{151}, 85--97 (2020)

 \bibitem{KR2006}
 {Kalisch, H., Raynaud, X.}:
 {Convergence of a spectral projection of the Camassa-Holm equation}.
 {Numer. Meth. Part. Differ. Equ.} {\textbf{22}}, 1197--1215 (2006)

 \bibitem{KLQ2021}
 {Kang, J., Liu, X., Qu, C.}:
 {On an integrable multi-component Camassa-Holm system arising from M\"{o}bius geometry}.
 {P. Roy. Soc. A Math. Phy.} {\textbf{477}} (2021) https://doi.org/10.1098/rspa.2021.0164

 \bibitem{Guo1981}
 {Kuo, P., Sanz-Serna, J.M.}:
 Convergence of methods for the numerical solution of the Korteweg-de Vries equation.
 {IMA J. Numer. Anal.} \textbf{1}, 215--221 (1981)

 \bibitem{LV1995}
 {Li, S., Vu-Quoc, L.}:
 {Finite difference calculus invariant structure of a class of algorithms for the nonlinear Klein-Gordon equation}.
 {SIAM J. Numer. Anal.} {\textbf{32}}, 1839--1875 (1995)

 \bibitem{LQZS2017}
 {Li, X., Qian, X., Zhang, B-Y., Song, S.}:
 {A multi-symplectic compact method for the two-component Camassa-Holm equation with singular solutions}.
 {Chin. Phys. Lett.} {\textbf{34}}, 090202 (2017)

 \bibitem{LLP2014}
 {Li, N., Liu, Q., Popowicz, Z.}:
 {A four-component Camassa-Holm type hierarchy}.
 {J. Geom. Phys.} {\textbf{85}}, 29--39 (2014)

 \bibitem{LP2016}
 {Liu, H., Pendleton, T.}:
 {On invariant-preserving finite difference schemes for the Camassa-Holm equation and the two-component Camassa-Holm system}.
 {Commun. Comput. Phys.} {\textbf{19}}, 1015--1041 (2016)

 \bibitem{LX2016}
 {Liu, H., Xing, Y.}:
 {An invariant preserving discontinuous Galerkin method for the Camassa-Holm equation}.
 {SIAM J. Sci. Comput.} {\textbf{38}}, A1919--A1934 (2016)

 \bibitem{LPZ2019}
 {Liu, J., Pucci, P., Zhang, Q.}:
 {Wave breaking analysis for the periodic rotation-two-component Camassa-Holm system}.
 {Nonlinear Anal.} {\textbf{187}}, 214--228 (2019)

 \bibitem{Moon2017}
 {Moon, B.}:
 {On the wave-breaking phenomena and global existence for the periodic rotation-two-component Camassa-Holm system}.
 {J. Math. Anal. Appl.} {\textbf{451}}, 84--101 (2017)

 \bibitem{OR1996}
 {Olver, P., Rosenau, P.}:
 {Tri-Hamiltonian duality between solitons and solitary-wave solutions having compact support}.
 {Phys. Rev. E.} {\textbf{53}}, 1900--1906 (1996)

 \bibitem{WX2015}
 {Wang, Z. Xiang, X.}:
 {Generalized Laguerre approximations and spectral method for the Camassa--Holm equation}.
 {IMA J. Numer. Anal.} \textbf{35}, 1456--1482 (2015)

 \bibitem{XS2008}
 {Xu, Y., Shu, C-W.}:
 {A local Discontinuous Galerkin method for the Camassa-Holm equation}.
 {SIAM J. Numer. Anal.} {\textbf{46}}, 1998--2021 (2008)

 \bibitem{YLQ2020}
 {Yang, M., Li, Y., Qiao, Z.}:
 {Persistence properties and wave-breaking criteria for a generalized two-component rotational b-family system}.
 {Discret. Contin. Dyn. Syst. A} {\textbf{40}}, 2475--2493 (2020)

 \bibitem{YFS2018}
 {Yu, C-H., Feng, B-F., Sheu, T.W.H.}:
 {Numerical solutions to a two-component Camassa-Holm equation}.
 {J. Comput. Appl. Math.} {\textbf{336}}, 317--337 (2018)

 \bibitem{YYW2021}
 {Yu, X., Ye, X., Wang, Z.}:
 {A fast solver of Legendre-Laguerre spectral element method for the Camassa--Holm equation}.
 {Numer. Algor.} \textbf{88}, 1--23 (2021)

 \bibitem{ZL2018}
 {Zhang, L., Liu, B.}:
 {Well-posedness, blow-up criteria and gevrey regularity for a rotation-two-component Camassa-Holm system}.
 {Discret. Contin. Dyn. Syst. A} {\textbf{38}}, 2655--2685 (2018)

  \bibitem{ZL2020}
 {Zhang, Q., Liu, L.}:
 {Convergence and stability in maximum norms of linearized fourth-order conservative compact scheme
 for Benjamin-Bona-Mahony-Burgers' equation}.
 {J. Sci. Comput.} {\textbf{87}}, Article 59 (2021)

 \bibitem{ZLZ2022}
{Zhang, Q., Liu, L., Zhang, Z.}:
{Linearly implicit invariant-preserving decoupled difference
scheme for the rotation-two-component Camassa-Holm system}.
{SIAM J. Sci. Comput.} {\textbf{44}}, A2226--A2252 (2022)

 \bibitem{Zhang2017}
 {Zhang, Y.}:
 {Wave breaking and global existence for the periodic rotation-Camassa-Holm system}.
 {Discret. Contin. Dyn. Syst. A} {\textbf{37}}, 2243--2257 (2017)

 \bibitem{ZW2021}
 {Zhao, K., Wen, Z.}:
 {Effect of the Coriolis force on bounded traveling waves of the
 rotation-two-component Camassa-Holm system}.
  {Appl. Anal.} (2021) https://doi.org/10.1080/00036811.2021.1965587

 \bibitem{ZST2011}
 {Zhu, H., Song, S., Tang, Y.}:
 {Multi-symplectic wavelet collocation method for the nonlinear Schr\"{o}dinger equation and the Camassa-Holm equation}.
 {Comput. Phys. Commun.} {\textbf{182}}, 616--627 (2011)
\end{thebibliography}


\end{document}